\documentclass[a4paper,10pt,leqno]{amsart}
\title[Twisting $L^2$-invariants with finite-dimensional representations]
{Twisting $L^2$-invariants with finite-dimensional representations}

\author{Wolfgang L\"uck}
        \address{Mathematisches Institut der Universit\"at Bonn\\
                Endenicher Allee 60\\
                53115 Bonn, Germany}
         \email{wolfgang.lueck@him.uni-bonn.de}
          \urladdr{http://www.him.uni-bonn.de/lueck}
         \date{March, 2017}

\usepackage{hyperref}
\usepackage{color}
\usepackage{pdfsync}
\usepackage{calc}
\usepackage{enumerate,amssymb}
\usepackage[arrow,curve,matrix,tips,2cell]{xy}
  \SelectTips{eu}{10} \UseTips
  \UseAllTwocells

\DeclareMathAlphabet\EuR{U}{eur}{m}{n}
\SetMathAlphabet\EuR{bold}{U}{eur}{b}{n}

\makeindex             % used for the subject index
\theoremstyle{plain}
\newtheorem{theorem}{Theorem}[section]
\newtheorem{lemma}[theorem]{Lemma}
\newtheorem{proposition}[theorem]{Proposition}

\newtheorem{problem}[theorem]{Problem}
\newtheorem{question}[theorem]{Question}

\theoremstyle{definition}
\newtheorem{assumption}[theorem]{Assumption}
\newtheorem{definition}[theorem]{Definition}
\newtheorem{example}[theorem]{Example}

\newtheorem{remark}[theorem]{Remark}
\newtheorem{notation}[theorem]{Notation}

{\catcode`@=11\global\let\c@equation=\c@theorem}

\newcommand{\comsquare}[8]                   % Produces a commutative square
{\begin{CD}
#1 @>#2>> #3\\
@V{#4}VV @V{#5}VV\\
#6 @>#7>> #8
\end{CD}
}

\newcommand{\xycomsquare}[8]                   % kommutatives Quadrat (xy-Version)
{\xymatrix
{#1 \ar[r]^{#2} \ar[d]^{#4} &
#3 \ar[d]^{#5}  \\
#6\ar[r]^{#7} &
#8
}
}

\newcommand{\xycomsquareminus}[8]                      % kommutatives Quadrat (xy-Version)
{\xymatrix{#1 \ar[r]^-{#2} \ar[d]^-{#4} &
#3 \ar[d]^-{#5}  \\
#6\ar[r]^-{#7} &
#8
}
}

\newcommand{\calfin}{\mathcal{FIN}}

\newcommand{\calvcyc}{\mathcal{VCY}}
\newcommand{\caltr}{\{ \! 1 \! \}}

\newcommand{\caln}{{\mathcal N}}

\newcommand{\calu}{\mathcal{U}}

\newcommand{\IC}{{\mathbb C}}

\newcommand{\IN}{{\mathbb N}}

\newcommand{\IQ}{{\mathbb Q}}
\newcommand{\IR}{{\mathbb R}}

\newcommand{\IZ}{{\mathbb Z}}

\newcommand{\bfK}{{\mathbf K}}

\newcommand{\curs}{\EuR}

\newcommand{\Or}{\curs{Or}}

\newcommand{\asmb}{\operatorname{asmb}}
\newcommand{\aut}{\operatorname{aut}}

\newcommand{\ch}{\operatorname{ch}}

\newcommand{\colim}{\operatorname{colim}}

\newcommand{\cone}{\operatorname{cone}}

\newcommand{\ev}{\operatorname{ev}}

\newcommand{\id}{\operatorname{id}}

\newcommand{\im}{\operatorname{im}}

\newcommand{\lead}{\operatorname{lead}}

\newcommand{\orb}{\operatorname{orb}}
\newcommand{\pr}{\operatorname{pr}}
\newcommand{\Real}{\operatorname{Real}}
\newcommand{\Rep}{\operatorname{Rep}}

\newcommand{\rk}{\operatorname{rk}}

\newcommand{\supp}{\operatorname{supp}}

\newcommand{\specrad}{\operatorname{srad}}

\newcommand{\tors}{\operatorname{tors}}
\newcommand{\tr}{\operatorname{tr}}
\newcommand{\trans}{\operatorname{trans}}

\newcommand{\vol}{\operatorname{vol}}

\newcommand{\Wh}{\operatorname{Wh}}

\newcommand{\EGF}[2]{E_{#2}(#1)}                   %classifying space of a family
\newcommand{\higherlim}[3]{{\setbox1=\hbox{\rm lim}
        \setbox2=\hbox to \wd1{\leftarrowfill} \ht2=0pt \dp2=-1pt
        \mathop{\vtop{\baselineskip=5pt\box1\box2}}
        _{#1}}^{#2}#3}

\newcommand{\version}[1]                       %marks the date of last editing and compilation
{\begin{center} last edited on #1\\
last compiled on \today\\
name of texfile: \jobname
\end{center}
}

\newcounter{commentcounter}

\newcommand{\FBMOD}[1]{#1\text{-}\operatorname{FBMOD}}
\newcommand{\FBCC}[1]{#1\text{-}\operatorname{FBCC}}
\newcommand{\FGHIL}[1]{#1\text{-}\operatorname{FGHIL}}
\newcommand{\FGHCC}[1]{#1\text{-}\operatorname{FGHCC}}

\begin{document}

\typeout{---------------------------- twist.tex ----------------------------}

%%%%%%%%%%%%%%%%%%%%%%%%%%%%%%%%%%%%%%%%%%%%%%%%%%%%%%%%%%%%%%%%%%%%%%%%%%%%%%%%%
%%%%%%%%%%%%%%%%%%%%%%%%%%%%%%%%%%% Abstract %%%%%%%%%%%%%%%%%%%%%%%%%%%%%%%%%%%%%%%
%%%%%%%%%%%%%%%%%%%%%%%%%%%%%%%%%%%%%%%%%%%%%%%%%%%%%%%%%%%%%%%%%%%%%%%%%%%%%%%%%

\typeout{------------------------------------ Abstract
  ----------------------------------------}

\begin{abstract}

  We investigate how one can twist $L^2$-invariants such as $L^2$-Betti numbers and
  $L^2$-torsion with finite-dimensional representations. As a special case we assign to
  the universal covering $\widetilde{X}$ of a finite connected $CW$-complex $X$ together
  with an element $\phi \in H^1(X;\IR)$ a $\phi$-twisted $L^2$-torsion function $\IR^{>0}
  \to \IR$, provided that the fundamental group of $X$ is residually finite and
  $\widetilde{X}$ is $L^2$-acyclic.
\end{abstract}

\maketitle

%%%%%%%%%%%%%%%%%%%%%%%%%%%%%%%%%%%%%%%%%%%%%%%%%%%%%%%%%%%%%%%%%%%%%%%%%%%%%%%%%
%%%%%%%%%%%%%%%%%%%%%%%%%%%%%%%%%% Introduction %%%%%%%%%%%%%%%%%%%%%%%%%%%%%%%%%%%%%
%%%%%%%%%%%%%%%%%%%%%%%%%%%%%%%%%%%%%%%%%%%%%%%%%%%%%%%%%%%%%%%%%%%%%%%%%%%%%%%%%

\typeout{------------------------------- Section 0: Introduction
  --------------------------------}

\setcounter{section}{-1}
\section{Introduction}

An interesting question is how sensitive $L^2$-invariants are under twisting with
finite-dimensional representations.  This means the following:

Let $C_*$ be a bounded $\IC G$-chain complex such that each chain module is a finitely
generated free $\IC G$-module with (an equivalence) class of $\IC G$-bases.  Then one can
pass by a Hilbert space completion to the associated bounded chain complex of finitely
generated Hilbert $\caln(G)$-modules which we denote by $\Lambda(C_*)$.

Let $V$ be a finite-dimensional (complex) $G$-representation whose underlying complex
vector space comes with an (equivalence class of) $\IC$-bases.  (No compatibility
conditions linking the equivalence class of $\IC$-bases to the $G$-action on $V$ are
required.) We can twist $C_*$ with $V$ and the result is another bounded $\IC G$-chain
complex such that each chain module is a finitely generated free $\IC G$-module with an
(equivalence class of) $\IC G$-bases which we will denote by $\eta(C_*)$. We will explain
$\Lambda(C_*)$ and $\eta(C_*)$ in detail in
Section~\ref{sec:Twisting_CG-modules_with_finite-dimensional_representations}.

The main problem is to understand the various $L^2$-invariants of the twisted finite
Hilbert $\caln(G)$-chain complex $\Lambda \circ \eta(C_*)$ in terms of the various
$L^2$-invariants of the untwisted finite Hilbert $\caln(G)$-chain complex $\Lambda(C_*)$.
A finite Hilbert $\caln(G)$-chain complex $C_*$ is called \emph{$L^2$-acyclic} if all its
$L^2$-Betti numbers vanish and is called \emph{$\det$-$L^2$-acyclic} if it is
$L^2$-acyclic and of determinant class.  The typical questions are:

\begin{question}[$L^2$-Betti numbers and twisting]
  \label{que:L2-Betti_number_and_twisting}
  Under which conditions on $G$ and $V$ do we get for every $n \in \IZ$
  \[
  b_n^{(2)}\bigl(\Lambda \circ \eta(C_*);\caln(G)\bigr) = \dim_{\IC}(V) \cdot
  b_n^{(2)}\bigl(\Lambda(C_*);\caln(G)\bigr)?
  \]
\end{question}

\begin{question}[Novikov-Shubin invariants and twisting]
  \label{que:Novikov-Shubin-invariants_and_twisting}
  Under which conditions on $G$ and $V$ do we get for every $n \in \IZ$
  \[
  \alpha_n\bigl(\Lambda \circ \eta(C_*);\caln(G)\bigr) =
  \alpha_n\bigl(\Lambda(C_*);\caln(G)\bigr)?
  \]
\end{question}

\begin{question}[Determinant class and twisting]
  \label{que:determinant_class}
  Is $\Lambda \circ \eta(C_*)$ of determinant class if $\Lambda(C_*)$ is of determinant
  class?  Has $\Lambda \circ \eta(C_*)$ the property $\det$-$L^2$-acyclic, if
  $\Lambda(C_*)$ does?
\end{question}

An affirmative answer will be given in the main technical result of this paper
Theorem~\ref{the:Determinant_class_and_twisting} in the special case, where the underlying
group $G$ is residually finite and the representation is the pullback of a representation
over $\IZ^d$ by a group homomorphism $\phi \colon G \to \IZ^d$. As an illustration, let us
state the contents of the main
Theorem~\ref{the:Determinant_class_and_twisting}
in the special case, where we start with a surjective group homomorphism $\phi \colon G \to \IZ$
with a residual finite group $G$ as source and a non-trivial  $(r,s)$-matrix A with entries in $\IZ G$. Denote
by $r_A^{(2)} \colon L^2(G)^r \to L^2(G)^s$ the induced bounded $G$-operator
given by right multiplication with $A$. Fix $t \in \IR^{> 0}$. We obtain a ring homomorphism
$\widehat{\phi}_t \colon \IC G \to \IC G$ by sending $\sum_{g \in G} \lambda_g \cdot g$ to
$\sum_{g \in G} t^{\phi(g)} \cdot \lambda_g \cdot g$.  Let $A(t)$ be the
image of $A$ under $\widehat{\phi}$. Then we get for the new bounded $G$-operator
$r_{A(t)}^{(2)} \colon L^2(G)^r \to L^2(G)^s$ for the Murray-von Neumann-dimension
of the closure of its image
\[
\dim_{\caln(G)}\bigl(\overline{\im(r_{A(t)}^{(2)})}\bigr) = \dim_{\caln(G)}\bigl(\overline{\im(r_{A}^{(2)})}\bigr),
\]
and its Fuglede -Kadison determinant can be estimated by
\[
{\det}_{\caln(G)}\bigl(r_{A(t)}^{(2)}\bigr) \ge \min \{t^{\phi(g)} \mid g \in \supp(A)\}^{\dim_{\caln(G)}\bigl(\overline{\im(r_{A}^{(2)}})\bigr)},
\]
where $\supp(A)$ is the set of group elements which have a non-trivial coefficient in one
of the entries in $A$. Notice that this gives already for $t = 1$ the non-trivial estimate
${\det}_{\caln(G)}\bigl(r_{A}^{(2)}\bigr) \ge 1$.

The basic strategy of the proof of Theorem~\ref{the:Determinant_class_and_twisting} is
first to handle the case, where $\phi \colon G \to \IZ^d$ is an isomorphism. This is done
by using basic input from the theory of the  Mahler measure of a complex polynomial which can be
identified with the Fuglede-Kadison determinant, see
Subsection~\ref{subsec:Determinants_over_Zd}.  In the next step one uses transfer
arguments to handle the case, where $\phi$ is surjective and has finite kernel.  The idea
is to choose an injective map $j \colon \IZ^d \to G$ and use the fact that the restriction of
Fulgede-Kadison determinant with $j^*$ is given by the Fuglede-Kadison determinant over
$G$ raised to the $[G:\im(j)]$th power, see
Subsection~\ref{subsec:The_special_case_of_finite_ker(phi)}.  Finally, in
Subsection~\ref{subsec:Proof_of_Theorem_ref(the:Determinant_class_and_twisting)_in_general},
we use approximation techniques to handle the general case.

Theorem~\ref{the:Determinant_class_and_twisting} allows us to define a reduced
$\phi$-twisted $L^2$-torsion function, if $G$ is finitely generated and residually finite
and $\phi \colon G \to \IR$ is a group homomorphism in
Section~\ref{subsec:The_twisted_L2_torsion_function}, where also the basic properties are
proved.  In particular the twisted $L^2$-torsion function of the universal covering of a
connected finite $CW$-complex with residually finite fundamental group for an element
$\phi \in H^1(X;\IR)$ seems to be a very intriguing notion whose main properties are
collected in
Theorem~\ref{the:Basic_properties_of_the_reduced_L2-torsion_function_for_universal_coverings}. We
recommend to the reader to browse through
Subsection~\ref{subsec:L2-torsion_function_for_universal_coverings} in order to get a
first impression about the potential of this notion.  The existence of the $\phi$-twisted
$L^2$-torsion function was independently proved by Liu~\cite{Liu(2015)}.  Some of the
results of~\cite{Liu(2015)} will be discussed in
Sections~\ref{sec:Continuity_of_the_Fuglede-Kadison_determinant}
and~\ref{sec:Some_open_problems}.

We also mention that in dimension $3$ all necessary assumptions are satisfied for prime
compact $3$-manifolds whose fundamental group is infinite and whose boundary is empty or
toroidal, see Subsection~\ref{subsec:The_L2-torsion_function_for_3-manifolds}.  In
particular the estimates appearing in Theorem~\ref{the:Determinant_class_and_twisting}
will be exploited in papers~\cite{Friedl-Lueck(2015l2+Thurston),
Friedl-Lueck(2016l2+poly),Friedl-Lueck(2016l2_universal)}, where we will apply the reduced $\phi$-twisted
$L^2$-torsion function to compact connected orientable irreducible $3$-manifolds with
infinite fundamental group and empty or toroidal boundary, and relate its degree to the
Thurston norm $x_M(\phi)$ and to the degree of higher order Alexander polynomials.
Similar results have been proved independently by Liu~\cite{Liu(2015)}.

The main difficulties in twisting with finite-dimensional representations are related to
the problem whether the Fuglede-Kadison determinant is continuous which will be discussed
in Section~\ref{sec:Continuity_of_the_Fuglede-Kadison_determinant}. Some further open
problems are presented in Section~\ref{sec:Some_open_problems}.

%%%%%%%%%%%%%%%%%%%%%%%%%%%%%%%%%%%%%%%%%%%%%%%%%%%%%%%%%%%%%%%%%%%%%%%%%%%%%%%%%

\subsection*{Acknowledgments.}
This paper is financially supported by the Leibniz-Preis  of the author  granted by the
Deutsche Forschungsgemeinschaft and the ERC-Advanced Grant  \emph{KL2MG-interactions}
of the author granted by the European Research Council. The author thanks Stefan Friedl, whose visions about
$3$-manifolds have been the main motivation for this paper, Florian
Funke, Holger Kammeyer, and Yi Liu for many fruitful conversations and helpful comments,
and the referee for his very detailed and very useful report.

%%%%%%%%%%%%%%%%%%%%%%%%%%%%%%%%%%%%%%%%%%%%%%%%%%%%%%%%%%%%%%%%%%%%%%%%%%%%%%%%%

\tableofcontents

%%%%%%%%%%%%%%%%%%%%%%%%%%%%%%%%%%%%%%%%%%%%%%%%%%%%%%%%%%%%%%%%%%%%%%%%%%%%%%%%%
%%%%%%%%%% Section 1: Twisting $\IC G$-modules with finite-dimensional
%%%%%%%%%% representations %%%%%%%%%%%%%%%%
%%%%%%%%%%%%%%%%%%%%%%%%%%%%%%%%%%%%%%%%%%%%%%%%%%%%%%%%%%%%%%%%%%%%%%%%%%%%%%%%%

\typeout{------- Section 1: Twisting $\IC G$-modules with finite-dimensional
  representations ------------}

\section{Twisting $\IC G$-modules with finite-dimensional representations}
\label{sec:Twisting_CG-modules_with_finite-dimensional_representations}

In this section we explain how to twist with a finite-dimensional representation.

Let $M$ and $V$ be $\IC G$-modules. Denote by $(M \otimes_{\IC} V)_1$ the $\IC G$-module
whose underlying vector space is $M \otimes_{\IC} V$ and on which $g \in G$ acts only on
the first factor, i.e., $g (u \otimes v) = gu \otimes v$.  Denote by $(M\otimes_{\IC}
V)_d$ the $\IC G$-module whose underlying vector space is $M \otimes_{\IC} V$ and on which
$g \in G$ acts diagonally, i.e., $g (u \otimes v) = gu \otimes gv$.  Notice that $(M
\otimes_{\IC} V)_1$ is independent of the $G$-action on $V$ and $\IC G$-isomorphic to the
direct sum of $\dim_{\IC}(V)$ copies of $M$, whereas $(\IC G \otimes_{\IC} M)_d$ does
depend on the $G$-action on $M$.  One easily checks:

\begin{lemma} \label{lem:diagonal_versus_first_coordinate} We obtain a $\IC G$-isomorphism
  \[
  \xi \colon (\IC G \otimes_{\IC} V)_1 \xrightarrow{\cong} (\IC G \otimes_{\IC} V)_d,
  \quad g \otimes v \mapsto g \otimes gv,
  \]
  whose inverse sends $g \otimes v$ to $g \otimes g^{-1}v$.
\end{lemma}

Let $R$ be a ring and let $M$ be a free $R$-module. We call two $R$-bases $B, B' \subseteq
M$ equivalent if there is a bijection $\sigma \colon B \to B'$ and for each $b \in B$ an
element $\epsilon(b) \in \{\pm 1\}$ such that $\sigma(b) = \epsilon(b) \cdot b$ holds for
all $b \in B$. We call a $R$-module $M$ \emph{based} if it is free and we have chosen an
equivalence class $[B_M]$ of $R$-bases in the sense above. Denote by $\FBMOD{R}$ the
category whose objects are based finitely generated free $R$-modules and whose morphisms
are $R$-linear maps. Obviously $\FBMOD{R}$ inherits the structure of an additive category
by the direct sum of $R$-modules and the union of bases and the obvious abelian group
structure on the set of morphisms. If $R$ is commutative, we actually get the structure of
an additive $R$-category on $\FBMOD{R}$.

Let $V$ be a finite-dimensional (left) $G$-representation.  If the underlying vector space
comes with an equivalence class of $\IC$-bases $[B_V]$, we call $V$ a \emph{based
  finite-dimensional $G$-representation}.  (No compatibility conditions linking the
equivalence class of $\IC$-bases to the $G$-action on $V$ are required.)  We want to
define a functor of $\IC$-additive categories
\begin{eqnarray}
  & \eta_{V,[B_V]}^G   \colon \FBMOD{\IC G} \to \FBMOD{\IC G} &
  \label{eta_V,[B]_modules}
\end{eqnarray}
by sending an object $(M,[B_M])$ to the object $(M \otimes_{\IC} V)_d$ with the
equivalence class of $\IC G$-bases represented by the $\IC G$-basis $\{b \otimes v \mid b
\in B_M, v \in B_V\}$.  The latter is indeed a finite $\IC G$-basis by
Lemma~\ref{lem:diagonal_versus_first_coordinate}.  A morphism $f \colon (M,[B_M]) \to
(N,[B_N])$ is sent to the morphism whose underlying map of $\IC G$-modules is given by $f
\otimes_{\IC} \id_V \colon (M \otimes_{\IC} V)_d \to (N \otimes_{\IC} V)_d$.

This extends in the obvious way to chain complexes yielding a functor
\begin{eqnarray}
  & \eta_{V,[B_V]}^G   \colon \FBCC{\IC G} \to \FBCC{\IC G} &
  \label{eta_V,[B]_chain_complexes}
\end{eqnarray}
where $\FBCC{\IC G}$ is the $\IC$-additive category of bounded chain complexes over
$\FBMOD{\IC G}$.

Let $\FGHIL{\caln(G)}$ be the additive $\IC$-category of finitely generated Hilbert
$\caln(G)$-modules.  Objects are Hilbert spaces $V$ together with $G$-actions by linear
isometries such that there exists an isometric $G$-embedding of $V$ into $L^2(G)^n$ for
some natural number $n$. Morphisms are bounded $G$-equivariant operators.  Next we define
a functor of $\IC$-additive categories
\begin{eqnarray*}
  & \Lambda^G \colon \FBMOD{\IC G}  \to \FGHIL{\caln(G)}. & 
  \label{Lambda_FBMOD_to_FGHIL}
\end{eqnarray*}
It sends
an object $(M,[B_M])$ to $L^2(G) \otimes_{\IC G} M$, where we use on 
$L^2(G) \otimes_{\IC G} M$ the Hilbert space structure for which after a choice of 
representative $B_M$ for the equivalence class of $\IC G$-bases $[B_M]$ of $M$ the map 
\[
\bigoplus_{b \in B_M} L^2(G) \xrightarrow{\cong} L^2(G) \otimes_{\IC G} M,
\quad (x_b)_{b \in B_M} \mapsto \sum_{b \in B_M} x_b \otimes b
\]
is an isometry of Hilbert spaces. 
One easily checks that the Hilbert space structure depends only on the equivalence class of $B_M$.
Obviously $L^2(G) \otimes_{\IC G} M$ is a finitely generated Hilbert $\caln(G)$-module.
A morphisms $f \colon (M,[B_M]) \to (N,[B_N])$ is sent to the operator
$\id_{L^2(G)} \otimes_{\IC} f \colon L^2(G) \otimes_{\IC G} M \to L^2(G) \otimes_{\IC G} M$.
One easily checks that it is indeed a bounded $G$-equivariant operator.

This extends in the obvious way to chain complexes yielding a functor
\begin{eqnarray*}
& \Lambda^G \colon \FBCC{\IC G}  \to \FGHCC{\caln(G)} & 
\label{Lambda_FBCC_to_FGHIL}
\end{eqnarray*}
where  $\FGHCC{\caln(G) }$ is the $\IC$-additive category of bounded chain complexes 
over $\FGHIL{\caln(G)}$.

\begin{notation}[Omitting $G$, ${[B_M]}$, ${[B_V]}$ and $(V,{[B_V]})$ from the notation]
  \label{not:Omitting_B_M_B_V_and_(V,B_V)_from_the_notation}
  Often we omit $G$, $[B_M]$ and $[B_V]$ in the notation and write briefly $M$ and $V$ for
  $(M,[B_M])$ and $V$ for $(V,[B_V])$. We also write $\eta$, $\eta_V$, or  $\eta_{V,[B_V]}$ 
  instead of $\eta_{V,[B_V]}^G$, and  $\Lambda$ instead of $\Lambda^G$.
\end{notation}

\begin{example}[Twisting with $\phi^*\IC_t$]
 \label{exa:phi-twisting_in_terms_of_matrices}
The following  example will be relevant for us.  Let $\phi \colon G \to \IR$ be a group homomorphism.
For $t \in \IR^{> 0}$,
let $\IC_t$ be the based $1$-dimensional
$\IR$-representation given by $\IC$ with the equivalence class of the standard $\IC$-basis
for which $r \in \IR$  acts by multiplication with $t^r$ on $\IC$. Denote by $\phi^*\IC_t$
the based $1$-dimensional $G$-representation obtained from $\IC_t$ by restriction with $\phi$.

The functor $\eta_{\phi^*\IC_t}$ of~\eqref{eta_V,[B]_modules} has on the level of matrices
  the following description.  For an $(r,s)$-matrix $A \in M_{r,s}(\IZ G)$, we have the
  morphism $r_A \colon \IZ G^r \to \IZ G ^s$ in $\FBMOD{\IC G}$ given by right
  multiplication with $A$. It is sent under $\eta_{\phi^*\IC_t}$ to the morphism
  $r_{\widehat{\phi}_*(A)} \colon \IZ G^r \to \IZ G ^s$ in $\FBMOD{\IC G}$, where the
  $(r,s)$-matrix $\widehat{\phi}_*(A) \in M_{m,n}(\IZ G)$ is obtained from $A$ by applying
  elementwise the ring homomorphism
  \[
  \widehat{\phi} \colon \IC G \to \IC G, \quad 
    \sum_{g \in G} \lambda_g \cdot g \mapsto \sum_{g \in G} t^{\phi(g)} \cdot \lambda_g \cdot g.
  \]
\end{example}

%%%%%%%%%%%%%%%%%%%%%%%%%%%%%%%%%%%%%%%%%%%%%%%%%%%%%%%%%%%%%%%%%%%%%%%%%%%%%%%%%
%%%%%%%%%%%%%% Section 2: Brief review of $L^2$-invariants  %%%%%%%%%%%%%%%%%%%%%%%%%%%%%%%%%
%%%%%%%%%%%%%%%%%%%%%%%%%%%%%%%%%%%%%%%%%%%%%%%%%%%%%%%%%%%%%%%%%%%%%%%%%%%%%%%%%

 \typeout{------------------------   Section 2: Brief review of $L^2$-invariants-------------------}

\section{Brief review of $L^2$-invariants}
\label{subsec:Brief_review_of_L2-invariants}

  We will assume that the reader is familiar with the notion of the \emph{$L^2$-Betti  number} 
  \[
   b_n^{(2)}(C_*) \in [0,\infty)
  \]
  of a finite Hilbert $\caln(G)$-chain complex.  A Hilbert $\caln(G)$-chain complex is
  called \emph{$L^2$-acyclic} if all its $L^2$-Betti numbers are trivial.

  We  define  the \emph{$L^2$-torsion} $\rho^{(2)}(C_*) \in \IR$ of a finite Hilbert chain complex
  $C_*$ to be
  \[\rho^{(2)}(C_*) = - \sum_{p \in \IZ} (-1)^p\cdot  \ln(\det(c_p)),
  \]
  provided that it satisfies a certain
  condition, namely,  being of $\emph{determinant class}$.  The last condition ensures that the so called
  \emph{Fuglede-Kadison determinant} $\det(c_p)$, which is a priori a real number satisfying $\det(c_p) \ge 0$,
  is strictly greater than zero and hence $\ln(\det(c_p))$ is a well-defined real number.  
  One can define it also in terms of the \emph{Laplace operator}
  $\Delta_p = c_{p+1} \circ c_{p+1}^* + c_p^* \circ c_p$ by
  \[
  \rho^{(2)}(C_*)  =  - \frac{1}{2}\cdot \sum_{p \in \IZ} (-1)^p \cdot p \cdot \ln(\det(\Delta_p)).
  \]
  For the definition of the Fuglede-Kadison determinant and its main properties we refer
  to~\cite[Section~3.2]{Lueck(2002)}.  The idea behind this concept is that it generalizes
  the classical determinant in the case that $G$ is finite as follows.  Namely, for finite
  $G$ the Fuglede-Kadison determinant $\det(f)$ is the $2|G|$-th root of the classical
  determinant of the endomorphism
  $f^*f|_{\ker(f)^{\perp}} \colon \ker(f)^{\perp} \to \ker(f)^{\perp}$ of the finite
  dimensional vector space $\ker(f)^{\perp}$ coming from the automorphism
  $f^*f \colon V \to V$, or, equivalently, the $2|G|$-th root of the product of the
  non-zero eigenvalues of the self-adjoint endomorphism $f^*f \colon V \to V$.  In
  particular for finite $G$ the condition being of determinant class is always
  satisfied. By definition the Fuglede-Kadison determinant of the zero map
  $0 \colon V \to V$ is $1$.

  If $G = \IZ^d$, then the von Neumann algebra $\caln(\IZ^d)$ can be identified with $L^{\infty}(\IZ^d)$.
  Given an element for $f \in L^{\infty}(T^d)$ we get for the Fuglede-Kadison determinant of the operator
  $r_f^{(2)} \colon L^2(\IZ^d) = L^2(T^d) \to L^2(\IZ^d) = L^2(T^d)$ given by pointwise multiplication with $f$
  \[\det\left(r_f^{(2)}\colon  L^2(T^d) \to L^2(T^d)\right)
  = 
  \exp\left(\int_{T^n} \ln(|f(z)|) \cdot
\chi_{\{u \in S^1\mid f(u) \not= 0\}} \;d\mu \right)
  \]
  where $\chi_A$ denotes the characteristic function of the set $A \subseteq S^1$, the
  integral is understood with respect to the standard Lebesque measure on $T^d$, and we
  use the convention $\exp(-\infty) = 0$, see~\cite[Example~3.13 on page~128]{Lueck(2002)}.

Consider a non-trivial element $p \in \IC[\IZ]$.
We can write
\[
p(z)  = 
 C \cdot z^n \cdot \prod_{k = 1}^l (z - a_k)
\]
for complex numbers $C, a_0, a_1, \ldots , a_l$ and
integers $n,l$ with $C \not= 0$ and $l \ge 0$. Its \emph{Mahler measure}
is defined to be 
\[
M(p) =  
|C| \cdot 
\prod_{1 \le k \le l, |a_k| \ge 1} |a_k|.
\]
If we consider $p$ as an element in $\caln(\IZ) =L^{\infty}(S^1)$,
then $R_p \colon L^2(\IZ) = L^2(S^1) \to L^2(\IZ) = L^2(S^1)$  is of determinant class
and we get from~\cite[Example~3.22 on page~135]{Lueck(2002)}
\[
M(p) = \det(r_p^{(2)}).
\]

  A Hilbert $\caln(G)$-chain
  complex is called \emph{$\det$-$L^2$-acyclic} if it is $L^2$-acyclic and of determinant
  class.  If $f_* \colon C_* \to D_*$ is a chain of finite Hilbert $\caln(G)$-chain
  complexes, then it is weak homology equivalence, i.e., it induces weak isomorphisms on
  the $L^2$-homology groups, if and only if the mapping cone $\cone(f_*)$ is an
  $L^2$-acyclic Hilbert $\caln(G)$-chain complex. We call $f_*$ of $\emph{of determinant class}$ 
  if $\cone(f_*)$ is of determinant class, and in this case we define the
  \emph{$L^2$-torsion} $\tau(f_*) \in \IR$ to be $\rho^{(2)}(\cone(f_*)_*)$.  

  The following trivial facts are sometimes useful. If $C_*$ is a finite Hilbert
  $\caln(G)$-chain complex and $0_*$ denotes the trivial Hilbert $\caln(G)$-chain complex,
  then $C_*$ is of determinant class if and only if $0_* \to C_*$ is of determinant class,
  and in this case we have $\rho^{(2)}(C_*) = \tau^{(2)}(0_* \to C_*)$.  A chain homotopy
equivalence $f_* \colon C_* \to D_*$ of finite Hilbert $\caln(G)$-chain complexes is
always of determinant class.

For more information and precise definitions about these notions we refer for instance
to~\cite[Chapter~1 and~3]{Lueck(2002)}.

%%%%%%%%%%%%%%%%%%%%%%%%%%%%%%%%%%%%%%%%%%%%%%%%%%%%%%%%%%%%%%%%%%%%%%%%%%%%%%%%%
%%%%%%%%%%%%%% Section 3: Changing the twisting representation %%%%%%%%%%%%%%%%%%%%%%%%%%%%%%%%%
%%%%%%%%%%%%%%%%%%%%%%%%%%%%%%%%%%%%%%%%%%%%%%%%%%%%%%%%%%%%%%%%%%%%%%%%%%%%%%%%%

 \typeout{------------------------   Section 3: Changing the twisting representation -------------------}

\section{Changing the twisting representation}
\label{subsec:Changing_the_twisting_representation}

In this section we collect some basic properties about how the $L^2$-invariants transforms
under changing $(V,[B_V])$.

Let $(V,[B_V])$ and $(W,[B_W])$ be two based  finite-dimensional  $G$-representations.
Let $u\colon V \to W$ be a linear map compatible with the $G$-actions, (but not necessarily with 
$[B_V]$ and $[B_W]$). Then we obtain for any object $M$ in $\FBMOD{\IC G}$ a morphism
in $\FBMOD{\IC G}$
\[
T_u \colon \eta_{V,[B_V]}(M) \to \eta_{W,[B_W]}(M)
\]
coming from the $\IC G$-homomorphism $(\id_M \otimes_{\IC} u)_d \colon (M \otimes_{\IC} V)_d
\to (M \otimes_{\IC} W)_d$.  One easily checks that this is natural in $M$, i.e., we
obtain a natural transformation $T_u \colon \eta_{V,[B_V]} \to \eta_{W,[B_W]}$ of functors
of $\IC$-additive categories $\FBMOD{\IC G} \to \FBMOD{\IC G}$.
Obviously this extends to chain complexes yielding a natural transformation of functors
of additive categories $\FBCC{\IC G} \to \FBCC{\IC G}$
\begin{eqnarray}
&
T_u \colon \eta_{V,[B_V]} \to \eta_{W,[B_W]}. 
&
\label{transformation_T}
\end{eqnarray}
We have $T_{u \circ v} = T_u \circ T_v$ if $(U,[B_U])$ is another based
finite-dimensional  $G$-repre\-sen\-ta\-tion and $v\colon U \to V$ is a linear map
compatible with the $G$-actions.

\begin{lemma}
\label{lem:dependency_on_V}
Let $(V,[B_V])$ and $W,[B_W])$ be two based  finite-dimensional  $G$-repre\-sen\-tations.
Let $u\colon V \to W$ be a linear isomorphism compatible with the $G$-actions
(but not necessarily with the equivalence class of $\IC$-basis).

\begin{enumerate}

\item \label{lem:dependency_on_V:formula_for_T_u}
We get for any object $C_*$  in $\FBCC{\IC G}$
\begin{multline*}
\tau^{(2)}\bigl(\Lambda(T_u) \colon \Lambda \circ \eta_{V,[B_V]}(C_*) 
\to \Lambda \circ \eta_{W,[B_W]}(C_*)\bigr)
\\
= \chi_{\IC G}(C_*) \cdot \ln \left(\bigl|{\det}_{\IC}\big(u \colon (V,[B_V]) 
\to (W,[B_W])\bigr)\bigr|\right),
\end{multline*}
where $\chi(C_*) = \sum_{i \in \IZ} (-1)^i \cdot \dim_{\IC G}(C_i)$ is the $\IC G$-Euler
characteristic of $C_*$, and we define 
$\bigl|\det_{\IC}\big(u\colon (V,[B_V]) \to (W,[B_W])\bigr)\bigr|$ 
to be $|\det_{\IC}(D)|$ for the matrix $D$ describing $u$ with
respect to $B_V$ and $B_W$ for any choice of representatives $B_V \in [B_V]$ and $B_W \in [B_W|$;

\item \label{lem:dependency_on_V:independence_of_basis}
Let $f_* \colon C_* \to D_*$ be a chain map in $\FBCC{\IC G}$ which induces
a weak chain homotopy equivalence $\Lambda \circ \eta_{V,[B_V]}(f_*) \colon \Lambda \circ \eta_{V,[B_V]}(C_*)
\to \Lambda \circ \eta_{V,[B_V]}(D_*)$ of determinant class. 

Then
$\Lambda \circ \eta_{W,[B_W]}(f_*) \colon \Lambda \circ \eta_{W,[B_W]}(C_*)
\to \Lambda \circ \eta_{W,[B_W]}(D_*)$ is a weak chain homotopy equivalence 
of determinant class and we get 
\[
\tau^{(2)}\bigl(\Lambda \circ \eta_{V,[B_V]}(f_*)\bigr)
= 
\tau^{(2)}\bigl(\Lambda \circ \eta_{W,[B_W]}(f_*)\bigr).
\]
\end{enumerate}
\end{lemma}
\begin{proof}~\eqref{lem:dependency_on_V:formula_for_T_u}
  One easily checks that $\bigl|\det_{\IC}\big(u \colon (V,[B_V]) \to  (W,[B_W])\bigr)\bigr|$ 
  is indeed independent of the choices of $B_V$ and $B_W$ since for
  two equivalent bases $B_V$ and $B_V'$ of $V$ the determinant of the base change is  $\pm 1$.

  Without loss of generality we can assume that $C_i = 0$ for $i < 0$, otherwise pass to
  an appropriate suspension of $C_*$.  We use induction over the dimension of $C_*$. The
  induction beginning $\dim (C_*) = 0$ is done as follows.

In the sequel we equip $\IC G$ with the equivalence class of the standard basis. Choose an
isomorphism $v \colon \IC G^r \xrightarrow{\cong} C_0$ such that the image of the standard
basis for $\IC G^r$ is a representative of the equivalence class of bases for $C_0$. Then
$\eta_{V,[B_V]}(v) \colon \eta_{V,[B_V]}(\IC G^r) \to \eta_{V,[B_V]}(C_0)$ is an
isomorphism in $\FBMOD{\IC G}$ which respects the equivalence classes of $\IC
G$-bases. Hence it induces an isometric $G$-equivariant bounded operator
\[
\Lambda \circ \eta_{V,[B_V]}(v) \colon  \bigl(\Lambda \circ \eta_{V,[B_V]}(\IC G)\bigr)^r 
= \Lambda \circ \eta_{V,[B_V]}(\IC G^r) \xrightarrow{\cong}  \eta_{V,[B_V]}(C_0),
\]
such that the following diagram of Hilbert $\caln(G)$-modules commutes
\[
\xymatrix@!C=12em{
\bigl(\Lambda \circ \eta_{V,[B_V]}(\IC G)\bigr)^r \ar[r]^{\Lambda \circ \eta_{V,[B_V]}(v)} \ar[d]_{\Lambda(T_u)(\IC G)^r}
&
\Lambda \circ  \eta_{V,[B_V]}(C_0)  \ar[d]^{\Lambda(T_u)(C_0)}
\\
\bigl(\Lambda \circ \eta_{V,[B_W]}(\IC G)\bigr)^r  \ar[r]_{\Lambda \circ \eta_{W,[B_W]}(v)}
&
\Lambda \circ  \eta_{W,[B_W]}(C_0)
}
\]
and has isometric $G$-equivariant bounded operators as horizontal arrows. We conclude 
from~\cite[Theorem~3.14~(1) on page~128 and Lemma~3.15~(7) on page~130]{Lueck(2002)}
\begin{multline*}
{\det}_{\caln(G)}\bigl(\Lambda(T_u) \colon \Lambda \circ \eta_{V,[B_V]}(C_0) 
\to \Lambda \circ \eta_{W,[B_W]}(C_0)\bigr)
\\ = 
\left({\det}_{\caln(G)}\bigl(\Lambda(T_u) \colon \Lambda \circ \eta_{V,[B_V]}(\IC G) 
\to \Lambda \circ \eta_{W,[B_W]}(\IC G)\bigr)\right)^r.
\end{multline*}
We conclude 
from~\cite[Example~3.12 on page~127 and Theorem~3.14~(6) on page~129]{Lueck(2002)}
\begin{multline*}
{\det}_{\caln(G)}\bigl(\Lambda(T_u) \colon \Lambda \circ \eta_{V,[B_V]}(\IC G) 
\to \Lambda \circ \eta_{W,[B_W]}(\IC G)\bigr)
\\
 =  
\bigl|{\det}_{\IC}\big(u \colon (V,[B_V]) \to (W,[B_W])\bigr)\bigr|,
\end{multline*}
since $\Lambda(T_u)$ is obtained by induction with the inclusion $\{1\} \to G$
from the isomorphism of Hilbert $\caln(\{1\})$-modules $u \colon V \to W$,
where the Hilbert space structure on $V$ and $W$ comes from $[B_V]$ and $[B_W]$.
This implies
\begin{eqnarray*}
\lefteqn{\tau^{(2)}\bigl(\Lambda(T_u) \colon \Lambda \circ \eta_{V,[B_V]}(C_*) 
\to \Lambda \circ \eta_{W,[B_W]}(C_*)\bigr)}
\\ 
& = &
\ln\left({\det}_{\caln(G)}\bigl(\Lambda(T_u) \colon \Lambda \circ \eta_{V,[B_V]}(C_0) 
\to \Lambda \circ \eta_{W,[B_W]}(C_0)\bigr)\right)
\\ 
& =  &
r \cdot \ln\left(\bigl|{\det}_{\IC}\big(u \colon (V,[B_V]) \to (W,[B_W])\bigr)\bigr|\right)
\\ 
& =  &
\chi(C_*) \cdot \ln\left(\bigl|{\det}_{\IC}\big(u \colon (V,[B_V]) \to (W,[B_W])\bigr)\bigr|\right).
\end{eqnarray*}
This finishes the  induction beginning. The induction step is done as follows.

Let $C_*|_{n-1}$ be the truncation of $C_*$ to a $(n-1)$-dimensional chain complex. There
is an obvious exact sequence of $\IC G$-chain complexes $0 \to C_*|_{n-1} \to C_* \to
n[C_n] \to 0$, where $n[C_n]$ is the object in $\FBCC{\IC G}$ whose underlying chain
complex is concentrated in dimension $n$ and there given by the object $C_n$ in
$\FBMOD{\IC G}$. One easily checks that we obtain a commutative diagram of finite Hilbert
$\caln(G)$-chain complexes
\[
\xymatrix@C=2em{0 \ar[r]
&
\Lambda \circ \eta_{V,[B_V]}(C_*|_{n-1})  \ar[r] \ar[d]^{\Lambda(T_u(C_*|_{n-1}))}
&
\Lambda \circ \eta_{V,[B_V]}(C_*)  \ar[r] \ar[d]^{\Lambda(T_u(C_*))}
&
\Lambda \circ \eta_{V,[B_V]}(n[C_n])  \ar[r] \ar[d]^{\Lambda(T_u(n[C_n]))}
& 0
\\
0 \ar[r]
&
\Lambda \circ \eta_{W,[B_W]}(C_*|_{n-1})  \ar[r]
&
\Lambda \circ \eta_{W,[B_W]}(C_*)  \ar[r]
&
\Lambda \circ \eta_{W,[B_W]}(n[C_n])  \ar[r]
& 0
}
\]
whose rows are short exact sequences of finite Hilbert $\caln(G)$-chain complexes.
We conclude from~\cite[Theorem~3.25~(1) on page~142]{Lueck(2002)}
\begin{eqnarray*}
\tau^{(2)}\bigl(\Lambda(T_u(C_*))\bigr)
& = & 
\tau^{(2)}\bigl(\Lambda(T_u(C_*|_{n-1})) \bigr) 
+  \tau^{(2)}\bigl(\Lambda(T_u(n[C_n]))\bigr).
\end{eqnarray*}
Since the induction hypothesis applies to $C_*|_{n-1}$ and to $n[C_n]$ and 
$\chi_{\IC G}(C_*) = \chi_{\IC G}(C_*|_{n-1}) + \chi_{\IC G}(n[C_n])$ holds, the induction step and hence
the proof of assertion~\eqref{lem:dependency_on_V:formula_for_T_u} is finished.
\\[2mm]~\eqref{lem:dependency_on_V:independence_of_basis}
We have the following commutative diagram  in $\FGHCC{\caln(G)}$
\[
\xymatrix@!C=12em{
\Lambda \circ \eta_{V,[B_V]}(C_*)  \ar[r]^{\Lambda \circ \eta_{V,[B_V]}(f_*)} \ar[d]_{\Lambda(T_u(C_*))}
&
\Lambda \circ \eta_{V,[B_V]}(D_*) \ar[d]^{\Lambda(T_u(D_*))}
\\
\Lambda \circ \eta_{W,[B_W]}(C_*)  \ar[r]_{\Lambda \circ \eta_{W,[B_W]}(f_*)}
&
\Lambda \circ \eta_{W,[B_W]}(D_*) 
}
\]
We conclude from~\cite[Theorem~3.35~(4) on page~142]{Lueck(2002)}
that $\Lambda \circ \eta_{W,[B_W]}(f_*)$ is a  weak chain homotopy equivalence of determinant class and
\[
\tau^{(2)}\bigl(\Lambda \circ \eta_{W,[B_W]}(f_*) \bigr) + \tau^{(2)}\bigl(\Lambda(T_u(C_*))\bigr)
 =  
\tau^{(2)}\bigl(\Lambda(T_u(D_*))\bigr) + \tau^{(2)}\bigl(\Lambda \circ \eta_{V,[B_V]}(f_*)\bigr).
\]
Since $\Lambda \circ \eta_{V,[B_V]}(f_*)$ is a weak chain homotopy equivalence, we get
\begin{multline*}
\chi_{\IC G}(C_*) \cdot \dim_{\IC}(V) 
=
\chi^{(2)}\bigl(\Lambda \circ \eta_{V,[B_V]}(C_*)\bigr)
\\
=
\chi^{(2)}\bigl(\Lambda \circ \eta_{V,[B_V]}(D_*)\bigr)
=
\chi_{\IC G}(D_*) \cdot \dim_{\IC}(V),
\end{multline*}
and hence $\chi_{\IC G}(C_*) = \chi_{\IC G}(D_*)$.
We conclude $\tau^{(2)}\bigl(\Lambda(T_u(C_*))\bigr) =
\tau^{(2)}\bigl(\Lambda(T_u(D_*))\bigr)$ from 
assertion~\eqref{lem:dependency_on_V:formula_for_T_u}.
Hence we get 
 \[
\tau^{(2)}\bigl(\Lambda \circ \eta_{W,[B_W]}(f_*) \bigr) 
 =  
\tau^{(2)}\bigl(\Lambda \circ \eta_{V,[B_V]}(f_*)\bigr).
\]
This finishes the proof of Lemma~\ref{lem:dependency_on_V}.
\end{proof}

\begin{lemma} \label{lem:additivity_under_exact_sequences_of_reps}
Let $0 \to U \to V \to W \to 0$ be an exact sequence of finite-dimensional  $G$-representations.
Choose any equivalence class of $\IC$-basis $[B_U]$ on $U$, $[B_V]$ on $V$, and $[B_W]$ on $W$.
Consider a chain map  $f_* \colon C_* \to D_*$ in $\FBCC{\IC G}$
Suppose that two of the chain maps $\Lambda \circ \eta_{U,[B_U]}(f_*)$, $\Lambda \circ \eta_{V,[B_V]}(f_*)$,  
and  $\Lambda \circ \eta_{W,[B_W]}(f_*)$ are 
weak chain homotopy equivalences of determinant class. 

Then all three the chain maps  $\Lambda \circ \eta_{U,[B_U]}(f_*)$, $\Lambda \circ \eta_{V,[B_V]}(f_*)$,  
and  $\Lambda \circ \eta_{W,[B_W]}(f_*)$ are  a weak chain homotopy equivalences
of determinant class, and we have
\[\tau^{(2)}\bigl(\Lambda \circ \eta_{V,[B_V]}(f_*) \bigr)  =
\tau^{(2)}\bigl(\Lambda \circ \eta_{U,[B_U]}(f_*) \bigr)  + \tau^{(2)}\bigl(\Lambda \circ \eta_{W,[B_W]}(f_*) \bigr).
\]
\end{lemma}
\begin{proof}
Choose representatives $B_U$ and $B_W$ for $[B_U]$ and $[B_W]$. 

Choose for any element $b \in B_W$ an element $\widetilde{b} \in V$
with $\pr(\widetilde{b}) = b$. Define a subset $B_V'$ of $V$ by 
$B_V' = i(B_U) \cup \{\widetilde{b} \mid b \in B_W\}$.
This is a $\IC$-basis for $V$.  We obtain a commutative diagram of
finite  Hilbert $\caln(G)$-chain complexes 
\[
\xymatrix{0 \ar[r]
&
\Lambda \circ \eta_{U,[B_U]}(C_*)  \ar[r] \ar[d]^{\Lambda \circ \eta_{U,[B_U]}(f_*)}
&
\Lambda \circ \eta_{V,[B_V']}(C_*)  \ar[r] \ar[d]^{\Lambda \circ \eta_{V,[B_V']}(f_*)}
&
\Lambda \circ \eta_{W,[B_W]}(C_*)  \ar[r] \ar[d]^{\Lambda \circ \eta_{W,[B_W]}(f_*)}
& 0
\\
0 \ar[r]
&
\Lambda \circ \eta_{U,[B_U]}(D_*)  \ar[r]
&
\Lambda \circ \eta_{V,[B_V']}(D_*)  \ar[r]
&
\Lambda \circ \eta_{W,[B_W]}(D_*)  \ar[r]
& 0
}
\]
The rows are in each dimension short exact sequences of finitely generated Hilbert
$\caln(G)$-module, where we call a sequence 
$0 \to P_0 \xrightarrow{i} P_1 \xrightarrow{p}  P_2 \to 0$ 
of finitely generated Hilbert $\caln(G)$-modules exact, if $i$ is an
isometry onto its image, the image $i$ is the kernel of $p$ and the morphism
$\ker(p)^{\perp} \to P_2$ induced by $p$ is an isometric isomorphism.  We conclude
from~\cite[Theorem~3.35~(2) on page~142]{Lueck(2002)}
that  all three the chain maps  $\Lambda \circ \eta_{U,[B_U]}(f_*)$, $\Lambda \circ \eta_{V,[B_V']}(f_*)$,  
and  $\Lambda \circ \eta_{W,[B_W]}(f_*)$ are a weak chain homotopy equivalence
of determinant class and 
\[
\tau^{(2)}\bigl(\Lambda \circ \eta_{V,[B_V']}(f_*)\bigr)
=
\tau^{(2)}\bigl(\Lambda \circ \eta_{U,[B_U]}(f_*)\bigr)
+
\tau^{(2)}\bigl(\Lambda \circ \eta_{W,[B_W]}(f_*)\bigr).
\]
Since $\tau^{(2)}\bigl(\Lambda \circ \eta_{V,[B_V']}(f_*)\bigr) = \tau^{(2)}\bigl(\Lambda \circ \eta_{V,[B_V]}(f_*)\bigr)$
holds by Lemma~\ref{lem:dependency_on_V}~\eqref{lem:dependency_on_V:independence_of_basis}
applied to $\id_V \colon (V,[B_V']) \to (V;[B_V])$,
Lemma~\ref{lem:additivity_under_exact_sequences_of_reps} follows.
\end{proof}

Next we deal with duality. 

Given an object $[M,B_M]$ in $\FBMOD{\IC G}$, we define its
dual $(M^*,[B_M^*])$ as follows. The underlying $\IC G$-module $M^*$ is given by 
the abelian group $\hom_{\IC G}(M,\IC G)$ with the left $ \IC G$-module structure defined by
$(u \cdot \phi)(x) := \phi(x) \cdot \overline{u}$ for $u \in \IC G$, $\phi \in \hom_{\IC G}(M,\IC G)$ and $x \in M$,
where we put  $\overline{u} := \sum_{g \in G} \overline{\lambda_g} \cdot g^{-1}$ for $u = \sum_{g \in G} \lambda_g \cdot g$.
The class $[B_M^*]$ is represented by the dual base $B_M^*$ for some representative $B_M \in
[B_M]$, where $B_M^* = \{b^* \mid b \in B_M\}$ and $b^* \colon M \to \IC G$ is the $\IC
G$-map sending $c \in B_M$ to $1$ if $c = b$, and to $0$ if $c \not= b$. 

Given a  based finite-dimensional $G$-representation $(V,[B_V])$, we define its dual
$(V^*,[B_V^*])$ to be the based  finite-dimensional $G$-representation defined as follows. The underlying
$G$-representation  $V^*$ is given by the complex vector space $\hom_{\IC}(V,\IC)$ with the $\IC$-module structure
defined by $(\lambda \cdot \psi)(v) := \overline{\lambda} \cdot \psi(v)$ and the left
$G$-action defined by $g\psi(v) := \psi(g^{-1} \cdot v)$ for $\lambda \in \IC$, $g \in G$, 
$\psi \in \hom_{\IC}(V;\IC)$ and $v \in V$.
The class $[B_V^*]$ is represented by the dual base $B_V^*$ for some representative $B_V \in [B_V]$, 
where $B_V^* = \{v^* \mid v \in B_V\}$ and $v^* \colon V \to \IC$ is the 
$\IC$-map sending $w \in B_V$ to $1$ if $w = v$, and to $0$ if $w \not= v$. 

\begin{lemma}\label{lem:twisting_and_duality}
  Given an object $(M,[B_M])$ in $\FBMOD{\IC G}$ and a based  finite-d\-imen\-sional 
  $G$-representation $(V,[B_V])$, there exists an isomorphism in $\FBMOD{\IC G}$ which is
  natural in $(M,[B_M])$ and $(V,[B_V])$ and respects the equivalence classes of $\IC
  G$-basis
\[
T((M,[B_M]),(V,[B_V]))  \colon \eta_{V^*,[B_V^*]}(M^*,[B_M^*]) 
\xrightarrow{\cong} \bigl(\eta_{V,[B_V]}(M,[B_M])\bigr)^*.
\]
\end{lemma}
\begin{proof}
  Consider a finite-dimensional $G$-representation $V$. Denote for $\lambda \in \IC$ and
  $g \in G$ by $\overline{\lambda} \cdot r_{g^{-1}} \colon \IC G \to \IC G$ the $\IC
  G$-homomorphism given by $x \mapsto \overline{\lambda} \cdot x \cdot g^{-1}$.  Notice
  that $\overline{\lambda} \cdot r_{g^{-1}} = (\lambda \cdot g) \cdot \id_{\IC G}$ and
  $\{\id_{\IC G}\}$ is a $\IC G$-basis for $\hom_{\IC G}(\IC G,\IC G)$. Define the $\IC  G$-isomorphism
\[
T = T(\IC G,V)  \colon \bigl(\hom_{\IC G}(\IC G;\IC G)\otimes_{\IC} V^*\bigr)_d
\xrightarrow{\cong} \hom_{\IC G}\bigl((\IC G \otimes_{\IC} V)_d,\IC G)
\]
by the formula $T(\overline{\lambda} \cdot r_{g^{-1}} \otimes \psi)(h \otimes v) = (\lambda \cdot \psi)(gh^{-1}v) \cdot hg^{-1} $
for $g \in G$, $\psi \in V^*$, $h \in G$ and $v \in V$.
Obviously $T(\overline{\lambda} \cdot r_{g^{-1}} \otimes \psi) \colon (\IC G \otimes_{\IC} V)_d \to \IC G$ is $\IC$-linear,
and it is $G$-equivariant by the following calculation for $g,h,k \in G$, $v \in V$ and $\psi \in V^*$
\begin{eqnarray*}
T(\overline{\lambda} \cdot r_{g^{-1}} \otimes \psi)\bigl(k \cdot (h \otimes v)\bigr)
& = & 
T(\overline{\lambda} \cdot  r_{g^{-1}} \otimes \psi)(kh \otimes kv)
\\
& = & 
(\lambda \cdot \psi)(g(kh)^{-1}kv) \cdot khg^{-1}
\\
& = & 
(\lambda \cdot \psi)(gh^{-1}v) \cdot khg^{-1}
\\
& = & 
k \cdot \bigl((\lambda \cdot \psi)(gh^{-1}v)  \cdot hg^{-1}\bigr)
\\
& = &
k \cdot T(\overline{\lambda} \cdot r_{g^{-1}} \otimes \psi)(h \otimes v).
\end{eqnarray*}

Next we show that $T$ is $\IC G$-linear. Since it is obviously
compatible with the $\IC$-module structures, this follows from the
calculation for $g,h,k \in G$, $v \in V$ and $\psi \in V^*$
\begin{eqnarray*}
T\bigl(k \cdot (r_{g^{-1}} \otimes \psi)\bigr)(h \otimes v) 
& = & 
T\bigl((k \cdot r_{g^{-1}}) \otimes (k \cdot \psi)\bigr)(h \otimes v) 
\\
& = & 
T\bigl(r_{(kg)^{-1}} \otimes (k \cdot \psi)\bigr)(h \otimes v) 
\\
& = & 
(k \cdot \psi)(kgh^{-1}v) \cdot h(kg)^{-1}
\\
& = & 
\psi(k^{-1}kgh^{-1}v) \cdot h(kg)^{-1}
\\
& = & 
\bigl(\psi(gh^{-1}v) \cdot hg^{-1}\bigr) \cdot k^{-1}
\\
& = & 
T(r_{g^{-1}} \otimes \psi\bigr)(h \otimes v) \cdot k^{-1}
\\
& = & 
\bigl(k \cdot T(r_{g^{-1}} \otimes \psi)\bigr)(h \otimes v).
\end{eqnarray*}
Let $e \in \IC G$ be the unit. Then $\{e\}$ is the standard $\IC G$-basis for $\IC G$.
Let $B_V$ be a $\IC$-basis for $V$. Then $\{e^* \otimes v^*\mid v \in B_V\}$ is a $\IC G$-basis
for the source of $T(\IC G,V)$ and $\{(e \otimes v)^* \mid v \in B_V\}$ is a $\IC G$-basis
for the target because of Lemma~\ref{lem:diagonal_versus_first_coordinate}. Since for all
$v \in B_V$ we have 
\begin{eqnarray*}
T(\IC G,V)(e^* \otimes v^*) & = & (e \otimes v)^*,
\end{eqnarray*}
the $\IC G$ map  $T(\IC G,V)$ respects these $\IC G$-basis and hence is a $\IC G$-isomorphism.

It is natural in $\IC G$ and $V$. Namely for $\IC G$-homomorphisms 
$f \colon \IC G \to \IC G$ and $u \colon V \to W$,
the following diagram of $\IC G$-modules commutes
\[
\xymatrix@!C=16em{
\bigl(\hom_{\IC G}(\IC G;\IC G)\otimes_{\IC} V^*\bigr)_d
\ar[r]^{T(\IC G, V)} 
&
 \hom_{\IC G}\bigl((\IC G \otimes_{\IC} V)_d,\IC G\bigr)
\\
\bigl(\hom_{\IC G}(\IC G;\IC G)\otimes_{\IC} W^*\bigr)_d
\ar[r]_{T(\IC G, W)}
\ar[u]^{\hom_{\IC G}(f,\id) \otimes u^*}
&
 \hom_{\IC G}\bigl((\IC G \otimes_{\IC} W)_d,\IC G\bigr)
\ar[u]_{\hom_{\IC G}(f \otimes u,\id)}
}
\]
To show commutativity it suffices  to check in the special case, where $f = r_{g^{-1}}$ for some $g \in G$,  
that for every $\psi \in W^*$ we have
\[
T(\IC G,V) \circ \bigl(\hom_{\IC G}(r_{g^{-1}},\id) \otimes u^*\bigr)(\id \otimes \psi) = 
\hom_{\IC G}(r_{g^{-1}} \otimes u,\id) \circ T(\IC G,W)(\id \otimes \psi).
\]
This equality of $\IC G$-maps $(\IC G \otimes_{\IC} V)_d \to \IC G$
has to be checked only elements of the shape $e \otimes v$ for $e \in G$ the unit and
$v \in V$. This is done by the following calculation
\begin{eqnarray*}
\lefteqn{\left(T(\IC G,V) \circ \bigl(\hom_{\IC G}(r_{g^{-1}},\id) \otimes u^*\bigr)(\id \otimes \psi)\right)(e \otimes v)}
& & 
\\
& = & 
\left(T(\IC G,V)(r_{g^{-1}} \otimes \psi \circ u)\right)(e \otimes v)
\\
& = & 
\psi \circ u(gv) \cdot g^{-1}
\\
& = & 
\psi( g \cdot u(v)) \cdot g^{-1}
\\
& = & 
T(\IC G,W)(\id \otimes \psi)\bigl(g^{-1} \otimes u(v))
\\
& = & 
\left(\hom_{\IC G}(r_{g^{-1}} \otimes u,\id) \circ T(\IC G,W)(\id \otimes \psi)\right)(e \otimes v).
\end{eqnarray*}
We can extend $T(\IC G,V)$ to 
\[
T(\IC G^n,V)  \colon \bigl(\hom_{\IC G}(\IC G^n;\IC G)\otimes_{\IC} V^*\bigr)_d
\xrightarrow{\cong} \hom_{\IC G}\bigl((\IC G^n \otimes_{\IC} V)_d,\IC G\bigr)
\]
by requiring that it is compatible with direct sums in $\IC G^n$, i.e., we require that
the following diagram of $\IC G$-modules commutes, where the vertical arrows are the
canonical $\IC G$-isomorphisms
\[
\xymatrix@!C= 19em{\bigoplus_{i = 1}^n \bigl(\hom_{\IC G}(\IC G;\IC G)\otimes_{\IC} V^*\bigr)_d
\ar[r]^{\bigoplus_{i= 1}^n T(\IC G,V)}_{\cong}
\ar[d]_{\cong}
&
\bigoplus_{i = 1}^n 
\hom_{\IC G}\bigl((\IC G \otimes_{\IC} V)_d,\IC G\bigr) 
\ar[d]^{\cong}
\\
\bigl(\hom_{\IC G}(\IC G^n;\IC G)\otimes_{\IC} V^*\bigr)_d
\ar[r]_{T(\IC G^n,V)}^{\cong} 
&
\hom_{\IC G}\bigl((\IC G^n \otimes_{\IC} V)_d,\IC G\bigr)
}
\]
Obviously $T(\IC G^n,V)$ is a $\IC G$-isomorphisms, natural in $\IC G^n$ and $V$.

Now consider an object $(M,[B_M])$ in $\FBMOD{\IC G}$ and a based finite-dimensional $G$-representation 
$(V,[B_V])$. Choose a $\IC G$-isomorphism
$\omega \colon \IC G^n \xrightarrow{\cong} M$ which sends the class of the standard basis of $\IC G^n$
to the class $[B_M]$. Define the $\IC G$-isomorphism
\[
T(M, \omega ,V)  \colon \bigl(\hom_{\IC G}(M;\IC G)\otimes_{\IC} V^*\bigr)_d
\xrightarrow{\cong} \hom_{\IC G}\bigl((M \otimes_{\IC} V)_d,\IC G\bigr)
\]
by requiring that the following $\IC G$-diagram commutes
\[
\xymatrix@!C= 19em{
\bigl(\hom_{\IC G}(M;\IC G)\otimes_{\IC} V^*\bigr)_d
\ar[r]^{T(M,\omega,V)}_{\cong} \ar[d]_{\hom(\omega,\id) \otimes \id_{V^*}}^{\cong}
&
\hom_{\IC G}\bigl((M\otimes_{\IC} V)_d,\IC G\bigr) 
\ar[d]^{\hom(\omega \otimes \id_V,\id)}_{\cong}
\\
\bigl(\hom_{\IC G}(\IC G^n;\IC G)\otimes_{\IC} V^*\bigr)_d
\ar[r]_{T(\IC G^n,V)}^{\cong} 
&
\hom_{\IC G}\bigl((\IC G^n \otimes_{\IC} V)_d,\IC G\bigr)
}
\]
We claim that $T(M,\omega,V)$ does not depend on $\omega$ (although the two vertical maps
do depend on $\omega$).  This follows from the observation that for any $n$, any
permutation $\sigma \colon \{1,2, \ldots, n\} \xrightarrow{\cong} \{1,2, \ldots, n\}$ and
any choice of signs $\epsilon_i \in \{\pm 1\}$ for $i = 1,2, \ldots, n$ the following
diagram of $\IC G$-modules commutes for the automorphism 
$\pi \colon \IC G^n \xrightarrow{\cong}\IC G^n$ sending $(x_1,x_2, \ldots , x_n)$ to 
$(\epsilon_1 \cdot x_{\sigma(1)}, \epsilon_2 \cdot x_{\sigma(2)}, \ldots , \epsilon_n \cdot x_{\sigma(n)})$
\[
\xymatrix@!C= 19em{
\bigl(\hom_{\IC G}(\IC G^n;\IC G)\otimes_{\IC} V^*\bigr)_d
\ar[r]^{T(\IC G^n,V)}_{\cong} \ar[d]_{\hom(\pi,\id) \otimes \id_{V^*}}^{\cong}
&
\hom_{\IC G}\bigl((\IC G ^n \otimes_{\IC} V)_d,\IC G\bigr) 
\ar[d]^{\hom(\pi \otimes \id_V,\id)}_{\cong}
\\
\bigl(\hom_{\IC G}(\IC G^n;\IC G)\otimes_{\IC} V^*\bigr)_d
\ar[r]_{T(\IC G^n,V)}^{\cong} 
&
\hom_{\IC G}\bigl((\IC G^n \otimes_{\IC} V)_d,\IC G\bigr)
}
\]
Now the desired $\IC G$-isomorphism
\begin{multline*}
T((M,[B_M]),(V,[B_V]) \colon \eta_{V^*,[B_V^*]}(M^*,[B_M^*]) =  \bigl(\hom_{\IC G}(M;\IC G)\otimes_{\IC} V^*\bigr)_d
\\
\xrightarrow{\cong} \bigl(\eta_{V,[B_V]}(M,[B_M])\bigr)^* = \hom_{\IC G}\bigl((M \otimes_{\IC} V)_d,\IC G\bigr)
\end{multline*}
is defined to be $T(M, \omega ,V)$ for any choice of $\IC G$-isomorphism
$\omega \colon \IC G^n \xrightarrow{\cong} M$ which sends the class of the standard basis of $\IC G^n$
to the class $[B_M]$. On easily checks that it is natural in $(M,[B_M])$ and $(V,[B_V])$
and respects the equivalence classes of $\IC G$-basis. This finishes the proof of
Lemma~\ref{lem:twisting_and_duality}
\end{proof}

%%%%%%%%%%%%%%%%%%%%%%%%%%%%%%%%%%%%%%%%%%%%%%%%%%%%%%%%%%%%%%%%%%%%%%%%%%%%%%%%%
%%%%%%%%%%%%%%%%%%%%%%% Section 4: Unitary representations %%%%%%%%%%%%%%%%%%%%%%%%%%%%%%%%%
%%%%%%%%%%%%%%%%%%%%%%%%%%%%%%%%%%%%%%%%%%%%%%%%%%%%%%%%%%%%%%%%%%%%%%%%%%%%%%%%%

 \typeout{-----------------------------   Section 4: Unitary representations -------------------------}

\section{Unitary representations}
\label{subsec:Unitary_representations}

We first deal with the rather elementary case, where $(V,[B_V])$ is unitary, i.e., the
$G$-action is isometric with respect to the Hilbert space structure for which one (and
hence all) representatives $B_V$ of $[B_V]$ is an orthonormal basis.
For the notions of  spectral density function and  Novikov-Shubin invariant, 
we refer for instance to~\cite[Chapter~2]{Lueck(2002)}.

\begin{theorem}[Twisting with unitary finite-dimensional representations]
\label{the:twisting_with_unitary_finite-dimensional_representations}
Let $C_*$ be an object in $\FBCC{\IC G}$ and let $(V,[B_V])$ be a based finite-dimensional 
 $G$-repre\-sen\-ta\-tion.  Suppose that $V$ is unitary. Then

\begin{enumerate}

\item \label{the:twisting_with_unitary_coefficients:density} For each $n \in \IZ$ and
  $\lambda \in [0,\infty)$ we get for the spectral  density functions

\[
F_n(\Lambda \circ \eta(C_*))(\lambda) = \dim_{\IC}(V) \cdot F_n(\Lambda (C_*))(\lambda);
\]

\item \label{the:twisting_with_unitary_coefficients:Betti_numbers}
For each $n \in \IZ$  we have
\[
b_n^{(2)}\bigl(\Lambda \circ \eta(C_*);\caln(G)\bigr) 
 = 
\dim_{\IC}(V) \cdot b_n^{(2)}\bigl(\Lambda(C_*);\caln(G)\bigr);
\]

\item \label{the:twisting_with_unitary_coefficients:Novikov-Shubin_invariants}
For each $n \in \IZ$  we have
\[
\alpha\bigl(\Lambda \circ \eta(C_*);\caln(G)\bigr) 
 = 
\alpha\bigl(\Lambda(C_*);\caln(G)\bigr);
\]

\item \label{the:twisting_with_unitary_coefficients:L2-torsion}
The finite Hilbert chain complex $\Lambda \circ \eta(C_*)$ is of determinant class if and only if
 $\Lambda(C_*)$ is of determinant class. In this case we get 
\[
\rho^{(2)}\bigl(\Lambda \circ \eta(C_*);\caln(G)\bigr) 
 = 
\dim_{\IC}(V) \cdot  \rho^{(2)}\bigl(\Lambda(C_*);\caln(G)\bigr).
\]
\end{enumerate}
\end{theorem}
\begin{proof}~\eqref{the:twisting_with_unitary_coefficients:density} Let $P$ be a finitely
  generated Hilbert $\caln(G)$-module. Then $P \otimes_{\IC} V$ inherits a Hilbert space
  structure from the given Hilbert space structure on $P$ and the Hilbert space
  structure on $V$ coming from the equivalence class of bases. Since $V$ is unitary, the
  diagonal action on $P \otimes_{\IC} V$ is isometric. We denote by $(P \otimes_{\IC}
  V)_d$ this Hilbert space with this isometric $G$-action. If we let $G$ act only on $P$, we
  denote the resulting Hilbert space with isometric $G$-action by $(P \otimes_{\IC} V)_1$. 
  Since $V$ is unitary, we   obtain an isometric $G$-equivariant isomorphism of Hilbert 
  spaces with isometric  $G$-action
\[
\xi^{(2)} \colon (L^2(G) \otimes_{\IC} V)_1 \xrightarrow{\cong} (L^2(G) \otimes_{\IC} V)_d,
\quad g \otimes v \mapsto g \otimes gv.
\]
Let $i \colon P \to L^2(G)^r$ be an isometric $G$-equivariant embedding.
Then the induced map $(i \otimes_{\IC} \id_V)_d \colon (P \otimes_{\IC} V)_d \to
(L^2(G)^r \otimes_{\IC} V)_d$ is an isometric $G$-equivariant embedding. 
We obtain from $\xi^{(2)}$ above and the basis for $V$ an isometric $G$-equivariant isomorphism
\begin{multline*}
(L^2(G)^r \otimes_{\IC} V)_d \xrightarrow{\cong} \bigl((L^2(G) \otimes_{\IC} V)_d\bigr)^r
\xrightarrow{\cong} \bigl((L^2(G) \otimes_{\IC} V)_1\bigr)^r 
\\
\xrightarrow{\cong} (L^2(G) \otimes_{\IC} V^r)_1
\xrightarrow{\cong} L^2(G)^{r \cdot \dim_{\IC}(V)}.
\end{multline*}
This shows that $(P \otimes_{\IC} V)_d$ is a finitely generated Hilbert $\caln(G)$-module. 
A morphism $f \colon P \to Q$ induces a morphism $(f \otimes_{\IC} \id_V)_d \colon (P \otimes_{\IC} V)_d
\to (Q \otimes_{\IC} V)_d$ of finitely generated Hilbert $\caln(G)$-modules. Thus we obtain a functor
of $\IC$-additive categories
\[
\eta^{(2)} \colon \FGHIL{\caln(G)} \to  \FGHIL{\caln(G)}, \quad P \mapsto (P \otimes_{\IC} V)_d,
\]
which is compatible with the involution given by taking adjoints, i.e., 
$\eta(f^*) = \eta(f)^*$. One easily checks the identity of functors $\FBMOD{\IC G} \to \FGHIL{\caln(G)}$
\begin{eqnarray}
\eta^{(2)} \circ \Lambda = \Lambda \circ \eta.
\label{eta_commutes_with:_Lambda}
\end{eqnarray}
Next we show for any finitely generated Hilbert $\caln(G)$-module $P$
\begin{eqnarray}
\dim_{\caln(G)}(\eta^{(2)}(P))
    & = & 
    \dim_{\IC}(V) \cdot  \dim_{\caln(G)}(P).
    \label{dim(rho(P))_is_dim(V)_cdot_dim(P)}
\end{eqnarray}
Let $i \colon P \to L^2(G)^r$ be an isometric $G$-equivariant embedding.
Let $\pr \colon L^2(G)^r \to L^2(G)^r$ be $i \circ i^*$. Then $\pr$ is a
$G$-equivariant projection whose image is isometrically $G$-isomorphic to $P$ and
satisfies
\[
\dim_{\caln(G)}(P) = \tr_{\caln(G)}(\pr).
\] 
Obviously $\eta^{(2)}(\pr) \colon \eta^{(2)}(L^2(G)^r) \to  \eta^{(2)}(L^2(G)^r)$ is a projection
whose image is isometrically isomorphic to $\eta^{(2)}(P)$. Hence 
\[
\dim_{\caln(G)}(\eta^{(2)}(P)) = \tr_{\caln(G)}(\eta^{(2)}(\pr)).
\]
One easily checks for any endomorphism $f \colon L^2(G)^r \to L^2(G)^r$ 
 \[
\tr_{\caln(G)}(\eta^{(2)}(f)) = \dim_{\IC}(V) \cdot \tr_{\caln(G)}(f),
\]
using the observation that for $g \in G$ and $v \in B_V$ we have
\[
\langle 1 \otimes v, g \otimes g^{-1} \cdot v\rangle_{L^2(G) \otimes_{\IC} V} = 
\begin{cases}
1 & g = 1;
\\
0 & g \not= 1.
\end{cases}
\]
If we apply this to  $f = \pr$, we get~\eqref{dim(rho(P))_is_dim(V)_cdot_dim(P)}.

Consider a morphism $f \colon P \to Q$ of finitely generated Hilbert $\caln(G)$-modules. 
Let $\{E_{\lambda} \mid \lambda \in \IR\}$ be the spectral family of $f^*f$. One
easily checks that then $\{\eta^{(2)}(E_{\lambda}) \mid \lambda \in \IR\}$ is the spectral family
of $\eta^{(2)}(f^*f) = \bigl(\eta^{(2)}(f)\bigr)^* \circ \eta^{(2)}(f)$. We conclude 
from~\eqref{dim(rho(P))_is_dim(V)_cdot_dim(P)}   applied to $P = \im(E_{\lambda})$ and
from~\cite[Lemma~2.3 on page~74]{Lueck(2002)}
\[
F(\eta^{(2)}(f))(\lambda)  = \dim_{\IC}(V) \cdot F(f)(\lambda).
\]
Now assertion~\eqref{the:twisting_with_unitary_coefficients:density}
follows from~\eqref{eta_commutes_with:_Lambda}.
\\[1mm]~\eqref{the:twisting_with_unitary_coefficients:Betti_numbers},~%
\eqref{the:twisting_with_unitary_coefficients:Novikov-Shubin_invariants}
and~\eqref{the:twisting_with_unitary_coefficients:L2-torsion} are direct consequences of
assertion~\eqref{the:twisting_with_unitary_coefficients:density}.
\end{proof}

%%%%%%%%%%%%%%%%%%%%%%%%%%%%%%%%%%%%%%%%%%%%%%%%%%%%%%%%%%%%%%%%%%%%%%%%%%%%%%%%%
%%%%%%%%%%%%%%%%%%%%%%%%% Section 5: $L^2$_Betti numbers %%%%%%%%%%%%%%%%%%%%%%%%%%%%%%%%%
%%%%%%%%%%%%%%%%%%%%%%%%%%%%%%%%%%%%%%%%%%%%%%%%%%%%%%%%%%%%%%%%%%%%%%%%%%%%%%%%%

 \typeout{-----------------------------   Section 5: $L^2$-Betti numbers ---------------------------}

\section{$L^2$-Betti numbers}
\label{subsec:L2-Betti_numbers}

%%%%%%%%%%%%%%%%%%%%%%%%%%%%%%%%%%%%%%%%%%%%%%%%%%%%%%%%%%%%%%%%%%%%%%%%%%%%%%%%%
\subsection{Twisting $L^2$-Betti numbers}
\label{subsec:Twisting_L2-Betti_numbers}

We will show in this section that the answer to Question~\ref{que:L2-Betti_number_and_twisting}
is positive if $G$ is torsionfree elementary amenable.
Another result about $L^2$-Betti numbers will be presented in 
Theorem~\ref{the:Determinant_class_and_twisting}~\eqref{the:Determinant_class_and_twisting:Betti_numbers}.

\begin{remark}[Field of fractions]
  \label{rem:Field_of_fractions}
  Let $F$ be any field of characteristic zero. Let $G$ be an amenable group for which
  $FG$ has no non-trivial zero-divisor. Examples for $G$ are torsionfree elementary
  amenable groups, see~\cite[Theorem~1.2]{Kropholler-Linnell-Moody(1988)},
  \cite[Theorem~2.3]{Linnell(2006)}.  Then $FG$ has a skew field of fractions $S^{-1}FG$
  given by the Ore localization with respect to the multiplicative
  closed subset $S$ of non-trivial elements in $FG$, see~\cite[ Example~8.16 on
  page~324]{Lueck(2002)}.  
\end{remark}

\begin{lemma}\label{lem:invariance_of_L2-Betti_numbers_under_twisting}
Suppose that $G$ is amenable and there are no non-trivial zero-divisors in $\IC G$, e.g.,
$G$ is torsionfree elementary amenable.
Then we get for  every object $C_*$ in $\FBCC{\IC G}$, every  based finite-dimensional  $G$-representation $V$  and every $n \in \IZ$
\begin{eqnarray*}
b_n^{(2)}\bigl(\Lambda \circ \eta_V(C_*);\caln(G)\bigr) 
& = &
\dim_{\IC}(V) \cdot b_n^{(2)}\bigl(\Lambda(C_*);\caln(G)\bigr)
\\
& = & 
\dim_{\IC}(V) \cdot  \dim_{S^{-1} \IC G}\bigl(S^{-1} \IC G \otimes_{\IC G} H_n(C_*)\bigr)
\\
& = & 
\dim_{S^{-1} \IC G}\bigl(S^{-1} \IC G \otimes_{\IC G} H_n(\eta(C_*))\bigr).
\end{eqnarray*}
\end{lemma}
\begin{proof}
We get for every object $C_*$ in $\FBCC{\IC G}$ and every $n \in \IZ$
\[
  b_n^{(2)}\bigl(\Lambda(C_*);\caln(G)\bigr) = \dim_{S^{-1}\IC G}\bigl(S^{-1}\IC G \otimes_{\IC G} H_n(C_*)\bigr)
  \]
from~\cite[Lemma~6.53 on page~264, (6.74) on page~275, Theorem~8.29 on page~330, Remark~10.30 on page~384]{Lueck(2002)}. 
Applying this to $\eta(C_*)$ instead of $C_*$ we get
 \[
  b_n^{(2)}\bigl(\Lambda \circ \eta(C_*);\caln(G)\bigr) 
= \dim_{S^{-1}\IC G}\bigl(S^{-1}\IC G \otimes_{\IC G} H_n(\eta(C_*))\bigr).
  \]
Hence it suffices to show every object $C_*$ in $\FBCC{\IC G}$ and every $n \in \IZ$
\begin{multline*}
\dim_{S^{-1}\IC G}\bigl((S^{-1}\IC G \otimes_{\IC G}   H_n(\eta(C_*))\bigr) 
\\
=
\dim_{\IC}(V)  \cdot \dim_{S^{-1}\IC G}\bigl(S^{-1}\IC G \otimes_{\IC G}  H_n(C_*)\bigr).
\end{multline*}
By the additivity of the dimension over the skew field  $S^{-1}\IC G$ this can be reduced to the case, where
$C_*$ is $1$-dimensional and $n = 1$, i.e., it suffices to show
\begin{multline*}
\dim_{S^{-1}\IC G}\left(\ker\bigl(\id \otimes \eta(c_1) \colon S^{-1}\IC G \otimes_{\IC G} \eta(C_1) 
\to S^{-1}\IC G \otimes_{\IC G} \eta(C_0)\bigr)\right) 
\\
=
\dim_{\IC}(V)  \cdot \dim_{S^{-1}\IC G}\left(\ker\bigl(\id \otimes c_1 \colon 
S^{-1}\IC G \otimes_{\IC G} C_1 \to S^{-1}\IC G \otimes_{\IC G} C_0\bigr)\right). 
\end{multline*}
Let $m$ and $n$ be the rank of the finitely generated free $\IC G$-modules $C_1$ and $C_0$.
Since $S^{-1} \IC G$ is a skew field of fractions of $\IC G$, we can find a natural number
$l$ with $l \le m$ and $l \le n$, non-trivial elements $x_1, x_2, \ldots , x_l$ in $\IC G$
and $\IC G$-maps $u$ and $v$ such that the following diagram of $\IC G$-modules commutes
\[
\xymatrix{
C_1 \ar[r]^{c_1} &
C_2 \ar[d]^{v}
\\
\IC G^m \ar[u]^u\ar[r]_d&
\IC G^n
}
\]
where $d \colon \IC G^m \to \IC G^n$ sends $(y_1, y_2, \ldots, y_m)$ to 
$(y_1 \cdot x_1, y_2 \cdot x_2, \ldots ,y_l \cdot x_l, 0,0 \ldots , 0)$
and $u$ and $v$ are injective.  It induces a commutative diagram of $S^{-1} \IC G$-modules
\[
\xymatrix{
S^{-1}\IC G \otimes_{\IC G}  C_1 \ar[r]^{\id \otimes c_1} &
S^{-1}\IC G \otimes_{\IC G} C_0 \ar[d]^{\id \otimes v}
\\
S^{-1}\IC G \otimes_{\IC G} \IC G ^m \ar[u]^{\id \otimes u}\ar[r]_{\id \otimes d}&
S^{-1}\IC G \otimes_{\IC G} \IC G ^n 
}
\]
Since $u$ and $v$ are injective, the vertical arrows are injective maps
whose  source and target have the same $S^{-1} \IC G$-dimension. Hence the vertical maps
are isomorphisms. Since the map $r_{x_i} \colon \IC G \to \IC G$ is injective,
the map $\id  \otimes r_{x_i} \colon S^{-1}\IC G \otimes_{\IC G} \IC G \to S^{-1}\IC G \otimes_{\IC G} \IC G$ 
is an isomorphism. We conclude
\[
\dim_{S^{-1}\IC G}\bigl(\ker(\id \otimes c_1)\bigr) = m -l.
\]
We also obtain a commutative diagram of $S^{-1} \IC G$-modules
\[
\xymatrix@!C=17em{
S^{-1}\IC G \otimes_{\IC G}  \eta(C_1) \ar[r]^{\id \otimes \eta(c_1)} &
S^{-1}\IC G \otimes_{\IC G} \eta(C_0)   \ar[d]^{\id \otimes \eta(v)}
\\
S^{-1}\IC G \otimes_{\IC G}  \eta(\IC G^m)\ar[u]^{\id \otimes \eta(u)}\ar[r]_{\id \otimes \eta(d)}&
S^{-1}\IC G \otimes_{\IC G}  \eta(\IC G^n) 
}
\]
Since $u$, $v$ and $r_{x_i}$  are injective, the maps $\eta(u)$, $\eta(v)$ and $\eta(r_{x_i})$ are injective.
By the same argument as above we conclude
that $\id \otimes \eta(u)$, $\id \otimes  \eta(v)$ and $\id \otimes \eta(r_{x_i})$ are isomorphisms.
Hence we get
\[
\dim_{S^{-1}\IC G}\bigl(\ker(\id \otimes \eta(c_1))\bigr) = \dim_{\IC}(V) \cdot (m -l).
\]
This finishes the proof of Lemma~\ref{lem:invariance_of_L2-Betti_numbers_under_twisting}.
\end{proof}

%%%%%%%%%%%%%%%%%%%%%%%%%%%%%%%%%%%%%%%%%%%%%%%%%%%%%%%%%%%%%%%%%%

\subsection{$L^2$-Betti numbers and fibrations}
\label{subsec:L2_Betti_numbers_and_fibrations}

Throughout this subsection let $F \to E \xrightarrow{p} B$ be a fibration of connected $CW$-complexes
of finite type.  We want to study the question

\begin{question}[Fibrations]
  \label{que:fibrations}
  Let $d$ be a natural number.  Suppose that the $n$th $L^2$-Betti number of the universal
  covering of $B$ with respect to the action of the fundamental group
  $b_q^{(2)}(\widetilde{B})$ vanishes for $n \le d$. Under which condition does this
  implies $b_n^{(2)}(\widetilde{E}) = 0$ for $n \le d$?
  \end{question}

  Under the assumption that $B$ is aspherical, some cases, where the answer is positive,
  are listed in~\cite[Theorem~7.4 on page~295]{Lueck(2002)}.  It includes the case, where $B$ is
  aspherical and the fundamental group is infinite elementary amenable. Moreover,
  inheritance properties under amalgamated products and normal subgroups are stated
  provided that $B$ is aspherical. Next we want to consider the case, where  $B$ may not be aspherical.

  \begin{lemma} \label{lem:fibrations_and_twistings} 
    Suppose that $\pi_1(p) \colon \pi_1(E) \to \pi_1(B)$ is bijective. Then the answer to 
    Question~\ref{que:fibrations} is positive if the answer to Question~\ref{que:L2-Betti_number_and_twisting} 
    is    positive for $G = \pi_1(E)$.
  \end{lemma}
\begin{proof}
  We will write $G = \pi_1(E)$ and identify $G = \pi_1(B)$ by the isomorphism 
  $\pi_1(p)   \colon \pi_1(E) \to \pi_1(B)$.  Consider the von Neumann algebra $\caln(G)$ just as a
  ring. In the sequel $\caln(G)$-module is to be understood in the purely algebraic sense,
  no topologies are involved.  The fiber transport along loops in $B$ induces a right 
  $\IC  G$-module structure on $H_q(F;\IC)$ for all $q \ge 0$. Let $(\caln(G) \otimes_{\IC}   H_q(F;\IC))_d$ 
  be the $\caln(G)$-$\IC G$-bimodule where for $x,y \in \caln(G)$, $z \in H_q(F;\IC)$ 
  and $g \in G$ we define the module structure by  $x \cdot (y \otimes_{\IC} z) \cdot g = xyg \otimes zg$.  
  Then there is    the Leray-Serre spectral sequence associated to the $\caln(G)$-$\IC G$-bimodule 
  $\caln(G)$ which converges to the left $\caln(G)$-module 
  $H_n\bigl(\caln(G)  \otimes_{\IC G} C_*(\widetilde{E})\bigr)$ and whose $E^2$-term can 
  be identified with the left $\caln(G)$-module
  \[
   E^2_{p,q} = H_p\bigl((\caln(G) \otimes H_q(F;\IC))_d \otimes_{\IC G} C_*(\widetilde{B})\bigr).
  \]
Let $V_q$ be the finite-dimensional $G$-representation obtained from the right $\IC G$-module
$H_q(F;\IC)$ by the action $g \cdot z := z \cdot g^{-1}$ for $g \in G$ and 
$z \in H_q(F;\IC)$.  If we choose any $\IZ$-basis $B^{\IZ}_q$ for $H_q(F;\IZ)/ \tors(H_q(F;\IZ))$,
then the induced complex basis of $\bigl(H_q(F;\IZ) /\tors(H_q(F;\IZ))\bigr) \otimes_{\IZ} \IC$ 
together with the canonical isomorphism
$\bigl(H_q(F;\IZ) /\tors(H_q(F;\IZ))\bigr) \otimes_{\IZ} \IC \xrightarrow{\cong} V_q = H_q(F;\IC)$  
yields a $\IC$-basis $B_{V_q}$ on $V_q$ whose equivalence class $[B_{V_q}]$ is independent of the choice 
of $B^{\IZ}_q$. Using the dimension theory and the equivalence of the categories of finitely generated
Hilbert $\caln(G)$-modules and finitely generated projective $\caln(G)$-modules described
in~\cite[Section~6.1 and~6.2]{Lueck(2002)}, we see that $b_n^{(2)}(\widetilde{E}) = 0$ for $n \le d$
if we can show for $q \le d$
\[
b_q^{(2)}\bigl(\Lambda^G \circ \eta_{V_q,[B_{V_q}]}(C_*(\widetilde{B}))\bigr) = 0.
\]
Since $b_q^{(2)}(\widetilde{B}) := b_q^{(2)}\bigl(\Lambda^G(C_*(\widetilde{B}))\bigr) = 0$
holds by assumption for $q \le d$, this follows if the answer to
Question~\ref{que:L2-Betti_number_and_twisting} is positive, i.e., if we have the equality
\[
b_q^{(2)}\bigl(\Lambda^G \circ \eta_{V_q,[B_{V_q}]}(C_*(\widetilde{B}))\bigr)
=
\dim_{\IC}(V_q) \cdot  b_q^{(2)}\bigl(\Lambda^G(C_*(\widetilde{B}))\bigr).
\]
This finishes the proof of Lemma~\ref{lem:fibrations_and_twistings}.
\end{proof}

We conclude from Lemma~\ref{lem:invariance_of_L2-Betti_numbers_under_twisting} and
Lemma~\ref{lem:fibrations_and_twistings} that the answer to Question~\ref{que:fibrations}
is positive if $\pi_1(p) \colon \pi_1(E) \to \pi_1(B)$ is bijective and $\pi_1(B)$ is a
torsionfree elementary amenable group.  (It is not hard to prove that one can replace
torsionfree by virtually torsionfree.)

%%%%%%%%%%%%%%%%%%%%%%%%%%%%%%%%%%%%%%%%%%%%%%%%%%%%%%%%%%%%%%%%%%%%%%%%%%%%%%%%%
%%%%%%%%%% Section 6: $L^2$-torsion twisted by a based finite-dimensional    representation %%%%%%%%%%%%%%
%%%%%%%%%%%%%%%%%%%%%%%%%%%%%%%%%%%%%%%%%%%%%%%%%%%%%%%%%%%%%%%%%%%%%%%%%%%%%%%%%

 \typeout{--------   Section 6: $L^2$-torsion twisted by a based finite-dimensional  representation --------}

\section{$L^2$-torsion twisted by a based  finite-dimensional representation}
\label{subsec:L2-torsion_twisted_by_a_based_finite-dimensional_representation}

Let $G$ be a (discrete) group $G$.  Consider a finite free $G$-$CW$-complex $X$.
The $G$-$CW$-complex structure yields a \emph{cellular equivalence class} 
  $[[B_n]]$ of $\IZ G$-bases $B_n$ on the finitely generated free
  $\IZ G$-module $C_n(X)$, where we call two
  $\IZ G$-bases $B$ and $B'$ for a finitely generated free $\IZ G$-module $M$
  \emph{cellularly equivalent} if there is a bijection $\sigma \colon B \to B'$ and for
  each $b \in B$ elements $\epsilon(b) \in \{\pm 1\}$ and $g(b) \in G$ such that
  $\sigma(b) = \epsilon(b) \cdot g(b) \cdot b$ holds for each $b \in B$. Obviously
  ``cellular equivalent'' is a weaker equivalence relation than the relation ``equivalent''
  introduced in Section~\ref{sec:Twisting_CG-modules_with_finite-dimensional_representations}
  since there $g(b) = 1$ for all $b \in B$.  In order
  to define $\Lambda \circ \eta_{\phi^*\IC_t}(C_*(X))$ we need equivalence classes
  $[B_n]$ of $\IZ G$-bases and not only a cellular equivalence class
  $[[B_n]]$ of $\IZ G$-bases. 

  \begin{definition}[Base refinement] \label{def:base_refinement}
  We call a choice $[B_X]$ of equivalence classes $[B_n]$ of
  $\IZ G$-bases for $C_n(X)$ such that $[[B_n]]$ represents the cellular equivalence class of
  $\IZ G$-bases coming from the $G$-$CW$-structure on $X$ for all $n \ge 0$ a \emph{base refinement}.
\end{definition}

\begin{remark}[Base refinements, spiders, Euler structures, and $\operatorname{Spin}^c$-structures]
\label{rem:Base_refinements,spiders_and_Euler_structures}
Recall that an open $n$-cell of a $CW$-complex $Y$ is the same as a path component of
$Y_n \setminus Y_{n-1}$.  Geometrically the choice of the base refinement corresponds to
choosing for every open cell $e \in X/G$ an open  cell $\widehat{e}$ of $X$ which
is mapped to $e$ under the projection $p \colon X \to X/G$. Such  choices have already
occurred as so called \emph{spiders} in connections with simple structures on the total space of a fibration
in~\cite[Section~2]{Lueck(1984)}, \cite[Section~3]{Farrell-Lueck-Steimle(2010)},  in the
more general context of equivariant $CW$-complexes in \cite[Section~15]{Lueck(1989)}, and
as so called \emph{Euler structures} as introduced by Turaev~\cite{Turaev(1990),Turaev(2001)}.
For a closed orientable $3$-manifold  there is a bijection between the set of  base refinements 
and $\operatorname{Spin}^c$-structures, see Turaev~\cite{Turaev(1997)}.
\end{remark}

\begin{definition}[$L^2$-torsion twisted by a based finite-dimensional representation]
\label{def:L2-torsion_twisted_by_a_based_finite-dimensional_representation}
  Consider a finite free $G$-$CW$-complex $X$ with a base refinement $[B_X]$.   
   Let $V = (V,[B_V])$ be a  based finite-dimensional $G$-representation.
  We call $X$ \emph{$L^2$-acyclic} if $b_n^{(2)}(X;\caln(G))$ vanishes for all $n \ge 0$.
  We call $X$ \emph{of $V$-twisted determinant class} or
  \emph{$V$-twisted $\det$-$L^2$-acyclic} respectively if the finite Hilbert
  $\caln(G)$-chain complex $\Lambda \circ \eta_{V}(C_*(X),[B_X])$
  is of determinant class or $\det$-$L^2$-acyclic. (This is independent of the choice of
  $[B_X]$ and $[B_V]$ by~\cite[Theorem~3.35~(4) on page~142]{Lueck(2002)}.)
 
  Provided that $X$ is of $V$-twisted determinant class, we define the
  \emph{$V$-twisted $L^2$-torsion} of $(X,V,[B_X])$ to be
  \begin{eqnarray*}
   \rho^{(2)}_G(X;V,[B_X]) & := & \rho^{(2)}\bigl(\Lambda \circ \eta_{V}(C_*(X),[B_X])\bigr).
  \end{eqnarray*} 
\end{definition}

In order to investigate how $\rho^{(2)}_G(X;V,[B_X])$ depends on the base refinement $[B_X]$,
we need the following notion.  Let $[B]$ and $[B']$ be equivalence
classes of $\IZ G$-basis for the finitely generated free $\IZ G$-module $M$ such that 
$[[B]] = [[B']]$ holds.  Choose a bijection $\sigma \colon B \to B'$ and for each $b \in B$ elements
$\epsilon(b) \in \{\pm 1\}$ and $g(b) \in G$ such that $\sigma(b) = \epsilon(b) \cdot
g(b) \cdot b$ holds for each $b \in B$.  Define an element in the abelianization of $G$
\begin{eqnarray}
\trans([B],[B'])  & = & \prod_{b \in B} \overline{g(b)} \quad  \in H_1(G),
\label{trans(b,B')}
\end{eqnarray}
where $\overline{g(b)}$ is the image of $g(b)$ under the canonical projection $G \to H_1(G)$.

Let $C_*$ be a bounded $\IZ G$-chain complex consisting of finitely generated free $\IZ G$-modules.
Let $[B_{C_*}]$ and $[B_{C_*}']$ be two equivalence classes of $\IZ G$-basis for $C_*$ satisfying
$[[B_{C_*}]]=[[B_{C_*}']]$.  Define
\begin{eqnarray}
\trans([B_{C_*}],[B_{C_*}']) 
& = & 
\prod_{n \in \IZ} \trans([B_{C_n}],[B_{C_n}'])^{(-1)^n} \quad  \in H_1(G).
\label{trans}
\end{eqnarray} 
We have
\begin{eqnarray*}
\trans([B_{C_*}],[B_{C_*}]) & = & 1;
\\
\trans([B_{C_*}'],[B_{C_*}]) & = & \trans([B_{C_*}],[B_{C_*}'])^{-1};
\\
\trans([B_{C_*}],[B''_{C_*}]) &= &  \trans([B_{C_*}],[B'_{C_*}]) \cdot \trans([B'_{C_*}],[B''_{C_*}]).
\end{eqnarray*}

\begin{remark}[The Farrell-Jones Conjecture]\label{rem:Farrell-Jones_Conjecture}
  In several theorems the condition will occur that the $K$-theoretic Farrell-Jones
  Conjecture holds for $\IZ G$.  This statement can be used as a black box. The reader should
  have in mind that it is known for a large class of groups, e.g., hyperbolic groups,
  CAT(0)-groups, solvable groups, lattices in almost connected Lie groups, fundamental
  groups of $3$-manifolds and passes to subgroups, finite direct products, free products, and
 colimits of directed systems of groups (with arbitrary structure maps). 
  For more information we refer for instance
  to~\cite{Bartels-Farrell-Lueck(2014), Bartels-Lueck(2012annals),Bartels-Lueck-Reich(2008hyper),
    Farrell-Jones(1993a),  Kammeyer-Lueck-Rueping(2016), Lueck-Reich(2005),Wegner(2015)}.
\end{remark}

\begin{theorem}[Basic properties of the $V$-twisted $L^2$-torsion  for finite  free  $G$-$CW$-com\-plexes]
\label{the:Basic_properties_of_the_V-twisted_L2-torsion}
Let $X$ be a finite free $G$-$CW$-complex and let $V = (V;[B_V])$ be a based finite-dimensional representation.

\begin{enumerate}

\item \label{the:Basic_properties_of_the_V-twisted_L2-torsion:changing_the_base_refinement}
\emph{Dependency on the base refinement.}\\
Let $[B_X]$ and $[B_X']$ be two base refinements.
Suppose that $X$ is  $V$-twisted $\det$-$L^2$-acyclic.
Put $H_1(G)_f := H_1(G)/\tors(H_1(G))$. Let
\[
D_V \colon H_1(G)_f \to \IR
\]
be the homomorphism of abelian groups
which sends $\overline{g} \in H_1(G)_f$ represented by $g \in G$ to
$\ln\bigl(\bigl|\det_{\IC}(l_g \colon V \to V)\bigr|\bigr)$. Let
$\trans([B_X],[B_X'])_f \in H_1(G)_f$ be given by the element 
$\trans([B_X],[B_X']) \in H_1(G) $ defined in~\eqref{trans} for $C_*(X)$.

Then we get
\[
\quad \quad \quad
\rho^{(2)}_G(X;V,[B_X'])  - \rho^{(2)}_G(X;V,[B_X])  = D_V\bigl(\trans([B_X],[B_X'])_f\bigr);
\]

\item \label{the:Basic_properties_of_the_V-twisted_L2-torsion:homotopy_invariance}
\emph{$G$-homotopy invariance.}\\
Let $X$ and $Y$ be  finite free $G$-$CW$-complexes.
Let $[B_X]$ and $[B_Y]$ be  base refinements for $X$ and $Y$.
Let $f \colon X \to Y$ be a $G$-homotopy equivalence.
Denote by 
\[
\tau\bigl(C_*(f) \colon (C_*(X);[B_X]) \to (C_*(Y);[B_Y])\bigr) \in \widetilde{K}_1(\IZ G)
\]
the Whitehead torsion of the $\IZ G$-chain homotopy equivalence $C_*(f)$. (This is well-defined as an
element in $\widetilde{K}_1(\IZ G)$ since we have fixed equivalence classes of $\IZ
G$-basis and not only a cellular equivalence classes of $\IZ G$-basis.) The projection 
$\pr \colon G \to H_1(G)_f$ and the determinant over the commutative ring 
$\IZ[H_1(G)_f]$ induce homomorphisms
\[
\widetilde{K}_1(\IZ G) \xrightarrow{\pr_*} \widetilde{K}_1(\IZ[H_1(G)_f]) 
\xrightarrow{{\det}_{\IZ[H_1(G)_f]}} \IZ[H_1(G)_f]^{\times}/\{\pm 1\}.
\]
The homomorphism
\[
\psi \colon H_1(G)_f  \xrightarrow{\cong} \IZ[H_1(G)_f]^{\times}/\{\pm 1\} \quad x \mapsto \pm x
\]
is an isomorphism. 
Let 
\[
m(f_*,[B_X],[B_Y]) \in H_1(G)_f
\] 
be the image of
$\tau(C_*(f))$ under the composite
\[
\psi^{-1} \circ {\det}_{ \IZ[H_1(G)_f]} \circ \pr_* \colon \widetilde{K}_1(\IZ G) \to H_1(G)_f.
\]
Suppose that the $K$-theoretic Farrell-Jones Conjecture is true for $\IZ G$ or that $f$ is a simple homotopy equivalence.
Assume that $X$ is $V$-twisted $\det$-$L^2$-acyclic.

Then $Y$ is $V$-twisted $\det$-$L^2$-acyclic and we get 
\[
\quad \quad \quad 
\rho^{(2)}_G(Y;V,[B_Y]) - \rho^{(2)}_G(X;V,[B_X])
= D_V\bigl(m(f_*,[B_X],[B_Y])\bigr);
\]

\item \label{the:Basic_properties_of_the_V-twisted_L2-torsion:sum_formula}
\emph{Sum formula.}\\
Consider a $G$-pushout of finite free $G$-$CW$-complexes
\[
\xymatrix{
X_0 \ar[r]^{i_1} \ar[d]_{i_2}
& 
X_1 \ar[d]^{j_1}
\\
X_2 \ar[r]_{j_2} 
&
X
}
\]
where $i_1$ is cellular, $i_0$ an inclusion of $G$-$CW$-complexes and $X$ has the obvious
$G$-$CW$-structure coming from the ones on $X_0$, $X_1$ and $X_2$.  Suppose that $X_0$,
$X_1$ and $X_2$ are $V$-twisted $\det$-$L^2$-acyclic.  Equip $X_i$ with base refinements
$[B_{X_i}]$ for $i = 0,1,2$ and $X$ with a base refinement $[B_{X}]$ which are
compatible with one another in the obvious sense. 

Then $X$ is $V$-twisted $\det$-$L^2$-acyclic and we get
\begin{multline*}
\quad \quad \quad \rho^{(2)}_G(X;V,[B_X]) 
\\
= \rho^{(2)}_G(X_1;V,[B_{X_1}]) 
+ \rho^{(2)}_G(X_2;V,[B_{X_2}]) - \rho^{(2)}_G(X_0;V,[B_{X_0}]);
\end{multline*}

\item \label{the:Basic_properties_of_the_V-twisted_L2-torsion:product_formula}
  \emph{Product formula.}\\
  Let $G$ and $H$ be groups. Let $X$ be a finite free $G$-$CW$-complex and $Y$ be a finite
  free $H$-$CW$-complex which come with base-refinements $[B_X]$ and $[B_Y]$.  Equip
  $X \times Y$ with the induced base refinement $[B_{X \times Y}]$.  Let $V$ be a based
  finite-dimensional  $G \times H$-representation. Let $i_G^* V$ be the 
  based  finite-dimensional  $G$-representation obtained from $V$ by restriction with the
  inclusion $i_G \colon G \to G \times H$. Suppose that $X$ is $i_G^*V$-twisted
  $\det$-$L^2$-acyclic.

Then $X \times Y$ is a finite free $G \times H$-$CW$-complex which is 
$V$-twisted $\det$-$L^2$-acyclic and we get
\[
\rho^{(2)}_{G \times H}(X \times Y; V,[B_{X \times Y}])  
= \chi(Y/H) \cdot \rho^{(2)}_G(X;i_G^*V,[B_{X}]);
\]

\item \label{the:Basic_properties_of_the_V-twisted_L2-torsion:induction}
  \emph{Induction.}\\
  Let $i \colon H \to G$ be the inclusion of a subgroup $H$ of $G$.  Let $i^*V$ be the restriction of $V$ to
  $H$. Let $X$ be a finite free $H$-$CW$-complex which is $V$-twisted
  $\det$-$L^2$-acyclic. Equip $X$ with a base refinement $[B_X]$.

  Then $G \times_HX$ is a finite free $G$-$CW$-complex which is $V$-twisted
  $\det$-$L^2$-acyclic. Moreover, if we equip $G \times_HX$ with the obvious base refinement
  $[B_{G\times_H X}]$ coming from $[B_X]$, we get
\[
\rho^{(2)}_G(G \times_H X; V,[B_{G \times_H X}]) = \rho^{(2)}_H(X; i^*V,[B_{ X}]);
\]

\item \label{the:Basic_properties_of_the_V-twisted_L2-torsion:restriction}
  \emph{Restriction.}\\
  Let $i \colon H \subset G$ be the inclusion of a subgroup  $H$  of $G$ of finite index.  Equip $X$ with a base
  refinement $[B_X]$. Denote by $i^*X$ the restriction of $X$ to $H$ which is a finite free
  $H$-$CW$-complex.  Fix a map of sets $\sigma \colon H\backslash G \to G$ whose composite
  with the projection $G \to H\backslash G$ is the identity. Choose any representative
  $B_X$ of $[B_X]$. Put $i^*B_X = \{ \sigma(z) \cdot b \mid z \in H\backslash G, b  \in B_X\}$ 
  and equip $i^*X$ with the base refinement $[i^* B_X]$ given
  by $i^* B_X$. Let $i^* V$ be the based finite-dimensional $H$-representation
  obtained from $V$ by restriction to $H$.

  Then $i^*X$ is $i^* V$-twisted $\det$-$L^2$-acyclic if and only if $X$ is $V$-twisted
  $\det$-$L^2$-acyclic, and we get in this case
\[
\rho^{(2)}_G(X; V, [B_X]) = [G:H] \cdot  \rho^{(2)}_H(i^*X; i^*V,[i^* B_{ X}]).
\]
In particular $\rho^{(2)}_H(i^*X; i^*V,[i^* B_{ X}])$ is independent of the choice of $\sigma \colon H\backslash G \to G$;

\item \label{the:Basic_properties_of_the_V-twisted_L2-torsion:Poincare_duality}
  \emph{Poincar\'e duality.}\\
  Let $X$ be a finite free $G$-$CW$-complex such that $X/G$ is a finite orientable $n$-dimensional (not necessarily simple)
  Poincar\'e complex, e.g., a 
  cocompact free proper smooth $G$-manifold of dimension $n$ without boundary such that $X$
  is orientable and the $G$-action is orientation preserving.
  Let $[B_X]$ be a base refinement. Denote by $V^*$ the dual of $V = (V,[B_V])$. Suppose
  that $X$ is $V$-twisted $\det$-$L^2$-acyclic.  Suppose that the
  $K$-theoretic Farrell-Jones Conjecture holds for $\IZ G$. Poincare duality induces a $\IZ G$-chain homotopy equivalence
  $P_* \colon C^{n-*}(X) \to C_*(X)$. Let $m(X,[B_X]) \in H_1(G)_f$ be the image of 
  $\tau\bigl(P \colon (C^{n-*}(X),[B_X^{n-*}]) \to (C_*(X),[B_X])\bigr) \in \widetilde{K_1}(\IZ G)$
  under the composite 
 \[
\psi^{-1} \circ {\det}_{\IZ[H_1(G)_f]} \circ \pr_* \colon \widetilde{K}_1(\IZ G) \to H_1(G)_f
 \]
 defined in assertion~\eqref{the:Basic_properties_of_the_V-twisted_L2-torsion:homotopy_invariance}.
 Let $D_V \colon H_1(G)_f \to \IR$ be the homomorphism defined in 
 assertion~\eqref{the:Basic_properties_of_the_V-twisted_L2-torsion:changing_the_base_refinement}.

Then $X$ is $V^*$-twisted $\det$-$L^2$-acyclic and we get 
\[
\quad \quad \quad 
\rho^{(2)}_G(X;V^*,[B_X]) - (-1)^{n+1} \cdot \rho^{(2)}_G(X;V,[B_X]) 
= D_V(m(X,[B_X])).
\]

\item \emph{Additivity in $V$ in the $\det$-$L^2$-acyclic case.}\\
\label{the:Basic_properties_of_the_V-twisted_L2-torsion:additivity_in_V}
Fix a base refinement $[B_X]$ of $X$.
Let $0 \to U_0 \to U_1 \to U_2 \to 0$ be an exact sequence of finite-dimensional  $G$-representations.
Choose any equivalence class of $\IC$-basis $[B_{U_i}]$ on $U_i$ for $i \in \{0,1,2\}$.
Suppose that  $X$ is  $U_i$-twisted $\det$-$L^2$-acyclic for at least two elements $i \in \{0,1,2\}$.

Then $X$ is  $U_i$-twisted $\det$-$L^2$-acyclic for all $i \in \{0,1,2\}$ and we get
\[
\rho^{(2)}_G(X;U_0,[B_X]) - \rho^{(2)}_G(X;U_1,[B_X]) + \rho^{(2)}_G(X;U_2,[B_X]) = 0.
\]  

In particular $\rho^{(2)}_G(X;V,[B_X])$ is independent of the equivalence class of basis  $[B_V]$ for $V$, if $X$ is
$V$-twisted $\det$-$L^2$-acyclic.

\end{enumerate}
\end{theorem}
\begin{proof}~\eqref{the:Basic_properties_of_the_V-twisted_L2-torsion:changing_the_base_refinement}
Let $B_n$ and $B_n'$ be representatives for the equivalence class of $\IZ G$-basis of $C_n(X)$ given by $[B_X]$ and $[B_X']$.
We conclude  from~\cite[Theorem~3.35~(5) on page~143]{Lueck(2002)}
\begin{eqnarray*}
\lefteqn{\rho^{(2)}(X;V,[B_X']) - \rho^{(2)}(X;V,[B_X])}
& & 
\\
& = & 
\tau^{(2)}\bigl(\Lambda \circ \eta_{V}(\id_{C_*(X)}) \colon \Lambda \circ \eta_{V}(C_*(X),[B_X]) 
\to \Lambda \circ \eta_{V}(C_*(X),[B_X'])\bigr)
\\
& = & 
\sum_{n \in \IZ} (-1)^n \cdot \ln\bigl({\det}_{\caln(G)}\bigl(\Lambda \circ \eta_{V}(\id_{C_n}) 
\colon \Lambda \circ \eta_{V}(C_n(X),[B_n])
\\ & & \hspace{80mm} 
\to \Lambda \circ \eta_{V}(C_n(X),[B_n'])\bigr)\bigr).
\end{eqnarray*}
Choose a bijection $\sigma \colon B_n \to B_n'$ and for each $b \in B_n$ elements
$\epsilon(b) \in \{\pm 1\}$ and $g(b) \in G$ such that $\sigma(b) = \epsilon(b) \cdot g(b) \cdot b$. 
Then there is a commutative diagram of finitely generated Hilbert
$\caln(G)$-modules with isometric isomorphisms as vertical arrows
\[
\xymatrix@!C=13em{\Lambda \circ \eta_{V}(C_n,[B_n]) \ar[r]^{\Lambda \circ \eta_{V}(\id_{\IC_n})} \ar[d]_{\cong}
&
\Lambda \circ \eta_{V}(C_n,[B_n']) \ar[d]^{\cong}
\\
\Lambda \circ \eta_{V}\left(\bigoplus_{b \in B_n} \IC G\right) \ar[r]_{\Lambda \circ \eta_{V}(r_A)} 
&
\Lambda \circ \eta_{V}\left(\bigoplus_{b' \in B_n'} \IC G\right) 
}
\]
where the entry of the matrix $A$ for $(b,b') \in B_n \times B_n'$ is the element $\epsilon(b)^{-1} \cdot g(b)^{-1}$ if
$b' = \sigma(b)$ and  zero otherwise. We conclude
\begin{eqnarray*}
\lefteqn{\ln\bigl({\det}_{\caln(G)}\bigl(\Lambda \circ \eta_{V}(\id_{C_n}) 
\colon \Lambda \circ \eta_{V}\bigl(C_n(X),[B_n]) 
\to \Lambda \circ \eta_{V}\bigl(C_n(X),[B_n']) \bigr)\bigr)}
& & 
\\
& = & 
\ln\bigl({\det}_{\caln(G)}\bigl(\Lambda \circ \eta_{V}(r_A) \colon \Lambda \circ \eta_{V}\bigl(\bigoplus_{b \in B_n} \IC G\bigr)  
\to \Lambda \circ \eta_{V}\bigl(\bigoplus_{b' \in B_n'} \IC G\bigr)\bigr)\bigr)
\\
& = & 
\sum_{b \in B_n}
\ln\bigl({\det}_{\caln(G)}\bigl(\Lambda \circ \eta_{V}(r_{\epsilon(b)^{-1} \cdot g(b)^{-1}}) \colon \Lambda \circ \eta_{V}(\IC G)  
\to \Lambda \circ \eta_{V}(\IC G)\bigr)\bigr).
\end{eqnarray*}
The following diagram in $\FBMOD{\IC G}$ commutes, where the vertical arrows are base
preserving isomorphisms coming from the isomorphisms appearing 
in Lemma~\ref{lem:diagonal_versus_first_coordinate}
\[
\xymatrix@!C=10em{\eta_V(\IC G) \ar[rr]^{\eta_{V}(r_{\epsilon(b)^{-1}\cdot g(b)^{-1}})} \ar[d]^{\cong}
& & 
\eta_V(\IC G) \ar[d]_{\cong}
\\
(\IC G \otimes V)_1 \ar[r]_{\epsilon(b)^{-1} \cdot r_{g(b)^{-1}} \otimes \id_V}
&
(\IC G \otimes V)_1 
\ar[r]_{\id_{\IC G} \otimes l_{g(b)}}
&
(\IC G \otimes V)_1 
}
\]
We conclude from~\cite[Theorem~3.14~(1) and~(6) on page~128]{Lueck(2002)}
\begin{eqnarray*}
\lefteqn{\ln\bigl({\det}_{\caln(G)}\bigl(\Lambda \circ \eta_{V}(r_{\epsilon(b)^{-1} \cdot g(b)^{-1}})\bigr)\bigr)}
& &
\\
& = & 
\ln\bigl({\det}_{\caln(G)}\bigl(\Lambda(\epsilon(b)^{-1} \cdot r_{g(b)^{-1}} \otimes \id_V)\bigr)\bigr) +
\ln\bigl({\det}_{\caln(G)}\bigl(\Lambda(\id_{\IC G} \otimes l_{g(b)})\bigr)\bigr)
\\
& = & 
0 + \ln\bigl(\bigl|{\det}_{\IC}\bigl(l_{\epsilon(b) \cdot g(b)} \colon V\to V \bigr)\bigr|\bigr).
\end{eqnarray*}
This implies
\begin{eqnarray*}
\lefteqn{\ln\bigl({\det}_{\caln(G)}\bigl(\Lambda \circ \eta_{V}(\id_{C_n}) 
\colon \Lambda \circ \eta_{V}(C_n(X),[B_n]) 
\to \Lambda \circ \eta_{V}(C_n(X),[B_n']) \bigr)\bigr)}
& & 
\\
& \hspace{50mm}  = &
\sum_{b \in B} \ln\bigl(\bigl|{\det}_{\IC}\bigl(l_{g(b)} \colon V\to V \bigr)\bigr|\bigr)
\\
& \hspace{50mm}  = &
\ln\bigl(\bigl|{\det}_{\IC}\bigl(l_{\prod_{b \in B} g(b)} \colon V\to V \bigr)\bigr|\bigr)
\\
& \hspace{50mm}  = & 
D_V\bigl(\trans([B_X],[B_X'])_f\bigr).
\end{eqnarray*}
This finishes the proof of 
assertion~\eqref{the:Basic_properties_of_the_V-twisted_L2-torsion:changing_the_base_refinement}.
\\[2mm]~\eqref{the:Basic_properties_of_the_V-twisted_L2-torsion:homotopy_invariance}
(Notice that 
assertion~\eqref{the:Basic_properties_of_the_V-twisted_L2-torsion:changing_the_base_refinement}
is a special case of 
assertion~\eqref{the:Basic_properties_of_the_V-twisted_L2-torsion:homotopy_invariance},
but we do not need the assumption that
the $K$-theoretic Farrell-Jones Conjecture holds for $\IZ G$.)
We get from~\cite[Theorem~3.35~(5) on page~142]{Lueck(2002)}
that  $Y$ is $V$-twisted $\det$-$L^2$-acyclic and
\begin{multline*}
\rho^{(2)}(Y;V,[B_Y]) - \rho^{(2)}(X;V,[B_X])
\\
= \tau^{(2)}\left(\Lambda \circ \eta_{V}(C_*(f)) \colon \Lambda \circ \eta_{V}\bigl(C_*(X),[B_X]) 
\to \Lambda \circ \eta_{V}\bigl(C_*(Y),[B_Y])\right).
\end{multline*}
One easily checks that we obtain a well-defined homomorphism
\[
\mu \colon \widetilde{K}_1(\IZ G) \to \IR^{>0}
\]
by sending the class in $\widetilde{K}_1(\IZ G)$ represented by the invertible matrix $A \in Gl_n(\IZ G)$ to 
$\det_{\caln(G)}\bigl(\Lambda \circ \eta_{V}\bigl(r_A \colon \IZ G^n \to \IZ G^n\bigr)\bigr)$.
We leave it to the reader to check by inspecting  the definitions of   Whitehead torsion and $L^2$-torsion 
that the homomorphism  $\mu$ sends the Whitehead torsion $\tau\bigl(C_*(f) \colon (C_*(X),[B_X]) \to (C_*(Y);[B_Y])\bigr)$
to the $L^2$-torsion $\tau^{(2)}\left(\Lambda \circ \eta_{V}(C_*(f) \colon \Lambda \circ \eta_{V}\bigl(C_*(X),[B_X]) 
\to \Lambda \circ \eta_{V}\bigl(C_*(Y),[B_Y])\right)$.

The map $\psi$ is obviously bijective if $H_1(G)_f$ is finitely generated and hence isomorphic to $\IZ^m$ for some $m$.
By a colimit argument over the directed set of finitely generated subgroups of $H_1(G)_f$ one shows that
$\psi$ is bijective in general. 

If $f$ is  simple, then 
$\tau\bigl(C_*(f) \colon (C_*(X);[B_X]) \to (C_*(Y);[B_Y])\bigr)$
is contained in the image of the map 
$\iota \colon H_1(G)\xrightarrow{\cong}  \widetilde{K}_1(\IZ G)$ sending $g \in G$ 
to the class of the invertible $(1,1)$-matrix having $g$ as non-trivial entry.
Then the claim follows by a direct inspection since $\psi^{-1} \circ \det_{\IZ[H_1(G)_f]} \circ \pr_*$ sends the class 
of the $(1,1)$-matrix $(g)$ in $\widetilde{K}_1(\IZ G)$
to the class of $g$ in $H_1(G)_f$.

The remaining and hard step in the proof is to show that $\mu$ agrees with
$D_V \circ \psi^{-1} \circ {\det}_{\IZ[H_1(G)_f]} \circ \pr_*$. In the case that $G$ is torsionfree,
this follows from the conclusion of the $K$-theoretic Farrell-Jones Conjecture 
that $\Wh(G)$ is trivial and hence $f$ is a simple $G$-homotopy equivalence.
The general case is done as follows.

Consider the homomorphism 
\[
\nabla(G) \colon  \IQ \otimes_{\IZ} K_1(\IZ G) \to \IR, 
\quad r \otimes x \mapsto r \cdot \left(D_V \circ \psi^{-1} \circ {\det}_{\IZ[H_1(G)_f]} \circ \pr_*(x) - \mu(x)\right).
\]
It remains to show that  $\nabla(G)$ is trivial.

We have the following composition of assembly maps
\begin{multline*}
H_1^G(EG;\bfK_{\IZ}) \xrightarrow{\asmb_{\caltr,\calfin}}
H_1^G(\EGF{G}{\calfin};\bfK_{\IZ}) 
\\
\xrightarrow{\asmb_{\calfin,\calvcyc}}
H_1^G(\EGF{G}{\calvcyc};\bfK_{\IZ}) \xrightarrow{\asmb}
H_1^G(G/G;\bfK_{\IZ}) = K_1(\IZ G).
\end{multline*}
The map $\asmb$ is an isomorphism since we assume that the 
$K$-theoretic Farrell-Jones Conjecture holds for $\IZ G$.
The map $\asmb_{\calfin,\calvcyc}$ is rationally a bijection, see~\cite[Theorem~5.11]{Grunewald(2008Nil)}
or~\cite[Theorem~0.3]{Lueck-Steimle(2016splitasmb)}. Hence it suffices to show that
the following composite is trivial.
\begin{multline*}
\IQ \otimes_{\IZ} H_1^G(\EGF{G}{\calfin};\bfK_{\IZ}) 
\xrightarrow{\id_{\IQ} \otimes_{\IZ}\asmb_{\calfin,\calvcyc}}
\IQ \otimes_{\IZ} H_1^G(\EGF{G}{\calvcyc};\bfK_{\IZ}) 
\\
\xrightarrow{\id_{\IQ} \otimes_{\IZ}\asmb}
\IQ \otimes_{\IZ} H_1^G(G/G;\bfK_{\IZ}) = \IQ \otimes_{\IZ} K_1(\IZ G) \xrightarrow{\nabla(G)} \IR.
\end{multline*}
The equivariant Chern character, see~\cite[Theorem~0.1]{Lueck(2002b)}, gives for every
proper $G$-$CW$-complex $X$ natural isomorphisms
\[
\ch^G_1(X)  \colon  \bigoplus_{i = 0}^{\infty} \IQ \otimes_{\IZ}  H_i^G(X;K_{1-i}(\IZ {?}))  
\xrightarrow{\cong}  \IQ \otimes_{\IZ} H_1^G(X;\bfK_{\IZ}),
\]
where $H_i^G(X;K_j(\IZ {?}))$ is the $i$th Bredon homology of $X$ with coefficients in the
covariant functor from the orbit category of $G$ to the category of $\IZ$-modules sending
$G/H$ to $K_j(\IZ H)$.  Denote by $\ch^G_1(X)_i$ the restriction of $\ch^G_1(X)$ to the
$i$th summand.  Then it remains to show for $i = 0,1,2,3, \ldots $ that the composite
\begin{multline*}
c_i \colon \IQ \otimes_{\IZ}  H_i^G(\EGF{G}{\calfin};K_{1-i}(\IZ {?}))  
\xrightarrow{\ch^G_1(X)_i}  \IQ \otimes_{\IZ} H_1^G(\EGF{G}{\calfin};\bfK_{\IZ}) 
\\
\xrightarrow{\id_{\IQ} \otimes_{\IZ}\asmb_{\calfin,\calvcyc}}
\IQ \otimes_{\IZ} H_1^G(\EGF{G}{\calvcyc};\bfK_{\IZ}) 
\\
\xrightarrow{\id_{\IQ} \otimes_{\IZ}\asmb}
\IQ \otimes_{\IZ} H_1^G(G/G;\bfK_{\IZ}) = \IQ \otimes_{\IZ} K_1(\IZ G) \xrightarrow{\nabla(G)} \IR
\end{multline*}
is trivial.

We begin with $c_0$. There is a natural isomorphism
\[
\beta \colon \colim_{G/H \in \Or(G;\calfin)} K_1(\IZ H) 
\xrightarrow{\cong} H_0^G(\EGF{G}{\calfin};K_1(\IZ {?})),
\]
where $\Or(G;\calfin)$ is the category whose objects are homogeneous spaces $G/H$ with
$|H| < \infty$ and whose morphisms are $G$-maps.  Since the canonical map
\[
\bigoplus_{\substack{H \subseteq G\\ |H| < \infty}} \alpha_H \colon 
\bigoplus_{\substack{H \subseteq G\\ |H| < \infty}} K_1(\IZ H) 
\to \colim_{G/H \in \Or(G;\calfin)} K_1(\IZ H)
\]
is surjective it suffices to show that the composite 
$c_0 \circ \beta \circ \alpha_H \colon K_1(\IZ H) \to \IR$ is trivial for each
finite subgroup $H \subseteq G$. Notice that the definition of $\nabla(G)$ makes sense for any subgroup $H \subseteq G$
and that the composite above can be identified with the map
$\nabla(H) \colon K_1(\IZ H) \to \IR$. But $\nabla(H)$ is trivial as $H_1(H)_f$ is the trivial group.
Hence we have shown that $c_0$ is trivial.

Next we treat $c_1$. If $H$ is a finite group, the inclusion $\IZ \to K_0(\IZ H)$ sending
$n$ to $n \cdot [\IZ H]$ is split injective with finite cokernel, 
see~\cite[Theorem~8.1 and Proposition~9.1]{Swan(1960a)}.  If $\underline{\IZ}$ is the constant 
covariant $\IZ\Or(G;\calfin)$-module with value $\IZ$, we obtain a transformation of covariant 
$\IZ\Or(G;\calfin)$-modules $u \colon \underline{\IZ} \to K_0(\IZ {?})$ whose evaluation at
each object is injective with finite cokernel.  Therefore it induces an isomorphism of
covariant $\IQ \Or(G;\calfin)$-modules
\[
\underline{\IQ} \xrightarrow{\cong} \IQ \otimes_{\IZ} K_0(\IZ {?}).
\]
Hence we obtain for every proper $G$-$CW$-complex a natural isomorphism
\[
H_1(X/G;\IQ) \xrightarrow{\cong} H_1^G(X;\underline{\IQ}) 
\xrightarrow{\cong} H_1^G(X;\IQ \otimes_{\IZ} K_0(\IZ {?}))
\xrightarrow{\cong} \IQ \otimes_{\IZ} H_1^G(X;K_0(\IZ {?})).
\]
We have the following commutative diagram
\[
\xymatrix{H_1(EG/G;\IQ) \ar[r]^-{\cong} \ar[d]
& 
\IQ \otimes_{\IZ} H_1^G(EG/G;K_{0}(\IZ {?})) \ar[d]
\\
H_1(\EGF{G}{\calfin}/G;\IQ) \ar[r]^-{\cong} 
& 
\IQ \otimes_{\IZ} H_1^G(\EGF{G}{\calfin}/G;K_{0}(\IZ {?}))
}
\]
The left vertical arrow is bijective since both $C_*(EG) \otimes_{\IZ} \IQ$ and
$C_*(\EGF{G}{\calfin}) \otimes_{\IZ} \IQ$ are projective $\IQ G$-resolutions of the
trivial $\IQ G$-module $\IQ$.  Hence the right vertical arrow is an isomorphisms.  
A direct inspection of the definitions shows that the composite
\begin{multline*}
H_1(G)_f \otimes_{\IZ} \IQ \xrightarrow{\cong} H_1(EG/G;\IQ) 
\xrightarrow{\cong} \IQ \otimes_{\IZ} H_1^G(EG/G;K_{0}(\IZ {?}))
\\
\xrightarrow{\cong} \IQ \otimes_{\IZ} H_1^G(\EGF{G}{\calfin}/G;K_{0}(\IZ {?}))
\xrightarrow{c_1} \IR
\end{multline*}
is trivial, since $\psi^{-1} \circ \det_{\IZ[H_1(G)_f]} \circ \pr_*$ sends the class 
of the $(1,1)$-matrix $(g)$ in $\widetilde{K}_1(\IZ G)$
to the class of $g$ in $H_1(G)_f$. Hence $c_1$ is trivial.

Next we show that $c_2$ is trivial.
Since the map $\nabla(G) \colon  \IQ \otimes_{\IZ} K_1(\IZ G) \to \IR$ factorizes through the 
change of rings map  $K_1(\IZ G) \to K_1(\IQ G)$, the map $c_2$ factorizes through the map
\[
\IQ \otimes_{\IZ}  H_2^G(X;K_{-1}(\IZ {?})) \to \IQ \otimes_{\IZ}  H_2^G(X;K_{-1}(\IQ {?})).
\]
For any finite group $H$ the ring $\IQ H$ is semisimple and in particular regular and hence
$K_{-1}(\IQ {?}) = 0$. This implies that $H_2^G(X;K_{-1}(\IQ {?}))$ vanishes. Hence the map 
$c_2$ is trivial.

Since the coefficient system $K_{1-i}(\IZ?)$ is identically zero for $i \ge 3$ 
by~\cite{Carter(1980b)}, the map $c_i$ is trivial for $i \ge 3$.
This finishes the proof of 
assertion~\eqref{the:Basic_properties_of_the_V-twisted_L2-torsion:homotopy_invariance}.
\\[2mm]~\eqref{the:Basic_properties_of_the_V-twisted_L2-torsion:sum_formula}
One obtains an exact sequence of based finite free  $\IZ G$-chain complexes
\[
0 \to (C_*(X_0),[B_{X_0}]) \xrightarrow{i_*} (C_*(X_1),[B_{X_1}])  \oplus (C_*(X_2),[B_{X_2}])  
\xrightarrow{p_*}  (C_*(X),[B_{X}])  \to 0
\]
where the $\IZ G$-bases are respected in the obvious way. It induces an exact sequence
of Hilbert $\caln(G)$-chain complexes
\begin{multline*}
0 \to \Lambda \circ \eta_V(C_*(X_0),[B_{X_0}])
\\
\xrightarrow{\Lambda \circ \eta_V(i_*)} 
\Lambda \circ \eta_V(C_*(X_1),[B_{X_1}])  \oplus \Lambda \circ \eta_V(C_*(X_2),[B_{X_2}])
\\
\xrightarrow{\Lambda \circ \eta_V(p_*)}  \Lambda \circ \eta_V(C_*(X),[B_{X}])  \to 0
\end{multline*}
such that $\Lambda \circ \eta_V(i_n)$ is an isometric embedding and the map 
$\Lambda \circ \eta_V(p_n)$ induces an isometric isomorphism 
$\ker(\Lambda \circ \eta_V(p_n))^{\perp} \to \Lambda \circ \eta_V\bigl(C_n(X),[B_n^{X}])\bigr)$ 
for all $n \ge 0$. Now assertion~\eqref{the:Basic_properties_of_the_V-twisted_L2-torsion:sum_formula} 
follows from~\cite[Theorem~3.35~(1) on page~142]{Lueck(2002)}.
\\[2mm]~\eqref{the:Basic_properties_of_the_V-twisted_L2-torsion:product_formula} The
product formula in the case $Y = H$ is actually equivalent to
assertion~\eqref{the:Basic_properties_of_the_V-twisted_L2-torsion:induction} applied to
the inclusion $G \to G \times H$.  By an elementary argument using homotopy
invariance (without the Farrell-Jones Conjecture) we get the product formula also for 
$Y = G \times D^n$.  Now one uses induction over the dimension $d$ of $Y$ and subinduction over
the number of $d$-cells. In the induction step one writes $Y$ as a cellular $G$-pushout
\[
\xymatrix{H \times S^{d-1} \ar[r]\ar[d] 
&
Y' \ar[d]
\\
H \times D^d \ar[r]
& Y
}
\]
observes that taking the product with $X$ yields a cellular $G \times H$-pushout
\[
\xymatrix{X \times H \times S^{d-1} \ar[r]\ar[d] 
&
X \times Y' \ar[d]
\\
X \times H \times D^d \ar[r]
& X \times Y
}
\]
applies the sum formula proved in 
assertion~\eqref{the:Basic_properties_of_the_V-twisted_L2-torsion:sum_formula}
using the fact that we know the product formula already for  $X \times H \times S^{d-1}$, $X \times H \times D^d$ 
and $X \times Y'$, and applies the sum formula for the (classical) Euler characteristic.
This finishes the proof of assertion~\eqref{the:Basic_properties_of_the_V-twisted_L2-torsion:product_formula}
\\[2mm]~\eqref{the:Basic_properties_of_the_V-twisted_L2-torsion:induction}
There is a canonical isomorphism of $\IZ G$-chain complexes
\[
\IC G \otimes _{\IC H} C_*(X) \xrightarrow{\cong} C_*(G \times_H X)
\]
If we equip the source with the equivalence class of basis $[i_* B^X]$ given by
$\{1 \otimes b \mid b \in B^X_*\}$, it is compatible with $[i_* B^X]$ and
$[B^{G \times_H X}_*]$. For any $\IC H$-module $M$ we have the isomorphism
of $\IZ G$-modules, natural in $M$
\[
\IC G \otimes_{\IC H} (M \otimes i^*V)_d \xrightarrow{\cong} \bigl((\IC G \otimes_{\IC H} M) \otimes V\bigr)_d,
\quad g \otimes m \otimes v \mapsto g \otimes m \otimes gv.
\]
We obtain a $\IZ G$-isomorphism 
\begin{multline*}
\IC G \otimes_{\IC H} \eta^H_{i^*V}(C_*(X),[B^X_*]) 
\xrightarrow{\cong}
\eta^G_V(\IC G \otimes _{\IC H} C_*(X),i_*[B^X]) 
\\
\xrightarrow{\cong} \eta^G_V(C_*(G \times_H X),[B^{G \times_H X}]),
\end{multline*}
which is compatible with the equivalence classes of $\IZ G$-basis.
This induces an isometric $G$-equivariant isomorphism of 
$\caln(G)$-Hilbert chain complexes
\[
i_* \left(\Lambda^H \circ \eta^H_{i^* V}(C_*(X),[B^X_*])\right) 
\xrightarrow{\cong} \Lambda^G \circ \eta_V^G(C_*(G \times_H X),[B^{G \times_H X}_*]),
\]
where $i_*$ denotes induction for Hilbert modules, see~\cite[Section~1.1.5]{Lueck(2002)}.
We conclude from~\cite[Theorem~3.35~(8) on page~143]{Lueck(2002)}
\begin{eqnarray*}
\rho^{(2)}_G\bigl(G \times_H X; V,[B_{G \times_H X}]\bigr) 
& = & 
\rho^{(2)}_{G}\bigl(i_* \Lambda^H \circ \eta^H_{i^* V}(C_*(X),[B^X_*])\bigr)
\\
& = &
\rho^{(2)}_{H}\bigl(\Lambda^H \circ \eta^H_{i^* V}(C_*(X),[B^X_*])\bigr)
\\
& = & 
\rho^{(2)}_H(X; i^*V,[B_{ X}]).
\end{eqnarray*}
% This finishes the proof of
% assertion~\eqref{the:Basic_properties_of_the_V-twisted_L2-torsion:induction}.
% \\[2mm]~
\eqref{the:Basic_properties_of_the_V-twisted_L2-torsion:restriction} 
Let $(M,[B_M])$ be an object in $\FBMOD{\IC G}$. Its restriction with $i \colon H \to G$ is the object in
$\FBMOD{\IC H}$ given by $(i^*M,i^*[B_M])$, where $i^*M$ is the $\IC H$-module obtained from
the $\IC G$-module $M$ by restricting the $G$-action to an $H$-action by $i$ and $i^*[B_M]$
is given by the class of $i^*B_M = \{\sigma(z) \cdot b \mid z \in H\backslash G\}$ for some
representative $B_M$ of $[B_M]$. Then the identity $\id \colon i^*(M \otimes V)_d \to (i^*M \otimes i^*V)_d$ induces 
an isomorphism in $\FBMOD{\IC H}$
\[
T(M) \colon i^*\eta^G_V(M,[B_M]) \xrightarrow{\cong}  \eta^H_{i^*V}(i^*M,i^*[B_M])
\]
which is \emph{not} compatible with the equivalence class of $\IC H$-basis. 
Namely, the source is  $i^*(M \otimes V)_d$ equipped with the equivalence class of  $\IC H$-basis 
given by $\{\sigma(z) \cdot b \otimes \sigma(z) \cdot v \mid b \in B_M, v \in B_V, z \in H\backslash G\}$
and the target is $(i^*M \otimes i^*V)_d$ equipped with the equivalence class of  $\IC H$-basis 
given by $\{\sigma(z) \cdot b \otimes  v \mid b \in B_M, v \in B_V, z \in H\backslash G\}$.
Hence we get a commutative diagram of based finitely generated free $\IZ H$-modules
\[
\xymatrix@!C=21em{\bigoplus_{b \in B_M}  \bigoplus_{z \in H \backslash G} (\IC H \otimes V)_d
\ar[r]^{\bigoplus_{b \in B_M}  \bigoplus_{z \in H \backslash G} \id_{\IC H} \otimes l_{\sigma(z)}} \ar[d]_{\omega_1}^{\cong}
&
\bigoplus_{b \in B_M}  \bigoplus_{z \in H \backslash G} (\IC H \otimes V)_d
\ar[d]^{\omega_2}_{\cong}
\\
i^*\eta^G_V(M,[B_M]) \ar[r]_{T(M)}^{\cong} 
&
\eta^H_{i^*V}(i^*M,i^*[B_M])
}
\]
where we equip $(\IC H \otimes V)_d$ with the $\IC H$-basis $\{1 \otimes v \mid v \in V\}$
and the left and right upper corner with the same basis given by the disjoint union over $B \times H\backslash G$ 
of the basis for the summands $\IC H \otimes V$
and $\omega_1$ and $\omega_2$ are the obvious base preserving isomorphisms.
This implies using~\cite[Theorem~3.14~(1) and~(6) on page~128]{Lueck(2002)}
\begin{eqnarray*}
\lefteqn{{\det}_{\caln(H)}\bigl(\Lambda^H(T(M)) \colon \Lambda^H(i^*\eta^G_V(M,[B_M]) 
\xrightarrow{\cong}  \Lambda^H(\eta^H_{i^*V}(i^*M,i^*[B_M]))\bigr)}
& & 
\\
& = & 
{\det}_{\caln(H)}\left(\bigoplus_{b \in B_M}  \bigoplus_{z \in H \backslash G} \Lambda^H(\id_{\IC H} \otimes l_{\sigma(z)}) \colon
\bigoplus_{b \in B_M}  \bigoplus_{z \in H \backslash G} \Lambda^H((\IC H \otimes V)_d)\right.
\\
& & \left.
\to \bigoplus_{b \in B_M}  \bigoplus_{z \in H \backslash G} \Lambda^H((\IC H \otimes V)_d)\right)
\\
& = & 
\prod_{b \in B_M}  \prod_{z \in H \backslash G}  {\det}_{\caln(H)}\left(\Lambda^H(\id_{\IC H} \otimes l_{\sigma(z)}) \colon
\Lambda^H((\IC H \otimes V)_d) \to  \Lambda^H((\IC H \otimes V)_d) \right)
\\
& = & 
\left(\prod_{z \in H \backslash G}  {\det}_{\caln(H)}\left(\Lambda^H(\id_{\IC H} \otimes l_{\sigma(z)}) \colon
\Lambda^H((\IC H \otimes V)_d) \to  \Lambda^H((\IC H \otimes V)_d) \right)\right)^{\dim_{\IC G}(M)}.
\end{eqnarray*}
If $(C_*,[B_{C_*}])$ is a  bounded chain complex over $\FBMOD{\IC G}$, we obtain an isomorphism of $\IC H$-chain complexes
\[
T(C_*,[B_{C_*}]) \colon i^*\eta^G_V(C_*,[B_{C_*}]) \xrightarrow{\cong}  \eta^H_{i^*V}(i^*C_*,i^*[B_{C_*}])
\]
satisfying
\begin{eqnarray*}
\tau^{(2)}\bigl(\Lambda^H(T(C_*,[B_{C_*}]))\bigr) = \chi_{\IC G}(C_*) \cdot \sum_{z \in H\backslash G} 
\ln\left({\det}_{\caln(H)}\bigl(\Lambda^H(\id_{\IC H} \otimes l_{\sigma(z)})\bigr)\right).
\end{eqnarray*}
If $\Lambda^G(C_*)$ is $\det$-$L^2$-acyclic, then 
\begin{eqnarray*}
\chi_{\IC G}(C_*) 
& = & 
\chi^{(2)}\bigl(\Lambda^G(C_*,[B_{C_*}])\bigr)
\\
 & = & 
\sum_{n \ge 0} (-1)^n \cdot b_n^{(2)}\bigl(\Lambda^G(C_*,[B_{C_*}])\bigr)
\\
& = &
0,
\end{eqnarray*}
and we conclude from~\cite[Theorem~3.35~(5) on page~142]{Lueck(2002)}
\begin{eqnarray*}
\lefteqn{\rho^{(2)}_H\bigl(\Lambda^H(\eta^H_{i^*V}(i^*C_*,i^*[B_{C_*}]))\bigr)
-
\rho^{(2)}_H\bigl(\Lambda^H(i^*\eta^G_V(C_*,[B_{C_*}]) )\bigr)}
& & 
\\
& = & 
\tau^{(2)}\bigl(\Lambda^H(T(C_*,[B_{C_*}]))\bigr) 
\\
& = & 
\chi_{\IC G}(C_*) \cdot \sum_{z \in H\backslash G} 
\ln\left({\det}_{\caln(H)}\bigl(\Lambda^H(\id_{\IC H} \otimes l_{\sigma(z)})\bigr)\right).
\\
& = & 0.
\end{eqnarray*}
If we apply this to $(C_*,[B_{C_*}]) = (C_*(X),[B^X])$ and use the obvious identifications
$(i^*C_*(X),i^*[B^X]) = (C_*(i^*X),[B^{i^*X}])$ and
$\Lambda^H(i^*\eta^G_V(C_*,[B_{C_*}]) = i^* \Lambda^G \circ \eta^G_V(C_*,[B_{C_*}])$,
we conclude using~\cite[Theorem~3.35~(7) on page~143]{Lueck(2002)}
\begin{eqnarray*}
\rho^{(2)}_G(X; V, [B_X]) 
& = & 
\rho^{(2)}_G\bigl(\Lambda^G \circ \eta^G_V(C_*(X),[B_X])\bigr)
\\
& = & 
[G:H] \cdot \rho^{(2)}_H\bigl(i^*\Lambda^G \circ \eta^G_V(C_*(X),[B_X])\bigr)
\\
& = & 
[G:H] \cdot \rho^{(2)}_H\bigl(\Lambda^H(i^*\eta^G_V(C_*(X),[B_X]) )\bigr)
\\
& = & 
[G:H] \cdot \rho^{(2)}_H\bigl(\Lambda^H(\eta^H_{i^*V}(i^*C_*(X),i^*[B_X]) )\bigr)
\\
& = & 
[G:H] \cdot \rho^{(2)}_H(i^*X; i^*V,[i^* B_X]).
\end{eqnarray*}
This finishes the proof of assertion~\eqref{the:Basic_properties_of_the_V-twisted_L2-torsion:restriction}.
\\[2mm]~\eqref{the:Basic_properties_of_the_V-twisted_L2-torsion:Poincare_duality}
Let $V = (V,[B_V])$ be a based finite-dimensional $G$-representation. We conclude from Lemma~\ref{lem:twisting_and_duality}
that we have an isometric $G$-equivariant isomorphism of finite Hilbert $\caln(G)$-chain complexes
\[
\Lambda \circ \eta_{V^*,[B_V^*]}(C^{n-*}(X),[B_X^{n-*}]) 
\xrightarrow{\cong} \Lambda\bigl(\eta_{V,[B_V])}(C_*(X),[B_X])\bigr)^{n-*}\bigr).
\]
This implies 
\[\rho^{(2)}_G\left(\Lambda \circ \eta_{V^*}(C^{n-*}(X),[B_X^{n-*}])\right)
=
(-1)^{n+1} \cdot \rho^{(2)}_G\left(\Lambda \circ \eta_V(C_*(X),[B_X])\right).
\]
We conclude from the chain complex version of 
Theorem~\ref{the:Basic_properties_of_the_V-twisted_L2-torsion}~%
\eqref{the:Basic_properties_of_the_V-twisted_L2-torsion:homotopy_invariance}
applied to the Poincar\'e $\IZ G$-chain homotopy equivalence $P_* \colon C^{n-*}(X) \to C_*(X)$
and $V^*$
\begin{multline*}
\rho^{(2)}_G\left(\Lambda \circ \eta_{V^*}(C_*(X),[B_X])\right)
- \rho^{(2)}_G\left(\Lambda \circ \eta_{V^*}(C^{n-*}(X),[B_X^{n-*}])\right)
\\
=
D_{V^*}(m(X,[B_X]))
= D_V(m(X;[B_X])).
\end{multline*}
We conclude
\[
\rho^{(2)}_G(X;V^*,[B_X]) - (-1)^{n+1} \cdot \rho^{(2)}_G(X;V,[B_X])  = D_V(m(X,[B_X])).
\]
\eqref{the:Basic_properties_of_the_V-twisted_L2-torsion:additivity_in_V} This follows from
Lemma~\ref{lem:additivity_under_exact_sequences_of_reps}.
Hence  the proof of Theorem~\ref{the:Basic_properties_of_the_V-twisted_L2-torsion} is finished.
\end{proof}

%%%%%%%%%%%%%%%%%%%%%%%%%%%%%%%%%%%%%%%%%%%%%%%%%%%%%%%%%%%%%%%%%%%%%%%%%%%%%%%%%
%%%%%%%%%%%%%%%%%%%%%%%%%Section 7: Fuglede-Kadison determinants %%%%%%%%%%%%%%%%%%%%%%%%%%%%%%%%%
%%%%%%%%%%%%%%%%%%%%%%%%%%%%%%%%%%%%%%%%%%%%%%%%%%%%%%%%%%%%%%%%%%%%%%%%%%%%%%%%%

 \typeout{-----------------------------   Section 7: Fuglede-Kadison determinants ---------------------------}

\section{Fuglede-Kadison determinants}
\label{subsec:Fuglede-Kadison_determinants}

In this section we give a positive answer to Questions~\ref{que:L2-Betti_number_and_twisting} 
and~\ref{que:determinant_class} in an interesting special case in Theorem~\ref{the:Determinant_class_and_twisting}.

Consider a matrix $A \in M_{r,s}(\IC G)$. Let 
\begin{eqnarray}
\supp_G(A) \subseteq G
\label{supp_G(A)}
\end{eqnarray}
be the finite set of elements $g \in G$ for which there is at least one entry $a_{i,j}$ of
$A$ such that  the coefficient $m_g$ of $g$ in  $a_{i,j} = \sum_{h \in G}  m_h \cdot h$ is
non-trivial.

For an element $x = \sum_{g \in G} r_g \cdot g$ in $\IC G$ define $|x|_1 := \sum_{g \in G}
|r_g|$. Given a matrix $A \in M_{r,s}(\IC G)$ define
\begin{eqnarray}
||A||_1 & = & r\cdot s  \cdot \max\left\{ |a_{j,k}|_1 \mid 1 \le j \le r, 1 \le k \le s\right\}.
\label{||A||_1}
\end{eqnarray}

\begin{lemma}
\label{lem:properties_of_L1-norm}

\begin{enumerate}

\item \label{lem:properties_of_L1-norm:submultiplicative}
We have for $A \in M_{r,s}(\IC G)$ and $B \in M_{s,t}(\IC G)$
\[
||A \cdot B||_1 \le ||A||_1 \cdot ||B||_1;
\]

\item \label{lem:properties_of_L1-norm:norm_estimate}
We have for $A \in M_{r,s}(\IC G)$
\[
\left|\left|\Lambda^G(r_A) \colon L^2(G)^r \to L^2(G)^s\right|\right|
\le ||A||_1;
\]

\item \label{lem:properties_of_L1-norm:group_homomorphisms}
Let $\mu \colon G \to H$ be a group homomorphism.
Consider $A \in M_{r,s}(\IC G)$. Let $\mu(A) \in M_{r,s}(\IC H)$ be the image of
$A$ under the map $M_{r,s}(\IC G) \to M_{r,s}(\IC H)$ induced by $\mu$. Then
\[
||\mu(A)||_1 \le ||A||_1.
\]
\end{enumerate}
\end{lemma}
\begin{proof}~\eqref{lem:properties_of_L1-norm:submultiplicative}
This follows from the inequality $||x \cdot y||_1 \le ||x||_1 \cdot ||y||_1$ for $x,y \in \IC G$ and the triangle inequality.
\\[1mm]~\eqref{lem:properties_of_L1-norm:norm_estimate}
See~\cite[Lemma~13.33 on page~466]{Lueck(2002)}.
\\[1mm]~\eqref{lem:properties_of_L1-norm:group_homomorphisms}
This follows from the triangle inequality.
\end{proof}

If $C_*$ is an object in $\FBCC{\IZ G}$, then the expression $\eta(C_*)$ has to be understood
that we apply it to $C_* \otimes_{\IZ} \IC$ with the induced equivalence class of $\IC$-basis.

\begin{notation}\label{not:nu_and_theta}
Let $V$ be a finite-dimensional $\IZ^d$-representation and $S \subseteq \IZ^d$ be a finite subset.
Define
\begin{eqnarray}
\theta(V,S) 
& :=  &
\min \big\{|{\det}_{\IC}(l_{s} \colon V \to V)|  \; \bigl| \;  s \in  S\bigr\},
\label{theta(V,S)}
\end{eqnarray}
where $l_{s}\colon V \to V$ is multiplication with $s \in \IZ^d$.  

If $e_l$ is the $l$-th element of the standard $\IZ$-basis for $\IZ^d$, we put for $l = 1,2, \ldots, d$
\begin{eqnarray*}
\delta_l & := & {\det}_{\IC}(l_{e_l} \colon V \to V).
\end{eqnarray*}
 Put
\begin{eqnarray*}
\epsilon_l 
& = & 
\begin{cases}  
+1 & \text{if} \; |\delta_l| \ge 1;
\\
-1 & \text{if} \; |\delta_l| < 1;
\end{cases}
\end{eqnarray*}

Let $M$ be the smallest integer for which $M \ge 1$ and 
\begin{eqnarray*}
S
&\subseteq &
\bigl\{(s_1, \ldots, s_d) \; \bigl|\; -M \le   s_l \le M\; \text{for}\; l = 1,2 \ldots, d\bigr\}.
\end{eqnarray*}
Define 
\begin{multline}
\label{nu(V,S)}
\nu(V,S)
 := 
\bigl|\bigl|l_{(-\epsilon_1 \cdot (M+1), -\epsilon_2 \cdot (M+1), \ldots, -\epsilon_d \cdot (M+1))} \colon 
V \to V\bigr|\bigr|^{-\dim_{\IC}(V)} \\
\cdot \prod_{l= 1}^d |\delta_l|^{- \epsilon_l \cdot 2M}.
\end{multline}
\end{notation}

The main result of this section is

\begin{theorem}[Determinant class and twisting]
  \label{the:Determinant_class_and_twisting}
  Fix a natural number $d$. Let $G$ be a countable residually finite group. Consider a
  surjective  group homomorphism $\phi \colon G \to \IZ^d$. Let $V$ be a based finite-dimensional
  $\IZ^d$-representation. Denote by $\phi^*V$ its pullback to $G$, i.e., the equivalence
  class of the $\IC$-bases is unchanged and $G$ acts on $V$ by $g \cdot v = \phi(g) \cdot
  v$ for $g \in G$ and $v \in V$.

  Then:

  \begin{enumerate}

   \item \label{the:Determinant_class_and_twisting:Betti_numbers} 
   We get for any object $C_*$ in $\FBCC{\IZ G}$ and all $n \in \IZ$
   \[
  b_n^{(2)}\bigl(\Lambda^G \circ \eta_{\phi^* V}(C_*);\caln(G)\bigr)
    =
   \dim_{\IC}(V) \cdot b_n^{(2)}\bigl(\Lambda^G(C_*);\caln(G)\bigr);
  \]

    \item \label{the:Determinant_class_and_twisting:determinant} 
   For every matrix $A \in M_{r,s}(\IZ G)$ the following inequality holds
    \begin{multline*}
  \quad \quad \quad \nu\bigl(V,\phi(\supp_G(A)\bigr)^{r - \dim_{\caln(G)}(\ker(\Lambda(r_A)))} 
  \\
  \le {\det}_{\caln(G)}\bigl(\Lambda^G \circ \eta_{\phi^* V}(r_A)\bigr) \le 
  \\
   \bigl(||A||_1 \cdot \max\bigl\{||l_s \colon V \to V|| 
   \mid s \in \phi(\supp_G(A))\bigr\}\bigr)^{r - \dim_{\caln(G)}(\ker(\Lambda(r_A)))},
 \end{multline*}
  where $\nu\bigl(V,\phi(\supp_G(A)\bigr)$ has been defined in~\eqref{nu(V,S)}.
   
  If $\phi \colon G \to \IZ^d$ has a section, i.e., a group homomorphism $i \colon \IZ^d
  \to G$ with $\phi \circ i = \id_{\IZ^d}$, then we can replace in the inequality above
  $\nu\bigl(V,\phi(\supp_G(A)\bigr)$ by the constant $\theta\bigl(V,\phi(\supp_G(A)\bigr)$
  defined in~\eqref{theta(V,S)};

   \item \label{the:Determinant_class_and_twisting:L2-det_acyclic}
   Consider any object $C_*$ in $\FBCC{\IZ G}$. Then  $\Lambda^G \circ \eta_{\phi^* V}(C_*)$
   is of determinant class or is $\det$-$L^2$-acyclic if $\Lambda(C_*)$ is.

\end{enumerate}
\end{theorem}

Our proof of Theorem~\ref{the:Determinant_class_and_twisting} relies on the very good
knowledge about Fuglede-Kadison determinants for matrices over $\IC[\IZ^d]$ which stems
from the fact that in this case they are given by Mahler measures. We will need this as
the starting point to extend some of the basic results to $G$ by approximation
techniques. Therefore we can only treat $G$-representations which come from $\IZ^d$ by
restrictions with a group homomorphism $\phi \colon G \to \IZ^d$.

\begin{remark}[Replacing $\IZ^d$ by a torsionfree abelian group $A$ and dropping surjectivity]
\label{rem:Replacing_Zd_by_a_torsionfree_abelian_group_A}
It is possible to replace in Theorem~\ref{the:Determinant_class_and_twisting} the  
group $\IZ^d$ by any  torsionfree  abelian  group  $A$, e.g.,  $\IR^d$,  and  drop  
the condition that  $\phi$ is surjective, provided that
the image of $\phi$ is finitely generated.  

This more general case reduces to the old one as follows.
Since the image of $\phi$  is finitely generated and $A$ is torsionfree abelian, the 
image of $\phi \colon G \to A$ is $\IZ^d$ for an appropriate natural number $d$.
Hence one can write $\phi$ as a composite
of an epimorphism $\psi\colon G \to \IZ^d$ and an injective map $i \colon \IZ^d \to A$ and then consider instead of
the pair $(\phi,V)$ the pair $(\psi,i^*V)$ since we have
\begin{eqnarray*}
b_n^{(2)}\bigl(\Lambda \circ \eta_{\phi^* V}(C_*);\caln(G)\bigr)
& = & 
b_n^{(2)}\bigl(\Lambda \circ \eta_{\psi^* i^*V}(C_*);\caln(G)\bigr);
\\
{\det}_{\caln(G)}\bigl(\Lambda \circ \eta_{\phi^* V}(r_A)\bigr) 
& = & 
{\det}_{\caln(G)}\bigl(\Lambda \circ \eta_{\psi^* i^*V}(r_A)\bigr).
\end{eqnarray*}
\end{remark}

The results of
Subsections~\ref{subsec:Estimating_the_Fuglede-Kadison_determinant_in_terms_of_the_norm},~%
\ref{subsec:Determinants_over_torsionfree_amenable_groups}
and~\ref{subsec:Determinants_over_Zd} give the claim in the case, where $G = \IZ^d$ and
$\phi = \id_G$, see Lemma~\ref{lem:lower_bound_for_det_over_Zd}. They will be used in
Subsection~\ref{subsec:The_special_case_of_finite_ker(phi)} to prove
Theorem~\ref{the:Determinant_class_and_twisting} in the special case, where the kernel of
$\phi$ is finite, see Proposition~\ref{prop:main_theorem_true_for_finite_kernel}.  The
main idea in this step is that we can find an inclusion $j \colon \IZ^d \to G$ of finite
index and reduce the computation of Fuglede-Kadison determinants over $G$ to $\IZ^d$ by
restriction with $j$, provided that $\phi$ has finite kernel.  The general case 
will follow from approximation techniques applied to a
chain of subgroups $\ker(\phi) = K_0 \supseteq K_1 \supseteq K_2 \supseteq \cdots $ of in $G$
normal subgroups $K_i \subseteq G$ with $[\ker(\phi) : K_i] < \infty$ and 
$\bigcap_{i \ge  0} K_i =\{1\}$, which enables us to deduce the claim for general $\phi$ from in $i$ uniform
estimates for $\phi_i \colon G/K_i \to \IZ$, where $\phi_i$ is induced by $\phi$. In order
to guarantee the existence of such a chain of subgroups we need the assumption that $G$ is
residually finite and countable.

%%%%%%%%%%%%%%%%%%%%%%%%%%%%%%%%%%%%%%%%%%%%%%%%%%%%%%%%%%%%%%%%%%%%%%%%%%%%%%%%%

\subsection{Estimating the Fuglede-Kadison determinant in terms of the norm}
\label{subsec:Estimating_the_Fuglede-Kadison_determinant_in_terms_of_the_norm}

\begin{lemma} \label{lem:det_estimate_in_terms_of_norm}
Let $f \colon L^2(G)^m \to L^2(G)^n$ be a bounded $G$-equivariant operator.  Then
\[
\ln({\det}_{\caln(G)}(f)) \le \dim_{\caln(G)}(\overline{\im(f)}) \cdot \ln(||f||).
\]
\end{lemma}
\begin{proof}
We get for the spectral density function of $F(\lambda)$ of $f$
\[
F(||f||) - F(0) =  \dim_{\caln(G)}(L^2(G)^m) - \dim_{\caln(G)}(\ker (f)) = \dim_{\caln(G)}(\overline{\im(f)}),
\]
and $F(\lambda) = F(||f||)$ for $\lambda \ge 0$. We conclude 
from~\cite[Lemma~3.15~(1) on page~128]{Lueck(2002)}.
\begin{eqnarray*}
\ln({\det}_{\caln(G)}(f)) 
& = & 
\int_{0+} ^{\infty} \ln(\lambda) \;dF
\\
& = & 
\int_{0+} ^{||f||} \ln(\lambda) \;dF
\\
 & = & 
- \int_{0+}^{||f||} \frac{F(\lambda) - F(0)}{\lambda} d\lambda
+ \ln(||f||) \cdot (F(||f||) - F(0))
\\
& \le &
(F(||f||) - F(0)) \cdot \ln(||f||)
\\
& = & 
\dim_{\caln(G)}(\overline{\im(f)}) \cdot \ln(||f||).
\end{eqnarray*}
\end{proof}

%%%%%%%%%%%%%%%%%%%%%%%%%%%%%%%%%%%%%%%%%%%%%%%%%%%%%%%%%%%%%%%%%%%%%%%%%%%%%%%%%

\subsection{Determinants over torsionfree amenable groups}
\label{subsec:Determinants_over_torsionfree_amenable_groups}

In this subsection
we give some tools how to reduce the computation determinants
for $(r,r)$-matrices to $(1,1)$-matrices. For this paper it would suffice to consider $G = \IZ^d$, but for
other purposes we include the more general case of a torsionfree elementary amenable group here.

In this subsection $G$ will denote a torsionfree amenable group.  Let $F$ be a field with
$\IQ \subseteq F \subseteq \IC$.  Suppose that $FG$ has no non-trivial zero-divisors
and has a skew field of fractions $S^{-1}FG$, as explained in Remark~\ref{rem:Field_of_fractions}.  Let $V$ be a based
finite-dimensional complex $G$-representation.  We will abbreviate $\eta_V$ by $\eta$ and
$\Lambda^G$ by $\Lambda$ throughout this subsection. Moreover, $\eta$ is to be understood
to be the composite of the functor defined in~\eqref{eta_V,[B]_chain_complexes} with the 
obvious induction functor  $\FBMOD{FG} \to \FBMOD{\IC G}$.

\begin{lemma}[Estimate in terms of minors]
\label{lem:Estimate_in_terms_of_minors}
Let $G$ be a torsionfree amenable group whose group ring $\IC G$ has no non-trivial
zero-divisor, e.g., a torsionfree elementary amenable group. Consider a matrix $A$ over $FG$.
Let $B$ be a quadratic submatrix of $A$ of maximal size $k$ such that the map $r_B \colon \IC G^k \to \IC G^k$ is
injective.

Then:

\begin{enumerate}

\item \label{lem:Estimate_in_terms_of_minors:rank}
The rank of $A$ over the skew field $S^{-1}\IC G$ is $k$;

\item \label{lem:Estimate_in_terms_of_minors:rank_and_weak}
The morphism
\[
\Lambda \circ \eta(r_{B}) \colon \Lambda \circ \eta(\IC G^k) \to \Lambda \circ \eta(\IC G^k)
\]
is  a weak isomorphism;

\item \label{lem:Estimate_in_terms_of_minors:spectral_density}
We get for the spectral density functions and every $\lambda \ge 0$
\[
F\bigl(\Lambda \circ \eta(r_A)\bigr)(\lambda) - F\bigl(\Lambda \circ \eta(r_A)\bigr)(0) 
\le 
F\bigl(\Lambda \circ \eta(r_{B})\bigr)(\lambda);
\]

\item \label{lem:Estimate_in_terms_of_minors:Novikov-Shubin}
We get for the Novikov-Shubin invariants 
\[
\alpha\bigl(\Lambda \circ \eta(r_A);\caln(G)\bigr)
\ge 
\alpha\bigl(\Lambda \circ \eta(r_{B});\caln(G)\bigr);
\]

\item \label{lem:Estimate_in_terms_of_minors:Fuglede_Kadison}
We have 
\[
{\det}_{\caln(G)}\bigl(\Lambda \circ \eta(r_A)\bigr)
\ge 
{\det}_{\caln(G)}\bigl(\Lambda \circ \eta(r_{B})\bigr).
\]

\end{enumerate}
\end{lemma}
\begin{proof}~\eqref{lem:Estimate_in_terms_of_minors:rank}
Let $i^{(2)} \colon \Lambda \circ \eta(\IC G^k) \to \Lambda \circ \eta(\IC G^r)$ be the obvious 
inclusion and let $\pr^{(2)} \colon \Lambda \circ \eta(\IC G^s) \to \Lambda \circ \eta(\IC G^k)$ be
the obvious projection corresponding to the columns and rows which we have not deleted 
when passing from $A$ to $B$. Then
$\Lambda \circ \eta(r_{B})  \colon \Lambda \circ \eta(\IC G^r) 
\to \Lambda \circ \eta(\IC G^s)$ agrees with the composite
\[
\Lambda \circ \eta(r_{B})  \colon \Lambda \circ \eta(\IC G^k)   
\xrightarrow{i^{(2)}} \Lambda \circ \eta(\IC G^r) 
\xrightarrow{\Lambda \circ \eta(r_A)} \Lambda \circ \eta(\IC G^s) 
\xrightarrow{\pr^{(2)}}\Lambda \circ \eta(\IC G^k). 
\]
Let $p^{(2)} \colon  \Lambda \circ \eta(\IC G^r)  
\to \ker(\Lambda \circ \eta(r_A))^{\perp}$ be the orthogonal projection
onto the orthogonal complement $\ker(\Lambda \circ \eta(r_A))^{\perp} \subseteq \Lambda \circ \eta(\IC G^r)$
of the kernel of $\Lambda \circ \eta(r_A)$.
Let $j^{(2)} \colon \overline{\im(\Lambda \circ \eta(r_A))} \to \Lambda \circ \eta(\IC G^s)$ 
be the inclusion of the closure
$\overline{\im(\Lambda \circ \eta(r_A))}$  of the image of $\Lambda \circ \eta(r_A)$.
Let $(\Lambda \circ \eta(r_A))^{\prime}  \colon  \ker(\Lambda \circ \eta(r_A))^{\perp}  
\to \overline{\im(\Lambda \circ \eta(r_A))}$ be the 
$G$-equivariant bounded operator uniquely determined by 
\begin{eqnarray*}
\Lambda \circ \eta(r_A) 
& = & 
j^{(2)} \circ (\Lambda \circ \eta(r_A))^{\prime} \circ p^{(2)}.
\end{eqnarray*}
Since $S^{-1} \IC G$ is a skew field and the functor $S^{-1} \IC G \otimes_{\IC G} -$ is exact,
we get for any natural number $l$ and $(l,l)$-submatrix $C$ of $A$
that the map $r_{C} \colon \IC G^l \to \IC G^l$ is injective if and only if the rank of the matrix $C$
considered as matrix over $S^{-1} \IC G$ is $l$.  This implies that the rank of the matrix
$A$ over  $S^{-1} \IC G$ is $k$.
\\[2mm]~\eqref{lem:Estimate_in_terms_of_minors:rank_and_weak}
Since also the rank of the matrix $B$ over  $S^{-1} \IC G$ is $k$, we get
\begin{eqnarray*}
\dim_{S^{^-1} \IC G}\bigl(\ker\bigl(r_A \colon S^{-1} \IC G^r \to S^{-1} \IC G^s\bigr)\bigr) 
& = & 
r - k;
\\
\dim_{S^{^-1} \IC G}\bigl(\ker\bigl(r_{B} \colon S^{-1} \IC G^k \to S^{-1} \IC G^k\bigr)\bigr) 
& = & 
0.
\end{eqnarray*}
We conclude from Lemma~\ref{lem:invariance_of_L2-Betti_numbers_under_twisting}
\begin{eqnarray*}
\dim_{\caln(G)}\bigl(\ker(\Lambda \circ \eta(r_A)\bigr)\bigr) 
& = & 
\dim_{\IC}(V)  \cdot (r -k);
\\
\dim_{\caln(G)}\bigl(\ker(\Lambda \circ \eta(r_B))\bigr)
& = &
0.
\end{eqnarray*}
Hence $\Lambda \circ \eta(r_{B}) \colon \Lambda \circ \eta(\IC G^k) \to \Lambda \circ \eta(\IC G^k)$ 
is a weak isomorphism. 
\\[2mm]~\eqref{lem:Estimate_in_terms_of_minors:spectral_density}
The morphism 
$(\Lambda \circ \eta(r_A))^{\prime} \colon \ker(\Lambda \circ \eta(r_A))^{\perp} 
\to \overline{\im(\Lambda   \circ \eta(r_A))}$ 
is a weak isomorphism by construction.  We have the decomposition
\begin{eqnarray}
& & \Lambda \circ \eta(r_{B}) = \pr^{(2)} \circ (\Lambda \circ \eta(r_A)) \circ i^{(2)} 
 = 
\pr^{(2)} \circ j^{(2)} \circ (\Lambda \circ \eta(r_A))^{\prime}  \circ p^{(2)} \circ i^{(2)}.
\label{lem:Estimate_in_terms_of_minors_decomposition}
\end{eqnarray}
This implies that the morphism
$p^{(2)} \circ i^{(2)}  \colon \Lambda \circ \eta(\IC G^k) \to  \ker(\Lambda \circ \eta(r_A))^{\perp}$ 
is injective and  the morphism 
$\pr^{(2)} \circ j^{(2)}  \colon \overline{\im(\Lambda \circ \eta(r_A))} \to \Lambda \circ \eta(\IC G^k)$
has dense image.   We conclude from Lemma~\ref{lem:invariance_of_L2-Betti_numbers_under_twisting}
\begin{eqnarray*}
\dim_{\caln(G)}\bigl(\ker(\Lambda \circ \eta(r_A))^{\perp}\bigr)
& = & 
\dim_{\caln(G)}\bigl(\Lambda \circ \eta(\IC G^r)\bigr) - \dim_{\caln(G)}\bigl(\ker(\Lambda \circ \eta(r_A))\bigr)
\\
& = & 
\dim_{\IC}(V) \cdot r - \dim_{\IC}(V) \cdot (r -k)
\\
& = & 
\dim_{\IC}(V) \cdot k
\\
& = & 
\dim_{\caln(G)}\bigl(\Lambda \circ \eta(\IC G^k)\bigr).
\end{eqnarray*}
This implies that both morphisms 
$p^{(2)} \circ i^{(2)}  \colon \Lambda \circ \eta(\IC G^k) \to  \ker(\Lambda \circ \eta(r_A))^{\perp}$ 
and  $\pr^{(2)} \circ j^{(2)}  \colon \overline{\im(\Lambda \circ \eta(r_A))} \to \Lambda \circ \eta(\IC G^k)$
are weak isomorphisms.  

Since the operator norm of $ \pr^{(2)} \circ j^{(2)}$ and of $p^{(2)} \circ i^{(2)}$ are less or equal to $1$, 
we conclude from~\cite[Lemma~2.13 on page~78]{Lueck(2002)} 
and~\eqref{lem:Estimate_in_terms_of_minors_decomposition}
\begin{eqnarray*}
\lefteqn{F\bigl(\Lambda \circ \eta(r_A)\bigr)(\lambda) - F\bigl(\Lambda \circ \eta(r_A)\bigr)(0)}
& & 
\\
& = &
F\bigl((\Lambda \circ \eta(r_A))^{\prime}\bigr)(\lambda) 
\\
& \le  & 
F\bigl(\pr^{(2)} \circ j^{(2)} \circ (\Lambda\circ \eta(r_A))^{\prime}  \circ p^{(2)} \circ i^{(2)}\bigr)
\bigl(||\pr^{(2)} \circ j^{(2)}|| \cdot  ||p^{(2)} \circ i^{(2)}|| \cdot \lambda\bigr) 
\\
& =  & 
F\bigl(\Lambda \circ \eta(r_{B})\bigr)\bigl(||\pr^{(2)} \circ j^{(2)}|| \cdot  ||p^{(2)} \circ i^{(2)}|| \cdot \lambda\bigr) 
\\
& =  & 
F\bigl(\Lambda \circ \eta(r_{B})\bigr)(\lambda).
\end{eqnarray*}
\eqref{lem:Estimate_in_terms_of_minors:Novikov-Shubin}
This follows from assertion~\eqref{lem:Estimate_in_terms_of_minors:rank_and_weak} 
and~\eqref{lem:Estimate_in_terms_of_minors:spectral_density}
\\[1mm]~\eqref{lem:Estimate_in_terms_of_minors:Fuglede_Kadison}
We conclude from~\cite[Theorem~3.14 on page~128]{Lueck(2002)}
\begin{eqnarray*}
\lefteqn{{\det}_{\caln(G)} \bigl(\Lambda \circ \eta(r_{B})\bigr)}
& & 
\\
& = & 
{\det}_{\caln(G)} \bigl(\pr^{(2)} \circ j^{(2)} \circ (\Lambda\circ \eta(r_A))^{\prime}  \circ p^{(2)} \circ i^{(2)} \bigr)
\\
& = & 
{\det}_{\caln(G)} \bigl(\pr^{(2)} \circ j^{(2)}\bigr) \cdot {\det}_{\caln(G)} \bigl((\Lambda\circ \eta(r_A))^{\prime}\bigr)
\cdot {\det}_{\caln(G)} \bigl(p^{(2)} \circ i^{(2)}\bigr).
\end{eqnarray*}
We get 
\begin{eqnarray*}
{\det}_{\caln(G)} \bigl((\Lambda\circ \eta(r_A))^{\prime}\bigr) 
& = & 
{\det}_{\caln(G)} \bigl(\Lambda\circ \eta(r_A)\bigr)
\end{eqnarray*}
from~\cite[Lemma~3.15~(3) on page~129]{Lueck(2002)}.
Since the operator norm of $ \pr^{(2)} \circ j^{(2)}$ and of $p^{(2)} \circ i^{(2)}$ are less or equal to $1$, we get
from Lemma~\ref{lem:det_estimate_in_terms_of_norm}
\begin{eqnarray*}
{\det}_{\caln(G)}\bigl(\pr^{(2)} \circ j^{(2)}\bigr) & \le & 1;
\\
{\det}_{\caln(G)}\bigl(p^{(2)} \circ i^{(2)}\bigr) & \le & 1.
\end{eqnarray*}
We conclude
\begin{eqnarray*}
{\det}_{\caln(G)} \bigl(\Lambda \circ \eta(r_{B})\bigr)
& \le &
{\det}_{\caln(G)} \bigl(\Lambda\circ \eta(r_A)\bigr).
\end{eqnarray*}
This finishes the proof of Lemma~\ref{lem:Estimate_in_terms_of_minors}.
\end{proof}

\begin{lemma} \label{lem_det_class_to_matrices_from_elements}
Let $G$ be a torsionfree amenable group whose group ring $FG$ has no non-trivial
zero-divisor, e.g., a torsionfree elementary amenable group. Suppose that for every element
$x \in FG$ the operator $\Lambda \circ \eta(r_x) \colon L^2(G) \to L^2(G)$ is of determinant class
or its Novikov-Shubin invariant satisfies $\alpha(\Lambda \circ \eta(r_x)) > 0$ respectively.

Then we get for every $r,s \in \IN$ and $A \in M_{r,s}(F G)$ that
$\Lambda \circ \eta(r_A) \colon L^2(G) \to L^2(G)$ is of determinant class
or satisfies $\alpha(\Lambda \circ \eta(r_A)) > 0$ respectively.
\end{lemma}
\begin{proof}
  Because of Lemma~\ref{lem:Estimate_in_terms_of_minors} we can assume without loss of
  generality that $r = s$ and $A$ invertible over $S^{-1} F G$, otherwise pass to an
  appropriate minor of $A$.  

  Given $a,b \in S^{-1} FG$ and $i,j \in \{1,2, \ldots, r\}$ with $i \not= j$, denote by
  $E_{i,j}[a,b]$ the matrix whose entry at $(i,j)$ is $b$, at $(k,k)$ is $a$ for $k = 1,2,
  \ldots, r$ and whose other entries are all zero.  Since $S^{-1}FG$ is a skew field, we
  can perform elementary row operations to transform $A$ into an upper triangular matrix
  over $S^{-1} FG$. By clearing denominators by multiplying with an appropriate diagonal
  matrix, we can construct a sequence of matrixes $B_0,B_1, B_2 , \ldots , B_m$ such that
  each matrix $B_n$ is of the form $E_{i,j}[a,b]$ for appropriate $i,j \in \{1,2,
  \ldots,r\}$ with $i \not= j$ and $a,b \in FG$ with $a \not = 0$ such that the matrix $C$
  given by
\begin{eqnarray*}
C 
& = & 
B_1 \cdot B_2 \cdot \,\cdots \, \cdot B_m \cdot A
\end{eqnarray*}
is upper triangular with non-zero entries in $FG$ on the diagonal. Notice that each $B_n$
is lower triangular  with non-trivial entries in $FG$ on the diagonal. Since $\Lambda \circ
\eta(r_x) \colon L^2(G) \to L^2(G)$ is a weak isomorphism of determinant class or a weak
isomorphism satisfying $\alpha(\Lambda \circ \eta(r_x)) > 0$ respectively for every $x \in
FG$ with $x \not= 0$ by assumption, we conclude 
from~\cite[Lemma~2.15~(2) on page~80 and Theorem~3.14~(2) on page~128]{Lueck(2002)} 
that $\Lambda \circ \eta(r_{C}) \colon L^2(G)^r \to L^2(G)^r$ and 
$\Lambda \circ \eta(r_{B_n}) \colon L^2(G)^r \to L^2(G)^r$ for $n = 1,2, \ldots , m$ 
are  weak isomorphisms of determinant class or a weak isomorphism satisfying
$\alpha(\Lambda \circ \eta(r_{B_n})) > 0$ respectively for every $n \in \{1,2, \ldots, m\}$.  
Now~\cite[Lemma~2.14 on page~79 and Theorem~3.14~(1) on page~128]{Lueck(2002)}
imply that $\Lambda \circ \eta(r_A) \colon L^2(G)^r \to L^2(G)^r$ is a weak isomorphism
of determinant class or a weak isomorphism satisfying $\alpha(\Lambda \circ \eta(r_{B_u}))
> 0$ respectively.
\end{proof}

We will make the following assumption
\begin{assumption}\label{ass:determinant_class_over_FG}
For any $x \in FG$ with $x \not= 0$ the operator
$\Lambda^G(r_x) \colon L^2(G) \to L^2(G)$ is of determinant class. 
\end{assumption}
Then we want to define a homomorphism 
\begin{eqnarray}
\Delta = \Delta_V \colon K_1(S^{-1} FG) \to \IR^{>0}
\label{homomorphism_Delta}
\end{eqnarray}
as follows. Consider any natural number $r$ and matrix $A \in GL_r(S^{-1} FG)$.  
We can choose $a \in FG$ with $a \not= 0$ such that 
\begin{eqnarray*}
A[a] & := & I_r[a] \cdot A
\end{eqnarray*} 
belongs to $M_{r,r}(FG)$, where $I_r[a]$ is the diagonal $(r,r)$-matrix whose entries on
the diagonal are all equal to $a$.  We conclude from
Lemma~\ref{lem:Estimate_in_terms_of_minors},
Lemma~\ref{lem_det_class_to_matrices_from_elements} and
Assumption~\ref{ass:determinant_class_over_FG} that $\Lambda \circ \eta(r_{A[a]})$ and
$\Lambda \circ \eta(r_{I_r[a]})$ are weak isomorphisms whose Fuglede-Kadison determinants
take values in $\IR^{> 0}$. If $[A]$ denotes the class represented by $A$ in
$K_1(S^{-1}FG)$, we want to define
\[
\Delta([A]) := \frac{{\det}_{\caln(G)}\bigl(\Lambda 
\circ \eta(r_{A[a]})\bigr)}{{\det}_{\caln(G)}\bigl(\Lambda \circ \eta(r_{I_r[a]})\bigr)}.
\]
We have to show that $\Delta$  is a well-defined homomorphism of abelian groups.

Consider $A \in M_{r,r}(S^{-1}F G)$ and two elements $a, a' \in F G$ different from $0$ such that  $A[a]$ and
  $A[a']$ belong to $M_{r,r}(FG)$. We want to show the equality
  \begin{eqnarray}
 \frac{{\det}_{\caln(G)}\bigl(\Lambda \circ \eta(r_{A[a]})\bigr)}{{\det}_{\caln(G)}\bigl(\Lambda \circ \eta(r_{I_r[a]})\bigr)}
 & = & 
 \frac{{\det}_{\caln(G)}\bigl(\Lambda \circ \eta(r_{A[a']})\bigr)}{{\det}_{\caln(G)}\bigl(\Lambda \circ \eta(r_{I_r[a']})\bigr)}.
  \label{Pichelsteiner}
\end{eqnarray}
since it implies that the choice of $a$ does not matter. 
Thanks to the Ore condition, we can choose elements $x,y \in F G$ with $x,y \not= 0$ satisfying
\[
xaa' = ya'a.
\]
Now~\eqref{Pichelsteiner} follows  from the calculation using~\cite[Theorem~3.14~(1) on page~128]{Lueck(2002)}
\begin{eqnarray*}
\frac{{\det}_{\caln(G)}\bigl(\Lambda \circ \eta(r_{A[a]})\bigr)}{{\det}_{\caln(G)}\bigl(\Lambda \circ \eta(r_{I_r[a]})\bigr)}
& = & 
\frac{{\det}_{\caln(G)}\bigl(\Lambda \circ \eta(r_{I_r[ya']})\bigr) \cdot {\det}_{\caln(G)}\bigl(\Lambda \circ \eta(r_{A[a]})\bigr)}
{{\det}_{\caln(G)}\bigl(\Lambda \circ \eta(r_{I_r[ya']})\bigr) \cdot  {\det}_{\caln(G)}\bigl(\Lambda \circ \eta(r_{I_r[a]})\bigr)}
\\
& = & 
 \frac{{\det}_{\caln(G)}\bigl(\Lambda \circ \eta(r_{A[ya'a]})\bigr)}{{\det}_{\caln(G)}\bigl(\Lambda \circ \eta(r_{I_r[ya'a]})\bigr)}
\\
& = & 
 \frac{{\det}_{\caln(G)}\bigl(\Lambda \circ \eta(r_{A[xaa']})\bigr)}{{\det}_{\caln(G)}\bigl(\Lambda \circ \eta(r_{I_r[xaa']})\bigr)}
\\
& = & 
\frac{{\det}_{\caln(G)}\bigl(\Lambda \circ \eta(r_{I_r[xa]})\bigr) \cdot {\det}_{\caln(G)}\bigl(\Lambda \circ \eta(r_{A[a']})\bigr)}
{{\det}_{\caln(G)}\bigl(\Lambda \circ \eta(r_{I_r[xa]})\bigr) \cdot  {\det}_{\caln(G)}\bigl(\Lambda \circ \eta(r_{I_r[a']})\bigr)}
\\
& = & 
 \frac{{\det}_{\caln(G)}\bigl(\Lambda \circ \eta(r_{A[a']})\bigr)}{{\det}_{\caln(G)}\bigl(\Lambda \circ \eta(r_{I_r[a']})\bigr)}.
\end{eqnarray*}

Consider $A,B \in GL_r(S^{-1}F G)$ and $a,b \in FG$  with $a, b \not= 0$ such that 
$A[a]$ and $B[b]$ belong to $M_{r,r}(F G)$. 
Choose $c \in F G$ such that 
\[
X := I_r[c] \cdot I_r[b] \cdot A[a] \cdot I_r[b]^{-1} \cdot A[a|^{-1}
\]
belongs to $M_{r,r}(F G)$. Then we get
\begin{eqnarray}
X \cdot A[a] \cdot I_r[b] 
& = &
I_r[c] \cdot I_r[b] \cdot A[a]
\label{Upsala}
\end{eqnarray}

We conclude
\begin{eqnarray*}
\lefteqn{{\det}_{\caln(G)}(\Lambda(r_{X})) \cdot {\det}_{\caln(G)}(\Lambda(r_{A[a]})) \cdot {\det}_{\caln(G)}(\Lambda(r_{I_r[b]}))} 
& & 
\\
& = & 
{\det}_{\caln(G)}\left(\Lambda(r_{X}) \circ\Lambda(r_{A[a]}) \circ \Lambda(r_{I_r[b]})\right)
\\
& = & 
{\det}_{\caln(G)}(\Lambda(r_{X \cdot A[a] \cdot I_r[b]}))
\\
& = & 
{\det}_{\caln(G)}(\Lambda(r_{I_r[c] \cdot I_r[b] \cdot A[a]}))
\\
& = & 
{\det}_{\caln(G)}\left(\Lambda(r_{I_r[c] }) \circ\Lambda(r_{I_r[b]}) \circ \Lambda(r_{A[a]})\right)
\\
& = & 
{\det}_{\caln(G)}(\Lambda(r_{I_r[c]})) \cdot {\det}_{\caln(G)}(\Lambda(r_{I_r[b]})) \cdot {\det}_{\caln(G)}(\Lambda(r_{A[a]}))
\\
& = & 
{\det}_{\caln(G)}(\Lambda(r_{I_r[c]})) \cdot {\det}_{\caln(G)}(\Lambda(r_{A[a]})) \cdot {\det}_{\caln(G)}(\Lambda(r_{I_r[b]})).
\end{eqnarray*}
This implies
\begin{eqnarray}
{\det}_{\caln(G)}(\Lambda(r_{X})) 
& = & 
{\det}_{\caln(G)}(\Lambda(r_{I_r[c]})).
\label{Tomatensaft}
\end{eqnarray}
We conclude from~\eqref{Upsala}
\begin{eqnarray}
X \cdot A[a] \cdot B[b]
& = &
(AB)[cba].
\label{Louder_than_words}
\end{eqnarray}
In particular we see that $(AB)[cba]$ belongs to $M_{r,r}(F G)$.
We compute
\begin{eqnarray}
\label{Pustekuchen}
& & 
\\
\lefteqn{\frac{{\det}_{\caln(G)}\bigl(\Lambda \circ \eta(r_{(AB)[cba]})\bigr)}{{\det}_{\caln(G)}\bigl(\Lambda \circ \eta(r_{I_r[cba]})\bigr)}}
& & 
\nonumber
\\
& \stackrel{\eqref{Louder_than_words}}{=}  & 
\frac{{\det}_{\caln(G)}\bigl(\Lambda \circ \eta(r_{X \cdot A[a] \cdot B[b]})\bigr)}{{\det}_{\caln(G)}\bigl(\Lambda \circ \eta(r_{I_r[cba]})\bigr)}
\nonumber
\\
& = & 
\frac{{\det}_{\caln(G)}\bigl(\Lambda \circ \eta(r_{X})\bigr) \cdot {\det}_{\caln(G)}\bigl(\Lambda \circ \eta(r_{A[a]})\bigr) 
\cdot {\det}_{\caln(G)}\bigl(\Lambda \circ \eta(r_{B[b]})\bigr)}
{{\det}_{\caln(G)}\bigl(\Lambda \circ \eta(r_{I_r[c]})\bigr)\cdot {\det}_{\caln(G)}\bigl(\Lambda \circ \eta(r_{I_r[b]})\bigr)
\cdot {\det}_{\caln(G)}\bigl(\Lambda \circ \eta(r_{I_r[a]})\bigr)}
\nonumber
\\
& \stackrel{\eqref{Tomatensaft}}{=} & 
\frac{{\det}_{\caln(G)}\bigl(\Lambda \circ \eta(r_{X})\bigr) \cdot {\det}_{\caln(G)}\bigl(\Lambda \circ \eta(r_{A[a]})\bigr) 
\cdot {\det}_{\caln(G)}\bigl(\Lambda \circ \eta(r_{B[b]})\bigr)}
{{\det}_{\caln(G)}\bigl(\Lambda \circ \eta(r_{X})\bigr)\cdot {\det}_{\caln(G)}\bigl(\Lambda \circ \eta(r_{I_r[a]})\bigr)
\cdot {\det}_{\caln(G)}\bigl(\Lambda \circ \eta(r_{I_r[b]})\bigr)}
\nonumber
\\ 
& = & 
\frac{{\det}_{\caln(G)}\bigl(\Lambda \circ \eta(r_{A[a]})\bigr)}{{\det}_{\caln(G)}\bigl(\Lambda \circ \eta(r_{I_r[a]})\bigr)}
\cdot 
\frac{{\det}_{\caln(G)}\bigl(\Lambda \circ \eta(r_{B[b]})\bigr)}{{\det}_{\caln(G)}\bigl(\Lambda \circ \eta(r_{I_r[b]})\bigr)}.
\nonumber
\end{eqnarray}
If $B$ is the block matrix $\begin{pmatrix} A & * \\ 0 & I_1 \end{pmatrix}$
and we have $a \in FG$ with $a \not= 0$ such that $A[a] := I_r[a] \cdot A$
belongs to $M_{r,r}(FG)$, then $B[a] \in M_{r+1,r+1}(F G)$.
Since $\Lambda \circ \eta(r_{A[a]})$ and $\Lambda \circ \eta(r_{I_r[a]})$ 
are weak isomorphisms by Lemma~\ref{lem:invariance_of_L2-Betti_numbers_under_twisting},
we conclude from~\cite[Theorem~3.14~(2) on page~128]{Lueck(2002)}
\begin{eqnarray}
\label{Schokoladenei}
& & 
\\
\frac{{\det}_{\caln(G)}\bigl(\Lambda \circ \eta(r_{B[a]})\bigr)}{{\det}_{\caln(G)}\bigl(\Lambda \circ \eta(r_{I_{r+1}[a]})\bigr)}
& = & 
\frac{{\det}_{\caln(G)}\bigl(\Lambda \circ \eta(r_{A[a]})\bigr) \cdot {\det}_{\caln(G)}\bigl(\Lambda \circ \eta(r_{a})\bigr)}
{{\det}_{\caln(G)}\bigl(\Lambda \circ \eta(r_{I_r[a]})\bigr) \cdot {\det}_{\caln(G)}\bigl(\Lambda \circ \eta(r_{a})\bigr)}
\nonumber
\\
& = & 
\frac{{\det}_{\caln(G)}\bigl(\Lambda \circ \eta(r_{A[a]})\bigr)}{{\det}_{\caln(G)}\bigl(\Lambda \circ \eta(r_{I_r[a]})\bigr)}.
\nonumber
\end{eqnarray}
We conclude from~\eqref{Pichelsteiner},~\eqref{Pustekuchen} and~\eqref{Schokoladenei} that
the map $\Delta$ announced in~\eqref{homomorphism_Delta} is a well-defined homomorphism of
abelian groups, provided that assumption~\ref{ass:determinant_class_over_FG} holds.

There is a Dieudonn\'e determinant for invertible matrices over a skew field $K$ which takes
values in the abelianization of the group of units of the skew field $K^{\times}/[K^{\times},K^{\times}]$ 
and induces an isomorphism, see~\cite[Corollary~4.3 in page~133]{Silvester(1981)} 
\begin{eqnarray}
{\det}_D \colon K_1(K) 
& \xrightarrow{\cong}  &
K^{\times}/[K^{\times},K^{\times}]. 
\label{K_1(K)_Dieudonne}
\end{eqnarray}
The inverse 
\begin{eqnarray}
J_D \colon  K^{\times}/[K^{\times},K^{\times}] & \xrightarrow{\cong} & K_1(K) 
\label{K_1(K)_Dieudonne_inverse}
\end{eqnarray}
sends the class of a unit in $K$ to the class of the corresponding $(1,1)$-matrix.  
The following result reduces the computation of Fuglede-Kadison determinants
for $(r,r)$-matrices to $(1,1)$-matrices.

\begin{lemma} \label{lem:Dieudonne}
Let $G$ be a torsionfree amenable group whose group ring $FG$ has no non-trivial
zero-divisor, e.g., a torsionfree elementary amenable group. Consider any natural number $r$ and
matrix $A \in M_{r,r}(FG)$.

\begin{enumerate}
\item \label{lem:Dieudonne:equivalent}
Then the following statements are equivalent:
\begin{enumerate}
\item \label{lem:Dieudonne:equivalent:r_A_inj}
$r_A \colon FG^r \to FG^r$ is injective;
\item \label{lem:Dieudonne:equivalent:r_A_S_bij}
$r_A \colon S^{-1}FG^r \to S^{-1}FG^r$ is bijective, or, equivalently, $A$ is invertible over $S^{-1}FG$;
\item \label{lem:Dieudonne:equivalent:r_A(2)_inj}
$\Lambda \circ \eta(r_A)\colon \Lambda \circ \eta(FG^r) \to \Lambda \circ \eta(FG^r)$ is injective;
\item \label{lem:Dieudonne:equivalent:r_A(2)_weak}
$\Lambda \circ \eta(r_A)\colon \Lambda \circ \eta(FG^r) \to \Lambda \circ \eta(FG^r)$ 
is a weak isomorphism;
\end{enumerate}
\item \label{lem:Dieudonne:dets} If Assumption~\ref{ass:determinant_class_over_FG}  and 
  one of the equivalent conditions above are satisfied, then $\Lambda \circ
  \eta(r_A)\colon \Lambda \circ \eta(FG^r) \to \Lambda \circ \eta(FG^r)$ is a weak
  isomorphism of determinant class and we get the equation
\[
{\det}_{\caln(G)}\bigl(\Lambda \circ \eta(r_A)\colon \Lambda \circ \eta(FG^r) \to \Lambda \circ \eta(FG^r)\bigr) 
=
\Delta \circ J_D\bigl({\det}_{D}(A)\bigr)
\]
where the homomorphisms $\Delta$ and $J_D$ have been defined in~\eqref{homomorphism_Delta}
and~\eqref{K_1(K)_Dieudonne_inverse}. In particular ${\det}_{\caln(G)}\bigl(\Lambda \circ \eta(r_A)\bigr)$ agrees
with the quotient $\frac{{\det}_{\caln(G)}(\Lambda \circ \eta(r_x))}{{\det}_{\caln(G)}(\Lambda \circ \eta(r_y))}$ 
for two appropriate elements $x,y \in FG$. 
\end{enumerate}
\end{lemma}
\begin{proof}~\eqref{lem:Dieudonne:equivalent}
Since localization is a flat functor and $S^{-1}FG$ is a field, we get
$\eqref{lem:Dieudonne:equivalent:r_A_inj} \Longleftrightarrow \eqref{lem:Dieudonne:equivalent:r_A_S_bij}$.
The equivalence 
$\eqref{lem:Dieudonne:equivalent:r_A_S_bij} \Longleftrightarrow \eqref{lem:Dieudonne:equivalent:r_A(2)_weak}$
follows directly from Lemma~\ref{lem:invariance_of_L2-Betti_numbers_under_twisting}.
The equivalence 
$\eqref{lem:Dieudonne:equivalent:r_A(2)_inj} \Longleftrightarrow \eqref{lem:Dieudonne:equivalent:r_A(2)_weak}$
is a direct consequence of the additivity of the von Neumann dimension. 
\\[1mm]~\eqref{lem:Dieudonne:dets}
This follows from  the existence of the isomorphism~\eqref{K_1(K)_Dieudonne}
and the fact that the homomorphism $\Delta$ of~\eqref{homomorphism_Delta}  is well-defined.
\end{proof}

\begin{remark}[About the Dieudonn\'e determinant] \label{rem:About_the_Dieudonne_determinant}
  If  $A = \begin{pmatrix} a & b \\ c &    d \end{pmatrix}$ is a $(2,2)$-matrix over a skew field $K$, its
  Dieudonn\'e determinant in $K^{\times}/[K^{\times},K^{\times}]$ is defined to be the class of $-cb$
  if $a = 0$ and to be the class of $ad -aca^{-1}b$ otherwise.  It can happen that
  for a $(r,r)$-matrix $A$  over $FG$ which is invertible over $S^{-1}FG$
  the standard  representative of the Dieudonn\'e determinant does not belong to $FG$.
  The following example is due to Peter Linnell.  Let G be the metabelian group 
\[
\IZ \wr \IZ  = \langle x_i,y \mid  x_ix_j=x_jx_i, y^{-1}x_iy = x_{i+1} \:\text{for all}\;  i,j \in \IZ\rangle.
\]
Then we have $\IQ G \subset L^1(G) \subset \calu(G)$, and $S^{-1} \IQ G \subseteq \calu(G)$,
where $\calu(G)$ is the algebra of affiliated operators, see for instance~\cite[Chapter~8]{Lueck(2002)}.
Consider the element $2-x_0 \in \IQ G$.  Then we can consider the element $(2-x_0)y(2-x_0)^{-1}$ in $S^{-1} FG$. 
It  agrees with $y(1-x_1/2)(1 - x_0/2)^{-1}$. We get in in the Banach algebra $L^1(G)$ the equality $(1-x_0/2) \cdot
(1+x_0/2+x_0^2/4+ \cdots) = 1$. Hence the element $(2-x_0)y(2-x_0)^{-1}$ in $\calu(G)$
agrees with the element $y(1-x_1/2)(1+x_0/2+x_0^2/4+ \cdots)$ which is already contained in
$L^1(G)$. If $(2-x_0)y(2-x_0)^{-1}$ would belong to $\IQ G$, also the element
$y(1-x_1/2)(1+x_0/2+x_0^2/4+ \cdots)$ in $L^1(G)$ would belong to $\IQ G$, what is
obviously not true. Hence $(2-x_0)y(2-x_0)^{-1}$ in $S^{-1} \IQ G$ is not contained in $\IQ G$. 

So the Dieudonn\'e
determinant of the matrix $A = \begin{pmatrix} 2- x_0 & 1 \\ y & 0 \end{pmatrix}$ is
represented by the element $(2-x_0)y(2-x_0)^{-1}$ which is not contained in $\IQ G$
although all entries of $A$ belong to $\IQ G$.
   \end{remark} 

%%%%%%%%%%%%%%%%%%%%%%%%%%%%%%%%%%%%%%%%%%%%%%%%%%%%%%%%%%%%%%%%%%%%%%%%%%%%%%%%%

\subsection{Determinants over $\IZ^d$}
\label{subsec:Determinants_over_Zd}

Next we consider the special case $G = \IZ^d$. Then 
Lemma~\ref{lem:Dieudonne} simplifies to the following  result.

\begin{lemma} \label{lem:K_1_and_det(C[Zd]}
Consider any matrix $A \in M_{r,r}(\IC[\IZ^d])$.

\begin{enumerate}
\item \label{lem:K_1_and_det(C[Zd]:equivalent}
Then the following statements are equivalent:
\begin{enumerate}
\item \label{lem:K_1_and_det(C[Zd]:equivalent:r_A_inj}
$r_A \colon \IC[\IZ^d]^r \to \IC[\IZ^d]^r$ is injective;
\item \label{lem:K_1_and_det(C[Zd]:equivalent:r_A_S_bij}
$r_A \colon S^{-1}\IC[\IZ^d]^r \to S^{-1}\IC[\IZ^d]^r$ is bijective;
\item \label{lem:K_1_and_det(C[Zd]:equivalent:r_A(2)_inj}
$\Lambda \circ \eta(r_A)\colon \Lambda \circ \eta(\IC[\IZ^d]^r) \to \Lambda \circ \eta(\IC[\IZ^d]^r)$ is injective;
\item \label{lem:K_1_and_det(C[Zd]:equivalent:r_A(2)_weak}
$\Lambda \circ \eta(r_A)\colon \Lambda \circ \eta(\IC[\IZ^d]^r) \to \Lambda \circ \eta(\IC[\IZ^d]^r)$ 
is a weak isomorphism;
\item \label{lem:K_1_and_det(C[Zd]:equivalent:r_det_inj}
$r_{{\det}_{\IC[\IZ^d]}(A)} \colon \IC[\IZ^d] \to \IC[\IZ^d]$ is injective;
\item \label{lem:K_1_and_det(C[Zd]:equivalent:r_det_S_bij}
${\det}_{\IC[\IZ^d]}(A)$ is a unit in $S^{-1}\IC[\IZ^d]$, or, equivalently, ${\det}_{\IC[\IZ^d]}(A)\not= 0$;
\item \label{lem:K_1_and_det(C[Zd]:equivalent:r_det(2)_inj}
$\Lambda \circ \eta(r_{{\det}_{\IC[\IZ^d]}(A)})\colon \Lambda \circ \eta(\IC[\IZ^d]) \to \Lambda \circ \eta(\IC[\IZ^d])$ 
is injective;
\item \label{lem:K_1_and_det(C[Zd]:equivalent:r_det(2)_weak}
$\Lambda \circ \eta(r_{{\det}_{\IC[\IZ^d]}(A)})\colon \Lambda \circ \eta(\IC[\IZ^d]) \to \Lambda \circ \eta(\IC[\IZ^d])$ 
is a weak isomorphism;
\end{enumerate}
\item \label{lem:K_1_and_det(C[Zd]:dets}
If one of the equivalent conditions above is satisfied, we get the equality of positive real numbers
\begin{multline*}
{\det}_{\caln(\IZ^d)}\bigl(\Lambda \circ \eta(r_A)\colon 
\Lambda \circ \eta(\IC[\IZ^d]^r) \to \Lambda \circ \eta(\IC[\IZ^d]^r)\bigr)
\\ =
{\det}_{\caln(\IZ^d)}\bigl(\Lambda \circ \eta(r_{{\det}_{\IC[\IZ^d]}(A)})\colon 
\Lambda \circ \eta(\IC[\IZ^d]) \to \Lambda \circ \eta(\IC[\IZ^d])\bigr).
\end{multline*}
\end{enumerate}
\end{lemma}

\begin{proof}~\eqref{lem:K_1_and_det(C[Zd]:equivalent} We get directly from
  Lemma~\ref{lem:K_1_and_det(C[Zd]}~\eqref{lem:K_1_and_det(C[Zd]:equivalent} applied in
  the case $F = \IC$ and $G = \IZ^d$
\[\eqref{lem:K_1_and_det(C[Zd]:equivalent:r_A_inj} 
\Longleftrightarrow \eqref{lem:K_1_and_det(C[Zd]:equivalent:r_A_S_bij}
\Longleftrightarrow \eqref{lem:K_1_and_det(C[Zd]:equivalent:r_A(2)_inj}
\Longleftrightarrow \eqref{lem:K_1_and_det(C[Zd]:equivalent:r_A(2)_weak}.
\]
As a special case we get
\[\eqref{lem:K_1_and_det(C[Zd]:equivalent:r_det_inj} 
\Longleftrightarrow \eqref{lem:K_1_and_det(C[Zd]:equivalent:r_det_S_bij}
\Longleftrightarrow \eqref{lem:K_1_and_det(C[Zd]:equivalent:r_det(2)_inj}
\Longleftrightarrow \eqref{lem:K_1_and_det(C[Zd]:equivalent:r_det(2)_weak}.
\]

By Cramer's rule we see that $r_A \colon S^{-1}\IC[\IZ^d]^r \to S^{-1}\IC[\IZ^d]^r$ is
bijective if and only if ${\det}_{\IC[\IZ^d]}(A)$ is a unit in
$S^{-1}\IC[\IZ^d]$.  This implies $\eqref{lem:K_1_and_det(C[Zd]:equivalent:r_A_S_bij}
\Longleftrightarrow \eqref{lem:K_1_and_det(C[Zd]:equivalent:r_det_S_bij}$ and hence
assertion~\eqref{lem:K_1_and_det(C[Zd]:equivalent} is proved.
\\[2mm]~\eqref{lem:K_1_and_det(C[Zd]:dets} This follows from 
Lemma~\ref{lem:Dieudonne}~\eqref{lem:Dieudonne:dets} applied in the case $F = \IC$
and $G = \IZ^d$ since Assumption~\ref{ass:determinant_class_over_FG} is satisfied
by~\cite[Theorem~1.2]{Lueck(2015spectral)}, over the commutative field
$S^{-1}\IC[\IZ^d]$ the Dieudonn\'e determinant reduces to the classical determinant
${\det}_{S^{-1}\IC[\IZ^d]}$ and for any matrix $A \in M_{r,r}(\IC[\IZ^d])$ we have the
equality of classical determinants ${\det}_{\IC[\IZ^d]}(A) =
{\det}_{S^{-1}\IC[\IZ^d]}(A)$. (The latter actually means that the difficulty discussed in
Remark~\ref{rem:About_the_Dieudonne_determinant} does not occur in the commutative case.)
This finishes the proof of Lemma~\ref{lem:K_1_and_det(C[Zd]}.
\end{proof}

Equip $\IZ^d$ with the lexicographical order, i.e., we put $(m_1, \ldots, m_d) < (n_1,
\ldots, n_d)$, if $m_d < n_d$, or if $m_d = n_d$ and $m_{d-1} < n_{d-1}$, or if $m_d = n_d$, $m_{d-1}
= n_{d-1}$ and $m_{d-2} < n_{d-2}$, or if $\ldots$, or if $m_i = n_i$ for $i = d, (d-1), \ldots, 2$ and
$m_1 < n_1$. We can write a non-trivial element $p \in \IC[\IZ^d]$ as a finite sum with 
complex coefficients $c_{n_1, \ldots,  n_d}$
\[
p(z_1^{\pm}, \ldots, z_d^{\pm}) = \sum_{(n_1, \ldots, n_d) \in \IZ^d} c_{n_1, \ldots,  n_d} 
\cdot z_1^{n_1} \cdot z_2^{n_2} \cdot \, \cdots \, \cdot z_d^{n_d}.
\]
Recall that its support is $\supp_{\IZ^d}(p)  :=  \{(n_1, \ldots , n_d) \in \IZ^d \mid c_{n_1, \ldots, n_d} \not= 0\}$.
Let $(m_1, \ldots, m_d) \in \IZ^d$ be maximal with respect to the lexicographical order
among the elements in $\supp_{\IZ^d}(p)$. The \emph{leading coefficient} of $p$ is defined to be
\begin{eqnarray}
\lead(p) & := & c_{m_1, \ldots, m_d}.
\label{lead(p)}
\end{eqnarray}

\begin{lemma} \label{lem:estimate_for_elements_over_Zd} Consider
  a non-trivial element 
   \[ p = \sum_{(n_1, \ldots , n_d) \in \IZ^d} c_{n_1, \ldots, n_d}
  \cdot z_1^{n_1} \cdot \, \cdots \, \cdot z_d^{n_d} \] in $\IC[\IZ^d]$. 

  Then $\Lambda(r_p) \colon \Lambda(\IC[\IZ^d]) \to
  \Lambda(\IC[\IZ^d])$ is a weak isomorphism and satisfies 
   \[
  {\det}_{\caln(\IZ^d)}\bigl(\Lambda(r_p)\bigr) \ge |\lead(p)|.  
   \] 
   \end{lemma} 
   
   \begin{proof} We begin with the case $d = 1$. Then we can
  write $p(z) = \sum_{n = n_0}^{n_1} c_n \cdot z^n$ for integers $n_0$ and $n_1$ with 
  $n_0   \le n_1$, complex numbers $c_{n_0}, c_{n_0 +1}, \ldots, c_{n_1}$ with $c_{n_0} \not= 0$
  and $c_{n_1} \not= 0$. We can also write \[ p(z) = c_{n_1} \cdot z^k \cdot \prod_{i=1}^r   (z - a_i) \] 
  for an integer $r \ge 0$, non-zero complex numbers $a_1$, $\ldots$, $a_r$
  and an integer $k$.  We get from~\cite[(3.23) on  page~136]{Lueck(2002)} 
  
  \begin{eqnarray}
   {\det}_{\caln(\IZ)}\bigl(\Lambda (r_p)\bigr)  
     & =  & 
    |c_{n_1}| \cdot \prod_{\substack{i=1, \ldots, r\\|a_i| \ge 1}} |a_i|
     \ge  |c_{n_1}| = \lead(p). 
   \label{Costa}
 \end{eqnarray}
 Next we reduce the case $d \ge 2$ to the case
  $d=1$.  Choose $(m_1, m_2, \ldots, m_d) \in \IZ^d$ such that for every 
  $(n_1, n_2, \ldots,  n_d) \in \supp(p)$ we have $m_i + n_i \ge 0$ for $i = 1,2, \ldots, d$. Put 
  $q =   z_1^{m_1} \cdot \cdots \cdot z_d^{m_d} \cdot p$.  Then we get
   from~\cite[Theorem~3.14~(1) on page~128]{Lueck(2002)} 
   \begin{eqnarray*}
    {\det}_{\caln(\IZ^d)}(\Lambda(r_p)) 
   & = &
    {\det}_{\caln(\IZ^d)}(\Lambda(r_q)).  
  \end{eqnarray*} 
  Obviously $\lead(p) = \lead(q)$. Hence we can assume without loss of generality that for each 
  $(n_1, \ldots, n_d) \in  \supp(p)$ we have $n_i \ge 0$ for $i =1,2, \ldots , d$, otherwise replace $p$ by $q$.
  
  For $i = 1, \ldots, (d-1)$ define 
  \[ 
   b_i :=  1  + \max\{n_i \mid \exists (n_1, \ldots , n_d) \in {\supp}_{\IZ^d}(p)\}.  
   \] 
   Fix a sequence of natural numbers $k_2, \ldots, k_d$ satisfying $k_2 \ge   b_1$, 
  $k_3 \ge b_2 \cdot k_2$, $k_4 \ge  b_3 \cdot k_3$, $\ldots$, $k_d \ge  b_{d-1} \cdot
  k_{d-1}$.  Next we prove for $j = 2,3, \ldots, d$ by induction 
  \begin{eqnarray}
  m_1 + \sum_{i = 2}^{j-1} k_i \cdot m_i < k_j \quad\text{for}\; (m_1, \ldots , m_d) \in {\supp}_{\IZ^d}(p).
  \label{estimate_for_m_s_versus_k_s}
\end{eqnarray}
  The induction beginning $j = 2$ follows directly from $k_2 \ge  b_1$ and the definition of $b_1$.
  The induction step from $j -1 \ge 2$ to $j \le d$ follows from the following calculation
   \begin{eqnarray*}
   m_1 + \sum_{i = 2}^{j-1} k_i \cdot m_i 
   & = & 
    m_1 + \sum_{i = 2}^{j-2} k_i \cdot m_i  + k_{j-1} \cdot m_{j-1}
    \\
     & < & 
     k_{j-1} + k_{j-1} \cdot m_{j-1}
     \\
     & = & 
     k_{j-1} \cdot  (1 + m_{j-1})
     \\
     & \le  & 
     k_{j-1} \cdot  b_{j-1}
     \\
      & \le & 
      k_j.
    \end{eqnarray*}
    Next we show for any two $d$-tuples $(m_1, m_2, \ldots, m_d)$ and 
   $(n_1, n_2, \ldots , n_d)$ in $\supp_{\IZ^d}(p)$
   \begin{eqnarray} 
     & & 
    \quad m_1 + \sum_{i =  2}^d k_i \cdot m_i \le  n_1 + \sum_{i = 2}^d k_i \cdot n_i 
    \Longleftrightarrow 
    (m_1, m_2, \ldots, m_d) \le  (n_1, n_2, \ldots, n_d). 
    \label{implication_coefficients} 
    \end{eqnarray} 
    Suppose  $m_1 + \sum_{i =  2}^d k_i \cdot m_i \le   n_1 + \sum_{i = 2}^d k_i \cdot n_i$.
    We want to show $(m_1, m_2, \ldots, m_d) \le  (n_1, n_2, \ldots, n_d)$.
    If $(m_1, m_2, \ldots, m_d) =  (n_1, n_2, \ldots, n_d)$, the claim is true. Hence we only have to consider the case
    $(m_1, m_2, \ldots, m_d) \not=  (n_1, n_2, \ldots, n_d)$. Then there exists 
    $j \in \{1,2, \ldots, d\}$ such that $m_i = n_i$ holds for $i >j$ and $m_j \not=  n_j$.  We have
    \begin{eqnarray*}
     m_1 + \sum_{i =  2}^j k_i \cdot m_i  & \le &  n_1 + \sum_{i = 2}^j k_i \cdot n_i. 
   \end{eqnarray*}
   This implies  using~\eqref{estimate_for_m_s_versus_k_s}
    \[
    -k_j  <  \left(m_1 + \sum_{i =  2}^{j-1} k_i \cdot m_i\right)  - \left(n_1 + \sum_{i = 2}^{j-1} k_i \cdot n_i \right) \le 
    (n_j - m_j) \cdot k_j, 
     \]
     and hence $m_j \le n_j$. Since this implies $m_j < n_j$, we conclude  $(m_1, m_2, \ldots, m_d) <  (n_1, n_2, \ldots, n_d)$.
     This finishes the proof of the implication $\Longrightarrow$. It  remains to prove the implication
     $\Longleftarrow$.

     Suppose $(m_1, m_2, \ldots, m_d) <   (n_1, n_2, \ldots, n_d)$. 
     Then there exists 
     $j \in \{1,2, \ldots, d\}$ such that $m_i = n_i$ holds for $i >j$ and $m_j <  n_j$.  If $j = 1$,
     the claim is obviously true. It remains to treat the case $j \ge 2$. 
     
    We estimate
     \begin{eqnarray*}
     \lefteqn{\left(n_1 + \sum_{i = 2}^d k_i \cdot n_i \right) - \left( m_1 + \sum_{i =  2}^d k_i \cdot m_i\right)}
     & &
     \\
     & = & 
     \left(n_1 + \sum_{i = 2}^j k_i \cdot n_i \right) - \left( m_1 + \sum_{i =  2}^j k_i \cdot m_i\right)
     \\
     & = & 
     (n_1 -m_1) + \sum_{i = 2} ^{j-1} (n_i - m_i) \cdot k_i + (n_j - m_j) \cdot k_j
     \\
     & \ge & 
     - m_1  - \sum_{i =  2}^{j-1} k_i \cdot m_i + k_j
     \\
      & \stackrel{\eqref{estimate_for_m_s_versus_k_s}}{\ge} & 
    0.
   \end{eqnarray*}
   This finishes the proof of~\eqref{implication_coefficients}.

  Let $p[k_2, \ldots k_d]$
  be the polynomial in one variable $z$ given by $p(z, z^{k_2}, \ldots , z^{k_d})$. We conclude 
   from~\eqref{implication_coefficients} 
   \begin{eqnarray} 
   \lead(p) 
   & = & 
   \lead(p[k_2, \ldots k_d]).  
   \label{equality_for_Gamma} 
   \end{eqnarray} 
   The following equality 
   \begin{eqnarray} {\det}_{\caln(\IZ^d)}\bigl(\Lambda(r_p)\bigr) 
   & = & 
   \lim_{k_2 \to \infty} \; \lim_{k_3 \to \infty} \; \ldots \; \lim_{k_d \to \infty} \;
    {\det}_{\caln(\IZ)}\bigl(r_{p[k_2, \ldots, k_d]}\bigr) 
     \label{reduction_by_iterated_limit} 
   \end{eqnarray} 
  is proved   in~\cite[Appendix~4]{Boyd(1981speculations)} and~\cite[Theorem~2]{Lawton(1983)}.  
  For every sequence of natural numbers $k_1, k_2, \ldots, k_d$ satisfying $k_2 \ge 
  b_1$, $k_3 \ge b_2 \cdot k_2$, $k_4 \ge  b_3 \cdot k_3$, $\ldots$, $k_d \ge  b_{d-1} \cdot
  k_{d-1}$ we  get from~\eqref{equality_for_Gamma} and the already proved special case $d = 1$ applied
  to $p[k_2, \ldots, k_d]$ 
   \[ 
   {\det}_{\caln(\IZ)}\bigl(r_{p[k_2, \ldots, k_d]}\bigr) \ge   \lead(p). 
   \] 
   We conclude from~\eqref{reduction_by_iterated_limit} 
  \[
   {\det}_{\caln(\IZ^d)}\bigl(\Lambda(r_p)\bigr)
  \ge \lead(p).  
  \] 
  This finishes the proof of   Lemma~\ref{lem:estimate_for_elements_over_Zd}. 
  \end{proof} 

  Recall that we have defined for a finite-dimensional $\IZ^d$-representation $V$ and a subset $S \subseteq \IZ^d$
 the real number  $\theta(V,S)$ in~\eqref{theta(V,S)}.

\begin{lemma}
\label{lem:main_properties_of_theta(V,S)}
\begin{enumerate}

\item \label{lem:main_properties_of_theta(V,S):multiplicativity_in_V}
If $0 \to U \to V \to W \to 0$ is an exact sequence of finite-dimensional $G$-representations
and $S \subseteq G$ is a finite subset, then
\[
\theta(V,S) \ge  \theta(U,S) \cdot \theta(W,S);
\]

\item \label{lem:main_properties_of_theta(V,S):monotone}
Let $S,T$ be  finite subsets of $G$ and let $V$ be  finite-dimensional $G$-representation. Then
\[
S \subseteq T \implies \theta(V,S) \ge  \theta(U,T);
\]

\item \label{lem:main_properties_of_theta(V,S):products}
Let $S,T$ be  finite subsets of $G$ and let $V$ be a finite-dimensional $G$-repre\-sen\-ta\-tions. 
Put $S \cdot T = \{ s \cdot t \mid s \in S, t \in T\}$
Then
\[
\theta(V,S \cdot T ) =   \theta(V,S) \cdot \theta(V,T). 
\]

\end{enumerate}
\end{lemma}
\begin{proof}~\eqref{lem:main_properties_of_theta(V,S):multiplicativity_in_V}
This follows from
\[
{\det}_{\IC}(l_s \colon V \to V) =  {\det}_{\IC}(l_s  \colon U \to U) \cdot {\det}_{\IC}(l_s  \colon W \to W)
\]
for $s \in S$.
\\[2mm]~\eqref{lem:main_properties_of_theta(V,S):monotone} This is obvious.
\\[2mm]~\eqref{lem:main_properties_of_theta(V,S):products} 
This follows from
\[
{\det}_{\IC}(l_{s \cdot t} \colon V \to V) =  {\det}_{\IC}(l_s  \colon V \to V) \cdot {\det}_{\IC}(l_t  \colon V \to V).
\]
\end{proof}

  \begin{lemma} \label{lem:estimate_for_Lambda_circ_eta(r_p)}
  Consider a non-trivial element 
   \[ 
    p = \sum_{(n_1, \ldots , n_d) \in \IZ^d} c_{n_1, \ldots, n_d} \cdot z_1^{n_1} \cdot \cdots \cdot z_d^{n_d} 
   \] 
   in $\IC[\IZ^d]$.  

   Then
  $\Lambda \circ \eta(r_p) \colon \Lambda \circ \eta(\IC[\IZ^d]) \to \Lambda \circ   \eta(\IC[\IZ^d]) $ 
   is a weak isomorphism and we get 
   \[
  {\det}_{\caln(\IZ^d)}\bigl(\Lambda \circ \eta(r_p) \bigr)
  \ge 
  |\lead(p)|^{\dim_{\IC}(V)} \cdot \theta\bigl(V,\supp_{\IZ^d}(p)\bigr).
   \] 
   \end{lemma} 
   \begin{proof} 
    We conclude from Lemma~\ref{lem:K_1_and_det(C[Zd]}~\eqref{lem:K_1_and_det(C[Zd]:equivalent}
   that  $\Lambda \circ \eta(r_p) \colon \Lambda \circ \eta(\IC[\IZ^d]) \to \Lambda \circ   \eta(\IC[\IZ^d])$. 
   is a weak isomorphism.
  
   Since $\IZ^d$ is abelian, we can find a sequence of $\IZ^d$-subrepresentations
   $0 = V_0 \subseteq V_1 \subseteq V_2 \subseteq \ldots \subseteq V_l = V$ such that each quotient
   is $1$-dimensional. For $i = 1,2, \ldots, l$ choose complex numbers 
   $\lambda_{i,1}, \lambda_{i,2}, \ldots,  \lambda_{i,d}$ such that $(s_1, s_2, \ldots, s_d) \in \IZ^d$
   acts on $V_i/V_{i-1}$ by multiplication with 
   $\lambda_{i,1}^{s_1} \cdot \lambda_{i,2}^{s_2} \cdot  \cdots \cdot \lambda_{i,d}^{s_d}$.
   We can equip $V_i$ for $i = 1,2, \ldots, l$ and $V_i/V_{i-1}$ for $i = 1,2, \ldots, l$ with 
   equivalence classes of $\IC$-basis such that the obvious exact sequence
   $0 \to V_{i-1} \to V_i \to V_i/V_{i-1} \to 0$ is compatible with the equivalence classes of 
   $\IC$-basis for $i = 1,2, \ldots, l$. This can be arranged without loss of generality since 
   ${\det}_{\caln(\IZ^d)}\bigl(\Lambda \circ \eta_{V,[B_V]}(r_p) \bigr)$ is independent of the choice  
   of equivalence class of $\IC$-basis on $V$ by 
   Lemma~\ref{lem:dependency_on_V}~\eqref{lem:dependency_on_V:independence_of_basis}.
    We obtain a commutative diagram of finitely generated Hilbert $\caln(\IZ^d)$-modules
   \[
   \xymatrix{0  \ar[r]
   & 
   \Lambda \circ \eta_{V_{i-1}}(\IC[\IZ^d]) \ar[r] \ar[d]^{\Lambda \circ \eta_{V_{i-1}}(r_p)}
   &
   \Lambda \circ \eta_{V_i}(\IC[\IZ^d]) \ar[r] \ar[d]^{\Lambda \circ \eta_{V_i}(r_p)}
   &
   \Lambda \circ \eta_{V_i/V_{i-1}}(\IC[\IZ^d]) \ar[r] \ar[d]^{\Lambda \circ \eta_{V_i/V_{i-1}}(r_p)}
   &
   0
   \\
   0  \ar[r]
   & 
   \Lambda \circ \eta_{V_{i-1}}(\IC[\IZ^d]) \ar[r] 
   &
   \Lambda \circ \eta_{V_i}(\IC[\IZ^d]) \ar[r] 
   &
   \Lambda \circ \eta_{V_i/V_{i-1}}(\IC[\IZ^d]) \ar[r] 
   &
   0
   }
   \]
   We conclude from~\cite[Theorem~3.14~(2) on page~128]{Lueck(2002)}
   \[
    {\det}_{\caln(\IZ^d)}\bigl(\Lambda \circ \eta_{V_i}(r_p)\bigr) 
    = {\det}_{\caln(\IZ^d)}\bigl(\Lambda \circ \eta_{V_{i-1}}(r_p)\bigr) 
    \cdot {\det}_{\caln(\IZ^d)}\bigl(\Lambda \circ \eta_{V_i/V_{i-1}}(r_p)\bigr).
    \]
    This implies 
    \[
    {\det}_{\caln(\IZ^d)}\bigl(\Lambda \circ \eta_{V}(r_p)\bigr) 
    = \prod_{i = 1}^l {\det}_{\caln(\IZ^d)}\bigl(\Lambda \circ \eta_{V_i/V_{i-1}}(r_p)\bigr).    
    \]
    There is an obvious identification $\eta_{V_i/V_{i-1}}(\IC[\IZ^d]) =  \IC[\IZ^d]$ coming from Lemma~\ref{lem:diagonal_versus_first_coordinate}. 
    Under this identification $\eta_{V_i/V_{i-1}}(r_p) \colon \eta(\IC[\IZ^d]) \to   \eta(\IC[\IZ^d])$ 
  can be identified with $r_{p_i} \colon \IC[\IZ^d] \to \IC[\IZ^d]$ for the
  finite Laurent series 
\[
p_i = \sum_{(m_1, \ldots , m_d) \in \IZ^d} \left(c_{m_1, \ldots, m_r} \cdot 
\prod_{j = 1}^d \lambda_{i,j}^{m_j} \right) \cdot z_1^{m_1} \cdot \, \cdots \, \cdot z_d^{m_d}.  
\] 
Let $(n_1, n_2, \ldots, n_d) \in \IZ^d$ be the maximal element in $\supp_{\IZ}(p)$ with respect to the
lexicographic order for $i = 1,2, \ldots, d$. Then 
\[
\lead(p_i) = c_{n_1, n_2, \ldots, n_d} \cdot  \prod_{j = 1}^d \lambda_{i,j}^{-n_j} =
c_{n_1, n_2, \ldots, n_d} \cdot  {\det}_{\IC}\bigl(l_{(-n_1, -n_2, \ldots ,-n_d)} \colon V_i/V_{i-1} \to  V_i/V_{i-1}\bigr).
\]
We conclude from   Lemma~\ref{lem:estimate_for_elements_over_Zd}  
 \begin{eqnarray*}
   \lefteqn{\bigl| {\det}_{\caln(\IZ^d)}\bigl(\Lambda \circ \eta(r_p)\bigr) \bigr|}
    & & 
    \\
    & = & 
    \left|\prod_{i = 1}^l {\det}_{\caln(\IZ^d)}\bigl(\Lambda \circ \eta_{V_i/V_{i-1}}(r_p)\bigr)\right|
    \\
    & = & 
    \prod_{i = 1}^l \left|{\det}_{\caln(\IZ^d)}\bigl(\Lambda(r_{p_i})\bigr)\right|
    \\
    & \ge & 
    \prod_{i = 1}^l |\lead(p_i)| 
     \\
    & = & 
     \prod_{i = 1}^l |c_{n_1, n_2, \ldots, n_d}| \cdot  
     {\det}_{\IC}\bigl(l_{(-n_1, -n_2, \cdots, -n_d)} \colon V_i/V_{i-1} \to  V_i/V_{i-1}\bigr)
   \\
    & = & 
     \left|c_{n_1, n_2, \ldots, n_d}\right|^l  \cdot \prod_{i = 1}^l 
   {\det}_{\IC}\bigl(l_{(-n_1, -n_2, \cdots, -n_d)} \colon V_i/V_{i-1} \to  V_i/V_{i-1}\bigr)
    \\
    & = & 
     |\lead(p)|^l  \cdot {\det}_{\IC}\bigl(l_{(-n_1, -n_2, \cdots, -n_d)} \colon V\to  V\bigr)
    \\
    & = &
    |\lead(p)|^{\dim_{\IC}(V)}  \cdot {\det}_{\IC}\bigl(l_{(-n_1, -n_2, \cdots, -n_d)} \colon V\to  V\bigr)
     \\
     & \ge & 
     |\lead(p)|^{\dim_{\IC}(V)}   \cdot  \theta\bigl(V,\supp_{\IZ^d}(p)\bigr).
  \end{eqnarray*} 
 This finishes the proof of Lemma~\ref{lem:estimate_for_Lambda_circ_eta(r_p)}.
\end{proof}

\begin{lemma} \label{lem:lower_bound_for_det_over_Zd}
   Consider any matrix $A \in M_{r,s}(\IZ[\IZ^d])$.  Let $k$ be its rank over the quotient field $S^{-1}\IC[\IZ^d]$ of
  $\IC[\IZ^d]$. 
  Then
  \[
  {\det}_{\caln(\IZ^d)}\left(\Lambda \circ \eta(r_A) \colon \Lambda \circ \eta(\IZ[\IZ^d]^r)
  \to  \Lambda \circ \eta(\IZ[\IZ^d]^s)\right) 
  \ge 
  \theta(V,\supp_{\IZ^d}(A))^k.
\]
\end{lemma}
\begin{proof}
By Lemma~\ref{lem:Estimate_in_terms_of_minors} we can choose  a $(k,k)$-submatrix $B$ of $A$ with the property that 
$\Lambda \circ \eta(r_B)\colon \Lambda \circ \eta(\IZ[\IZ^d]^k) \to \Lambda \circ \eta(\IZ[\IZ^d]^k) $ is 
a weak isomorphism  and we get 
\[
{\det}_{\caln(\IZ^d)}\bigl(\Lambda \circ \eta(r_A) \bigr)
\ge 
{\det}_{\caln(\IZ^d)}\bigl(\Lambda \circ \eta(r_B)\bigr).
\]
We conclude from Lemma~\ref{lem:K_1_and_det(C[Zd]}  that
${\det}_{\IC[\IZ^d]}(B)$ is non-trivial and
\[
{\det}_{\caln(\IZ^d)}\bigl(\Lambda \circ \eta(r_B)\bigr)
=  
{\det}_{\caln(\IZ^d)}\bigl(\Lambda \circ \eta({r_{{\det}_{\IC[\IZ^d]}(B)}})\bigr).
\]
Since $\supp_{\IZ^d}(B) \subseteq \supp_{\IZ^d}(A)$, 
Lemma~\ref{lem:main_properties_of_theta(V,S)}~\eqref{lem:main_properties_of_theta(V,S):monotone} implies
\begin{eqnarray*}
\theta(V,\supp_{\IZ^d}(A))^k 
& \le &
\theta\bigl(V;\supp_{\IZ^d}(B)\bigr)^k.
\end{eqnarray*}
Hence it remains to show
\begin{eqnarray}
{\det}_{\caln(G)}\bigl(\Lambda \circ \eta(r_{{\det}_{\IC[\IZ^d]}(B)})\bigr)
& \ge  &
\theta\bigl(V;\supp_{\IZ^d}(B)\bigr)^k.
\label{lem:lower_bound_for_det_over_Zd:to_show}
\end{eqnarray}
One easily checks by inspecting the definition of ${\det}_{\IC[\IZ^d]}$ that
\[
\supp_{\IZ^d}\bigl({\det}_{\IC[\IZ^d]}(B)\bigr) 
\subseteq 
\{x_1 + x_2+ \cdots + x_k \mid x_i \in \supp_{\IZ^d}(B)\}.
\]
We conclude from 
Lemma~\ref{lem:main_properties_of_theta(V,S)}~\eqref{lem:main_properties_of_theta(V,S):monotone} 
and~\eqref{lem:main_properties_of_theta(V,S):products}
\[
\theta\bigl(V,\supp_{\IZ^d}\bigl({\det}_{\IC[\IZ^d]}(B))\bigr) 
\ge
\theta\bigl(V,\supp_{\IZ^d}(B)\bigr)^k.
\]
Now~\eqref{lem:lower_bound_for_det_over_Zd:to_show}  follows from 
Lemma~\ref{lem:estimate_for_Lambda_circ_eta(r_p)}
applied to ${\det}_{\IC[\IZ^d]}(B)$ since the assumption $A \in M_{r,s}(\IZ[\IZ^d])$ implies that
$\lead\bigl({\det}_{\IC[\IZ^d]}(B)\bigr)$ is a non-trivial integer and hence
we get the inequality  $\left|\lead\bigl({\det}_{\IC[\IZ^d]}(B)\bigr)\right| \ge 1$.
This finishes the proof of Lemma~\ref{lem:lower_bound_for_det_over_Zd}.
\end{proof}

%%%%%%%%%%%%%%%%%%%%%%%%%%%%%%%%%%%%%%%%%%%%%%%%%%%%%%%%%%%%%%%%%%%%%%%%%%%%%%%%%

\subsection{The special case, where the kernel $K$ of $\phi$ is finite}
\label{subsec:The_special_case_of_finite_ker(phi)}

The main result of this section is

\begin{proposition} \label{prop:main_theorem_true_for_finite_kernel}
Theorem~\ref{the:Determinant_class_and_twisting} possibly except the second inequality appearing in 
assertion~\eqref{the:Determinant_class_and_twisting:determinant} is true in the special case that  the kernel of $\phi$ is finite.
\end{proposition}

Its proof needs some preparation. Let $K$ be the kernel of $\phi$.

In the sequel we fix a group homomorphism $j \colon \IZ^d \to G$ and an integer $N \ge 1$
such that $\phi \circ j = N \cdot \id_{\IZ^d}$.  We will \emph{not} assume that the image
of $j$ is a normal subgroup of $G$. This will enable is to take $N = 1$ in the case where
$\phi$ has a section. 

The existence of $j$ is proved as follows. Let $|K|$ be the order of
the kernel $K$ of $\phi$.  Then for any $g \in G$ conjugation with $g$ defines an automorphism $c_g
\colon K \to K$.  We have $c_{g^{|K|!}}  = (c_g)^{|K|!}  = \id$ since $|\aut(K)|$ divides
${|K|!}$.  Consider any element $h \in G$.  Then $hg^{|K|!}h^{-1} = k \cdot g^{|K|!}$ for
some $k \in K$ since the image of $\phi$ is abelian.  We have $k \cdot g^{|K|!}  = g^{|K|!} \cdot k$.  We compute for $N = |K|!
\cdot |K|$
\[
hg^Nh^{-1} = \bigl(h g^{|K|!}h^{-1}\bigr)^{|K|} = \bigl(k \cdot g^{|K|!}\bigr)^{|K|} =
k^{|K|} \cdot \bigl(g^{|K|!}\bigr)^{|K|} = g^N.
\]
Hence $g^N$ is in the center of $G$. Now choose elements  $g_1, g_2, \ldots, g_d$ in $G$ such that
$\{\phi(g_1), \phi(g_2), \ldots , \phi(g_d)\}$ is the standard base of $\IZ^d$. Then we
can define the desired homomorphism $j \colon \IZ^d \to G$ by sending the $i$-th element of the
standard basis of $\IZ^d$ to $g_i^N$.

Let $\sigma' \colon \IZ/N \cdot \IZ \to \IZ$ be the map of sets sending the class of
$\overline{n}$ to the representative $n \in \IZ$ uniquely determined by $0 \le n \le N-1$.
Define
\begin{eqnarray}
& \overline{\sigma} := \prod_{l = 1}^{d} \sigma'  \colon (\IZ/N)^d \to \IZ^d. &
  \label{section_s}
\end{eqnarray}
The composite of  $\overline{\sigma }$ with the projection $\IZ^d \to (\IZ/N)^d$ is the identity

Denote by $\pr \colon G \to \im(j)\backslash G$ and by $\overline{\pr} \colon \IZ^d \to
(\IZ/N)^d$ the obvious projections, where $\im(j)\backslash G$ is the quotient of $G$ by
the obvious left $\im(j)$ action.

Choose maps of sets $\sigma \colon \im(j)\backslash G \to G$ and  
$\overline{\phi} \colon \im(j)\backslash G \to (\IZ/N)^d$
such that  the following diagram commutes
\begin{eqnarray}
& 
\xymatrix{\im(j)\backslash G \ar[r]^{\overline{\phi}} \ar[d]^{\sigma} \ar@/_{8mm}/[dd]_{\id}
& 
(\IZ/N)^d \ar[d]^{\overline{\sigma}} \ar@/^{8mm}/[dd]^{\id}
\\
G \ar[r]^{\phi} \ar[d]^{\pr}
&
\IZ^d \ar[d]^{\overline{\pr}}
\\
\im(j)\backslash G \ar[r]^{\overline{\phi}} 
& 
(\IZ/N)^d
}
&
\label{sigma_im(j)_to_G}
\end{eqnarray}

The map $\overline{\phi} \colon \im(j)\backslash G \to (\IZ/N)^d$ is uniquely determined by the
commutativity of the  lower square in the diagram above, and the map $\sigma$  making the 
upper square commutative exists since  $\ker(\phi) \cap \im(j) = \{1\}$ and hence the  map 
$\phi^{-1}(s) \mapsto \overline{\phi}^{-1}(\overline{\pr}(s))$ induced by $\pr$ is bijective for all $s \in \IZ^d$.

For any object $(M,[B_M])$ in $\FBMOD{\IZ G}$ we obtain an object
$j^*(M,[B_M])$ in $\FBMOD{\IZ[\IZ^d]}$ by restricting the $G$-action on $M$ to a
$\IZ^d$-action of $M$ by $j$ and equip $j^*M$ with the equivalence class of
$\IZ[\IZ^d]$-bases represented by the $\IZ[\IZ^d]$-basis 
$\{\sigma(\overline{g}) \cdot b \mid b \in B_M,  \overline{g}  \in \im(j)\backslash G \}$. 
Thus we get a functor of additive $\IC$-categories
\[
j^* \colon \FBMOD{\IC G} \to \FBMOD{\IC[\IZ^d]}.
\]
There is also an obvious restriction functor
\[
j^* \colon \FGHIL{\caln(G)} \to \FGHIL{\caln(\IZ^d)}.
\]
In the sequel we equip $\IZ G^m$ with the equivalence class of the
standard $\IZ G$-basis and $\IZ[\IZ^d]^n$ with the equivalence class of the standard
$\IZ[\IZ^d]$-basis. 
Consider $A \in M_{r,s}(\IZ G)$.  Let $r_A \colon \IC G^r \to \IC G^s$ be the 
associated morphism in  $\FBMOD{\IC G}$. As described above
it induces a morphism $j^*r_A \colon j^*\IC G^r \to j^*\IC G^s$ in $\FBMOD{\IC[\IZ^d]}$.

We have the $\IZ[\IZ^d]$-isomorphism
\[
\omega \colon \bigoplus_{\overline{g} \in \im(j)\backslash G} \IC[\IZ^d] \xrightarrow{\cong} j^*\IC G, \quad
(x_{\overline{g}})_{\overline{g} \in \im(j)\backslash G} 
\mapsto  \sum_{\overline{g} \in \im(j)\backslash G} x_{\overline{g}} \cdot \sigma(\overline{g}).
\]
It induces for each $m \ge 1$ a $\IC[\IZ^d]$-isomorphism
\[
\omega^m \colon \IC[\IZ^d]^{|\im(j)\backslash G|\cdot m}  \xrightarrow{\cong} j^*(\IC G^m)
\]
which respects the equivalence classes of $\IC[\IZ^d]$-bases. For an appropriate matrix
$B \in  M_{|\im(j)\backslash G|\cdot r, |\im(j)\backslash G|\cdot s}(\IZ[\IZ^d])$, we obtain a commutative diagram in
$\FBMOD{\IC[\IZ^d]}$
\begin{eqnarray}
&
\xymatrix{j^* \IC G^r \ar[r]^{j^*r_A} \ar[d]_{\omega^r}^{\cong}
& 
j^* \IC G^s  \ar[d]^{\omega^s}_{\cong}
\\
\IC[\IZ^d]^{|\im(j)\backslash G|\cdot r} \ar[r]_{r_B} 
&
\IC[\IZ^d]^{|\im(j)\backslash G|\cdot s}
}
&
\label{res_of_A__with_j_abstract}
\end{eqnarray}

\begin{lemma} \label{phi(supp(B)_in_terms_of_supp(A)}
Every element in $\phi \circ j(\supp_{\IZ^d}(B)) = N \cdot \supp_{\IZ^d}(B)$ can be written
as a sum $z +z'$ for elements
\begin{eqnarray*}
z & \in &\phi(\supp_G(A)),
\\
z' & \in &
\{\overline{\sigma}(y_0) -   \overline{\sigma}(y_0 + y) \mid y_0 \in (\IZ/N)^d, y \in \overline{\pr} \circ \phi(\supp_{G}(A))\}.
\end{eqnarray*}

In particular we get in the case where $N = 1$, or equivalently, $j$ is a section of $\phi$,
\[
\phi \circ j(\supp_{\IZ^d}(B)) = \supp_{\IZ^d}(B)\subseteq \phi(\supp_G(A)). 
\]
\end{lemma}
\begin{proof}
It suffices to check this in the special case, where $A$ is a
$(1,1)$-matrix whose only entry we write as 
$\sum_{\overline{g} \in \im(j)\backslash G}  \sum_{x_{\overline{g}}   \in \IZ^d} \; 
m_{\overline{g},x_{\overline{g}}} \cdot  j(x_{\overline{g}}) \cdot \sigma(\overline{g})$ 
for $m_{\overline{g},x_{\overline{g}}} \in \IZ$. Then $j^*r_A \colon j^* \IZ G \to j^* \IZ G$ sends the
element $\sigma(\overline{g_0})$ of the $\IZ[\IZ^d]$-basis 
$\{\sigma(\overline{g_0}) \mid \overline{g_0}  \in \im(j)\backslash G\}$ for
$j^*\IZ G$ to 
\begin{multline*}
\sum_{\overline{g} \in \im(j)\backslash G}  \; \sum_{x_{\overline{g}}   \in \IZ^d} \; m_{\overline{g},x_{\overline{g}}} 
\cdot  \sigma(\overline{g_0})  \cdot j(x_{\overline{g}}) \cdot \sigma(\overline{g})
\\
 = 
\sum_{\overline{g} \in \im(j)\backslash G}  \; \sum_{x_{\overline{g}}   \in \IZ^d} \; m_{\overline{g},x_{\overline{g}}} 
\cdot  \sigma(\overline{g_0})  \cdot j(x_{\overline{g}}) \cdot \sigma(\overline{g}) 
\cdot \sigma \circ \pr\bigl(\sigma(\overline{g_0})  \cdot j(x_{\overline{g}}) \cdot \sigma(\overline{g})\bigr)^{-1} 
\\
\cdot \sigma \circ \pr\bigl(\sigma(\overline{g_0})  \cdot j(x_{\overline{g}}) \cdot \sigma(\overline{g})\bigr).
\end{multline*}
Since we get in $\im(j)\backslash G$
\begin{eqnarray*}
\pr\left(\sigma \circ \pr\bigl(\sigma(\overline{g_0})  \cdot j(x_{\overline{g}}) \cdot \sigma(\overline{g})\bigr)\right)
& = & 
\pr\bigl(\sigma(\overline{g_0})  \cdot j(x_{\overline{g}}) \cdot \sigma(\overline{g})\bigr),
\end{eqnarray*}
there exists for each $\overline{g_0}, \overline{g} \in \im(j)\backslash G$ and $x \in \IZ^d$
an element $y_{\overline{g_0},\overline{g},x} \in \IZ^d$
uniquely determined by 
\begin{eqnarray*}
j(y_{\overline{g_0},\overline{g},x}) 
& = & 
\sigma(\overline{g_0})  \cdot j(x) \cdot \sigma(\overline{g}) 
\cdot \sigma \circ \pr\bigl(\sigma(\overline{g_0})  \cdot j(x) \cdot \sigma(\overline{g})\bigr)^{-1}.
\end{eqnarray*}
We get
\begin{eqnarray*}
j^*r_A(\sigma(\overline{g_0})) 
& = & 
\sum_{\overline{g} \in \im(j)\backslash G}  \;\sum_{x_{\overline{g}}   \in \IZ^d} \; m_{\overline{g},x_{\overline{g}}} 
\cdot  j(y_{\overline{g_0},\overline{g},x_{\overline{g}}})
\cdot \sigma \circ \pr\bigl(\sigma(\overline{g_0})  \cdot j(x_{\overline{g}}) \cdot \sigma(\overline{g})\bigr).
\end{eqnarray*}
Hence the matrix $B$ has as entry for $\overline{g_0},\overline{g_1} \in \im(j)\backslash G$
\begin{eqnarray}
B_{\overline{g_0},\overline{g_1}}  & = & 
\sum_{\overline{g} \in \im(j)\backslash G} \;\;
\sum_{x_{\overline{g}}   \in \IZ^d, \overline{g_1} = \pr(\sigma(\overline{g_0})  \cdot j(x_{\overline{g}}) \cdot \sigma(\overline{g}))}  \;
 m_{\overline{g},x_{\overline{g}}} 
\cdot  y_{\overline{g_0},\overline{g},x_{\overline{g}}}.
\label{res_of_A__with_j_explicit}
\end{eqnarray}
Suppose that $s \in \IZ^d$ belongs to $\supp_{\IZ^d}(B)$. Then
there must exist $\overline{g},\overline{g_0} \in \im(j)\backslash G$ with $m_{\overline{g},x_{\overline{g}}} \not= 0$ and
$s = y_{\overline{g_0},\overline{g},x_{\overline{g}}}$, or, equivalently,  with 
\begin{eqnarray*}
j(s) 
& =  & 
\sigma(\overline{g_0})  \cdot j(x_{\overline{g}}) \cdot \sigma(\overline{g}) 
\cdot \sigma \circ \pr\bigl(\sigma(\overline{g_0})  \cdot j(x_{\overline{g}}) \cdot \sigma(\overline{g})\bigr)^{-1}.
\end{eqnarray*}
Since $m_{\overline{g},x_{\overline{g}}} \not= 0$, we have $j(x_{\overline{g}}) \cdot \sigma(\overline{g}) \in \supp_{G}(A)$. This implies
\begin{eqnarray*}
&\phi\bigl(j(x_{\overline{g}}) \cdot \sigma(\overline{g})\bigr) \in \phi(\supp_{G}(A)); &
\\
&\overline{\phi}(\overline{g}) \in \overline{\pr} \circ \phi(\supp_{G}(A)). &
\end{eqnarray*}
Notice that $\im(j)\backslash G$ is not necessarily a group and $\pr \colon G \to \im(j)\backslash G$
not necessarily  a group homomorphism, whereas $\overline{\pr} \colon \IZ^d \to (\IZ/N)^d$ is a group homomorphism.
We conclude 
\begin{eqnarray*}
\lefteqn{\phi \circ j(s)}
& & 
\\  
& = & 
\phi\left(\sigma(\overline{g_0})  \cdot j(x_{\overline{g}}) \cdot \sigma(\overline{g}) 
\cdot \sigma \circ \pr\bigl(\sigma(\overline{g_0})  \cdot j(x_{\overline{g}}) \cdot \sigma(\overline{g})\bigr)^{-1}\right)
\\
& = & 
\phi(\sigma(\overline{g_0}))  + \phi\bigl(j(x_{\overline{g}}) \cdot \sigma(\overline{g})\bigr)
- \phi \circ \sigma \circ \pr\bigl(\sigma(\overline{g_0})  \cdot j(x_{\overline{g}}) \cdot \sigma(\overline{g})\bigr)
\\
& = & 
\phi(\sigma(\overline{g_0}))  + \phi\bigl(j(x_{\overline{g}}) \cdot \sigma(\overline{g})\bigr)
- \overline{\sigma} \circ \overline{\pr} \circ \phi \bigl(\sigma(\overline{g_0})  \cdot j(x_{\overline{g}}) \cdot \sigma(\overline{g})\bigr)
\\
& = & 
\phi(\sigma(\overline{g_0}))  + \phi\bigl(j(x_{\overline{g}}) \cdot \sigma(\overline{g})\bigr)
- \overline{\sigma} \bigl(\overline{\pr} \circ \phi \circ \sigma(\overline{g_0})  + \overline{\pr} \circ \phi(j(x_{\overline{g}})) 
+ \overline{\pr} \circ \phi \circ \sigma(\overline{g})\bigr)
\\
& = & 
\phi(\sigma(\overline{g_0}))  + \phi\bigl(j(x_{\overline{g}}) \cdot \sigma(\overline{g})\bigr)
- \overline{\sigma} \bigl(\overline{\phi} \circ \pr \circ \sigma(\overline{g_0})  + \overline{\phi} \circ \pr(j(x_{\overline{g}}))
+  \overline{\phi} \circ \pr \circ \sigma(\overline{g})\bigr)
\\
& = & 
\phi(\sigma(\overline{g_0}))  + \phi\bigl(j(x_{\overline{g}}) \cdot \sigma(\overline{g})\bigr)
- \overline{\sigma} \bigl(\overline{\phi} (\overline{g_0})  + 0 + \overline{\phi}(\overline{g})\bigr)
\\
& = & 
\phi(\sigma(\overline{g_0}))  + \phi\bigl(j(x_{\overline{g}}) \cdot \sigma(\overline{g})\bigr)
- \overline{\sigma} \bigl(\overline{\phi} (\overline{g_0})  + \overline{\phi}(\overline{g})\bigr).
\end{eqnarray*}

Put $z = \phi \bigl(j(x_{\overline{g}}) \cdot \sigma(\overline{g})\bigr)$, $y = \overline{\phi}(\overline{g})$ 
and $y_0 = \overline{\phi}(\overline{g_0})$.
Then we have 
\begin{eqnarray*}
z  & \in & 
\phi(\supp_{G}(A));
\\
y & \in & \overline{\pr} \circ \phi(\supp_{G}(A));
\\
y_0  & \in & (\IZ/N)^d;
\\
\phi \circ j(s)  
& =  & 
z  +  \overline{\sigma}(y_0) -   \overline{\sigma}(y_0 + y).
\end{eqnarray*}
This finishes the proof of Lemma~\ref{phi(supp(B)_in_terms_of_supp(A)}.
\end{proof}

Recall that we have associated to $V$ numbers $\delta_i$ and elements $\epsilon_i \in \{\pm 1\}$ for $i = 1,2 \ldots, d$
in Notation~\ref{not:nu_and_theta}.

\begin{lemma}\label{lem:more_on_supp(B)}
Suppose that there is an integer $M$ satisfying $1 \le M \le N-1$ and
\[
\phi\bigl(\supp_{G}(A)\bigr) \subseteq
\bigl\{(s_1, \ldots, s_d) \; \bigl|\; 1 \le \epsilon_l \cdot s_l \le M\; \text{for}\; l = 1,2, \ldots, d\bigr\}.
\]
Then we get
\[
\phi \circ j\bigl(\supp_{\IZ^d}(B)\bigr) \subseteq 
\bigl\{(s_1, \ldots, s_d) \; \bigl| \;  - \epsilon_l \cdot s_l \le (M-1) \; \text{for} \; l = 1,2, \ldots, d \bigr\}
\]
and
\[
\theta\bigl((\phi \circ j)^* V,\supp_{\IZ^d}(B)\bigr)
\ge \prod_{l= 1}^d |\delta_l|^{- \epsilon_l \cdot (M-1)}.
\]
\end{lemma}
\begin{proof}
We get for the section $\sigma' \colon \IZ/N \to \IZ$ introduced before~\eqref{section_s}
\[
\sigma'(y_0) -   \sigma'(y_0 + y)
= 
\begin{cases}  
-\sigma'(y)   & \text{if} \;\sigma'(y_0)  + \sigma'(y) \le N-1;
\\
N -\sigma'(y) & \text{if} \; \sigma'(y_0)  + \sigma'(y) \ge N.
\end{cases}
\]
We conclude for $y,y_0 \in \IZ/N$
\[
\begin{array}{lclccl}
\sigma'(y_0) -   \sigma'(y_0 + y) & \ge &   -M   & & \text{if} & \sigma'(y)  \in [0,M];
\\
 \sigma'(y_0) -   \sigma'(y_0 + y)  & \le  &   +M   & &\text{if}  & \sigma'(y)  \in [N-M,N-1].
\end{array}
\]

Fix $l \in \{1,2, \ldots, d\}$. Consider  $s = (s_1, s_2, \ldots, s_d) \in \phi\bigl(\supp_{G}(A)\bigr) $.
For each $l \in \{1,2, \ldots, d\}$ we have by assumption 
\[
1 \le \epsilon_l \cdot s_l \le M,
\] 
and hence
$\sigma'(s_l) \in [0,M]$ if $\epsilon_l = 1$ and $\sigma'(s_l) \in [N-M,N-1]$, if $\epsilon_l = -1$.
This implies
\begin{multline*}
\{\overline{\sigma}(y_0) -   \overline{\sigma}(y_0 + y) \mid y_0 \in (\IZ/N)^d, y \in \overline{\pr} \circ \phi(\supp_{G}(A))\}
\\ \subseteq
\bigl\{(s_1, \ldots, s_d) \; \bigl| \;  - \epsilon_l \cdot s_l \le M \; \text{for} \; l = 1,2 \ldots, d \bigr\}.
\end{multline*}
Since we have
\begin{eqnarray*}
\phi\bigl(\supp_{G}(A)\bigr) 
& \subseteq &
\bigl\{(s_1, \ldots, s_d) \; \bigl| \;  -\epsilon_l \cdot s_l \le -1 \; \text{for} \; l = 1,2 \ldots, d \bigr\}
\end{eqnarray*}
by assumption, Lemma~\ref{phi(supp(B)_in_terms_of_supp(A)} implies
\[
\phi \circ j\bigl(\supp_{\IZ^d}(B)\bigr) \subseteq 
\bigl\{(s_1, \ldots, s_d) \; \bigl| \;  -\epsilon_l \cdot s_l \le (M-1) \; \text{for} \; l = 1,2 \ldots, d \bigr\}.
\]
Since $|\delta_l|^{\epsilon_l} \ge 1$ holds, we get  $|\delta_l|^{n} \ge |\delta_l|^{-\epsilon_l \cdot (M-1)}$ 
for any integer $n$ satisfying  $-\epsilon_l \cdot n \le (M-1)$. We have
\begin{eqnarray*}
\lefteqn{\theta\bigl((\phi \circ j)^* V,\supp_{\IZ^d}(B)\bigr)}
& & 
\\
& := & 
\min\bigl\{|{\det}_{\IC}(l_{-s} \colon (\phi \circ j)^* V \to (\phi \circ j)^* V)|\; \bigl|\; s\in \supp_{\IZ^d}(B)\bigr\}
\\
& = & 
\min\bigl\{|{\det}_{\IC}(l_{-s}\colon V \to V)|\; \bigl|\; s \in \phi \circ j\bigl(\supp_{\IZ^d}(B)\bigr)\bigr\}
\\
& \ge & 
\min\bigl\{|{\det}_{\IC}(l_{(-s_1,-s_2, \ldots, -s_d)} \colon V \to V)| \; \bigl|\; (s_1, \ldots, s_d) \in \IZ^d, 
 -\epsilon_l \cdot s_l \le (M-1) \; \text{for} \; l = 1,2 \ldots, d \bigr\}
\\
& = & 
\min\bigl\{|\delta_1^{s_1} \cdot \delta_2^{s_2} \cdot \cdots \cdot \delta_d^{s_d}| \; \bigl|\; (s_1, \ldots, s_d) \in \IZ^d, 
 -\epsilon_l \cdot s_l \le (M-1) \; \text{for} \; l = 1,2 \ldots, d \bigr\}
\\
& \ge & 
\prod_{l= 1}^d |\delta_l|^{- \epsilon_l \cdot (M-1)}.
\end{eqnarray*}
This finishes the proof of Lemma~\ref{lem:more_on_supp(B)}.
\end{proof}

\begin{lemma}\label{lem:det_G(A)_with_positive_support}\
\begin{enumerate}
\item \label{lem:det_G(A)_with_positive_support:with_section}
Suppose that $\phi \colon G \to \IZ^d$ has a section. Then
\begin{eqnarray*}
{\det}_{\caln(G)}\bigl(\Lambda^G \circ \eta_{\phi^*V}(r_{A})\bigr)
& \ge  & 
\theta\bigl(V,\phi(\supp_{G}(A))\bigr)^{r - \dim_{\caln(G)}(\ker(\Lambda(r_A)))};
\end{eqnarray*}

\item \label{lem:det_G(A)_with_positive_support:without_section}
Suppose that there is an integer $M$ satisfying $1 \le M \le N-1$ and
\[
\phi\bigl(\supp_{G}(A)\bigr) \subseteq
\bigl\{(s_1, \ldots, s_d) \; \bigl|\; 1 \le \epsilon_l \cdot s_l \le M\; \text{for}\; l = 1,2 \ldots, d\bigr\}.
\]
Then we get
\[
{\det}_{\caln(G)}\bigl(\Lambda^G \circ \eta_{\phi^*V}(r_{A})\bigr) 
\ge 
\left(\prod_{l= 1}^d |\delta_l|^{- \epsilon_l \cdot (M-1)}\right)^{r - \dim_{\caln(G)}(\ker(\Lambda^G(r_A)))}.
\]
\end{enumerate}
\end{lemma}
\begin{proof}
Let $j \colon \IZ^d \to G$ be a homomorphism and $N \ge 1$ be an integer such that
$\phi \colon G \to \IZ^d$ is $N \cdot \id_{\IZ^d}$.
We get from Lemma~\ref{lem:invariance_of_L2-Betti_numbers_under_twisting},
Lemma~\ref{lem:lower_bound_for_det_over_Zd} 
and~\cite[Theorem~1.12~(6) on page~22]{Lueck(2002)}

\begin{eqnarray*}
\lefteqn{{\det}_{\caln(\IZ^d)}\bigl(j^*\bigl(\Lambda^G \circ \eta_{\phi^*V}(r_{A})\bigr)\bigr)}
& & 
\\
& = & 
{\det}_{\caln(\IZ^d)}\bigl(\Lambda^{\IZ^d}\bigl(j^*\eta_{\phi^*V}(r_{A})\bigr)\bigr)
\\
& = & 
{\det}_{\caln(\IZ^d)}\bigl(\Lambda^{\IZ^d} \circ \eta_{(\phi \circ j)^*V}(j^*r_{A})\bigr)
\\
& = & 
{\det}_{\caln(\IZ^d)}\bigl(\Lambda^{\IZ^d} \circ \eta_{(\phi \circ j)^*V}(r_B)\bigr)
\\
& \ge & 
\theta\bigl((\phi \circ j)^*V,\supp_{\IZ^d}(B)\bigr)^{\dim_{S^{-1}\IC[\IZ^d]}(B)}
\\
& = & 
\theta\bigl((\phi \circ j)^*V,\supp_{\IZ^d}(B)\bigr)^{r \cdot [G : \im(j)] -\rk_{S^{-1}\IC[\IZ^d]}(\ker(r_B))}
\\
& = & 
\theta\bigl((\phi \circ j)^*V,\supp_{\IZ^d}(B)\bigr)^{r \cdot [G : \im(j)] - \dim_{\caln(\IZ^d)}(\ker(\Lambda^{\IZ^d}(r_B)))}
\\
& = & 
\theta\bigl((\phi \circ j)^*V,\supp_{\IZ^d}(B)\bigr)^{[G : \im(j)] \cdot r - \dim_{\caln(\IZ^d)}(\ker(\Lambda^{\IZ^d}(j^*r_A)))}
\\
& = & 
\theta\bigl((\phi \circ j)^*V,\supp_{\IZ^d}(B)\bigr)^{[G : \im(j)] \cdot r - \dim_{\caln(\IZ^d)}(j^*\ker(\Lambda^G(r_A)))}
\\
& = & 
\theta\bigl((\phi \circ j)^*V,\supp_{\IZ^d}(B)\bigr)^{[G : \im(j)] \cdot r-  [G : \im(j)] \cdot \dim_{\caln(G)}(\ker(\Lambda^G(r_A)))}
\\
& = & 
\left(\theta\bigl((\phi \circ j)^*V,\supp_{\IZ^d}(B)\bigr)^{r - \dim_{\caln(G)}(\ker(\Lambda^G(r_A)))}\right)^{[G : \im(j)]}
\\
& = & 
\left(\theta\bigl(V,\phi \circ j(\supp_{\IZ^d}(B))\bigr)^{r - \dim_{\caln(G)}(\ker(\Lambda^G(r_A)))}\right)^{[G : \im(j)]}.
\end{eqnarray*}
Since  ${\det}_{G}\bigl(\Lambda^{G} \circ \eta_{\phi^*V}(r_{A})\bigr)
= {\det}_{\caln(G)}\bigl(j^*\bigl(\Lambda^{G} \circ \eta_{\phi^*V}(r_{A})\bigr)\bigr)^{[G : \im(j)]^{-1}}$
follows from~\cite[Theorem~3.14~(5) on page~128]{Lueck(2002)}, we conclude
\begin{eqnarray}
& & 
{\det}_{G}\bigl(\Lambda^{G} \circ \eta_{\phi^*V}(r_{A})\bigr)
\ge  
\theta\bigl(V,\phi \circ j(\supp_{\IZ^d}(B))\bigr)^{r - \dim_{\caln(G)}(\ker(\Lambda^G(r_A)))}.
\label{lem:det_G(A)_with_positive_support:some_inequality}
\end{eqnarray}
\eqref{lem:det_G(A)_with_positive_support:with_section}
By assumption we can choose $j$ such that $\phi \circ j = \id_{\IZ^d}$ and $N = 1$. 
We conclude from Lemma~\ref{phi(supp(B)_in_terms_of_supp(A)}
\begin{eqnarray*}
\phi \circ j(\supp_{\IZ^d}(B)) & \subseteq & \phi(\supp_G(A)). 
\end{eqnarray*}
Hence we get from 
Lemma~\ref{lem:main_properties_of_theta(V,S)}~\eqref{lem:main_properties_of_theta(V,S):monotone} 
and~\eqref{lem:det_G(A)_with_positive_support:some_inequality}
\begin{eqnarray*}
{\det}_{G}\bigl(\Lambda^{G} \circ \eta_{\phi^*V}(r_{A})\bigr)
& \ge  & 
\theta\bigl(V,\phi(\supp_{G}(A))\bigr)^{r - \dim_{\caln(G)}(\ker(\Lambda^G(r_A)))}.
\end{eqnarray*}
This finishes the proof of assertion~\eqref{lem:det_G(A)_with_positive_support:with_section}.
\\[1mm]~\eqref{lem:det_G(A)_with_positive_support:without_section}
This follows from Lemma~\ref{lem:more_on_supp(B)} and~\eqref{lem:det_G(A)_with_positive_support:some_inequality}.
 \end{proof}

\begin{lemma}\label{lem:final_estimate_for_det(r_z)_for_z_in_supp(A)}
Let $M$ be an integer such that $3 \le 2M + 1 \le N-1$ holds and we have
\begin{eqnarray*}
\phi\bigl(\supp_{G}(A)\bigr) 
&\subseteq &
\bigl\{(s_1, \ldots, s_d) \; \bigl|\; -M \le   s_l \le M\; \text{for}\; l = 1,2 \ldots, d\bigr\}.
\end{eqnarray*}

Then 
\begin{multline*}
{\det}_{\caln(G)}\bigl(\Lambda^{G} \circ \eta_{\phi^*V}(r_{A})\bigr)
\\
\ge 
\left(\bigl|\bigl|l_{(-\epsilon_1 \cdot (M+1), \ldots, -\epsilon_d \cdot (M+1))} \colon 
V \to V\bigr|\bigr|^{-\dim_{\IC}(V)} \cdot \prod_{l= 1}^d |\delta_l|^{- \epsilon_l \cdot 2M}\right)^{r - \dim_{\caln(G)}(\ker(\Lambda^{G}(r_A)))}.
\end{multline*}
\end{lemma}
\begin{proof}
Since $\phi$ is surjective by assumption, we can choose $g \in G$ with 
\begin{eqnarray}
\phi(g) 
& = & 
\bigl(\epsilon_1 \cdot (M+1), \epsilon_2 \cdot (M+1), \ldots, \epsilon_d \cdot (M+1)\bigr).
\label{Schluckauf}
\end{eqnarray}
Put $A' = A \cdot (g \cdot I_s)$. Then we get
\[
\phi\bigl(\supp_{G}(A')\bigr) \subseteq
\bigl\{(s_1, \ldots, s_d) \; \bigl|\; 1  \le   \epsilon_l \cdot s_l \le 2 M+1\; \text{for}\; l = 1,2 \ldots, d\bigr\}.
\]
Let $k \colon \overline{\im(\Lambda^{G} \circ \eta_{\phi^*V}(r_{A})\bigr)} \to \Lambda^{G} \circ \eta_{\phi^*V}(\IC G^s)$
be the inclusion of the closure of the image of $\Lambda^{G} \circ \eta_{\phi^*V}(r_{A})$ into
$\Lambda^{G} \circ \eta_{\phi^*V}(\IC G^s)$. Notice that $\Lambda^{G} \circ \eta_{\phi^*V}(r_{g \cdot I_s})$ is an isomorphism and
hence we have
\begin{eqnarray*}
0  
& < & 
{\det}_{\caln(G)}\bigl(\Lambda^{G} \circ \eta_{\phi^*V}(r_{g \cdot I_s}));
\\
0 
& < & 
{\det}_{\caln(G)}\bigl(\Lambda^{G} \circ \eta_{\phi^*V}(r_{g \cdot I_s}) \circ k\bigr).
\end{eqnarray*}
 We conclude from~\cite[Theorem~3.14~(1) on page~128 and Lemma~3.15~(3) on page~129]{Lueck(2002)}
\begin{eqnarray}
\lefteqn{{\det}_{G}\bigl(\Lambda^{G} \circ \eta_{\phi^*V}(r_{A'})\bigr)}
& & 
\label{lem:final_estimate_for_det(r_z)_for_z_in_supp(A):(1)}
\\
& = & 
{\det}_{G}\bigl(\Lambda^{G} \circ \eta_{\phi^*V}(r_{A \cdot (g \cdot I_s)}\bigr)
\nonumber
\\
& = &
{\det}_{\caln(G)}\bigl(\Lambda^{G} \circ \eta_{\phi^*V}(r_{g \cdot I_s} \circ r_A)\bigr)
\nonumber
\\
& = &
{\det}_{\caln(G)}\bigl((\Lambda^{G} \circ \eta_{\phi^*V}(r_{g \cdot I_s})) \circ (\Lambda^{G} \circ \eta_{\phi^*V}(r_{A}))\bigr)
\nonumber
\\
& = &
{\det}_{\caln(G)}\bigl(\Lambda^{G} \circ \eta_{\phi^*V}(r_{g \cdot I_s}) \circ k\bigr)
\cdot {\det}_{\caln(G)}\bigl(\Lambda^{G} \circ \eta_{\phi^*V}(r_{A})\bigr).
\nonumber
\end{eqnarray}
The morphism  $\Lambda^{G} \circ \eta_{\phi^*V}(r_{g \cdot I_s}) \colon \Lambda^{G} \circ \eta_{\phi^*V}(\IC G^s) 
\to \Lambda^{G} \circ \eta_{\phi^*V}(\IC G^s)$  can be identified with the composite
\begin{multline*}
\Lambda^G\bigl((\IC G^s \otimes_{\IC} \phi^*V)_1\bigr) \xrightarrow{\Lambda^G(r_g \otimes \id_V)} 
\Lambda^G\bigl((\IC G^s \otimes_{\IC} \phi^*V)_1\bigr) 
\\
\xrightarrow{\Lambda^G(\id_{\IC G^s} \otimes l_{g^{-1}})} \Lambda^G\bigl((\IC G^s \otimes_{\IC} \phi^*V)_1\bigr),
\end{multline*}
where $(\IC G^s \otimes_{\IC} \phi^*V)_1$ has been introduced before
Lemma~\ref{lem:diagonal_versus_first_coordinate}.
The first one is an isometric isomorphism of finitely generated Hilbert $\caln(G)$-modules.
The second one is an isomorphism of finitely generated Hilbert $\caln(G)$-modules. We get for the operator norm
 of $\Lambda^{G} \circ \eta_{\phi^*V}(r_{g \cdot I_s})$
\begin{eqnarray*}
||\Lambda^{G} \circ \eta_{\phi^*V}(r_{g \cdot I_s})||
& = & 
||\Lambda^G(\id_{\IC G^s} \otimes l_{g^{-1}})||
\\
& =  &
||l_{g^{-1}} \colon \phi^*V \to \phi^*V||
 \\
& =  &
||l_{\phi(g^{-1})} \colon V \to V||
\\
& \stackrel{\eqref{Schluckauf}}{=} & 
\bigl|\bigl|l_{(-\epsilon_1 \cdot (M+1), -\epsilon_2 \cdot (M+1), \ldots, -\epsilon_d \cdot (M+1))} \colon V \to V\bigr|\bigr|.
\end{eqnarray*}
We compute using~\cite[Theorem~1.12~(6) on page~22]{Lueck(2002)} and 
Lemma~\ref{lem:invariance_of_L2-Betti_numbers_under_twisting}
\begin{eqnarray}
\label{Hollywood}
& & 
\\
\lefteqn{{\dim}_{\caln(G)}\big(\ker(\Lambda^G \circ \eta_{\phi^* V}(r_A))\bigr)}
& & 
\nonumber
\\
& = &
[G : \im(j)] \cdot {\dim}_{\caln(\IZ^d)}\big(j^*(\ker(\Lambda^G \circ \eta_{\phi^* V}(r_A)))\bigr)
\nonumber
\\
& = &
[G : \im(j)] \cdot {\dim}_{\caln(\IZ^d)}\big(\ker(j^*(\Lambda^G \circ \eta_{\phi^* V}(r_A))\bigr)
\nonumber
\\
& = &
[G : \im(j)] \cdot {\dim}_{\caln(\IZ^d)}\big(\ker(\Lambda^{\IZ^d}(j^*\eta_{\phi^* V}(r_A))\bigr)
\nonumber
\\
& = &
[G : \im(j)] \cdot {\dim}_{\caln(\IZ^d)}\big(\ker(\Lambda^{\IZ^d} \circ \eta_{(\phi \circ j)^* V}(j^*r_A))\bigr)
\nonumber
\\
& = &
[G : \im(j)] \cdot \dim_{\IC}(V) \cdot {\dim}_{\caln(\IZ^d)}\big(\ker(\Lambda^{\IZ^d}(j^* r_A))\bigr)
\nonumber
\\
& = &
[G : \im(j)] \cdot \dim_{\IC}(V) \cdot {\dim}_{\caln(\IZ^d)}\big(\ker(j^*\Lambda^{G}(r_A))\bigr)
\nonumber
\\
& = &
[G : \im(j)] \cdot \dim_{\IC}(V) \cdot {\dim}_{\caln(\IZ^d)}\big(j^*(\ker(\Lambda^{G}(r_A)))\bigr)
\nonumber
\\
& = & 
\dim_{\IC}(V) \cdot  {\dim}_{\caln(G)}\bigl(\ker(\Lambda^G (r_A))\bigr).
 \nonumber
\end{eqnarray}

We conclude using Lemma~\ref{lem:det_estimate_in_terms_of_norm}
\begin{eqnarray}
\label{lem:final_estimate_for_det(r_z)_for_z_in_supp(A):(2)}
& & 
\\
\lefteqn{{\det}_{\caln(G)}\bigl(\Lambda^{G} \circ \eta_{\phi^*V}(r_{g \cdot I_s}) \circ k \bigr)}
& &
\nonumber 
\\
& \le & 
||\Lambda^{G} \circ \eta_{\phi^*V}(r_{g \cdot I_s}) \circ k)||^{\dim_{\caln(G)}(\im(\Lambda^{G} \circ \eta_{\phi^*V}(r_{g \cdot I_s})  \circ k))} 
\nonumber
\\
& = & 
||\Lambda^{G} \circ \eta_{\phi^*V}(r_{g \cdot I_s}) \circ k)||^{\dim_{\caln(G)}(\im(k))}
\nonumber
\\
& \le & 
||\Lambda^{G} \circ \eta_{\phi^*V}(r_{g \cdot I_s})||^{\dim_{\caln(G)}(\im(k))}
\nonumber
\\
& = & 
\bigl|\bigl|l_{(-\epsilon_1 \cdot (M+1), -\epsilon_2 \cdot (M+1), \ldots, -\epsilon_d \cdot (M+1))} \colon 
V \to V\bigr|\bigr|^{\dim_{\caln(G)}(\overline{\im(\Lambda^{G} \circ \eta_{\phi^*V}(r_{A}))})}
\nonumber
\\
& = & 
\bigl|\bigl|l_{(-\epsilon_1 \cdot (M+1), -\epsilon_2 \cdot (M+1), \ldots, -\epsilon_d \cdot (M+1))} \colon 
V \to V\bigr|\bigr|^{r \cdot \dim_{\IC}(V) -\dim_{\caln(G)}(\ker(\Lambda^{G} \circ \eta_{\phi^*V}(r_{A})))}
\nonumber
\\
& \stackrel{\eqref{Hollywood}}{=} & 
\bigl|\bigl|l_{(-\epsilon_1 \cdot (M+1), -\epsilon_2 \cdot (M+1), \ldots, -\epsilon_d \cdot (M+1))} \colon 
V \to V\bigr|\bigr|^{(r - \dim_{\caln(G)}(\ker(\Lambda^{G}(r_A)))) \cdot \dim_{\IC}(V)}.
\nonumber
\end{eqnarray}

This implies
\begin{eqnarray*}
\lefteqn{{\det}_{\caln(G)}\bigl(\Lambda^{G} \circ \eta_{\phi^*V}(r_{A})\bigr)}
& & 
\\
& \stackrel{\eqref{lem:final_estimate_for_det(r_z)_for_z_in_supp(A):(1)}}{=} & 
{\det}_{\caln(G)}\bigl(\Lambda^{G} \circ \eta_{\phi^*V}(r_{A'})\bigr) 
\cdot \left({\det}_{\caln(G)}\bigl(\Lambda^{G} \circ \eta_{\phi^*V}(r_{g \cdot I_s}) \circ k\bigr)\right)^{-1}
\\
& \stackrel{\eqref{lem:final_estimate_for_det(r_z)_for_z_in_supp(A):(2)}}{\ge}  & 
{\det}_{\caln(G)}\bigl(\Lambda^{G} \circ \eta_{\phi^*V}(r_{A'})\bigr) 
\\
& & 
\cdot \left(\bigl|\bigl|l_{(-\epsilon_1 \cdot (M+1), -\epsilon_2 \cdot (M+1), \ldots, -\epsilon_d \cdot (M+1))} \colon 
V \to V\bigr|\bigr|^{(r - \dim_{\caln(G)}(\ker(\Lambda^{G}(r_A)))) \cdot \dim_{\IC}(V)}\right)^{-1}.
\end{eqnarray*}
Since $\dim_{\caln(G)}(\ker(\Lambda(r_A))) = \dim_{\caln(G)}(\ker(\Lambda(r_{A'})))$,
Lemma~\ref{lem:det_G(A)_with_positive_support}~\eqref{lem:det_G(A)_with_positive_support:without_section}
applied to $A'$ implies
\[
{\det}_{G}\bigl(\Lambda^{G} \circ \eta_{\phi^*V}(r_{A'})\bigr)
\ge 
\left(\prod_{l= 1}^d |\delta_l|^{- \epsilon_l \cdot 2M}\right)^{r - \dim_{\caln(G)}(\ker(\Lambda^G(r_A)))}.
\]
We conclude from the last two inequalities
\begin{multline*}
{\det}_{\caln(G)}\bigl(\Lambda^{G} \circ \eta_{\phi^*V}(r_{A})\bigr)
\ge \left(\bigl|\bigl|l_{(-\epsilon_1 \cdot (M+1), -\epsilon_2 \cdot (M+1), \ldots, -\epsilon_d \cdot (M+1))} \colon 
V \to V\bigr|\bigr|^{-\dim_{\IC}(V)}\right.
\\ 
\left.\cdot \prod_{l= 1}^d |\delta_l|^{- \epsilon_l \cdot 2M}\right)^{r - \dim_{\caln(G)}(\ker(\Lambda^{G}(r_A)))}.
\end{multline*}
This finishes the proof of Lemma~\ref{lem:final_estimate_for_det(r_z)_for_z_in_supp(A)}.
\end{proof}

Now we can finish the proof of Proposition~\ref{prop:main_theorem_true_for_finite_kernel}.
\begin{proof}[Proof of Proposition~\ref{prop:main_theorem_true_for_finite_kernel}]~%
Consider the situation of Theorem~\ref{the:Determinant_class_and_twisting} in the special case that
the kernel $K$ of $\phi \colon G \to \IZ^d$ is finite.

Assertion~\eqref{the:Determinant_class_and_twisting:Betti_numbers} 
of Theorem~\ref{the:Determinant_class_and_twisting} follows from~\eqref{Hollywood}.

Recall that we have fixed a group homomorphism $j \colon \IZ^d \to G$ and an integer $N \ge 1$
such that $\phi \circ j = N \cdot \id_{\IZ^d}$. By composing $j$ with $(2M+2) \cdot \id_{\IZ^d}$,
we can arrange $2M +1  \le N -1$.
Assertion~\eqref{the:Determinant_class_and_twisting:determinant}  of 
Theorem~\ref{the:Determinant_class_and_twisting} follows from  
Lemma~\ref{lem:det_G(A)_with_positive_support}~\eqref{lem:det_G(A)_with_positive_support:with_section}
and Lemma~\ref{lem:final_estimate_for_det(r_z)_for_z_in_supp(A)}.
This finishes the proof of Proposition~\ref{prop:main_theorem_true_for_finite_kernel}.
\end{proof}

\begin{remark}
\label{rem:complexity_of_proof}
The proof of Proposition~\ref{prop:main_theorem_true_for_finite_kernel} simplifies a lot
if we assume that $\phi \colon G \to \IZ^d$ has a section. Then one only needs
Lemma~\ref{lem:det_G(A)_with_positive_support}~\eqref{lem:det_G(A)_with_positive_support:with_section}
and in its proof one only uses $\phi \circ j(\supp_{\IZ^d}(B)) \subseteq \phi(\supp_G(A))$
which is easy to check directly without going through the proof of
Lemma~\ref{lem:more_on_supp(B)}.
Lemma~\ref{lem:final_estimate_for_det(r_z)_for_z_in_supp(A)} is not needed at all. The
proof without assuming a section is rather complicated since we must first choose our $M$
as above only depending on $V$  and $\phi\bigl(\supp_{G}(A)\bigr)$ so that we
can later choose $N$ as large as we want without destroying the estimates.
\\[1mm]
But it is important to consider the more general case since for links we may
have $d \ge 2$ and $\phi \colon G \to \IZ^d$ may not have a section. Moreover, if we want
to allow $\phi \colon G \to \IR$, we need to handle the case $G \to \IZ^d$ as explained in
Remark~\ref{rem:Replacing_Zd_by_a_torsionfree_abelian_group_A}.  The main motivation for
us is that we will need the version for $\IZ^d$ in the proof of the equality of the
Thurston norm $X_M(\phi)$ and the degree of the $\phi$-twisted $L^2$-torsion function 
of the universal covering of a compact irreducible orientable connected $3$-manifold $M$
 with infinite fundamental group and empty or toroidal boundary for $\phi \in H^1(M;\IZ)$
in the forthcoming paper~\cite{Friedl-Lueck(2015l2+Thurston)}.
\end{remark}

%%%%%%%%%%%%%%%%%%%%%%%%%%%%%%%%%%%%%%%%%%%%%%%%%%%%%%%%%%%%%%%%%%%%%%%%%%%%%%%%%

\subsection{Proof of Theorem~\ref{the:Determinant_class_and_twisting} in general}
\label{subsec:Proof_of_Theorem_ref(the:Determinant_class_and_twisting)_in_general}

Now we are ready to give the proof of Theorem~\ref{the:Determinant_class_and_twisting}
in general. We will use approximation techniques to reduce it to the special
case, that the kernel of $\phi$ is finite which we have already taken care of in
Proposition~\ref{prop:main_theorem_true_for_finite_kernel}.

\begin{theorem}[Twisted Approximation inequality]
\label{the:Twisted_Approximation_inequality}
Let $\phi \colon G \to Q$  be  a group homomorphism.

Consider a nested sequence of in $G$ normal subgroups
\[
G  \supseteq G_0 \supseteq G_1 \supseteq G_2 \supseteq \cdots 
\]
such that $G_i$ is contained in $\ker(\phi) $ and  the intersection $\bigcap_{i \ge 0} G_i$ is trivial. Put $Q_i := G/G_i$.
Let $\phi_i \colon Q_i \to Q$ be  the homomorphism uniquely determined by $\phi_i \circ \pr_i = \phi$,
where $\pr_i \colon G \to Q_i$ is the canoncial projection.

Let $V$ be a based finite-dimensional $Q$-representation. Fix an $(r,s)$-matrix $A \in M_{r,s}(\IZ G)$.
Denote by $A[i]$ the image of $A$ under the map $M_{r,s}(\IZ G) \to M_{r,s}(\IZ Q_i)$ 
induced by the projection $\pr_i \colon G \to Q_i = G/G_i$. 

Suppose that one of the following conditions is satisfied:
\begin{enumerate}

\item \label{the:Twisted_Approximation_inequality:det_bound}
There is a real number $C > 0$  such that we get for all
$i \in I$
\[
{\det}_{\caln(Q_i)}\bigl(\Lambda^{Q_i} \circ \eta_{\phi_i^* V}(r_{A[i]})\bigr) \ge C;
\]

\item \label{the:Twisted_Approximation_inequality:Q_i-condition}
We have $Q = \IZ^d$, the map $\phi$ is surjective  and the index $[\ker(\phi) : G_i]$ is finite for all $i \ge 0$. 
\end{enumerate}

Then we get 
\begin{eqnarray*}
\dim_{\caln(G)}\bigl(\ker(\Lambda^G \circ \eta_{\phi^* V}(r_A))\bigr)
& = & 
\lim_{i \to \infty} \dim_{\caln(Q_i)}\bigl(\ker(\Lambda^{Q_i} \circ \eta_{\phi_i^* V}(r_{A[i]}))\bigr);
\\
{\det}_{\caln(G)}\bigl(\Lambda^G \circ \eta_{\phi^* V}(r_A)\bigr)
& \ge  & 
\limsup_{i \to \infty} {\det}_{\caln(Q_i)}\bigl(\Lambda^{Q_i} \circ \eta_{\phi_i^* V}(r_{A[i]})\bigr).
\end{eqnarray*}
\end{theorem}
\begin{proof}
Put $n := \dim_{\IC}(V)$. Let $B \in M_{nr,ns}(\IC G)$ be the matrix for which the morphism
$\eta_{\phi^*V}(r_{A}) \colon \eta_{\phi^*V}(\IC G^r) \to \eta_{\phi^* V}(\IC G^s)$ becomes
$r_B \colon \IC G^{nr} \to \IC G^{ns}$ under the obvious identifications
$\eta_{\phi^* V}(\IC G^r) = \IC G^{nr}$ and $\eta_{\phi^* V}(\IC G^s) = \IC G^{ns}$.
As before let $B[i] \in M_{nr,ns}(\IZ Q_i)$ be the matrix obtained from $B$ by applying to each entry the ring
homomorphism $\IC G \to \IC Q_i$ induced by the projection $\pr_i \colon G \to Q_i$.
One easily checks that $\eta_{\phi_i^* V}(r_{A[i]}) \colon \eta_{\phi_i^* V}(\IC Q_i^r) \to \eta_{\phi_i^* V}(\IC Q_i^s)$ becomes
$r_{B[i]} \colon \IC Q_i^{nr} \to \IC Q_i^{ns}$ under the obvious identifications
$\eta_{\phi_i^* V}(\IC Q_i^r) = \IC Q_i^{nr}$ and $\eta_{\phi_i^* V}(\IC Q_i^s) = \IC Q_i^{ns}$.

Next we finish the proof of Theorem~\ref{the:Twisted_Approximation_inequality}
provided that condition~\eqref{the:Twisted_Approximation_inequality:det_bound} is satisfied.

We can arrange that  $nr = ns$ and that
$\Lambda^G(r_B)$ and $\Lambda^{Q_i}(r_{B[i]})$  are positive operators, otherwise replace $A$ by $A^*A$ and 
use~\cite[Lemma~2.11~(11) on page~77 and Lemma~3.15~(4) on page~129]{Lueck(2002)}.

We want to apply~\cite[Theorem~13.19~(2) on page~461]{Lueck(2002)}  in the case, where 
$G_i$ in the notation of~\cite[Theorem~13.19 on page~461]{Lueck(2002)} is $Q_i$, the matrix 
$A$ in the notation of~\cite[Theorem~13.19 on page~461]{Lueck(2002)}  is the matrix 
$B$, the matrix $A_i$ in the notation of~\cite[Theorem~13.19 on page~461]{Lueck(2002)}  is the matrix 
$B[i]$ and $\tr_i = \tr_{\caln(Q_i)}$.  
For this purpose we have to check
the three conditions in~\cite[Theorem~13.19 on page~461]{Lueck(2002)}.

Put $K = ||B||_1$ We conclude from
Lemma~\ref{lem:properties_of_L1-norm}~\eqref{lem:properties_of_L1-norm:norm_estimate}
and~\eqref{lem:properties_of_L1-norm:group_homomorphisms}
the  inequalities $||\Lambda^G(r_B)|| \le K$ and $||\Lambda^{Q_i}(r_{B[i]})|| \le K$
for the operator norms of $\Lambda^G(r_B)$ and $\Lambda^{Q_i}(r_{B[i]})$. 

Consider a polynomial $p$ with complex coefficients. Notice that $\supp_G(p(B))$ is a finite subset of $G$.
Choose $i_1 \in I$ such for all
$i \ge i_1$ and $g  \in \supp_G(p(B))$ the implication $g \not= e \implies  \pr_i(g) \not= e $ holds
for the projection $\pr_i \colon G \to Q_i$. This implies
\[
\tr_{\caln(G)}(p(B)) = \tr_{\caln(Q_i)}(p(B[i])) \quad \text{for}\; i \ge i_0.
\]
In particular we get
\[
\tr_{\caln(G)}(p(B)) = \lim_{i \in I} \tr_{\caln(Q_i)}(p(B[i])).
\]

The third condition appearing in~\cite[Theorem~13.19~(2) on page~461]{Lueck(2002)}  is the inequality
\[
{\det}_{\caln(Q_i)}\bigl(\Lambda^{Q_i}(r_{B[i]})) \ge C
\]
with $C = 1$. One easily checks that for the proof it only matters that the lower bound $C$
is greater than zero and independent of $i$. Or one can argue that one may replace $A$ by
$N \cdot A$ for some large enough integer $N$ to arrange $C = 1$ and it is obvious that
the claim holds for $A$ if it holds for $N \cdot A$.

Now the equality
\begin{eqnarray*}
\dim_{\caln(G)}\bigl(\ker(\Lambda^G \circ \eta_{\phi^* V}(r_A))\bigr)
& = & 
\lim_{i \to \infty} \dim_{\caln(Q_i)}\bigl(\ker(\Lambda^G \circ \eta_{\phi_i^* V}(r_{A[i]}))\bigr)
\end{eqnarray*}
follows directly 
from~\cite[Theorem~13.19~(2) on page~461]{Lueck(2002)}. The inequality
\[
{\det}_{\caln(G)}^{(2)}\bigl(\Lambda^G(r_A)\bigr) \ge 
\limsup_{i \to \infty} {\det}_{\caln(Q_i)}\bigl(\Lambda^G \circ \eta_{\phi_i^* V}(r_{A[i]})\bigr)
\]
does not follow directly
from the assertion in~\cite[Theorem~13.19~(2) on page~461]{Lueck(2002)}, but from the inequality 
appearing at the very end of the proof  of~\cite[Theorem~13.19 on page~465]{Lueck(2002)}.
This finishes the proof of Theorem~\ref{the:Twisted_Approximation_inequality}
provided that condition~\eqref{the:Twisted_Approximation_inequality:det_bound} is satisfied.

It remains to explain why condition~\eqref{the:Twisted_Approximation_inequality:Q_i-condition}
implies condition~\eqref{the:Twisted_Approximation_inequality:det_bound}.

Since $\bigcap_{i \ge 0} G_i = \{1\}$ and $\supp_G(B)$ is finite, there exists an index
$i_0$ such that the canonical projection $\pr_i \colon G \to Q_i$ restricted to
$\supp_G(B)$ is injective for $i \ge i_0$. Hence we get $\pr_i(\supp_G(B)) =
\supp_{Q_i}(B[i])$ for $i \ge i_0$. Without loss of generality we can assume 
$\pr_i(\supp_G(B)) = \supp_{Q_i}(B[i])$ for all $i$, otherwise ignore the finitely many $i$-s with $i < i_0$.

Since $\phi$ is surjective by assumption, $\phi_i$ is surjective for all $i$.
The kernel of $\phi_i$ is finite, since $[\ker(\phi) : G_i]$ is finite by assumption. 
Hence we can apply Proposition~\ref{prop:main_theorem_true_for_finite_kernel}
in the case, where  the group $G$ appearing in Proposition~\ref{prop:main_theorem_true_for_finite_kernel}
is $Q_i$ and the homomorphism $\phi \colon G \to \IZ^d$ appearing in 
Proposition~\ref{prop:main_theorem_true_for_finite_kernel}
is $\phi_i \colon Q_i  \to \IZ^d$, and thus we obtain
\begin{eqnarray*}
 \quad  {\det}_{\caln(Q_i)}\bigl(\Lambda^{Q_i}(r_{B[i]})\bigr)
 &   \ge  &
   \nu\bigl(V,\phi_i(\supp_{Q_i}(B[i]))\bigr)^{r - \dim_{\caln(Q_i)}(\ker(\Lambda^{Q_i}(r_{B[i]})))},
\end{eqnarray*}
This implies
\begin{eqnarray*}
 \quad  {\det}_{\caln(Q_i)}\bigl(\Lambda^{Q_i}(r_{B[i]})\bigr)
 &   \ge  &
 \min\{1,\nu\bigl(V,\phi_i(\supp_{Q_i}(B[i]))\bigr)\}^r.
\end{eqnarray*}
Since  $\pr_i(\supp_G(B)) = \supp_{Q_i}(B[i])$ holds for all $i$,
we get $\phi_i(\supp_{Q_i}(B[i])) = \phi(\supp_G(B))$ for all  $i$. Hence we get for all $i \ge 0$
\[
\nu\bigl(V,\phi_i(\supp_{Q_i}(B[i]))\bigr) = \nu\bigl(V,\phi(\supp_G(B))\bigr).
\]
Thus we have shown for each $i \ge 0$
\begin{eqnarray*}
 \quad  {\det}_{\caln(Q_i)}\bigl(\Lambda^{Q_i}(r_{B[i]})\bigr)
 &   \ge  &
   \min\{1,\nu\bigl(V,\phi(\supp_{G}(B))\bigr)\}^r.
\end{eqnarray*}
Hence condition~\eqref{the:Twisted_Approximation_inequality:det_bound}
holds  if we put $C =  \min\{1,\nu\bigl(V,\phi(\supp_{G}(B))\bigr)\}^r$.
This finishes the proof of Theorem~\ref{the:Twisted_Approximation_inequality}.
\end{proof}

\begin{remark} \label{rem:Approximation_Lueck_versus_Liu}
One should compare Theorem~\ref{the:Twisted_Approximation_inequality}
with~\cite[Lemma~3.2]{Liu(2015)}. Theorem~\ref{the:Twisted_Approximation_inequality}
has a stronger conclusion and gives non-trivial conclusions in more cases,
since we are dealing with the non-regularized determinants
and also get a formula for the $L^2$-Betti numbers. 
On the other hand, the proof of~\cite[Lemma~3.2]{Liu(2015)} works also for matrices over $\IC G$,
and does not require the  condition~\eqref{the:Twisted_Approximation_inequality:det_bound}.
\end{remark}

\begin{proof}[Proof of Theorem~\ref{the:Determinant_class_and_twisting}]
Analogously to the proof of~\cite[Lemma~13.33 on page~466]{Lueck(2002)} one shows
that the operator norm of $\Lambda^G \circ \eta_{\phi^* V}(r_A)$
is bounded by $||A||_1 \cdot \max\bigl\{||l_s \colon V \to V||    \mid s \in \phi(\supp_G(A))\bigr\}$.
We conclude from Lemma~\ref{lem:det_estimate_in_terms_of_norm}
\begin{multline*}
{\det}_{\caln(G)}\bigl(\Lambda^G \circ \eta_{\phi^* V}(r_A)\bigr) 
\\
\le 
\bigl(||A||_1 \cdot \max\bigl\{||l_s \colon V \to V|| 
   \mid s \in \phi(\supp_G(A))\bigr\}\bigr)^{r - \dim_{\caln(G)}(\ker(\Lambda^G(r_A)))}.
 \end{multline*}
This proves the second inequality appearing in assertion~\eqref{the:Determinant_class_and_twisting:determinant}.

Consider a matrix $A \in M_{r,s}(\IZ G)$. Because of the additivity of the von Neumann dimension,
assertion~\eqref{the:Determinant_class_and_twisting:Betti_numbers}  follows
if we can show
\begin{eqnarray}
 \quad  {\dim}_{\caln(G)}\bigl(\ker(\Lambda^G \circ \eta_{\phi^* V}(r_A))\bigr)
  &  = &
  \dim_{\IC}(V) \cdot  {\dim}_{\caln(G)}\bigl(\ker(\Lambda^G(r_A))\bigr).
\label{the:Determinant_class_and_twisting:to_show_(1)}
\end{eqnarray}
 Obviously assertion~\eqref{the:Determinant_class_and_twisting:L2-det_acyclic} follows from
 assertions~\eqref{the:Determinant_class_and_twisting:Betti_numbers}
and~\eqref{the:Determinant_class_and_twisting:determinant}.
Hence Theorem~\ref{the:Determinant_class_and_twisting} follows if 
we can show~\eqref{the:Determinant_class_and_twisting:to_show_(1)}, the inequality
\begin{eqnarray}
\quad \quad \nu\bigl(V,\phi(\supp_G(A))\bigr)^{r - \dim_{\caln(G)}(\ker(\Lambda^G(r_A)))} 
& \le & 
{\det}_{\caln(G)}\bigl(\Lambda^G \circ \eta_{\phi^* V}(r_A)\bigr),
\label{the:Determinant_class_and_twisting:to_show_(2)}
\end{eqnarray}
and, provided that $\phi \colon G \to \IZ^d$ has a section, the better inequality
\begin{eqnarray}
\quad \theta\bigl(V,\phi(\supp_G(A))\bigr)^{r - \dim_{\caln(G)}(\ker(\Lambda^G(r_A)))} 
& \le &
{\det}_{\caln(G)}\bigl(\Lambda^G \circ \eta_{\phi^* V}(r_A)\bigr).
\label{the:Determinant_class_and_twisting:to_show_(3)}
\end{eqnarray}

 Since $G$ is residually finite and countable, we can choose a nested sequence of in
 $G$ normal subgroups
\[
G = G_0 \supseteq G_1 \supseteq G_2 \supseteq \cdots 
\]
such that $[G:G_i]$ is finite for each $i \ge 0$ and $\bigcap_{i \ge 0} G_i = \{1\}$.
Let $K$ be the kernel of $\phi \colon G \to \IZ^d$.
Put $K_i = G_i \cap K$. Then $K_i$ is a normal
subgroup of both $K$ and $G$ and has finite index in $K$. We have $\bigcap_{i \ge 0} K_i = \{1\}$. Put
$Q_i = G/K_i$.  The epimorphism  $\phi \colon G \to \IZ^d$ factorizes over the projection $\pr_i \colon G \to Q_i$
to  a epimorphism $\phi_i \colon Q_i \to \IZ^d$ such that the kernel of $\phi_i$ is $K/K_i$ 
and in particular is finite.  Let $A[i] \in M_{r,s}(\IC Q_i)$ 
be obtained from $A$ by applying to each entry the ring homomorphism
$\IC G \to \IC Q_i$ induced by the projection $\pr_i \colon G \to Q_i$.  We can apply 
Proposition~\ref{prop:main_theorem_true_for_finite_kernel}
to $\phi_i \colon Q_i \to \IZ^d$ and $A[i]$ and obtain 
\begin{multline}
  {\dim}_{\caln(Q_i)}\big(\ker(\Lambda^{Q_i} \circ \eta_{\phi_i^* V}(r_{A[i]}))\bigr)
  \\ = 
   \dim_{\IC}(V) \cdot  {\dim}_{\caln(Q_i)}\bigl(\ker(\Lambda^{Q_i} (r_{A[i]}))\bigr),
\label{the:Determinant_class_and_twisting:(dim_1)}
\end{multline}
and
\[
{\det}_{\caln(Q_i)}\bigl(\Lambda^{Q_i} \circ \eta_{\phi_i^* V}(r_{A[i]} )\bigr)
\ge 
\nu\bigl(V,\phi_i(\supp_{Q_i}(A[i]))\bigr)^{r - \dim_{\caln(Q_i)}(\ker(\Lambda^{Q_i}(r_{A[i]})))},
\]
and, provided that $\phi \colon G \to \IZ^d$ and hence $\phi_i \colon Q_i \to \IZ^d$ has a section,
\[
{\det}_{\caln(Q_i)}\bigl(\Lambda^{Q_i} \circ \eta_{\phi_i^* V}(r_{A[i]} )\bigr)
\ge 
\theta\bigl(V,\phi_i(\supp_{Q_i}(A[i]))\bigr)^{r - \dim_{\caln(Q_i)}(\ker(\Lambda^{Q_i}(r_{A[i]})))}.
\]
Since $\bigcap_{i \ge 0} K_i = \{1\}$ and $\supp_G(A)$ is finite, there exists an index
$i_0$ such that the canonical projection $\pr_i \colon G \to Q_i$ restricted to
$\supp_G(A)$ is injective for $i \ge i_0$. Hence we get $\pr_i(\supp_G(A)) =
\supp_{Q_i}(A[i])$ for $i \ge i_0$. We conclude  that $\phi_i(\supp_{Q_i}(A[i])) = \phi(\supp_{G}(A))$ 
holds for $i \ge i_0$. This implies for $i \ge i_0$
\begin{multline}
{\det}_{\caln(Q_i)}\bigl(\Lambda^{Q_i} \circ \eta_{\phi_i^* V}(r_{A[i]} )\bigr)
\ge 
\nu\bigl(V,\phi(\supp_{G}(A))\bigr)^{r - \dim_{\caln(Q_i)}(\ker(\Lambda^{Q_i}(r_{A[i]})))},
\label{the:Determinant_class_and_twisting:(nu_1)}
\end{multline}
and, provided that $\phi \colon G \to \IZ^d$ has section,
\begin{multline}
{\det}_{\caln(Q_i)}\bigl(\Lambda^{Q_i} \circ \eta_{\phi_i^* V}(r_{A[i]} )\bigr)
\ge 
\theta\bigl(V,\phi(\supp_G(A))\bigr)^{r - \dim_{\caln(Q_i)}(\ker(\Lambda^{Q_i}(r_{A[i]})))}.
\label{the:Determinant_class_and_twisting:(theta)}
\end{multline}
Theorem~\ref{the:Twisted_Approximation_inequality} applied to 
$G \supseteq K_0 \supseteq K_1 \supseteq K_2 \supseteq \cdots$, $\phi \colon G \to \IZ^d$ 
and the matrix $A\in M_{r,s}(\IZ G)$ yields
\begin{eqnarray}
\quad \quad \quad {\dim}_{\caln(G)}\big(\ker(\Lambda^{G} \circ \eta_{\phi^* V}(r_{A}))\bigr)
& = &
\lim_{i \to \infty} {\dim}_{\caln(Q_i)}\big(\ker(\Lambda^{Q_i} \circ \eta_{\phi_i^* V}(r_{A[i]}))\bigr);
\label{the:Determinant_class_and_twisting:(dim_2)}
\\
{\det}_{\caln(G)}\bigl(\Lambda^G  \circ \eta_{\phi^* V}(r_A)\bigr)
& \ge  &
\limsup_{i \in I} {\det}_{\caln(Q_i)}\bigl(\Lambda^{Q_i} \circ \eta_{\phi_i^* V} (r_{A[i]})\bigr).
\label{the:Determinant_class_and_twisting:(limsup_estimate)}
\end{eqnarray}
If $V$ is the trivial $1$-dimensional $\IZ^d$-representation, we get as a special case 
of~\eqref{the:Determinant_class_and_twisting:(dim_2)}
\begin{eqnarray}
\quad \quad \quad  {\dim}_{\caln(G)}\big(\ker(\Lambda^{G}(r_{A}))\bigr)
& = &
\lim_{i \to \infty} {\dim}_{\caln(Q_i)}\big(\ker(\Lambda^{Q_i} (r_{A[i]}))\bigr).
\label{the:Determinant_class_and_twisting:(dim_3)}
\end{eqnarray}
Hence~\eqref{the:Determinant_class_and_twisting:to_show_(1)}
follows from from~\eqref{the:Determinant_class_and_twisting:(dim_1)},~%
\eqref{the:Determinant_class_and_twisting:(dim_2)} and~\eqref{the:Determinant_class_and_twisting:(dim_3)}.
We conclude
\begin{eqnarray*}
\lefteqn{{\det}_{\caln(G)}\bigl(\Lambda^G  \circ \eta_{\phi^* V}(r_A)\bigr)}
& & 
\\
& \stackrel{\eqref{the:Determinant_class_and_twisting:(limsup_estimate)}}{\ge}  &
\limsup_{i \in I} {\det}_{\caln(Q_i)}\bigl(\Lambda^{Q_i} \circ \eta_{\phi_i^* V} (r_{A[i]})\bigr)
\\
& \stackrel{\eqref{the:Determinant_class_and_twisting:(nu_1)}}{\ge} &
\limsup_{i \in I} \nu\bigl(V,\phi(\supp_{G}(A))\bigr)^{r - \dim_{\caln(Q_i)}(\ker(\Lambda^{Q_i}(r_{A[i]})))}
\\
& \stackrel{\eqref{the:Determinant_class_and_twisting:(dim_2)}}{=} & 
\nu\bigl(V,\phi(\supp_{G}(A))\bigr)^{r - \dim_{\caln(G)}(\ker(\Lambda^{G}(r_A)))},
\end{eqnarray*}
and, provided that $\phi$ has a section, we conclude analogously using~\eqref{the:Determinant_class_and_twisting:(theta)}
instead of~\eqref{the:Determinant_class_and_twisting:(nu_1)} 
\begin{eqnarray*}
{\det}_{\caln(G)}\bigl(\Lambda^G  \circ \eta_{\phi^* V}(r_A)\bigr)
& \ge &
\theta\bigl(V,\phi(\supp_{G}(A))\bigr)^{r - \dim_{\caln(G)}(\ker(\Lambda^{G}(r_A)))}.
\end{eqnarray*}
This finishes the proof of  Theorem~\ref{the:Determinant_class_and_twisting}.
\end{proof}

%%%%%%%%%%%%%%%%%%%%%%%%%%%%%%%%%%%%%%%%%%%%%%%%%%%%%%%%%%%%%%%%%%%%%%%%%%%%%%%%%
%%%%%%%%%%%%%%Section 8: Twisting with  a cocycle in the first cohomology  %%%%%%%%%%%%%%%%%%%%%%%%%
%%%%%%%%%%%%%%%%%%%%%%%%%%%%%%%%%%%%%%%%%%%%%%%%%%%%%%%%%%%%%%%%%%%%%%%%%%%%%%%%%

 \typeout{----------------   Section 8: Twisting with  a cocycle in the first cohomology -------------------}

\section{Twisting with  a cocycle in the first cohomology}
\label{subsec:Twisting_with_a_cocycle_in_the_first_cohomology}

%%%%%%%%%%%%%%%%%%%%%%%%%%%%%%%%%%%%%%%%%%%%%%%%%%%%%%%%%%%%%%%%%%%%%%%%%%%%%%%%%

\subsection{The twisted $L^2$-torsion function}
\label{subsec:The_twisted_L2_torsion_function}

Let $\phi \colon G \to \IR$ be a group homomorphism.

Consider  a finite free $G$-$CW$-complex $X$.
For $t \in \IR^{> 0}$ let $\IC_t$ be the based $1$-dimensional
$\IR$-representation given by $\IC$ with the equivalence class of the standard $\IC$-basis
for which $r \in \IR$  acts by multiplication with $t^r$ on $\IC$. Denote by $\phi^*\IC_t$
the based $1$-dimensional $G$-representation obtained from $\IC_t$ by restriction with $\phi$.

The following function is of
interest and versions of it have already been studied  in low-dimensions in
\cite{Dubois-Friedl-Lueck(2016),Dubois-Friedl-Lueck(2015symmetric),Dubois-Friedl-Lueck(2015flavors),
Dubois-Wegner(2010),Dubois-Wegner(2013),Li-Zhang(2006volume),
Li-Zhang(2006Alexander),Li-Zhang(2008)}. Recall that we have 
made $\eta_{\phi^*\IC_t}$ more explicit in  Example~\ref{exa:phi-twisting_in_terms_of_matrices}.

\begin{definition}[Twisted $L^2$-torsion function]
  \label{def:Twisted_L2-torsion_function}
  We call $X$ \emph{of $\phi$-twisted determinant class} or
  \emph{$\phi$-twisted $\det$-$L^2$-acyclic} respectively if the finite Hilbert
  $\caln(G)$-chain complex $\Lambda\circ \eta_{\phi^*\IC_t}(C_*(X),[B_X])$
  is of determinant class or $\det$-$L^2$-acyclic respectively for every $t \in \IR^{>0}$
  for one (and hence all) choices of base refinements $[B_X]$.
 
  Provided that $X$ is of $\phi$-twisted determinant class, we define the
  \emph{$\phi$-twisted $L^2$-torsion function}
  \begin{eqnarray}
    \rho^{(2)}(X;\phi,[B_X]) \colon \IR^{>0} \to \IR, 
    \quad t \mapsto \rho^{(2)}\bigl(\Lambda \circ \eta_{\phi^*\IC_t}(C_*(X),[B_X])\bigr)
    \label{L2-torsion_function}
  \end{eqnarray} 
  by sending $t \in \IR^{>0}$ to the $L^2$-torsion of $X$ twisted with $\phi^*\IC_t$, see 
   Definition~\ref{def:L2-torsion_twisted_by_a_based_finite-dimensional_representation}.
\end{definition}

The evaluation of $\rho^{(2)}(X;\phi,[B_X])$ at $t = 1$ is the $L^2$-torsion $\rho^{(2)}(X;\caln(G))$
of the finite free $G$-$CW$-complex $X$ which is known to be independent of the choice of base
refinement.

Before one can study the function, one has to address the question whether it is well-defined.  
This concerns the following problems:

\begin{itemize}

\item Determinant class\\
  Under which conditions is $X$   of $\phi$-twisted determinant class?

\item $\det$-$L^2$-acyclicity\\
  In order to ensure that $\rho^{(2)}(X)$ depends essentially only on the homotopy
  type of $X$, it is also convenient to know that $X$ is $\det$-$L^2$-acyclic.
 This raises the problem  under which conditions $X$ is $\phi$-twisted $\det$-$L^2$-acyclic.
\end{itemize}

So far in the literature one could prove determinant class only for
representations over $\IQ$ and hence could only verify  determinant
class for rational values of $t$ using the Determinant Conjecture,
see~\cite[Chapter~13]{Lueck(2002)}.  Also the $L^2$-acyclicity was not
known from the beginning and had to be assumed or to be checked case
by case.

Therefore the following result is very useful and opens the door to systematically study
the twisted $L^2$-torsion function without having to worry whether it is well-defined.

\begin{theorem}[Properties of the twisted $L^2$-torsion function {$\rho^{(2)}(X,\phi,[B_X])$}]
\label{the:Properties_of_the_twisted_L2-torsion_function}

Let $X$ be a  finite free $G$-$CW$-complex. Let $[B_X]$ be  a base refinement for $X$.

\begin{enumerate}

\item \label{the:Properties_of_the_twisted_L2-torsion_function:determinant_class} 
\emph{Determinant class}\\
Suppose that $G$ is finitely generated and residually finite. 
Then there exist constants $C \ge 0$ and $D \ge 0$, 
depending on $X$ and $[B_X]$ 
but not on the parameter $t$, such that we get  for  $0 < t \le 1$
\[
C \cdot \ln(t) -D
\le 
\rho^{(2)}(X;\phi,[B_X])(t) 
\le
- C \cdot \ln(t) + D,
\]
and for  $t \ge 1$
\[
- C \cdot \ln(t)  -D
\le 
\rho^{(2)}(X;\phi,[B_X])(t) 
\le
C  \cdot \ln(t) + D.
\]
In particular $X$ is of $\phi$-twisted determinant class;

\item \label{the:Properties_of_the_twisted_L2-torsion_function:determinant_class:L2-acyclic_and_res_fin}
\emph{$\det$-$L^2$-acyclic}\\
Suppose that $G$ is finitely generated and residually finite.  Assume that $X$ is $L^2$-acyclic, i.e., 
$b_n^{(2)}(X;\caln(G)) = 0$ holds for all $n \ge 0$.  

Then  $X$  is $\phi$-twisted $\det$-$L^2$-acyclic;

\item \label{the:Properties_of_the_twisted_L2-torsion_function:determinant_class:change_of_base_refinement}
\emph{Change of base refinement}\\
Suppose that $X$ is $\phi$-twisted $\det$-$L^2$-acyclic.
Let $[B_X]$ and $[B_X']$ be two base refinements for $C_*(X)$.

Then we get for each $t \in \IR^{> 0}$
\[
\rho^{(2)}(X;\phi,[B_X'])(t)  - \rho^{(2)}(X;\phi,[B_X])(t) 
= \phi(\trans([B_X],[B_X'])) \cdot \ln(t);
\]

\item \label{the:Properties_of_the_twisted_L2-torsion_function:homotopy-invariance}
\emph{$G$-homotopy invariance}\\
Let $X$ and $Y$ be  finite free $G$-$CW$-complexes.
Let $[B_X]$ and $[B_Y]$ be  base refinements for $X$ and $Y$. Let $f \colon X \to Y$ be a $G$-homotopy equivalence.
Denote by 
\[
\tau\bigl(C_*(f) \colon (C_*(X),[B_X]) \to (C_*(Y),[B_Y]\bigr) \in \widetilde{K}_1(\IZ G)
\]
the Whitehead torsion of the $\IZ G$-chain homotopy equivalence $C_*(f)$. (This is well-defined as an
element in $\widetilde{K}_1(\IZ G)$ since we have fixed equivalence classes of $\IZ
G$-basis and not only  cellular equivalence classes of $\IZ G$-basis.)
The projection 
$\pr \colon G \to H_1(G)_f := H_1(G;\IZ)/\tors(H_1(G;\IZ))$ and the determinant over the commutative ring 
$\IZ[H_1(G)_f]$ induce homomorphisms
\[
\widetilde{K}_1(\IZ G) \xrightarrow{\pr_*} \widetilde{K}_1(\IZ[H_1(G)_f]) 
\xrightarrow{{\det}_{\IZ[H_1(G)_f]}} \IZ[H_1(G)_f]^{\times}/\{\pm 1\}.
\]
The homomorphism
\[
\psi \colon H_1(G)_f  \xrightarrow{\cong} \IZ[H_1(G)_f]^{\times}/\{\pm 1\} \quad x \mapsto \pm x
\]
is an isomorphism. 
Let 
\[
m(f_*,[B_X],[B_Y]) \in H_1(G)_f
\] 
be the image of
$\tau(C_*(f))$ under the composite
\[
\psi^{-1} \circ {\det}_{ \IZ[H_1(G)_f]} \circ \pr_* \colon \widetilde{K}_1(\IZ G) \to H_1(G)_f.
\]

Suppose that the $K$-theoretic Farrell-Jones Conjecture holds for $\IZ G$ or that
$f$ is a simple $G$-homotopy equivalence.
Assume that $X$ is $\phi$-twisted $\det$-$L^2$-acyclic.

Then $Y$ is $\phi$-twisted $\det$-$L^2$-acyclic and we get  for every $t \in \IR^{>0}$
\[
\rho^{(2)}(Y;\phi,[B_Y])(t) - \rho^{(2)}(X;\phi,[B_X])(t)
= \phi_f\bigl(m(f_*,[B_X],[B_Y])\bigr)  \cdot \ln(t).
\]
where $\phi_f \colon H_1(G)_f \to \IR$ is the homomorphism induced by $\phi \colon G \to \IR$;

\item  \label{the:Properties_of_the_twisted_L2-torsion_function:scaling} \emph{Scaling $\phi$}\\
Let $r \in \IR$ be a real number. Put $\phi_r := (r \cdot \id_{\IZ}) \circ \phi \colon G \to \IR$.
Then $X$ is $\phi$-twisted $\det$-$L^2$-acyclic
if and only if it is $\phi_r$-twisted $\det$-$L^2$-acyclic, and in this case
\[
\rho^{(2)}(X;\phi_r,[B_X])(t) = \rho^{(2)}(X;\phi,[B_X])(t^r).
\]  
\end{enumerate}
\end{theorem}
\begin{proof}~\eqref{the:Properties_of_the_twisted_L2-torsion_function:determinant_class}
Since $G$ is finitely generated and any finitely generated subgroup of $\IR$ is isomorphic to $\IZ^d$ for some
natural number $d$, we can factorize $\phi \colon G \to \IR$ as the composite $i \circ \phi'$
for an epimorphism $\phi' \colon G \to \IZ^d$ and an injective group homomorphism $i \colon \IZ^d \to \IR$. 
We can arrange that $i(e_l ) > 0$ holds for each element $e_i$ of the standard basis $\{e_1, \ldots, e_d\}$.
Obviously we have $(\phi')^*i^*\IC_t = \phi^* \IC_t$. Hence we get if we put $V = i^*\IC_t$
\begin{eqnarray*}
\rho^{(2)}(X;\phi ,[B_X])(t)
& = &
\rho^{(2)}\bigl(\Lambda \circ \eta_{(\phi')^*V}(C_*(X),[B_X])\bigr).
\end{eqnarray*}
Since $G$ is by assumption residually finite, it satisfies the Determinant 
Conjecture, see~\cite[Conjecture~13.2 on page~454 and Theorem~13.3~(2) on page~454]{Lueck(2002)}.
In particular $C_*(X)$ is of determinant class.

If $\phi$ is trivial, then $\rho^{(2)}(X;\phi,[B_X])(t)$ is constant with value $\rho^{(2)}(X;\caln(G))$
and the claim is obviously true. Hence we can assume without loss of generality that $\phi$ is non-trivial,
in other words that $d \ge 1$.

Let $r_n$ be the number of  equivariant $n$-cells of $X$, or, equivalently, 
the number of $n$-cells in $G \backslash X$. 
Let $M_n \ge 1$ be an integer  such that 
\[
\phi'(\supp_G(c_n)) \subseteq \{(s_1, s_2, \ldots, s_d) \in \IZ^d \mid -M_n \le s_i \le M_n\}
\]
holds for the $n$th differential $c_n$ of $(C_*(X),[B_X])$ for each $n \ge 0$. 
Theorem~\ref{the:Determinant_class_and_twisting}~\eqref{the:Determinant_class_and_twisting:L2-det_acyclic} 
applied to $c_n$, $\phi'$ and $V$ implies for  $t \in \IR^{> 0}$
\begin{multline*}
  \quad \quad \quad \nu\bigl(V,\phi'(\supp_G(c_n))\bigr)^{r_n - \dim_{\caln(G)}(\ker(\Lambda(c_n)))} 
  \\
  \le {\det}_{\caln(G)}\bigl(\Lambda \circ \eta_{\phi^* \IC_t}(c_n)\bigr)  =  {\det}_{\caln(G)}\bigl(\Lambda \circ \eta_{(\phi')^*V}(c_n)\bigr)
\le 
  \\
   \bigl(||c_n||_1 \cdot \max\bigl\{||l_s \colon V \to V|| 
   \mid s \in \phi'(\supp_G(c_n))\bigr\}\bigr)^{r_n - \dim_{\caln(G)}(\ker(\Lambda(c_n)))}.
 \end{multline*}
 
One easily checks using $i(e_l) > 0$ for $l = 1,2, \ldots, d$
\begin{eqnarray*}
\nu\bigl(V,\phi'(\supp_G(c_n)),\IZ^d\bigr)
 &= & 
\prod_{l = 1}^d \max\{t^{i(e_l)},t^{-i(e_l)}\}^{-(M_n+1)} \cdot \prod_{i = 1}^d \min\{t^{i(e_l)},t^{-i(e_l)}\}^{2M_n}
\\
& = & 
\prod_{l = 1}^d \min\{t^{i(e_l)},t^{-i(e_l)}\} ^{M_n+1}\cdot \prod_{i = 1}^d \min\{t^{i(e_l)},t^{-i(e_l)}\}^{2M_n}
\\
& = & 
\prod_{i = 1}^d \min\{t^{i(e_l)},t^{-i(e_l)}\}^{3M_n+1}
\\
& = & 
\begin{cases}
t^{(3M_n +1) \cdot \sum_{l= 1}^d i(e_l)} & \text{for} \; t \le 1;
\\
t^{-(3M +1) \cdot \sum_{l= 1}^d i(e_l)} & \text{for} \; t \ge 1,
\end{cases}
\end{eqnarray*}
and
\begin{eqnarray*}
\lefteqn{\max\bigl\{||l_s \colon \IC_t \to \IC_t|| \mid s \in \phi(\supp_G(A))\bigr\}}
& & 
\\
& \le  & 
\max\bigl\{||l_s \colon \IC_t \to \IC_t|| 
   \mid s \in \{(s_1, s_2, \ldots, s_d\} \mid -M_n \le s_i \le M_n\}
\\
& = &
\prod_{l = 1} \max\{t^{i(e_l)},t^{-i(e_l)}\}^{M_n}
\\
& = & 
\begin{cases}
t^{- M_n \cdot \sum_{l = 1}^d i(e_l)} & t \le 1;
\\
t^{M_n \cdot \sum_{l = 1}^d i(e_l)}  & t \ge 1.
\end{cases}
\end{eqnarray*}
Hence we get for every $n \ge 1$ and  $0 < t \le 1$
\begin{multline*}
\bigl(r_n - \dim_{\caln(G)}(\ker(\Lambda(c_n)))\bigr) \cdot (3M_n +1) \cdot \left(\sum_{l= 1}^d i(e_l)\right) \cdot \ln(t) 
\\
\le 
\ln\bigl({\det}_{\caln(G)}(\Lambda \circ \eta_{\IC_t}(c_n))\bigr)
\le
\\
\ln\bigl(||c_n||_1)  -  \bigl(r_n - \dim_{\caln(G)}(\ker(\Lambda(c_n)))\bigr) \cdot M_n \cdot \left(\sum_{l = 1}^d i(e_l) \right) \cdot \ln(t), 
\end{multline*}
and for every $n \ge 1$ and  $ t \ge 1$
\begin{multline*}
- \bigl(r_n - \dim_{\caln(G)}(\ker(\Lambda(c_n)))\bigr) \cdot (3M_n +1) \cdot \left(\sum_{l= 1}^d i(e_l)\right) \cdot \ln(t) 
\\
\le 
\ln\bigl({\det}_{\caln(G)}(\Lambda \circ \eta_{\IC_t}(c_n))\bigr)
\le
\\
\ln\bigl(||c_n||_1)  + \bigl(r_n - \dim_{\caln(G)}(\ker(\Lambda(c_n)))\bigr) \cdot M_n \cdot \left(\sum_{l = 1}^d i(e_l) \right) \cdot \ln(t), 
\end{multline*}
Since $M_n \le 3M_n +1$ and  by definition
\[\rho^{(2)}(X;\phi,[B_X])(t) 
= - \sum_{n \ge 0} (-1)^n \cdot \ln\bigl({\det}_{\caln(G)}(\Lambda \circ \eta_{\phi^*\IC_t}(c_n))\bigr),
\]
we conclude for  $0 < t \le 1$
\begin{multline*}
- \sum_{n \ge 0} \ln\bigl(||c_n||_1)  + \sum_{n \ge 0}  \bigl(r_n - \dim_{\caln(G)}(\ker(\Lambda(c_n)))\bigr) \cdot (3M_n+1) 
\cdot \left(\sum_{l = 1}^d i(e_l) \right) \cdot \ln(t) 
\\
\le 
\rho^{(2)}(X;\phi,[B_X])(t) 
\le
\\
\sum_{n \ge 0} \ln\bigl(||c_n||_1)  - \sum_{n \ge 0}  \bigl(r_n - \dim_{\caln(G)}(\ker(\Lambda(c_n)))\bigr) \cdot (3M_n +1) 
\cdot \left(\sum_{l = 1}^d i(e_l) \right)\cdot \ln(t), 
\end{multline*}
and for  $t \ge 1$
\begin{multline*}
- \sum_{n \ge 0} \ln\bigl(||c_n||_1)  - \sum_{n \ge 0}  \bigl(r_n - \dim_{\caln(G)}(\ker(\Lambda(c_n)))\bigr) \cdot (3M_n+1) 
\cdot \left(\sum_{l = 1}^d i(e_l) \right)\cdot \ln(t) 
\\
\le 
\rho^{(2)}(X;\phi,[B_X])(t) 
\le
\\
\sum_{n \ge 0} \ln\bigl(||c_n||_1)  + \sum_{n \ge 0}  \bigl(r_n - \dim_{\caln(G)}(\ker(\Lambda(c_n)))\bigr) \cdot (3M_n+1) 
\cdot \left(\sum_{l = 1}^d i(e_l) \right) \cdot\ln(t), 
\end{multline*}
If we put
\begin{eqnarray*}
C 
& := &
\bigl(r_n - \dim_{\caln(G)}(\ker(\Lambda(c_n)))\bigr) \cdot (3M_n+1)  \cdot \left(\sum_{l = 1}^d i(e_l) \right);
\\
D 
& := & 
 \sum_{n \ge 0} \ln\bigl(||c_n||_1),
\end{eqnarray*}
assertion~\eqref{the:Properties_of_the_twisted_L2-torsion_function:determinant_class}  follows.
\\[2mm]~\eqref{the:Properties_of_the_twisted_L2-torsion_function:determinant_class:L2-acyclic_and_res_fin}
We conclude from
Theorem~\ref{the:Determinant_class_and_twisting}~\eqref{the:Determinant_class_and_twisting:L2-det_acyclic}
and assertion~\eqref{the:Properties_of_the_twisted_L2-torsion_function:determinant_class}
that $X$ is $\det$-$L^2$-acyclic.
\\[2mm]\eqref{the:Properties_of_the_twisted_L2-torsion_function:determinant_class:change_of_base_refinement}
This follows from Theorem~\ref{the:Basic_properties_of_the_V-twisted_L2-torsion}~%
\eqref{the:Basic_properties_of_the_V-twisted_L2-torsion:changing_the_base_refinement}.
\\[2mm]\eqref{the:Properties_of_the_twisted_L2-torsion_function:homotopy-invariance}
This follows from Theorem~\ref{the:Basic_properties_of_the_V-twisted_L2-torsion}~%
\eqref{the:Basic_properties_of_the_V-twisted_L2-torsion:homotopy_invariance}.
\\[1mm]~\eqref{the:Properties_of_the_twisted_L2-torsion_function:scaling} 
This follows directly from the definitions.
Hence the proof of Theorem~\ref{the:Properties_of_the_twisted_L2-torsion_function} is finished.
\end{proof}

\begin{remark}[$L^2$-acyclicity]
  \label{rem:L2-acyclicity}
  Notice that in assertion~\eqref{the:Properties_of_the_twisted_L2-torsion_function:determinant_class} of
  Theorem~\ref{the:Properties_of_the_twisted_L2-torsion_function} we do not require that
  $X$ is $L^2$-acyclic. So in context of a smooth Riemannian manifold $X$ one may
  consider the $L^2$-torsion by taking the structure of a Hilbert module on the homology
  into account which comes from the inner product of the $\phi$-twisted harmonic $L^2$-forms
  induced by the metric, similar to the classical definition of Ray-Singer torsion or
  $L^2$-torsion for closed Riemannian manifolds which are not necessarily acyclic. However,
  in this paper we will concentrate on the $L^2$-acyclic case and not deal with a possible analytic interpretation.
\end{remark}

%%%%%%%%%%%%%%%%%%%%%%%%%%%%%%%%%%%%%%%%%%%%%%%%%%%%%%%%%%%%%%%%%%%%%%%%%%%%%%%%%

\subsection{The reduced twisted $L^2$-torsion function for finite free $G$-$CW$-com\-plexes}
\label{subsec:The_reduced_twisted_L2_torsion_function_for_finite_free_G-CW-complexes}

In order to get rid of the choice of base refinement and to ensure homotopy invariance we 
introduce the following notions.  Let $A \subseteq \IR$ an abelian group.
We call two functions $f_0, f_1 \colon \IR^{>0} \to \IR$ \emph{$A$-equivalent}
if there exists an element $a \in A$ such that $f_0(t) - f_1(t) = a \cdot \ln(t)$ holds for all
$t \in \IR^{>0}$. Given a finite free $G$-$CW$-complex $X$ with base refinement $[B_X]$
and a homomorphism $\phi \colon G \to A$, 
we have introduced the $\phi$-twisted $L^2$-torsion function 
$\rho^{(2)}(X;\phi,[B_X]) \colon \IR^{>0} \to \IR$ in Definition~\ref{def:Twisted_L2-torsion_function}.
Let $\overline{\rho}^{(2)}(X;\phi)$ be its $A$-equivalence class. It is independent of the choice 
of base refinement by 
Theorem~\ref{the:Properties_of_the_twisted_L2-torsion_function}~%
\eqref{the:Properties_of_the_twisted_L2-torsion_function:determinant_class:change_of_base_refinement}
and depends only on $X$ and $\phi$, as the notation suggests. 

\begin{definition}[The reduced twisted $L^2$-torsion function]
In the case $A = \IR$, we call $\overline{\rho}^{(2)}(X;\phi)$ the \emph{reduced $\phi$-twisted
$L^2$-torsion function} of the finite free $G$-$CW$-complex $X$ and $\phi \colon G \to \IR$.
\end{definition}

Obviously the set of equivalence classes of functions $\IR^{>0} \to \IR$ inherits the structure of  
an abelian group from the abelian group structure on the set of maps $\IR^{>0} \to \IR$ coming from the
standard  abelian group structure on $\IR$. In the sequel  we will write $\doteq$ instead of $=$ to indicate that the equality 
is to be understood as an equality of $\IR$-equivalence classes of functions.

The next result follows from Theorem~\ref{the:Basic_properties_of_the_V-twisted_L2-torsion} 
and Theorem~\ref{the:Properties_of_the_twisted_L2-torsion_function}.

\begin{theorem}[Basic properties of the reduced $L^2$-torsion function for finite  free  $G$-$CW$-com\-plexes]
\label{the:Basic_properties_of_the_reduced_L2-torsion_function_for_finite_free_G-CW_complexes}
Consider  a  group homomorphism $\phi \colon G \to \IR$. 
Let $X$ be a free finite $G$-$CW$-complex.

\begin{enumerate}

\item \label{the:Basic_properties_of_the_reduced_L2-torsion_function_for_finite_free_G-CW_complexes:value_at_1}
\emph{Twisted $\det$-$L^2$-acyclic implies $\det$-$L^2$-acyclic}\\
Suppose that $X$ is $\phi$-twisted $\det$-$L^2$-acyclic. Then $X$ is $\det$-$L^2$-acyclic and
the real number given by $\overline{\rho}^{(2)}_G(X;\phi)(1)$  is the $L^2$-torsion $\rho^{(2)}(X;\caln(G))$ of $X$;

\item \label{the:Basic_properties_of_the_reduced_L2-torsion_function_for_finite_free_G-CW_complexes:trivial_phi}
\emph{Trivial twisting}\\
Suppose that $\phi$ is trivial. Then $X$ is $\phi$-twisted $\det$-$L^2$-acyclic if and only 
if $X$ is $\det$-$L^2$-acyclic. In this case
$\overline{\rho}^{(2)}_G(X;\phi)$ is constant with value the $L^2$-torsion $\rho^{(2)}(X;\caln(G))$;

\item \label{the:Basic_properties_of_the_reduced_L2-torsion_function_for_finite_free_G-CW_complexes:L2-acyclic_and_res_fin}
 \emph{$L^2$-acyclic implies twisted $\det$-$L^2$-acyclic}\\
Suppose that $X$ is $L^2$-acyclic, i.e., the $n$th $L^2$-Betti number $b_n^{(2}(X;\caln(G))$ vanishes for all $n \ge  0$. 
Assume that $G$ is finitely generated residually finite. 

Then $X$ is $\phi$-twisted $\det$-$L^2$-acyclic;

\item \label{the:Basic_properties_of_the_reduced_L2-torsion_function_for_finite_free_G-CW_complexes:(homotopy_invariance)}
\emph{$G$-homotopy invariance}\\
Let $X$ and $Y$ be finite free $G$-$CW$-complexes which are $G$-homotopy equivalent.
Suppose that $X$ is $\phi$-twisted $\det$-$L^2$-acyclic.

Then $Y$ is $\phi$-twisted $\det$-$L^2$-acyclic.
If we additionally  assume that the $K$-theoretic Farrell-Jones Conjecture holds for $\IZ G$ or that $X$ and 
$Y$ are simple $G$-homotopy equivalent, then we get
\[
\overline{\rho}^{(2)}_G(X;\phi) \doteq \overline{\rho}^{(2)}_G(Y;\phi);
\]

\item \label{the:Basic_properties_of_the_reduced_L2-torsion_function_for_finite_free_G-CW_complexes:(sum_formula)}
\emph{Sum formula}\\
Consider a $G$-pushout of finite free $G$-$CW$-complexes
\[
\xymatrix{
X_0 \ar[r]^{i_1}  \ar[d]_{i_2}
& 
X_1 \ar[d]^{j_1}
\\
X_2 \ar[r]_{j_2} 
&
X
}
\]
where $i_1$ is cellular, $i_0$ an inclusion of $G$-$CW$-complexes and $X$ has the obvious 
$G$-$CW$-structure coming from the ones on $X_0$, $X_1$ and $X_2$.
Suppose that $X_0$, $X_1$ and $X_2$ are $\phi$-twisted $\det$-$L^2$-acyclic.

Then $X$ is $\phi$-twisted $\det$-$L^2$-acyclic and we get
\[
\overline{\rho}^{(2)}_G(X;\phi) \doteq \overline{\rho}^{(2)}_G(X_1;\phi) 
+ \overline{\rho}^{(2)}_G(X_1;\phi) - \overline{\rho}^{(2)}_G(X_0;\phi);
\]

\item \label{the:Basic_properties_of_the_reduced_L2-torsion_function_for_finite_free_G-CW_complexes:(product_formula)} 
  \emph{Product formula}\\
Let $G$ and $H$ be groups. Let $X$ be a finite free $G$-$CW$-complex and $Y$
  be a finite free $H$-$CW$-complex.  Let $\phi \colon G \times H \to \IR$ be a group
  homomorphism. Denote by $\phi_G$ the  restriction of $\phi$ to $G = G \times \{1\} \subseteq G \times H$.
  Suppose that $X$ is $\phi_G$-twisted $\det$-$L^2$-acyclic.

Then $X \times Y$ is a finite free $G \times H$-$CW$-complex which is 
$\phi$-twisted $\det$-$L^2$-acyclic and we get
\[
\overline{\rho}^{G \times H}(X \times Y; \phi) \doteq \chi(Y/H) \cdot \overline{\rho}^{(2)}_G(X;\phi_G);
\]

\item \label{the:Basic_properties_of_the_reduced_L2-torsion_function_for_finite_free_G-CW_complexes:(induction)}
\emph{Induction}\\
Let $H \subset G$ be a subgroup of $G$.
Let $\phi_H$ be the restriction of $\phi\colon G \to \IR$ 
to $H$. Let $X$ be a finite free $H$-$CW$-complex which is
$\phi_H$-twisted $\det$-$L^2$-acyclic. 

Then $G \times_HX$ is a finite free $G$-$CW$-complex which is
$\phi$-twisted $\det$-$L^2$-acyclic and we get
\[\overline{\rho}^{(2)}_G(G \times_H X;\phi) \doteq \overline{\rho}^{(2)}(X;\phi_H);
\]

\item \label{the:Basic_properties_of_the_reduced_L2-torsion_function_for_finite_free_G-CW_complexes:restriction}
\emph{Restriction}\\
Let $i\colon H \to  G$ be the inclusion of a subgroup $H$ of $G$ of finite index.
Put $i^* \phi := \phi \circ i \colon H \to \IR$.
Let $X$ be a finite free $G$-$CW$-complex

Then the restriction $i^* X$ of the $G$-space $X$ to an $H$-space is a finite free $H$-$CW$-complex.
It is $i^*\phi$-twisted $\det$-$L^2$-acyclic if and only if $X$ is $i^*\phi$-twisted $L^2$, and in this case we get
\[\overline{\rho}^{(2)}_G(X;\phi) \doteq [G:H] \cdot \overline{\rho}^{(2)}_H(i^*X;i^*\phi);
\]

\item \label{the:Basic_properties_of_the_reduced_L2-torsion_function_for_finite_free_G-CW_complexes:Poincare_duality}
  \emph{Poincar\'e duality}\\
  Let  $X$ be  a  finite free  $G$-$CW$-complex such  that  $X/G$ is  a finite  orientable
  $n$-dimensional  simple Poincar\'e  complex, e.g.,  a  cocompact free
  proper smooth $G$-manifold $X$ of dimension $n$ without boundary such that $X$ is orientable
  and the $G$-action  is orientation preserving. Suppose that $X$ is $\phi$-twisted $L^2$-acyclic.
 
 Then 
\[
\overline{\rho}^{(2)}_G(X;\phi)(t) \doteq (-1)^{n+1} \cdot \overline{\rho}^{(2)}_G(X;\phi)(t^{-1}).
\]
The formula still holds if we drop the condition simple but assume that the $K$-theoretic
Farrell-Jones Conjecture holds for $\IZ G$;

\item  \label{the:Basic_properties_of_the_reduced_L2-torsion_function_for_finite_free_G-CW_complexes:scaling}
\emph{Scaling $\phi$}\\
Let $r \in \IR$ be a real number. Put $\phi_r := (r \cdot \id_{\IR}) \circ \phi \colon G \to \IR$.
Then $X$ is $\phi$-twisted $\det$-$L^2$-acyclic
if and only if it is $\phi_r$-twisted $\det$-$L^2$-acyclic, and in this case
\[
\overline{\rho}^{(2)}(X;\phi_r)(t) \doteq \overline{\rho}^{(2)}(X;\phi)(t^r).
\]

\end{enumerate}
\end{theorem}

\begin{remark}\label{rem:Z-equivalence_class}
If we consider only elements $\phi$ in $H^1(X;\IZ)$, 
then we can  replace $\IR$-equivalence class by $\IZ$-equivalence class in the definition of
$\overline{\rho}^{(2)}(X;\phi)$ and everywhere in 
Theorem~\ref{the:Basic_properties_of_the_reduced_L2-torsion_function_for_finite_free_G-CW_complexes}.
The same remark applies also to $\overline{\rho}^{(2)}(\widetilde{X};\phi)$
introduced in Definition~\ref{def:The_L2-torsion_function_for_universal_coverings}
and to Theorem~\ref{the:Basic_properties_of_the_reduced_L2-torsion_function_for_universal_coverings}.
This does play a crucial role in~\cite{Dubois-Friedl-Lueck(2015symmetric)}.
\end{remark}

We leave the details of the proof of Theorem~\ref{the:S1-actions} 
to the reader since it is very similar to the proof
of~\cite[Theorem~3.105 on page~168]{Lueck(2002)}. 

\begin{theorem}[$S^1$-actions]\label{the:S1-actions}
Let $X$ be a connected finite $S^1$-$CW$-complex.
Let $\mu \colon \pi_1(X) \to G$ be a group homomorphism and let $\overline{X} \to X$ be
the associated $G$-covering. Consider a homomorphism $\phi \colon G \to \IZ$. 
Suppose that for one and hence all $x \in X$ the composite  
$\pi_1(S^1,1) \xrightarrow{\pi_1(\ev_x,1)} \pi_1(X,x) \xrightarrow{\mu} G$
is injective, where $\ev_x \colon S^1 \to X$ sends $z$ to $z \cdot x$. If the composite 
\[
\pi_1(S^1,1) \xrightarrow{\pi_1(\ev_x,1)} \pi_1(X,x) \xrightarrow{\mu} G  \xrightarrow{\phi} \IZ
\]
has infinite image, define $k$ to be the index of this image in $\IZ$ and if the image is trivial, put $k = 0$. 
Define the $S^1$-orbifold Euler characteristic of $X$ by
\[
\chi^{S^1}_{\orb}(X) = \sum_{n \ge 0} (-1)^n \cdot \sum_{e \in I_n} \frac{1}{|S^1_e|},
\]
where $I_n$  is the set of open $n$-dimensional $S^1$-cells of $X$ and for such a $S^1$-cell
$e \in I_n $ we denote by $S^1_e$ the isotropy group of any point in $e$. 

Then $X$ is up to simple homotopy type a finite $CW$-complex.  (If $X$ is a compact smooth
manifold with smooth $S^1$-action, one can equip $X$ with the $S^1$-$CW$-structure coming
from an $S^1$-equivariant smooth triangulation.)  Moreover, $X$ is $\phi$-twisted
$\det$-$L^2$-acyclic and there is a representative $\rho(t)$ of $\overline{\rho}^{(2)}_G(\overline{X};\phi)$
such that we get for $t > 0$ 
\[
\rho(t) = \begin{cases} 
\chi^{S^1}_{\orb}(X) \cdot k \cdot \ln(t) 
& 
t \ge 1;
\\
0 & 
t \le 1.
\end{cases}
\]
\end{theorem}

\begin{theorem}[Fibrations]\label{the:fibrations_rho_G-covering}
Let $F \xrightarrow{i}  E \xrightarrow{p} B$ be a fibration of connected finite $CW$-complexes.
Let  $\overline{E} \to E$ be a $G$-covering.
Denote by $\overline{F} \to F$ the $G$-covering obtained by the pullback construction applied
to $\overline{E} \to E$ and  $i$.  Let $\phi \colon G \to \IR$ be a group homomorphism.
Assume that $\overline{F}$ is $\phi$-twisted $\det$-$L^2$-acyclic.
Suppose that  the $K$-theoretic Farrell-Jones Conjecture holds for $\IZ G$ 
or that the class $\Theta(p)$ occurring in~\cite[Theorem~3.100 on page~166]{Lueck(2002)}
or~\cite[Section~3]{Farrell-Lueck-Steimle(2010)}
vanishes. (The last condition is automatically satisfied if $p$ is a locally trivial fiber bundle.)

Then $\overline{E} $ is up to (simple) $G$-homotopy a finite free $G$-$CW$-complex, 
see~\cite[Section~3]{Farrell-Lueck-Steimle(2010)},
 is $\phi$-twisted $\det$-$L^2$-acyclic and we get
\[
\overline{\rho}_G^{(2)}(\overline{E};\phi) \doteq \chi(B) \cdot \overline{\rho}^{(2)}_G(\overline{F};\phi).
\]
\end{theorem}
\begin{proof}
The proof is analogous to the proof of~\cite[Theorem~3.100 on page~166]{Lueck(2002)}.
\end{proof}

%%%%%%%%%%%%%%%%%%%%%%%%%%%%%%%%%%%%%%%%%%%%%%%%%%%%%%%%%%%%%%%%%%%%%%%%%%%%%%%%%

\subsection{Mapping tori}
\label{subsec:Mapping_tori}

In this  subsection we want to study mapping tori of self homotopy equivalences. We start with the chain complex version.

\begin{lemma}
\label{lem:torsion_of_chain-complexes-perturbated}
Let $f_*,g_* \colon C_* \to C_*$ be a chain homotopy equivalence of finite Hilbert $\caln(G)$--chain complexes.
Then there exists real numbers $0 < t_0 < t_{\infty}$ 
such that also $g_* + t \cdot f_* \colon C_* \to C_*$
is a chain homotopy equivalence for $0 \le t \le t_0$ and for $t_{\infty} \le t$.  We get for the $L^2$-torsion
\begin{eqnarray*}
\lim_{t \to 0+} \tau^{(2)}\bigl((g_* + t \cdot f_*)\colon C_*\to C_*\bigr)
& = & \tau^{(2)}(g_*);
\\
\lim_{t \to \infty}  \left(\tau^{(2)}\bigl((g_* + t \cdot f_*) \colon C_*\to C_*\bigr) - \chi^{(2)}(C_*) \cdot \ln(t)\right)
& = & 
\tau^{(2)}(f_*).
\end{eqnarray*}
\end{lemma}
\begin{proof} 
  We begin with the case $t \to 0+$.  Since $g_* + t \cdot f_*$ is chain homotopic to 
$g _*   \circ (\id + t \cdot g_*^{-1} \circ f_*)$ for any chain homotopy inverse $g_*^{-1}$ of $g_*$, it
  suffices to consider the special case $g_* = \id$ by~\cite[Theorem~3.35~(3) and~(4) on   page~142]{Lueck(2002)}.

Since $\id \colon C_n \to C_n$ is an isomorphism, there exists $t_0 > 0$ such that 
$\id_{C_n} + t \cdot f_n$ is an isomorphism for all $n$ and $t \in [0,t_0]$. We conclude 
from~\cite[Theorem~3.35~(2) on page~142]{Lueck(2002)} for all $t \in [0,t_0]$
\[
\tau^{(2)}(\id_* + t \cdot f_*) = 
\sum_{n \ge 0} (-1)^n \cdot \ln\bigl({\det}_{\caln(G)}(\id_{C_n} + t \cdot f_n \colon C_n \to C_n)\bigr).
\]
Because of Lemma~\ref{lem:continuity_of_det_on_GL_n(caln(G))} the function
\[
[0,t_0] \to \IR, \quad t \mapsto  \tau^{(2)}\bigl(\id_*+ t \cdot f_*)
\]
is continuous and hence
\[
\lim_{t \to 0+} \tau^{(2)}\bigl((\id_* + t \cdot f_*)\colon C_*\to C_*\bigr) = \tau^{(2)}(\id_*) = 0.
\]
This finishes the proof of the case $t \to 0+$. 

The case $t \to \infty$ follows from the
equation, see~\cite[Theorem~3.35~(1) and~(4) on page~142, Theorem~3.35~(6c) on page~143]{Lueck(2002)}
\begin{eqnarray*}
\tau^{(2)}\bigl(g_* + t \cdot f_*) 
 & = &
\tau^{(2)}\bigl(t \cdot \id) \circ (f_* + t^{-1} \cdot g_*) 
\\
& = &
\tau^{(2)}(t \cdot \id)  + \tau^{(2)}(f_* + t^{-1} \cdot g_*) 
\\
& = &
\chi^{(2)}(C_*) \cdot \ln(t) + \tau^{(2)}(f_* + t^{-1} \cdot g_*),
\end{eqnarray*}
and the conclusion from the already proved case
\[
\lim_{t \to \infty} \tau^{(2)}(f_* + t^{-1} \cdot g_*) = \tau^{(2)}(f_*).
\]
\end{proof}

Next consider the following situation. Let $h \colon F \to F$ be a
  cellular selfhomotopy equivalence of the connected finite $CW$-complex $F$.  Denote by
  $T_h$ its mapping torus equipped with the $CW$-structure coming from the one on $F$. 
  Fix an element $\gamma \in \pi_1(T_h)$ which is mapped under the homomorphism
  $\pi_1(T_h) \to \IZ = \pi_1(S^1)$ induced by the canonical projection $T_h \to S^1$ to a
  generator of $\IZ$. Suppose that $\pi_1(T_h) \to \IZ = \pi_1(S^1)$ factorizes as 
  $\pi_1(T_h) \xrightarrow{\mu} G \xrightarrow{\phi} \IZ$, where $G$ is residually finite.
  Let $\overline{\gamma} \in G$ be the image of $\gamma$ under $\mu
  \colon \pi_1(T_h) \to G$.  

  Denote by $K$ the kernel of $\phi \colon G \to \IZ$.  Let $p
  \colon \overline{T_h} \to T_h$ be the $G$-covering associated to $\mu \colon \pi_1(T_h)
  \to G$ and let $q \colon \overline{F} \to F$ be the $K$-covering of $F$ associated to
  the map $\mu' \colon \pi_1(F) \to K$ induced by $\mu \colon \pi_1(T_h) \to G$.

  Fix a base refinement for the $\IZ \pi_1(F)$-chain complex $C_*(\widetilde{F})$.  It
  induces a based refinement for the $\IZ[\pi_1(T_h)]$-chain complex
  $C_*(\widetilde{T_h)}$.  In the sequel we use the induced based refinement for the $\IZ
  G$-chain complex $C_*(\overline{T_h})$ and denote $\rho \colon \IR^{>0} \to \IR$ the
    associated representative of the equivalence class given by the $L^2$-torsion function
    $\overline{\rho}^{(2)}_G(\overline{T_h};\phi)$. (Actually, $\rho \colon \IR^{>0} \to \IR$ is
    independent of the choice of base refinement for the $\IZ\pi_1(F)$-chain complex
    $C_*(\widetilde{F})$.)

  We have the group automorphism $c_{\gamma} \colon  \pi_1(F)  \to \pi_1(F) , \; w \mapsto  \gamma^{-1} \cdot w \cdot \gamma$
  which is just $\pi_1(h)$. Let $\widetilde{h} \colon \widetilde{F} \to \widetilde{F}$ the $c_{\gamma}$-equivariant lift of $h$.
  We have the group automorphism
  $c_{\overline{\gamma}} \colon K \to K, \; k \mapsto  \overline{\gamma}^{-1} \cdot k \cdot \overline{\gamma}$. 
  Denote by $\overline{h} \colon \overline{F} \to \overline{F}$ be the $c_{\overline{\gamma}}$-equivariant 
 homotopy equivalence induced by $\widetilde{h}$.
  Let $T_0 > 0$ be a real number
  such that the spectral radius $\specrad(C_n^{(2)}(\overline{h}))$
  of the bounded $c_{\overline{\gamma}}$-equivariant operator
   $C_n^{(2)}(\overline{h}) \colon C_n^{(2)}(\overline{F}) \to C_n^{(2)}(\overline{F})$ is
  bounded by $T_0$. Let $T_{\infty} > 0$ be a real number for which there exists a
  $(c_{\overline{\gamma}})^{-1}$-equivariant chain map 
$C_*(\overline{h})^{-1}_* \colon   C_n^{(2)}(\overline{F}) \to C_n^{(2)}(\overline{F})$ 
  such that the composites   $C_*(\overline{h}) \circ C_*(\overline{h})^{-1}_*$ and 
$C_*(\overline{h})^{-1}_* \circ C_*(\overline{h})$ are $\IZ K$-chain
  homotopic to the identity and the spectral radius 
  $\specrad\bigl(\Lambda(C_*(\overline{h})_n^{-1} )\bigr)$ of the bounded
  $(c_{\overline{\gamma}})^{-1}$-equivariant operator 
  $\Lambda(C_*(\overline{h})_n^{-1}) \colon C_n^{(2)}(\overline{F}) \to  C_n^{(2)}(\overline{F})$ 
  is bounded by $T_{\infty}$ for each $n \ge 0$.

\begin{theorem}[Mapping tori] \label{the:mapping_tori} Let $G$ be residually finite.
Then $\overline{T_h}$ is $\phi$-twisted $\det$-$L^2$-acyclic and 
\[
\begin{array}{lclcl}–
\rho(t) & = & 0 & \text{if} & t \le T_0^{-1};
\\
\rho (t) & = & \chi(F) \cdot \ln(t) & \text{if} & t \ge T_{\infty}.
\end{array}
\]
\end{theorem}
\begin{proof}
Let $G'$ be the image of $\mu$. Let $\mu' \colon \pi = \pi_1(T_h) \to G'$ be the epimorphism induced by
$\mu$ and $\phi' \colon G' \to \IZ$  be the restriction of $\phi$ to $G'$.
The $G'$-covering $\overline{T_h}' \to T_h$ associated to $\mu'$ is $L^2$-acyclic by~\cite[Theorem~2.1]{Lueck(1994b)},
The $G$-covering $\overline{T_h} \to T_h$ associated to $\mu$ is given by $G \times_{G'} \overline{T_h}'$ 
and hence $L^2$-acyclic as well by~\cite[Theorem~1.35~(10) on page~37]{Lueck(2002)}.
Since $G$ is residually finite by assumption, 
the $G$-covering $\overline{T_h}$ is $\phi$-twisted $\det$-$L^2$-acyclic by
Theorem~\ref{the:Basic_properties_of_the_reduced_L2-torsion_function_for_finite_free_G-CW_complexes}~%
\eqref{the:Basic_properties_of_the_reduced_L2-torsion_function_for_finite_free_G-CW_complexes:L2-acyclic_and_res_fin}.

Let $f_* \colon \IZ G \otimes_{\IZ  \pi} C_*(\widetilde{F}) \to \IZ G \otimes_{\IZ \pi} C_*(\widetilde{F})$ 
be the $\IZ G$-chain map sending $g \otimes x$ to $- g\gamma \otimes C_*(\widetilde{h})(x)$.
Choose any base refinement for the finite free $\IZ K$-chain complex $C_*(\overline{F})$. 
It induces an equivalence class of $\IZ K$-basis
on the finite free $\IZ K$-chain complex $C_*(\overline{F})$ and hence on the mapping cone of 
$\id  + f_* \colon \IZ G \otimes_{\IZ K} C_*(\widetilde{F}) \to \IZ G \otimes_{\IZ K} C_*(\widetilde{F})$ which is the
cellular $\IZ G$-chain complex of the free $G$-$CW$-complex $\overline{T_h}$. 
Thus we obtain an explicit base refinement for 
the $\IZ G$-chain complex $C_*(\overline{T_h})$  for which in the sequel the representative $\rho$ 
of $\overline{\rho}^{(2)}(\widetilde{T_h},\phi)$ will refer to.

One easily checks that $\eta_{\IC_t}(C_*(\overline{T_h}))$ is the mapping cone of
$\id +  t \cdot f_* \colon \IZ G \otimes_{\IZ K} C_*(\overline{F}) \to \IZ G \otimes_{\IZ K} C_*(\overline{F})$.
We conclude from Lemma~\ref{lem:torsion_of_chain-complexes-perturbated}
that $\rho$ is continuous near $0$ and $\infty$ and satisfies
\begin{eqnarray*}
\lim_{t \to 0+} \rho^{(2)}(t)   & = & 0;
\\
\lim_{t \to \infty} \left(\rho^{(2)}(t)    -\chi(F) \cdot \ln(t) \right) & = & \tau^{(2)}(\Lambda(f_*)).
\end{eqnarray*}
Since $G$ is residually finite, $G$ 
satisfies the Determinant Conjecture,
see see~\cite[Conjecture~13.2 on page~454 and Theorem~13.3 on page~454]{Lueck(2002)},
 and hence the map induced by the Fuglede-Kadison determinant
$\Wh(G) \to \IR$ is trivial. Since this map sends the Whitehead torsion $\tau(f_*) \in \Wh(G)$ of $f_*$ to 
$\tau^{(2)}(\Lambda(f_*))$
we conclude 
\begin{eqnarray}
\tau^{(2)}(\Lambda(f_*)) & = & 0
\label{tau(2)(Lambda(f_ast))_vanishes}
\end{eqnarray}
Hence we get
\begin{eqnarray*}
\lim_{t \to 0+} \rho^{(2)}(t)   & = & 0;
\\
\lim_{t \to \infty} \left(\rho^{(2)}(t)    -\chi(F) \cdot \ln(t) \right) & = & 0.
\end{eqnarray*}
In order to check that $\rho^{(2)}(t)$ is constant near $0$ and is $\chi(F) \cdot \ln(t)$
near $\infty$ one has to analyze the proof of
Lemma~\ref{lem:torsion_of_chain-complexes-perturbated} in more detail and use
Theorem~\ref{the:CFM} as described next. Obviously the spectral radius of $f_n$ agrees
with the spectral radius of $C_n(\overline{h})$.  Hence $\specrad(t \cdot f_n) < 1$ for
every $t \in (0,T_0^{-1})$ and $n \ge 0$.  Therefore $\id + t \cdot f_n$ is invertible for
every $t \in (0,T_0^{-1})$ and $n \ge 0$, an inverse is given by $\sum_{k \ge 0} (-1)^k \cdot t^k f^k_n$. 
This follows from the fact that $\sum_{k \ge 0} ||(-1)^k \cdot t^k f^k_n|| < \infty$. 
In the notation of Lemma~\ref{lem:torsion_of_chain-complexes-perturbated} we get
in the special case $g_* = \id$ for $t \in (0,T_0)$
\begin{eqnarray*}
\lefteqn{\tau^{(2)}(\Lambda(\id_{\IZ G \otimes_{\IZ K} C_*} + t \cdot f_*))}
& & 
\\
& = & 
\sum_{n \ge 0} (-1)^n \cdot \ln\bigl({\det}_{\caln(G)}(\Lambda(\id_{\IZ G \otimes_{\IZ K} C_*} 
+ t \cdot f_n) \colon \Lambda(C_n) \to \Lambda(C_n))
\\
& = & 
\sum_{n \ge 0} (-1)^n \cdot \left( \ln\bigl({\det}_{\caln(G)}(\id_{\Lambda(\IZ G \otimes_{\IZ K} C_n)} 
+ 0\cdot \Lambda(f_n))\right.
\\ 
& & \quad + \int_0^t 
\Real \biggl(\tr_{\caln(G)}\biggl((\id_{\Lambda(\IZ G \otimes_{\IZ K} C_n)} 
\\
& & \quad \quad + s \cdot \Lambda(f_n))^{-1} 
\cdot \left. \frac{d(\id_{\Lambda(\IZ G \otimes_{\IZ K} C_n)} 
+ s \circ  \Lambda(f_n))}{ds}\right|_s \,ds\biggr)\biggr)\biggr)
\\
& = & 
\sum_{n \ge 0} (-1)^n \cdot \left(\int_0^t 
\Real \left(\tr_{\caln(G)}\left(\left(\sum_{k \ge 0} (-1)^k \cdot s^k \cdot \Lambda(f_n)^k\right) \circ  \Lambda(f_n)\right)  \,ds\right)\right)
\\
& = & 
\sum_{n \ge 0} (-1)^n \cdot \left(\int_0^t 
\Real \left(\tr_{\caln(G)}\biggl(\sum_{k \ge 0} (-1)^k \cdot s^k \cdot \Lambda(f_n)^{k+1}\biggr)  \,ds\right)\right)
\\
& = & 
\sum_{n \ge 0} (-1)^n \cdot \left(\sum_{k \ge 0} (-1)^k \cdot t^k \cdot  \int_0^t  \Real \left(\tr_{\caln(G)}(\Lambda(f_n)^{k+1})  \,ds\right)\right).
\end{eqnarray*}
A direct inspection shows $\tr_{\caln(G)}(\Lambda(f_n)^{k+1}) = 0$ for all $k \ge 0$. This implies
\[
\tau^{(2)}(\id_{\Lambda(\IZ G \otimes_{\IZ K} C_*)} + \;t \cdot \Lambda(f_*)) 
= \tau^{(2)}(\id_{\Lambda(\IZ G \otimes_{\IZ K}C_*)} + \; 0 \cdot f_*)  = 0
\]
for $t \in (0,T_0)$. 

Next we treat the case $t \to \infty$. 
Let $f_*^{-1} \colon \IZ G \otimes_{\IZ  \pi} C_*(\widetilde{F}) \to \IZ G \otimes_{\IZ \pi} C_*(\widetilde{F})$ 
be the $\IZ G$-chain map sending $g \otimes x$ to $- g\gamma^{-1} \otimes C_*(\widetilde{h})^{-1}(x)$.
Since $C_*(h)^{-1}$ is a chain homotopy inverse for $C_*(\overline{h})$, the $\IZ$-chain map
$f_*^{-1}$ is a $\IZ G$-chain homotopy inverse of $f_*$.

We have the equation, 
see~\cite[Theorem~3.35~(1) and~(4) on page~142, Theorem~3.35~(6c) on page~143]{Lueck(2002)}
\begin{eqnarray*}
\lefteqn{\tau^{(2)}(\Lambda(\id_{\IZ G \otimes_{\IZ K} C_*} + t \cdot f_*))}
& & 
\\
 & = &
\tau^{(2)}\bigl(\Lambda(t \cdot \id_{\IZ G \otimes_{\IZ K} C_*}) \circ f_*   \circ (t^{-1} \cdot (f_*)^{-1} + \id_{\IZ G \otimes_{\IZ K} C_*}))\bigr)
\\
& = &
\tau^{(2)}\bigl(\Lambda(t \cdot \id_{\IZ G \otimes_{\IZ K} C_*})\bigr) + \tau^{(2)}(\Lambda(f_*)) 
+ \tau^{(2)}\bigl(\Lambda((t^{-1} \cdot (f_*)^{-1} + \id_{\IZ G \otimes_{\IZ K} C_*}))\bigr)
\\
& = & 
\chi(F) \cdot \ln(t) +  \tau^{(2)}(\Lambda(f_*))   +  \tau^{(2)}\bigl(\Lambda(\id_{\IZ G \otimes_{\IZ K} C_*} + (t^{-1} \cdot (f_*)^{-1})\bigr)
\\
& \stackrel{\eqref{tau(2)(Lambda(f_ast))_vanishes}}{=} & 
\chi(F) \cdot \ln(t) +  \tau^{(2)}\bigl(\Lambda(\id_{\IZ G \otimes_{\IZ K} C_*} + (t^{-1} \cdot (f_*)^{-1})\bigr).
\end{eqnarray*}
Similarly as above for the case $t \to 0+$, one proves 
\[
\tau^{(2)}\bigl(\Lambda(\id_{\IZ G \otimes_{\IZ K} C_*} + (t^{-1} \cdot (f_*)^{-1})\bigr) = 0 \quad \text{for} \; t^{-1} \le T_{\infty}^{-1}.
\]
Hence we get
\[
\tau^{(2)}(\Lambda(\id_{\IZ G \otimes_{\IZ K} C_*} + t \cdot f_*))
= \chi(F) \cdot \ln(t)  \quad \text{for} \; t \ge T_{\infty}.
\]
Because of~\eqref{tau(2)(Lambda(f_ast))_vanishes} Theorem~\ref{the:mapping_tori} is proved.
\end{proof}

\begin{remark}[Entropy of a surface homomorphism]
  \label{rem:entropy_of_a_surface_homomorphism}
  Consider the situation and notation described in Theorem~\ref{the:mapping_tori} in the
  special case that $F$ is a closed surface. Denote by $h(\pi_1(h))$ the entropy of the
  automorphisms $\pi_1(h)$ of $\pi_1(F)$ which agrees with the
  dilatation of $f$ in the special case that $h$ is pseudo-Anosov and $\chi(F) < 0$.  Then we can choose
  $T_0 = T_{\infty} = h(\pi_1(f))$ in Theorem~\ref{the:mapping_tori}, as proved
  in~\cite[Theorem~8.5]{Dubois-Friedl-Lueck(2016)}.
\end{remark}

%%%%%%%%%%%%%%%%%%%%%%%%%%%%%%%%%%%%%%%%%%%%%%%%%%%%%%%%%%%%%%%%%%%%%%%%%%%%%%%%%
\subsection{The reduced $L^2$-torsion function for universal coverings}
\label{subsec:L2-torsion_function_for_universal_coverings}

The most interesting case is the one of the universal covering, where also some statements simplify.

\begin{definition}[The $L^2$-torsion function for universal coverings]
\label{def:The_L2-torsion_function_for_universal_coverings}
If $X$ is a connected finite $CW$-complex with universal covering $\widetilde{X}$ and we
have an element $\phi \in H^1(X;\IR) = \hom(\pi_1(X),\IR)$, then we 
say that  $\widetilde{X}$ is \emph{$\phi$-twisted $\det$-$L^2$-acyclic}
if  $\widetilde{X}$ is $\phi$-twisted $\det$-$L^2$-acyclic as  finite free
$\pi_1(X)$-$CW$-complex. If $\widetilde{X}$ is $\phi$-twisted $\det$-$L^2$-acyclic, we abbreviate
\[
\overline{\rho}^{(2)}(\widetilde{X};\phi) :\doteq \overline{\rho}_{\pi_1(X)}^{(2)}(\widetilde{X};\phi).
\]
If $X$ is  a finite $CW$-complex and we  have $\phi \in H^1(X)$, then we say that  $\widetilde{X}$ 
is \emph{$\phi$-twisted $\det$-$L^2$-acyclic} if for any component
$C \in \pi_0(X)$ its universal covering $\widetilde{C}$ 
is $\phi_C$-twisted $\det$-$L^2$-acyclic,
where $\phi_C$ is the restriction of $\phi$ to $C$. If $\widetilde{X}$ 
is $\phi$-twisted $\det$-$L^2$-acyclic,  we define the equivalence class of maps $\IR^{>0} \to \IR$.
\[
\overline{\rho}^{(2)}(\widetilde{X};\phi) :\doteq \sum_{C \in \pi_0(X)} \overline{\rho}^{(2)}(\widetilde{C};\phi_C).
\]
\end{definition}

In the following
Theorem~\ref{the:Basic_properties_of_the_reduced_L2-torsion_function_for_universal_coverings}
equality is to be understood as equality of equivalence classes of functions $\IR^{>0} \to \IR$. 
It follows from
Theorem~\ref{the:Basic_properties_of_the_reduced_L2-torsion_function_for_finite_free_G-CW_complexes}.

\begin{theorem}[Basic properties of the reduced $L^2$-torsion function for universal coverings]
  \label{the:Basic_properties_of_the_reduced_L2-torsion_function_for_universal_coverings}
  Let $X$ be a finite $CW$-complex and $\phi$ an element in $H^1(X)$.

\begin{enumerate}

\item \label{the:Basic_properties_of_the_reduced_L2-torsion_function_for_universal_coverings:value_at_1}
  \emph{Twisted $\det$-$L^2$-acyclic implies $\det$-$L^2$-acyclic}\\
  Suppose that $X$ is connected and $\widetilde{X}$ is $\phi$-twisted $\det$-$L^2$-acyclic. 

  Then $\widetilde{X}$ is $\det$-$L^2$-acyclic and the real number $\overline{\rho}^{(2)}(\widetilde{X};\phi)(1)$ is the
  $L^2$-torsion $\overline{\rho}^{(2)}(\widetilde{X})$ of the universal covering $\widetilde{X}$ of $X$;

\item \label{the:Basic_properties_of_the_reduced_L2-torsion_function_for_universal_coverings:trivial_phi}
\emph{Trivial twisting}\\
Suppose that $\phi$ is trivial.

Then $\widetilde{X}$ is $\phi$-twisted $\det$-$L^2$-acyclic if and only if $\widetilde{X}$
is $\det$-$L^2$-acyclic. In this case $\overline{\rho}^{(2)}(\widetilde{X};\phi)$ is
constant with value the $L^2$-torsion $\overline{\rho}^{(2)}(\widetilde{X})$;

\item \label{the:Basic_properties_of_the_reduced_L2-torsion_function_for_universal_coverings:L2-acyclic_and_res_fin}
  \emph{$L^2$-acyclic implies twisted $\det$-$L^2$-acyclic}\\
  Suppose that $X$ is connected and $\widetilde{X}$ is $L^2$-acyclic, i.e., 
 the $n$th $L^2$-Betti number $b_n^{(2}(\widetilde{X})$ vanishes for all $n \ge  0$.
  Assume that $\pi_1(X)$ is residually finite.
   
 Then $\widetilde{X}$ is $\phi$-twisted $\det$-$L^2$-acyclic;

\item \label{the:Basic_properties_of_the_reduced_L2-torsion_function_for_universal_coverings:homotopy_invariance} 
   \emph{Homotopy invariance}\\
   Let $X$ and $Y$ be finite $CW$-complexes together with elements 
   $\phi_X  \in H^1(X;\IR)$ and $\phi_Y \in H^1(Y;\IR)$. Suppose that there is a homotopy equivalence 
  $f  \colon X \to Y$ such that $f^* \phi_Y := H^1(f;\IR)(\phi_Y)$ agrees with $\phi_X$. 
  Assume that $X$ is $\phi_X$-twisted $\det$-$L^2$-acyclic.

  Then  $Y$ is $\phi_Y$-twisted $\det$-$L^2$-acyclic. 

  If we additionally assume that the $K$-theoretic Farrell-Jones Conjecture holds for
  $\IZ \pi_1(X)$ or that $f$ is a simple homotopy equivalence, then we get
\[
\overline{\rho}^{(2)}(\widetilde{X};\phi_X) \doteq \overline{\rho}^{(2)}(\widetilde{Y};\phi_Y);
\]

\item \label{the:Basic_properties_of_the_reduced_L2-torsion_function_for_universal_coverings:sum_formula}
\emph{Sum formula}\\
Consider a cellular pushout of finite  $CW$-complexes
\[
\xymatrix{
X_0 \ar[r]^{i_1} \ar[rd]^{j_0} \ar[d]_{i_2}
& 
X_1 \ar[d]^{j_1}
\\
X_2 \ar[r]_{j_2} 
&
X
}
\]
where $i_1$ is cellular, $i_0$ an inclusion of $CW$-complexes and $X$ has the obvious
$CW$-structure coming from the ones on $X_0$, $X_1$ and $X_2$. Suppose that for $i =0,1,2$
the map $j_i$ is $\pi_1$-injective, i.e., for any choice of bases point $x_i \in X_i$ the
induced map $\pi_1(j_i,x_i) \colon \pi_1(X_i,x_i) \to \pi_1(X,j_i(x_i))$ is injective.
Suppose we are given elements $\phi_i \in H^1(X_i;\IR)$ and $\phi \in H^1(X;\IR)$ such that
$j_i^*(\phi) = \phi_i$ holds for $i = 0,1,2$.  Assume that $\widetilde{X_i}$ is $\phi_i$-twisted
$\det$-$L^2$-acyclic for $i = 0,1,2$.

Then $\widetilde{X}$ is $\phi$-twisted $\det$-$L^2$-acyclic and we get
\[
\overline{\rho}^{(2)}(\widetilde{X};\phi) \doteq\overline{\rho}^{(2)}(\widetilde{X_1};\phi_1) 
+ \overline{\rho}^{(2)}(\widetilde{X_2};\phi_2) - \overline{\rho}^{(2)}(\widetilde{X_0};\phi_0);
\]

\item \label{the:Basic_properties_of_the_reduced_L2-torsion_function_for_universal_coverings:product_formula}
\emph{Product formula}\\
Let $X$ and $Y$ be a finite $CW$-complexes. Suppose that $Y$ is connected.
Consider an element $\phi \in H^1(X \times Y;\IR)$. Define
$\phi_X$  to be the image of $\phi$ under the map $H^1(X \times Y;\IR)   \to H^1(X;\IR)$ induced
by the inclusion $X \to X \times Y, \; x \mapsto (x,y)$ for any choice of base point $y \in Y$.
Suppose that $\widetilde{X}$ is $\phi_X$-twisted $\det$-$L^2$-acyclic.

Then $\widetilde{X \times Y}$ is $\phi$-twisted $\det$-$L^2$-acyclic and we get
\[
\overline{\rho}^{(2)}(\widetilde{X \times Y}; \phi) \doteq \chi(Y) \cdot \overline{\rho}^{(2)}(\widetilde{X};\phi_X);
\]

\item \label{the:Basic_properties_of_the_reduced_L2-torsion_function_for_universal_coverings:Poincare_duality}
\emph{Poincar\'e duality}\\
  Let  $X$ be  a finite  orientable $n$-dimensional  simple Poincar\'e  complex, e.g.,  
  a  closed orientable manifold of dimension $n$ without boundary.
  Suppose that $X$ is  $\phi$-twisted $\det$-$L^2$-acyclic. 
Then 
\[
\overline{\rho}^{(2)}(\widetilde{X};\phi)(t) \doteq (-1)^{n+1} \cdot \overline{\rho}^{(2)}(\widetilde{X};\phi)(t^{-1}). 
\]
 The formula still holds if we drop the condition simple but assume that the $K$-theoretic 
Farrell-Jones Conjecture holds for $\IZ\pi_1(X,x)$ for all base points $x \in X$;

\item \label{the:Basic_properties_of_the_reduced_L2-torsion_function_for_universal_coverings:(Multiplicativity)}
  \emph{Finite coverings}\\
  Let $p \colon X \to Y$ be a $d$-sheeted covering of finite connected $CW$-complexes for
  some natural number $d$.  Let $\phi_Y \in H^1(Y);\IR)$ and $\phi_X \in H^1(X;\IR)$ be
  elements satisfying $p^*\phi_Y = \phi_X$.  Then $\widetilde{X}$ is $\phi_X$-twisted
  $\det$-$L^2$-acyclic if and only if $\widetilde{X}$ is $\phi_X$-twisted $\det$-$L^2$-acyclic and in
  this case we get
\[
\overline{\rho}^{(2)}(\widetilde{X};\phi_X) \doteq d \cdot \overline{\rho}^{(2)}(\widetilde{Y};\phi_Y);
\]

\item  \label{the:Basic_properties_of_the_reduced_L2-torsion_function_for_universal_coverings:scaling}
\emph{Scaling $\phi$}\\ Let $r \in \IR$ be a real number.
Then $\widetilde{X}$ is $\phi$-twisted $\det$-$L^2$-acyclic
if and only if $\widetilde{X}$ is $(r \cdot \phi)$-twisted $\det$-$L^2$-acyclic, and in this case
\[
\overline{\rho}^{(2)}(\widetilde{X};r \cdot \phi)(t) \doteq \overline{\rho}^{(2)}(\widetilde{X};\phi)(t^r).
\]  

\end{enumerate}
\end{theorem}

For
assertion~\eqref{the:Basic_properties_of_the_reduced_L2-torsion_function_for_universal_coverings:Poincare_duality}
of
Theorem~\ref{the:Basic_properties_of_the_reduced_L2-torsion_function_for_universal_coverings}
in dimension $3$ see also~\cite{Dubois-Friedl-Lueck(2015symmetric)}, where a relation is
given already on the representatives of the equivalence class of the $L^2$-torsion
functions of the shape $\rho(t) = t^k \cdot \rho(t^{-1})$ for some $k$ which is modulo $2$ the
Thurston norm of $\phi$.

Notice that in Theorems~\ref{the:S1-actions} and~\ref{the:mapping_tori} we can take $G = \pi_1(X)$ and $\mu = \id$ 
and thus obtain a version for universal coverings.

We record a version of Theorem~\ref{the:fibrations_rho_G-covering} for universal coverings.

\begin{theorem}[Fibrations]\label{the:fibrations_rho}
Let $F \xrightarrow{i}  E \xrightarrow{p} B$ be a fibration of connected finite $CW$-complexes
such that $\pi_1(i) \colon \pi_1(F) \to \pi_1(E)$ is injective. Consider $\phi \in H^1(E,\IR)$.
Let $i^*\phi \in H^1(F;\IR)$ be its pullback with $i$. 
Suppose that $\pi_1(F)$ is residually finite. 
Assume the $K$-theoretic Farrell-Jones Conjecture holds for $\IZ G$
or that the class $\Theta(p)$ occurring in~\cite[Theorem~3.100 on page~166]{Lueck(2002)}
or~\cite[Section~3]{Farrell-Lueck-Steimle(2010)}
vanishes. (The last condition is automatically satisfied if $p$ is a locally trivial fiber bundle.)
Suppose that $\widetilde{F}$ is $i^*\phi$-twisted $\det$-$L^2$-acyclic.

Then $E$ is up to (simple) homotopy type a finite $CW$-complex, 
see~\cite[Section~3]{Farrell-Lueck-Steimle(2010)},
 $\widetilde{E}$ is $\phi$-twisted $\det$-$L^2$-acyclic and we get
\[
\overline{\rho}^{(2)}(\widetilde{E};\phi) \doteq \chi(B) \cdot \overline{\rho}^{(2)}(\widetilde{F};i^*\phi).
\]
\end{theorem}

%%%%%%%%%%%%%%%%%%%%%%%%%%%%%%%%%%%%%%%%%%%%%%%%%%%%%%%%%%%%%%%%%%%%%%%%%%%%%%%%%

\subsection{The degree of the $L^2$-torsion function}
\label{subsec:The_degree_of_the_L2-torsion_function}

\begin{definition}[Degree of an equivalence class of functions $\IR^{> 0} \to \IR$]
  \label{def:Degree_of_an_equivalence_class_of_functions_IR_greater_0_to_IR}
  Let $\overline{\rho}$ be an equivalence class of functions $\IR^{> 0} \to \IR$. Let
  $\rho$ be a representative.  Assume that 
  $\liminf_{t \to 0+}   \frac{\rho(t)}{\ln(t)}  \in \IR$ and 
  $\limsup_{t \to \infty}  \frac{\rho(t)}{\ln(t)}  \in \IR$.

  Then define the \emph{degree at zero} and the \emph{degree at infinity} of $\rho$ to
  be the real numbers 
  \begin{eqnarray*}
    \deg_0(\rho) 
   & := & 
    \liminf_{t \to 0+} \frac{\rho(t)}{\ln(t)};
    \\
    \deg_{\infty}(\rho) 
    & := & 
    \limsup_{t \to \infty} \frac{\rho(t)}{\ln(t)}.
  \end{eqnarray*}
  Define the \emph{degree} of $\overline{\rho}$  to be the real number
  \begin{eqnarray*}
    \deg(\overline{\rho}) 
    & := & 
    \deg_{\infty}(\rho)  -  \deg_0(\rho)  = 
    \limsup_{t \to \infty} \frac{\rho(t)}{\ln(t)} - \liminf_{t \to 0+} \frac{\rho(t)}{\ln(t)}.
    \\
  \end{eqnarray*}
\end{definition}

The answer to the question, whether $\liminf_{t \to 0+} \frac{\rho(t)}{\ln(t)} $ 
and $\limsup_{t \to \infty} \frac{\rho(t)}{\ln(t)}$ belong to $\IR$, and the definition of 
 $\deg(\overline{\rho}) \in \IR $ are independent of the choice of the representative $\rho$.
Suppose that there exist real numbers $C_i$ and $D_i$  for $i = 1,2,3,4$ and a real number $K \ge 1$ such that  we have 
\[
\begin{array}{rccclcl}
C_1 \cdot \ln(t) - D_1  & \le &  \rho(t)  & \le & C_2 \cdot \ln(t) + D_2 & & \text{for}\;  0 <  t  \le K^{-1};
\\
C_3 \cdot \ln(t) - D_3  & \le &  \rho(t)  & \le & C_4 \cdot  \ln(t) + D_3 & & \text{for}\;  K \le  t  < \infty.
\end{array}
\]
Then $C_2 \le \liminf_{t \to 0+} \frac{\rho(t)}{\ln(t)} \le  C_1 $ 
and $C_3 \le \limsup_{t \to \infty } \frac{\rho(t)}{\ln(t)}  \le C_4$ holds
and we get
\begin{eqnarray}
& C_3 - C_1  \le  \deg(\overline{\rho}) \le  C_4 - C_2.& 
\label{degree_and_C_i-s}
\end{eqnarray}
In particular we see that in the situations of 
Theorem~\ref{the:Properties_of_the_twisted_L2-torsion_function}~%
\eqref{the:Properties_of_the_twisted_L2-torsion_function:determinant_class} 
and Theorem~\ref{the:Basic_properties_of_the_reduced_L2-torsion_function_for_universal_coverings}~%
\eqref{the:Basic_properties_of_the_reduced_L2-torsion_function_for_universal_coverings:L2-acyclic_and_res_fin}
the degree of the $L^2$-torsion function is well-defined. It is a homotopy invariant
in the situations of
Theorem~\ref{the:Properties_of_the_twisted_L2-torsion_function}~%
\eqref{the:Basic_properties_of_the_reduced_L2-torsion_function_for_finite_free_G-CW_complexes:(homotopy_invariance)}
and Theorem~\ref{the:Basic_properties_of_the_reduced_L2-torsion_function_for_universal_coverings}~%
\eqref{the:Basic_properties_of_the_reduced_L2-torsion_function_for_universal_coverings:homotopy_invariance}.
In the situation of Theorem~\ref{the:S1-actions} we get
\begin{eqnarray*}
\deg\bigl(\overline{\rho}^{(2)}_G(\overline{X};\phi)\bigr) & = & \chi^{S^1}_{\orb}(X) \cdot k,
\end{eqnarray*} 
and in the situation of Theorem~\ref{the:mapping_tori} we get
\begin{eqnarray*}
\deg\bigl(\overline{\rho}_G(\overline{T_h};\phi)\bigr) & = & \chi(F).
\end{eqnarray*}

%%%%%%%%%%%%%%%%%%%%%%%%%%%%%%%%%%%%%%%%%%%%%%%%%%%%%%%%%%%%%%%%%%%%%%%%%%%%%%%%%

\subsection{The $L^2$-torsion function for $3$-manifolds}
\label{subsec:The_L2-torsion_function_for_3-manifolds}

The situation is pretty fortunate in dimension $3$ for the following reasons. The
fundamental group $\pi_1(M)$ of any compact $3$-manifold $M$ is residually finite.
This was first shown for Haken manifolds~\cite[Theorem~82]{Hempel(1987)}.
Since Perelman's proof of the Geometrization Conjecture, see~\cite{Morgan-Tian(2008)},
the proof now extends to all 3-manifolds.  The
$K$-theoretic Farrell-Jones Conjecture holds for $\pi_1(M)$ for any $3$-manifold $M$,
see~\cite[Corollary~0.3]{Bartels-Farrell-Lueck(2014)}. 

Any prime compact  $3$-manifold $M$ whose fundamental group is infinite 
and whose boundary is empty or toroidal, is $L^2$-acyclic,
see~\cite[Theorem~0.1]{Lott-Lueck(1995)}. Hence for any such $M$ and any $\phi \in H^1(M;\IR)$
we conclude that $M$ is $\phi$-twisted $\det$-$L^2$-acyclic, the equivalence class $\overline{\rho}^{(2)}(\widetilde{M})
\colon \IR^{>0} \to \IR$ is well-defined, see
Definition~\ref{def:The_L2-torsion_function_for_universal_coverings}, and is a homotopy
invariant of $(M,\phi)$ by
Theorem~\ref{the:Basic_properties_of_the_reduced_L2-torsion_function_for_universal_coverings}~%
\eqref{the:Basic_properties_of_the_reduced_L2-torsion_function_for_universal_coverings:homotopy_invariance}.
If $M_1$, $M_2$, $\ldots$, $M_r$ are the pieces in its Jaco-Shalen-Johannson splitting,
then
Theorem~\ref{the:Basic_properties_of_the_reduced_L2-torsion_function_for_universal_coverings}~%
\eqref{the:Basic_properties_of_the_reduced_L2-torsion_function_for_universal_coverings:L2-acyclic_and_res_fin}~%
\eqref{the:Basic_properties_of_the_reduced_L2-torsion_function_for_universal_coverings:sum_formula},
and~\eqref{the:Basic_properties_of_the_reduced_L2-torsion_function_for_universal_coverings:product_formula}
imply
\[
\overline{\rho}^{(2)}(\widetilde{M}) \doteq \sum_{i =1}^r \overline{\rho}^{(2)}(\widetilde{M_i}).
\]
If $M_i$ is Seifert, one can compute $\overline{\rho}^{(2)}(\widetilde{M_i})$ from
Theorem~\ref{the:Basic_properties_of_the_reduced_L2-torsion_function_for_universal_coverings}%
~\eqref{the:Basic_properties_of_the_reduced_L2-torsion_function_for_universal_coverings:(Multiplicativity)}
and Theorem~\ref{the:S1-actions} since an appropriate  finite covering of $M_i$ carries a free $S^1$-action
for which the inclusion of any $S^1$ orbit   induces an injection on fundamental  groups.
More generally, one can compute it for any graph manifold, see~\cite[Theorem~1.2]{Dubois-Friedl-Lueck(2016)}.
If $M_i$ is hyperbolic, one knows at least
$\overline{\rho}^{(2)}(\widetilde{M_i})(1) = \frac{1}{6\cdot \pi} \cdot \vol(M_i)$
by~\cite[Theorem~0.7]{Lueck-Schick(1999)}.  

The $L^2$-torsion function appears and is investigated
for compact connected orientable irreducible $3$-manifolds $M$ with infinite fundamental
group and empty or toroidal boundary 
in~\cite{Dubois-Friedl-Lueck(2016),Dubois-Friedl-Lueck(2015symmetric),Dubois-Friedl-Lueck(2015flavors)}, 
where  a multiplicative version is used and called $L^2$-Alexander torsion.  In Friedl-L\"uck~\cite{Friedl-Lueck(2015l2+Thurston)}
and other forthcoming papers   we will relate the degree of
$\overline{\rho}^{(2)}(\widetilde{M};\phi)$ and more generally of
$\overline{\rho}_G^{(2)}(\overline{M};\phi)$ to the Thurston norm $x_M(\phi)$ of $\phi$
and to the degree of higher order Alexander polynomials in the sense of Cochrane and
Harvey~\cite{Cochran(2004),Harvey(2005)}. Such results about the Thurston norm  are  also proved independently
by Liu~\cite{Liu(2015)}.

%%%%%%%%%%%%%%%%%%%%%%%%%%%%%%%%%%%%%%%%%%%%%%%%%%%%%%%%%%%%%%%%%%%%%%%%%%%%%%%%%
%%%%%%%%%%%%%%%%%Section 9: Continuity of the Fuglede-Kadison determinant      %%%%%%%%%%%%%%%%%%%%
%%%%%%%%%%%%%%%%%%%%%%%%%%%%%%%%%%%%%%%%%%%%%%%%%%%%%%%%%%%%%%%%%%%%%%%%%%%%%%%%%

 \typeout{---------------------   Section 9: Continuity of the Fuglede-Kadison determinant  --------------}

\section{Continuity of the Fuglede-Kadison determinant}
\label{sec:Continuity_of_the_Fuglede-Kadison_determinant}

Many arguments in this and in forthcoming papers would simplify considerably if good
continuity properties of the Fuglede-Kadison determinant would be available. Next we
present a discussion of what properties are true, may possibly hold under good circumstances, but 
have not been proved so far, or are known to be false.

%%%%%%%%%%%%%%%%%%%%%%%%%%%%%%%%%%%%%%%%%%%%%%%%%%%%%%%%%%%%%%%%%%%%%%%%%%%%%%%%%

\subsection{Upper semi-continuity}
\label{subsec:Upper_semi-continuity}

\begin{lemma}\label{lem:upper_semi_continuity}
Let $f_n$ for $n \in \IN$ and $f$ be morphisms $U \to V$ of finite-dimensional Hilbert $\caln(G)$-modules.
Suppose that the sequence $(f_n)_{n \ge 0}$ converges to $f$ in the norm topology.
Then:
\medskip
\begin{enumerate}

\item \label{lem:upper_semi_continuity:Betti_numbers}
\hspace{8mm} \,$\limsup_{n \to \infty}  {\dim}_{\caln(G)}(\ker(f_n)) 
 \le 
{\dim}_{\caln(G)}(\ker(f))$; 
\smallskip
\item \label{lem:upper_semi_continuity:det}
\hspace{17mm} $\limsup_{n \to \infty}  {\det}_{\caln(G)}(f_n) 
\le 
{\det}_{\caln(G)}(f)$.
\end{enumerate}
\smallskip
\end{lemma}

\begin{proof}~\eqref{lem:upper_semi_continuity:Betti_numbers}
  Consider any real number  $\epsilon > 0$. Let $F(f) \colon [0,\infty) \to
  [0,\infty)$ be the spectral density function of $f$ in the sense of~\cite[Definition~2.1
  on page~73]{Lueck(2002)}. This function is monotone increasing and right-continuous
  by~\cite[Lemma~2.3 on page~74]{Lueck(2002)}, Hence there exists $\lambda >0$ satisfying
\[
F(f)(\lambda) \le \dim_{\caln(G)}(\ker(f)) + \epsilon.
\]
Choose $N = N(\epsilon)$ such that $||f-f_n|| \le \lambda/2$ holds for $n \ge N$. Using
the notation and the statement of~\cite[Lemma~2.2~(2) on page~73]{Lueck(2002)}, we get for
$x \in U$ with $E_{\lambda^2}^{f^*f}(x) = 0$ and $x \not= 0$
\[
|f_n(x)|  \ge  |f(x)| - |f_n(x) -f(x)| 
 \ge 
\lambda \cdot |x| - ||f-f_n|| \cdot |x|
\ge 
\lambda \cdot |x| - \lambda/2 \cdot |x|
=  
\lambda/2 \cdot |x|
> 
0.
\]
Hence $E_{\lambda^2}^{f^*f} \colon U \to \im(E_{\lambda^2}^{f^*f})$ is injective on $\ker(f_n)$. This implies
using~\cite[Lemma~2.3 on page~74]{Lueck(2002)} that 
\[\dim_{\caln(G)}(\ker(f_n)) \le \dim_{\caln(G)}(\im(E_{\lambda^2}^{f^+f})) 
= F(f)(\lambda) \le \dim_{\caln(G)}(\ker(f)) + \epsilon
\]
holds for $n \ge N$. Hence we get 
\[
\limsup_{n \to \infty}  {\dim}_{\caln(G)}(\ker(f_n)) \le {\dim}_{\caln(G)}(\ker(f_n)) + \epsilon.
\]
Since this is true for all $\epsilon > 0$, assertion~\eqref{lem:upper_semi_continuity:Betti_numbers} follows.
\\[1mm]~\eqref{lem:upper_semi_continuity:det}
Choose a  real number $a \ge 1$ for which $||f||_{\infty}$ and $||f_n||_{\infty} $ for all $n\ge 0$ are less or equal to $a$.
Fix a real number $K >  0$. 
We conclude from~\cite[Lemma~3.15~(1) on page~129]{Lueck(2002)} that we can find  a real number $\epsilon$
satisfying $0 < \epsilon < 1 \le a$ and 
\begin{eqnarray}
\int_{\epsilon}^a \frac{F(f)(\lambda) - F(f)(0)}{\lambda} \, d\lambda 
& \ge &  
K, \quad 
\label{lem:upper_semi_continuity:(Ia)}
\end{eqnarray}
if ${\det}_{\caln(G)}(f) = 0$ and 
\begin{eqnarray}
\label{lem:upper_semi_continuity:(Ib)}
\int_{\epsilon}^a \frac{F(f)(\lambda) - F(f)(0)}{\lambda} \, d\lambda 
& \ge & 
- K^{-1} + \int_{0+}^a \frac{F(f)(\lambda) - F(f)(0)}{\lambda} \, d\lambda,
\end{eqnarray}
if ${\det}_{\caln(G)}(f) > 0$.
Fix any  real number $\delta$ with $0 < \delta < \epsilon$.
Then there is a natural number $N(\delta)$ satisfying 
\begin{eqnarray} 
||f_n - f||_{\infty} & \le & \delta \quad \text{for}\; n \ge N(\delta).
\label{lem:upper_semi_continuity:(II)}
\end{eqnarray}
Consider $x \in \im(E_{\lambda^2}^{f^*f})$. Then we get $|f(x)| \le \lambda \cdot |x|$ 
from~\cite[Lemma~2.2~(2) on page~73]{Lueck(2002)}.
This implies for $\lambda \ge 0$
\[
|f_n(x)|  
\le
 |f(x)| + |f_n(x) -f(x)|
\le  
\lambda \cdot |x| + ||f_n -f||_{\infty} \cdot |x|
\le  
\lambda \cdot |x| + \delta \cdot |x|
= 
(\lambda + \delta).
\]
We conclude from~\cite[Definition~2.1 on page~73 and Lemma~2.3 on page~74]{Lueck(2002)} that
\[
F(f)(\lambda) = {\dim}_{\caln(G)}\bigl( \im(E_{\lambda^2}^{f^*f})\bigr)
\le F(f_n)(\lambda + \delta).
\]
Now we estimate for $n \ge N(\delta)$

\begin{eqnarray}
\label{lem:upper_semi_continuity:(III)}
& & 
\\
\lefteqn{\int_{\epsilon}^a \frac{F(f)(\lambda) - F(f)(0)}{\lambda} \, d\lambda}
& & 
\nonumber
\\
& \le &
\int_{\epsilon}^a \frac{F(f_n)(\lambda+ \delta) - F(f)(0)}{\lambda} \, d\lambda
\nonumber
\\
& = &
\int_{\epsilon}^a \frac{F(f_n)(\lambda+ \delta) - F(f_n)(0)}{\lambda} \, d\lambda 
+ \int_{\epsilon}^a \frac{F(f_n)(0) - F(f)(0)}{\lambda} \, d\lambda
\nonumber
\\
& = &
\int_{\epsilon}^a \frac{F(f_n)(\lambda+ \delta) - F(f_n)(0)}{\lambda + \delta} 
\cdot \frac{\lambda + \delta}{\lambda} \, d\lambda 
+ \bigl(F(f_n)(0) - F(f)(0)\bigr) \cdot \bigl(\ln(a) - \ln(\epsilon)\bigr)
\nonumber
\\
& \le  &
\int_{\epsilon}^a \frac{F(f_n)(\lambda+ \delta) - F(f_n)(0)}{\lambda + \delta} 
\cdot \frac{\epsilon + \delta}{\epsilon} \, d\lambda 
+ \bigl(F(f_n)(0) - F(f)(0)\bigr) \cdot \bigl(\ln(a) - \ln(\epsilon)\bigr)
\nonumber
\\
& =  &
\frac{\epsilon + \delta}{\epsilon} \cdot \int_{\epsilon+ \delta }^{a+\delta} \frac{F(f_n)(\lambda) - F(f_n)(0)}{\lambda}  \, d\lambda 
+ \bigl(F(f_n)(0) - F(f)(0)\bigr) \cdot \bigl(\ln(a) - \ln(\epsilon)\bigr)
\nonumber
\\
& \le  &
\frac{\epsilon + \delta}{\epsilon} \cdot \int_{0+}^{a+\delta} \frac{F(f_n)(\lambda) - F(f_n)(0)}{\lambda}  \, d\lambda 
+ \bigl(F(f_n)(0) - F(f)(0)\bigr) \cdot \bigl(\ln(a) - \ln(\epsilon)\bigr)
\nonumber
\\
& =& 
\frac{\epsilon + \delta}{\epsilon} \cdot \int_{0+}^{a} \frac{F(f_n)(\lambda) - F(f_n)(0)}{\lambda}  \, d\lambda 
\nonumber
+ \frac{\epsilon + \delta}{\epsilon} \cdot \int_{a}^{a+\delta} \frac{F(f_n)(\lambda) - F(f_n)(0)}{\lambda}  \, d\lambda 
\\
& & 
\quad + \bigl(F(f_n)(0) - F(f)(0)\bigr) \cdot \bigl(\ln(a) - \ln(\epsilon)\bigr)
\nonumber
\\
& \le & 
\frac{\epsilon + \delta}{\epsilon} \cdot \int_{0+}^{a} \frac{F(f_n)(\lambda) - F(f_n)(0)}{\lambda}  \, d\lambda 
\nonumber
\quad + \frac{\epsilon + \delta}{\epsilon} \cdot \int_{a}^{a+\delta} \frac{\dim_{\caln(G)}(U)}{\lambda}  \, d\lambda 
\\
& & 
\quad + \bigl(F(f_n)(0) - F(f)(0)\bigr) \cdot \bigl(\ln(a) - \ln(\epsilon)\bigr)
\nonumber
\\
& =& 
\frac{\epsilon + \delta}{\epsilon} \cdot \int_{0+}^{a} \frac{F(f_n)(\lambda) - F(f_n)(0)}{\lambda}  \, d\lambda 
\nonumber
+ \frac{\epsilon + \delta}{\epsilon} \cdot \dim_{\caln(G)}(U) \cdot \bigl(\ln(a+\delta) - \ln(a)\bigr) 
\nonumber
\\
& & 
\quad + \bigl(F(f_n)(0) - F(f)(0)\bigr) \cdot \bigl(\ln(a) - \ln(\epsilon)\bigr).
\nonumber
\end{eqnarray}

Next we finish the proof in the case $\det_{\caln(G)}(f) = 0$. 
Because of assertion~\eqref{lem:upper_semi_continuity:Betti_numbers} there exists a number $N(\epsilon)$ such that
\begin{eqnarray}
\bigl(F(f_n)(0) - F(f)(0)\bigr) \cdot \bigl(\ln(a) - \ln(\epsilon)\bigr)
& \le & 1
\label{lem:upper_semi_continuity:(IV)} 
\end{eqnarray}
holds for  $n \ge N(\epsilon)$.
Hence we get from~\eqref{lem:upper_semi_continuity:(III)} and~\eqref{lem:upper_semi_continuity:(IV)}
that 
\begin{eqnarray}
\label{lem:upper_semi_continuity:(V)}
\lefteqn{\int_{\epsilon}^a \frac{F(f)(\lambda) - F(f)(0)}{\lambda} \, d\lambda}
& & 
\\
& \le & 
\frac{\epsilon + \delta}{\epsilon} \cdot \int_{0+}^{a} \frac{F(f_n)(\lambda) - F(f_n)(0)}{\lambda}  \, d\lambda 
\nonumber
\\ 
& & \quad + \frac{\epsilon + \delta}{\epsilon} 
\cdot \dim_{\caln(G)}(U) \cdot \bigl(\ln(a+\delta) - \ln(a)\bigr)  + 1
\nonumber\end{eqnarray}
holds for $n \ge \max\{N(\delta),N(\epsilon)\}$. 
We conclude from~\eqref{lem:upper_semi_continuity:(Ia)} and~\eqref{lem:upper_semi_continuity:(V)}
that
\[
\frac{(K -1) \cdot \epsilon}{\epsilon + \delta} - \dim_{\caln(G)}(U) \cdot \bigl(\ln(a+\delta) - \ln(a)\bigr)
\le 
\int_{0+}^{a} \frac{F(f_n)(\lambda) - F(f_n)(0)}{\lambda}  \, d\lambda 
\]
holds for $n \ge \max\{N(\delta),N(\epsilon)\}$. This implies 
\begin{multline*}
\frac{(K -1) \cdot \epsilon}{\epsilon + \delta} - \dim_{\caln(G)}(U) \cdot \bigl(\ln(a+\delta) - \ln(a)\bigr)
\\
\le 
\liminf_{n \to \infty} \int_{0+}^{a} \frac{F(f_n)(\lambda) - F(f_n)(0)}{\lambda}  \, d\lambda. 
\end{multline*}
Taking $\delta \to 0+$ implies
\[
K-1 \le \liminf_{n \to \infty} \int_{0+}^{a} \frac{F(f_n)(\lambda) - F(f_n)(0)}{\lambda}  \, d\lambda.
\]
Hence we obtain  from~\cite[Lemma~3.15~(1) on page~129]{Lueck(2002)}
\begin{eqnarray*}
\lefteqn{\limsup_{n \to \infty} \ln\bigl({\det}_{\caln(G)}(f_n)\bigr) }
& & 
\\
& \le  & 
\limsup_{n \to \infty} \left(- \int_{0+}^{a} \frac{F(f_n)(\lambda) - F(f_n)(0)}{\lambda}  \, d\lambda 
+ \ln(a) \cdot \bigl(F(f_n)(a) - F(f_n)(0)\bigr)\right)
\\
& \le & 
\limsup_{n \to \infty} \left(- \int_{0+}^{a} \frac{F(f_n)(\lambda) - F(f_n)(0)}{\lambda}  \, d\lambda 
+ \ln(a) \cdot \dim_{\caln(G)}(U) \right)
\\
& =  & 
- \left(\liminf_{n \to \infty} \int_{0+}^{a} \frac{F(f_n)(\lambda) - F(f_n)(0)}{\lambda}  \, d\lambda \right) 
+ \ln(a) \cdot \dim_{\caln(G)}(U)
\\
& \le &
-K +1 + {\dim}_{\caln(G)}(U) \cdot \ln(a).
\end{eqnarray*}
Since this holds for all $K > 0$, we conclude
\[
\lim_{n \to \infty} {\det}_{\caln(G)}(f_n) = 0 = {\det}_{\caln(G)}(f).
\]

It remains to treat the case $\dim_{\caln(G)}(f) > 0$.
We get  from~\eqref{lem:upper_semi_continuity:(Ib)} and~\eqref{lem:upper_semi_continuity:(III)} that

\begin{eqnarray*}
& & 
\\
\lefteqn{\int_{0+}^a \frac{F(f)(\lambda) - F(f)(0)}{\lambda} \, d\lambda}
& & 
\\
& \le &
\int_{\epsilon }^a \frac{F(f)(\lambda) - F(f)(0)}{\lambda} \, d\lambda  + K^{-1}
\\
& \le & 
K^{-1} +
\frac{\epsilon + \delta}{\epsilon} \cdot \int_{0+}^{a} \frac{F(f_n)(\lambda) - F(f_n)(0)}{\lambda}  \, d\lambda 
+ \frac{\epsilon + \delta}{\epsilon} \cdot \dim_{\caln(G)}(U) \cdot \bigl(\ln(a+\delta) - \ln(a)\bigr) 
\\
& & 
\quad + \bigl(F(f_n)(0) - F(f)(0)\bigr) \cdot \bigl(\ln(a) - \ln(\epsilon)\bigr),
\end{eqnarray*}
and hence
\begin{eqnarray*}
\lefteqn{\int_{0+}^{a} \frac{F(f_n)(\lambda) - F(f_n)(0)}{\lambda}  \, d\lambda} 
& & 
\\
& \ge &
\frac{\epsilon}{\epsilon + \delta} \cdot \left(\int_{0+}^a \frac{F(f)(\lambda) - F(f)(0)}{\lambda} \, d\lambda 
- K^{-1} - \bigl(F(f_n)(0) - F(f)(0)\bigr) \cdot \bigl(\ln(a) - \ln(\epsilon)\bigr)\right)
\\ 
& & 
\quad - \dim_{\caln(G)}(U) \cdot \bigl(\ln(a+\delta) - \ln(a)\bigr)
\end{eqnarray*}
holds for $n \ge N(\delta)$.
Using $F(f_n)(a) = F(f)(a) = \dim_{\caln(G)}(U)$ we conclude from~\cite[Lemma~3.15~(1) on page~129]{Lueck(2002)}
for $n \ge N(\delta)$

\begin{eqnarray*}
& & 
\\
\lefteqn{\ln\bigl({\det}_{\caln(G)}(f_n)\bigr)}
& & 
\\
& = & 
- \int_{0+}^{a} \frac{F(f_n)(\lambda) - F(f_n)(0)}{\lambda}  \, d\lambda  + \ln(a) \cdot \bigl(F(f_n)(a) - F(f_n)(0)\bigr)
\\ 
& \le &
- \frac{\epsilon}{\epsilon + \delta} \cdot 
\left(\int_{0+}^a \frac{F(f)(\lambda) - F(f)(0)}{\lambda} \, d\lambda - K^{-1} 
- \bigl(F(f_n)(0) - F(f)(0)\bigr) \cdot \bigl(\ln(a) - \ln(\epsilon)\bigr)\right)
\\
& & 
\quad + \dim_{\caln(G)}(U) \cdot \bigl(\ln(a+\delta) - \ln(a)\bigr) +  \ln(a) \cdot \bigl(F(f_n)(a) - F(f_n)(0)\bigr)
\\ 
& = &
\frac{\epsilon}{\epsilon + \delta} \cdot \left(-\int_{0}^a \frac{F(f)(\lambda) - F(f)(0)}{\lambda} \, d\lambda 
+ \ln(a) \cdot \bigl(F(f)(a) - F(f)(0)\bigr)\right)
\\
& & 
\quad - \frac{\epsilon}{\epsilon + \delta} \cdot \ln(a) \cdot \bigl(F(f)(a) - F(f)(0)\bigr) 
\\
& & 
\quad
+ \frac{\epsilon \cdot K^{-1}}{\epsilon + \delta} + \frac{\epsilon}{\epsilon 
+ \delta} \cdot \bigl(F(f_n)(0) - F(f)(0)\bigr) \cdot \bigl(\ln(a) - \ln(\epsilon)\bigr)
\\
& & 
\quad+ \dim_{\caln(G)}(U) \cdot \bigl(\ln(a+\delta) - \ln(a)\bigr) +  \ln(a) \cdot \bigl(F(f_n)(a) - F(f_n)(0)\bigr)
\\ 
& = &
\frac{\epsilon}{\epsilon + \delta} \cdot \ln\bigl({\det}_{\caln(G)}(f)\bigr) 
+  \frac{\epsilon \cdot K^{-1}}{\epsilon + \delta} + \dim_{\caln(G)}(U) \cdot \bigl(\ln(a+\delta) - \ln(a)\bigr)
\\
& & 
\quad 
- \frac{\epsilon}{\epsilon + \delta} \cdot \ln(a) \cdot \bigl(F(f)(a) - F(f)(0)\bigr)  
+ \frac{\epsilon}{\epsilon + \delta} \cdot \bigl(F(f_n)(0) - F(f)(0)\bigr) \cdot \bigl(\ln(a) - \ln(\epsilon)\bigr)
\\
& & 
\quad  +  \ln(a) \cdot \bigl(F(f)(a) - F(f_n)(0)\bigr)
\\ 
& = &
\frac{\epsilon}{\epsilon + \delta} \cdot \ln\bigl({\det}_{\caln(G)}(f)\bigr) 
+  \frac{\epsilon \cdot K^{-1}}{\epsilon + \delta} + \dim_{\caln(G)}(U) \cdot \bigl(\ln(a+\delta) - \ln(a)\bigr)
\\
& & 
\quad 
+ \left(1 - \frac{\epsilon}{\epsilon + \delta}\right) \cdot \left(\ln(a) \cdot (F(f)(a) - F(f)(0)\bigr)\right)  - \ln(a) \cdot \bigl(F(f)(a) - F(f)(0)\bigr)
\\
& & + \frac{\epsilon}{\epsilon + \delta} \cdot \bigl(F(f_n)(0) - F(f)(0)\bigr) \cdot \bigl(\ln(a) - \ln(\epsilon)\bigr)
+  \ln(a) \cdot \bigl(F(f)(a) - F(f_n)(0)\bigr)
\\  
& = &
\frac{\epsilon}{\epsilon + \delta} \cdot \ln\bigl({\det}_{\caln(G)}(f)\bigr) 
+  \frac{\epsilon \cdot K^{-1}}{\epsilon + \delta} + \dim_{\caln(G)}(U) \cdot \bigl(\ln(a+\delta) - \ln(a)\bigr)
\\
& & 
\quad 
+ \left(1 - \frac{\epsilon}{\epsilon + \delta}\right) \cdot \ln(a) \cdot (F(f)(a) - F(f)(0)\bigr)
+ \frac{\epsilon}{\epsilon + \delta}  \cdot \ln(\epsilon)  \cdot \bigl(F(f)(0) - F(f_n)(0)\bigr) 
\\
& & 
\quad  + \left(1 - \frac{\epsilon}{\epsilon + \delta} \right) \cdot \ln(a) \cdot \bigl(F(f)(0) - F(f_n)(0)\bigr)
\\ 
& \le  &
\frac{\epsilon}{\epsilon + \delta} \cdot \ln\bigl({\det}_{\caln(G)}(f)\bigr) 
+  \frac{\epsilon \cdot K^{-1}}{\epsilon + \delta} + \dim_{\caln(G)}(U) \cdot \bigl(\ln(a+\delta) - \ln(a)\bigr)
\\
& & 
\quad 
+ \left(1 - \frac{\epsilon}{\epsilon + \delta}\right) \cdot \ln(a) \cdot  2 \cdot \dim_{\caln(G)}(U) 
+ \frac{\epsilon}{\epsilon + \delta}  \cdot \ln(\epsilon)  \cdot \bigl(F(f)(0) - F(f_n)(0)\bigr) 
\\
& & 
\quad  + \left(1 - \frac{\epsilon}{\epsilon + \delta} \right) \cdot \ln(a) \cdot 2 \cdot \dim_{\caln(G)}(U).
\\ 
& \le  &
\frac{\epsilon}{\epsilon + \delta} \cdot \ln\bigl({\det}_{\caln(G)}(f)\bigr) 
+  \frac{\epsilon \cdot K^{-1}}{\epsilon + \delta} + \dim_{\caln(G)}(U) \cdot \bigl(\ln(a+\delta) - \ln(a)\bigr)
\\
& & 
\quad 
+ \left(1 - \frac{\epsilon}{\epsilon + \delta}\right) \cdot \ln(a) \cdot  4 \cdot \dim_{\caln(G)}(U) 
+   \frac{\epsilon}{\epsilon + \delta}  \cdot  \ln(\epsilon)  \cdot \bigl(F(f)(0) - F(f_n)(0)\bigr).
\end{eqnarray*}

This implies

\begin{eqnarray*}
\lefteqn{\limsup_{n \to \infty} \ln\bigl({\det}_{\caln(G)}(f_n)\bigr)}
& & 
\\
& \le &
\frac{\epsilon}{\epsilon + \delta} \cdot \ln\bigl({\det}_{\caln(G)}(f)\bigr) 
+  \frac{\epsilon \cdot K^{-1}}{\epsilon + \delta} + \dim_{\caln(G)}(U) \cdot \bigl(\ln(a+\delta) - \ln(a)\bigr)
\\
& & 
\quad 
+ \left(1 - \frac{\epsilon}{\epsilon + \delta}\right) \cdot \ln(a) \cdot  4 \cdot \dim_{\caln(G)}(U) 
+  \frac{\epsilon}{\epsilon + \delta}  \cdot  \limsup_{n \to \infty} \ln(\epsilon)  \cdot \bigl(F(f)(0) - F(f_n)(0)\bigr).
\end{eqnarray*}
Since $\limsup_{n \to \infty} F(f_n)(0) \le F(f)(0)$ holds by assertion~\eqref{lem:upper_semi_continuity:Betti_numbers}
and $\ln(\epsilon) \le 0$, we get 
\[
\limsup_{n \to \infty} \ln(\epsilon)  \cdot \bigl(F(f)(0) - F(f_n)(0)\bigr) \le 0.
\]
Hence we get 
\begin{eqnarray*}
\lefteqn{\limsup_{n \to \infty} \ln\bigl({\det}_{\caln(G)}(f_n)\bigr)}
& & 
\\
& \le &
\frac{\epsilon}{\epsilon + \delta} \cdot \ln\bigl({\det}_{\caln(G)}(f)\bigr) 
+  \frac{\epsilon \cdot K^{-1}}{\epsilon + \delta} + \dim_{\caln(G)}(U) \cdot \bigl(\ln(a+\delta) - \ln(a)\bigr)
\\
& & 
\quad 
+ \left(1 - \frac{\epsilon}{\epsilon + \delta}\right) \cdot \ln(a) \cdot  4 \cdot \dim_{\caln(G)}(U).
\end{eqnarray*}
Taking the limit for $\delta \to 0+$ implies
\[
\limsup_{n \to \infty} \ln\bigl({\det}_{\caln(G)}(f_n)\bigr) \le \ln\bigl({\det}_{\caln(G)}(f)\bigr) + K^{-1}.
\]
Since $K > 0$ was arbitrary, we conclude
\[
\limsup_{n \to \infty} \ln\bigl({\det}_{\caln(G)}(f_n)\bigr) \le \ln\bigl({\det}_{\caln(G)}(f)\bigr).
\]
This finishes the proof of Lemma~\ref{lem:upper_semi_continuity}.
\end{proof}

Define the \emph{regular Fuglede-Kadison determinant} of a morphism  $f \colon U \to U$ 
of a finite-dimensional Hilbert $\caln(G)$-modules
\begin{eqnarray} \quad \quad \quad 
{\det}_{\caln(G)}^r(f) & := & 
\begin{cases} {\det}_{\caln(\Gamma)}(f) 
     & 
    \text{if} \; f \;\text{is injective and of determinant class;}
      \\ 
      0 & 
       \text{otherwise,}
  \end{cases}
\label{regular_Fuglede-Kadison_determinant}
\end{eqnarray}

\begin{remark}[Fuglede-Kadison determinant versus regular Fuglede-Kadison determinant]
\label{rem:versus}
One should not confuse the Fuglede-Kadison determinant $\det_{\caln(G)}(f)$ and the regular
Fuglede-Kadison determinant $\det_{\caln(G)}^r(f)$ of a morphism  $f \colon U \to V$
of  finite-dimensional Hilbert $\caln(G)$-modules. They are equal  
if and only if one of the following statements holds: i.) $f$ is injective or
ii.) $f$ is not of determinant class which is by definition the same as $\det_{\caln(G)}(f) = 0$.
For instance, for the zero operator $0
\colon U \to U$ we have $\det^r_{\caln(G)}(0) = 0$ and $\det_{\caln(G)}(0) = 1$. In
particular we see that
\begin{eqnarray*}
0 = \lim_{\epsilon \to 0+} {\det}^r_{\caln(G)}(\epsilon \cdot \id_U) 
& = &
{\det}^r_{\caln(G)}\left(\lim_{\epsilon \to 0+} \epsilon \cdot \id_U\right) = 0;
\\
0 = \lim_{\epsilon \to 0+} {\det}_{\caln(G)}(\epsilon \cdot \id_U) 
& \not= & 
{\det}_{\caln(G)}\left(\lim_{\epsilon \to 0+} \epsilon \cdot \id_U\right) = 1.
\end{eqnarray*}
In general the regular Fuglede-Kadison determinant has nicer properties, 
see for instance 
Lemma~\ref{lem:upper_semi_continuity_det_regular}~\eqref{lem:upper_semi_continuity_det_regular:epsilon},
but there are
many instances, where one needs to consider the Fuglede-Kadison determinant instead. The
definition of $L^2$-torsion requires the Fuglede-Kadison determinant if the underlying
$G$-$CW$-complex is not $L^2$-acyclic.  If it is $L^2$-acyclic, the definition in terms of
the differentials still requires the use of the Fuglede-Kadison determinant and is sometimes more
useful than the one in terms of the Laplace operators. 

In the situation of
Theorem~\ref{the:Twisted_Approximation_inequality} it is in general unrealistic to assume
that $\dim_{\caln(Q_i)}\bigl(\ker(\Lambda^{Q_i} \circ \eta_{\phi_i^* V}(r_{A[i]}))\bigr) =
0$ holds for all $i \ge 0$ and therefore one needs to work with the Fuglede-Kadison
determinant in order to get a non-trivial statement. A typical example is when all the indices $[G:G_i]$ are finite.
Notice that some of our proofs become rather involved since we need the
condition~\eqref{the:Twisted_Approximation_inequality:det_bound} in
Theorem~\ref{the:Twisted_Approximation_inequality}.

However, there are  special cases, where
$\dim_{\caln(Q_i)}\bigl(\ker(\Lambda^{Q_i} \circ \eta_{\phi_i^* V}(r_{A[i]}))\bigr)$
vanishes for all $i \ge 0$,  and one can replace the Fuglede-Kadison determinant by the
regular one without loosing information, and, moreover, can also drop
condition~\eqref{the:Twisted_Approximation_inequality:det_bound} in
Theorem~\ref{the:Twisted_Approximation_inequality},
see~\cite[Lemma~3.2]{Liu(2015)}.

\end{remark}

In the case of the regular Fuglede-Kadison determinant the proof of upper semi-continuity
is simpler and  can be found in Liu~\cite[Lemma~3.1]{Liu(2015)}.  For the reader's convenience we give the proof
how it follows from Lemma~\ref{lem:upper_semi_continuity}.

\begin{lemma}\label{lem:upper_semi_continuity_det_regular}
Let $f_n$ for $n \in \IN$ and $f$ be morphisms $U \to V$ of finite-dimensional Hilbert $\caln(G)$-modules.
Suppose that the sequence $(f_n)_{n \ge 0}$ converges to $f$ in the norm topology.
Then:

\begin{enumerate}

\item \label{lem:upper_semi_continuity_det_regular:epsilon}
If  $ U = V$ and $f$ is positive, then
\[
\lim_{\epsilon \to 0+} {\det}^r_{\caln(G)}(f+ \epsilon \cdot \id_U) = {\det}^r_{\caln(G)}(f);
\]

\item \label{lem:upper_semi_continuity_det_regular:limsup}
We have 
\[\limsup_{n \to \infty}  {\det}_{\caln(G)}^r(f_n)  
\le  
{\det}^r_{\caln(G)}(f).
\]
\end{enumerate}
\end{lemma}
\begin{proof}~\eqref{lem:upper_semi_continuity_det_regular:epsilon}
If $f$ is injective, this follows from~\cite[Lemma~3.15~(5) on page~129]{Lueck(2002)}.

Suppose that $f$ is not injective. We conclude from~\cite[Lemma~3.15~(7) on page~130]{Lueck(2002)} and
Lemma~\ref{lem:det_estimate_in_terms_of_norm}
\begin{eqnarray*}
{\det}_{\caln(G)}(f + \epsilon \cdot \id_U)
& = &
{\det}_{\caln(G)}\left(\bigl(f|_{\ker(f)^{\perp}} + \epsilon \cdot \id_{\ker(f)^{\perp}}\bigr) \oplus \; \epsilon \cdot \id_{\ker(f)}\right)
\\
& = & 
{\det}_{\caln(G)}\bigl(f|_{\ker(f)^{\perp}} + \epsilon \cdot \id_{\ker(f)^{\perp}} \bigr)
\cdot 
{\det}_{\caln(G)}\bigl(\epsilon \cdot \id_{\ker(f)}\bigr)
\\
& \le &
\left(||f|_{\ker(f)^{\perp}}||_{\infty} + \epsilon\right)^{\dim_{\caln(G)}(\ker(f)^{\perp})} \cdot \epsilon^{\dim_{\caln(G)}(\ker(f))}.
\end{eqnarray*}
Since $\dim_{\caln(G)}(\ker(f)) > 0$,  this implies
\[
\lim_{\epsilon \to 0+}  {\det}_{\caln(G)}(f + \epsilon \cdot \id_U)  = 0 =  {\det}_{\caln(G)}^r(f).
\]
This finishes the proof of
assertion~\eqref{lem:upper_semi_continuity_det_regular:epsilon}.
\\[1mm]~\eqref{lem:upper_semi_continuity_det_regular:limsup} It suffices to treat the case
$\limsup_{n \to \infty} {\det}_{\caln(G)}^r(f_n) > 0$. Then there exists a number $N$ such
that $f_n$ is injective and $\det_{\caln(G)}(f_n) = \det_{\caln(G)}^r(f_n) > 0$ for $n \ge N$.  Fix $\epsilon > 0$. 
We conclude from~\cite[Lemma~3.15~(4) and (6) on page~129]{Lueck(2002)} for $n \ge N$
\[
{\det}_{\caln(G)}^r(f_n) = \sqrt{{\det}_{\caln(G)}(f_n^*f_n)} \le  \sqrt{{\det}_{\caln(G)}(f_n^*f_n + \epsilon \cdot \id_U)}.
\] 
This together with Lemma~\ref{lem:upper_semi_continuity} implies for all $\epsilon > 0$
\[
\limsup_{n \to \infty} {\det}_{\caln(G)}^r(f_n) \le  \sqrt{{\det}_{\caln(G)}(f^*f + \epsilon \cdot \id_U)}.
\] 
Assertion~\eqref{lem:upper_semi_continuity_det_regular:epsilon} 
and~\cite[Lemma~3.15~(4) on page~129]{Lueck(2002)} imply
\[
\lim_{\epsilon \to 0+}
\sqrt{{\det}_{\caln(G)}(f^*f + \epsilon \cdot \id_U)} 
= \sqrt{{\det}^r_{\caln(G)}(f^*f)} = {\det}^r_{\caln(G)}(f).
\]
This finishes the proof of Lemma~\ref{lem:upper_semi_continuity_det_regular}.
\end{proof}

%%%%%%%%%%%%%%%%%%%%%%%%%%%%%%%%%%%%%%%%%%%%%%%%%%%%%%%%%%%%%%%%%%%%%%%%%%%%%%%%%

\subsection{On the continuity of the regular determinant for matrices over group rings}
\label{subsec:On_the_continuity_of_the_regular_determinant_for_matrices_over_group_rings}

\begin{question}[Continuity of the regular determinant]
\label{que:Continuity_of_the_regular_determinant}
Let $G$ be a group for which there exists a natural number $d$, such that the order of any finite subgroup
$H \subseteq G$ is bounded by $d$, e.g., $G$ is torsionfree. Let $S \subseteq G$ be a finite subset. 
Put $\IC[n,S] := \{A \in M_{n,n}(\IC G) \mid \supp_G(A) \subseteq S\}$ and equip it 
with the standard topology  coming from the structure of a finite-dimensional complex vector space.

\begin{enumerate}

\item \label{que:Continuity_of_the_regular_determinant:det} Is then the function given by the  regular Fuglede-Kadison determinant
\[
   \IC[n,S] \to [0,\infty], \quad
     A  \mapsto  {\det}_{\caln(G)}^r(\Lambda^{G}(r_A)) 
\]
  continuous?

\item \label{que:Continuity_of_the_regular_determinant:L2-acyclic} Consider $A \in
  \IC[S]$ such that $\Lambda^{\Gamma}(r_A) \colon L^2(G)^n \to L^2(G)^n$ is a weak
  isomorphism of determinant class.  Does there exist an open neighbourhood $U$ of $A$ in
  $\IC[S]$ such that for every element $B \in U$ also 
  $\Lambda^{\Gamma}(r_B)\colon L^2(G)^n \to L^2(G)^n $ is a weak
  isomorphism of determinant class?
\end{enumerate}
\end{question}

Suppose that $G$ is a finitely generated abelian group. Then the answer to 
Question~\ref{que:Continuity_of_the_regular_determinant}
is positive. This is proven for 
Question~\ref{que:Continuity_of_the_regular_determinant}~\eqref{que:Continuity_of_the_regular_determinant:det} 
in~\cite[Corollary~2.7]{Dubois-Friedl-Lueck(2016)}.
Moreover, the answer  for Question~\ref{que:Continuity_of_the_regular_determinant}~%
\eqref{que:Continuity_of_the_regular_determinant:L2-acyclic} is positive 
by Lemma~\ref{lem:K_1_and_det(C[Zd]},~\cite[Theorem~3.14~(5) on page~128]{Lueck(2002)},
 and~\cite[Theorem~1.2]{Lueck(2015spectral)}
since the determinant  over $\IC[\IZ^d]$ yields  for each finite subset $S \subseteq \IZ^d$ 
a continuous function $\IC[n,S] \to \IC[1,S^d]$ for 
$S^d = \{s_1 \cdot s_2 \cdot \, \cdots \, \cdot s_d \mid s_1, s_2, \ldots ,s_d \in S\}$.

\begin{remark}[A bound on the order of finite subgroups is necessary]
  \label{rem:conditions_for_approximation_are_necessary}
  If one drops in Question~\ref{que:Continuity_of_the_regular_determinant} the condition
  that there is a bound on the order of finite subgroups,
  then there are counterexamples as explained next.

The following example comes from a discussion with Thomas Schick. 
Let $L = \left(\bigoplus_{n \in \IZ} \IZ/2\right) \rtimes \IZ$ be the so called
lamplighter group which is the semi-direct product with respect to the shift automorphism
of $\bigoplus_{n \in \IZ} \IZ/2$. Then there exists an element $x \in \IZ L$ such that for
a real number $r$ the operator $\Lambda^L(r_{x -r}) \colon L^2(L) \to L^2(L)$ is injective
if and only if $r$ does not belong to the subset $T:= \{2 \cdot \cos(p/q \cdot \pi) \mid p,q \in
\IZ, 1 \le p \le q-1, p \; \text{coprime to }\; q\}$, see~\cite{Dicks-Schick(2002),Grigorchuk-Zuk(2001)}. 
This example was used to produce a counterexample to a question of Atiyah
in~\cite{Grigorchuk-Linnell-Schick-Zuk(2000)}.

Define the finite subset $S \subseteq L$ to be the union of the support of $x$ and
$\{1\}$.  Then the support of $(x-r)$ belongs to $S$ for all real numbers $r$.  We get
from~\cite{Varona(2006)} that $\IQ \cap T = \{0\}$.   For every $r \in \IQ$ 
with $r \not= 0$ the operator $r_{x -r}$ is a weak isomorphism.
For every $r \in \IQ$ we have
$\det_{\caln(L)}(\Lambda^G(r_{x -r})) > 0$ since $L$ satisfies the Determinant
Conjecture, see~\cite[Conjecture~13.2 on page~454 and Theorem~13.3 on page~454]{Lueck(2002)}. 
Hence  the function appearing
in Question~\ref{que:Continuity_of_the_regular_determinant}~\eqref{que:Continuity_of_the_regular_determinant:det} 
has values greater than $0$ on $\IQ \setminus \{0\}$.
Obviously $T$ is a dense subset of $[-2,2]$.  If 
Question~\ref{que:Continuity_of_the_regular_determinant}~\eqref{que:Continuity_of_the_regular_determinant:det}
would have a positive answer for $L$ and $S$, the function appearing
in Question~\ref{que:Continuity_of_the_regular_determinant}~\eqref{que:Continuity_of_the_regular_determinant:det} 
is  constant zero on $T$ and hence on $[-2,2]$, a contradiction.   If 
Question~\ref{que:Continuity_of_the_regular_determinant}~\eqref{que:Continuity_of_the_regular_determinant:L2-acyclic}
would have a positive answer, the set of elements $r \in \IR$ for which
$\Lambda^G(r_{x-r})\colon L^2(G) \to L^2(G)$ is a weak isomorphism with
$\det_{\caln(L)}(\Lambda^G(r_{x -r})) \not= 0$ would be an open subset of $\IR$ containing $\IQ \setminus \{0\}$, 
a contradiction, since its complement contains the in $[-2,2]$ dense subset $T$.  Hence the answers to
Question~\ref{que:Continuity_of_the_regular_determinant}~\eqref{que:Continuity_of_the_regular_determinant:det}
and~\eqref{que:Continuity_of_the_regular_determinant:L2-acyclic}~are negative.

Notice that the lamplighter group is finitely generated and residually finite. So one
should compare the counterexample to continuity above with the statement about continuity
of Theorem~\ref{the:Continuity_of_the_L2_torsion_function}.  Both are dealing with a
$1$-parameter families of elements in the group ring. This show that the $\phi$-twisting
is a very special twisting.
\end{remark}

\begin{remark}[Discarding the finite subset $S$]
  \label{rem:Discarding_the_finite_subset_S}
  If we take $G = \IZ$ but discard the finite set $S$, we have several choices for a
  topology on $M_{n,n}(\IC G)$.  We get  from
  Lemma~\ref{lem:FK-det_and_Cast-Norm_and_weak-isos} a counterexample to
  Question~\ref{que:Continuity_of_the_regular_determinant}~\eqref{que:Continuity_of_the_regular_determinant:det}
  already for $G = \IZ$ if we take the topology with
  respect to $C^*_r(\IZ)$.
\end{remark}

%%%%%%%%%%%%%%%%%%%%%%%%%%%%%%%%%%%%%%%%%%%%%%%%%%%%%%%%%%%%%%%%%%%%%%%%%%%%%%%%%%%

\subsection{The Fuglede-Kadison determinant is continuous with respect to the norm topology
for isomorphisms}
\label{subsec:The_Fuglede-Kadison_determinant_continuous_norm_topology_for_isomorphisms}

\begin{lemma} \label{lem:continuity_of_det_on_GL_n(caln(G))}
The map
\[
GL_n(\caln(G)) \to \IR, \quad A \mapsto {\det}_{\caln(G)}\bigl(r_A \colon L^2(G)^n \to L^2(G)^n\bigr)
\]
is continuous with respect to the norm topology on $GL_n(\caln(G))$.
\end{lemma}
\begin{proof}
See~\cite[Theorem~1.10~(d)]{Carey-Farber-Mathai(1997)} or~\cite[Theorem~1 (3)]{Fuglede-Kadison(1952)}.
\end{proof}

We also mention the following result~\cite[Theorem~1.10~(e)]{Carey-Farber-Mathai(1997)}

\begin{theorem}\label{the:CFM}
Let $t > 0$ be a real number and let $A \colon [0,t] \to Gl_n(\caln(G))$ be a continuous piecewise smooth map.
Then
\[
{\det}_{\caln(G}(A(t)) = \det(A(0)) \cdot 
\exp\left(\int_0^t \Real \left(\tr_{\caln(G)}\left(A(s)^{-1} \cdot \left.\frac{dA}{ds}\right|_s \,ds\right)\right)\right).
\]
\end{theorem}

%%%%%%%%%%%%%%%%%%%%%%%%%%%%%%%%%%%%%%%%%%%%%%%%%%%%%%%%%%%%%%%%%%%%%%%%%%%%%%%%%%%

\subsection{The Fuglede-Kadison determinant is not continuous with respect to the norm topology
for weak isomorphisms}
\label{subsec:The_Fuglede-Kadison_determinant_is_not_continuous_Cast-norm_weak_isomorphisms}

\begin{lemma} \label{lem:FK-det_and_Cast-Norm_and_weak-isos}
The Fuglede-Kadison determinant is not continuous on $\IQ[\IZ] - \{0\}$ with  respect to
the $C^*_r$-norm.
\end{lemma}
\begin{proof}Consider
$p(z) = (z-1)(z^{-1}-1) \in \IZ[\IZ] = \IZ[z,z^{-1}]$. If we define
\[
f \colon [0,1] \to [0,4], \quad t \mapsto 2 -2 \cos(2\pi t)
\]
then $p(\exp(2 \pi i t)) = f(t)$ for $t \in [0,1]$. For $n \in \IZ, n \ge 4$ define  the function
$f_n \colon [0,1] \to [0,4] $ by
\[
f_n(t) = 
\begin{cases}
\exp(-n^2) & 0 \le t \le n^{-1};\\
\exp(-n^2) \cdot (2- nt) + f(2n^{-1}) \cdot  (nt - 1)& n^{-1} \le t \le 2n^{-1}; 
\\
f(t) & 2n^{-1} \le t \le 1- n^{-1};
\\
f(1 - n^{-1}) \cdot (n - nt) + \exp(-n^2) \cdot  (nt  - n + 1) & 1- n^{-1} \le t \le 1.
\end{cases}
\]
Notice that this function is constant on $[0,n^{-1}]$, affine on $[n^{-1}, 2n^{-1}]$ and $[1- n^{-1},1]$,
agrees with $f$ on $[2n^{-1}, 1- n^{-1}]$ and satisfies $f_n(0) = f_n(1)$. Hence we obtain a continuous function
\[g_n \colon S^1 \to [0,4], \quad \exp(2\pi i t)  \mapsto f_n(t).
\]
Obviously we get
\[
|f_n(t) - f(t)|
\le 
\begin{cases}
\exp(-n^2) + 2 - 2 \cos(2\pi i n^{-1}) &  0 \le t \le n^{-1};\\
\exp(-n^2) + 2 - 2 \cos(2\pi i 2n^{-1}) & n^{-1} \le t \le 2n^{-1}; \\
0 & 2n^{-1} \le t \le 1- n^{-1};
\\
\exp(-n^2) + 2 - 2 \cos(2\pi i (1 - n^{-1})) & 1- n^{-1} \le t \le 1.
\end{cases}
\]
Hence we obtain for all $t \in [0,1]$
\[
|f_n(t) - f(t)| \le \exp(-n^2) + 2 - 2 \cos(2\pi i 2n^{-1}). 
\]
Since $\lim_{n \to \infty} \exp(-n^2) + 2 - 2 \cos(2\pi i 2n^{-1}) = 0$, we conclude in
the real $C^*$-algebra $C(S^1)$ of continuous functions $S^1 \to \IR$ equipped with the
supremum norm, which agrees with $C^*_r(\IZ)$,
\[ 
\lim_{n \to \infty} g_n = p.
\]
Since $\IR[z,z^{-1}]$ and hence also $\IQ[z,z^{-1}]$ is dense in $C(S^1)$, we can choose
for every $n$ an element $p_n(z) \in \IQ[z,z^{-1}]$ such that for all $z \in S^1$ we have
\[
|p_n(z) - g_n(z)| \le \frac{\exp(-n^2)}{2}.
\]
Since $\lim_{n \to \infty} \frac{\exp(-n^2)}{2} = 0$, we get in $C(S^1)$
\[
\lim_{n \to \infty} p_n = p.
\]
We get 
\[
\begin{array}{lcccll}
0  &  <   & |p_n(\exp(2\pi i t))| & \le & \frac{3 \exp(-n^2)}{2} & \quad \quad  0 \le t \le n^{-1};
\\
0  & \le  & |p_n(\exp(2\pi i t))| &  \le & 3                                   & \quad \quad n^{-1} \le t \le 1.
\end{array}
\]
This implies
\[
\begin{array}{ccccc}
-\infty   &  <   & \int_0^{ n^{-1}} \ln(|p_n(\exp(2\pi i t)|) dt & \le  & -n + \frac{\ln(3/2)}{n};

\\
-\infty  & \le  & \int_{n^{-1}}^1 \ln(|p_n(\exp(2\pi i t)|) dt &  \le & \ln(3).   
\end{array}
\]
We conclude
\[
\int_{S^1} \ln(|p_n(z)|) dz = \int_0^1\ln(|p_n(\exp(2\pi i t)|) dt \le - n + \frac{\ln(3/2)}{n}  + \ln(3).
\] 
This implies using~\cite[Example~3.13 on page~128]{Lueck(2002)}
\[
{\det}_{\caln(\IZ)}(p_n) \le \exp\bigl(-n + \frac{\ln(3/2)}{n} + \ln(3)\bigr),
\]
and hence
\[
\lim_{n \to \infty} {\det}_{\caln(\IZ)}(p_n) = 0.
\]
On the other hand
\[
{\det}_{\caln(\IZ)}(p) = 1.
\]
\end{proof}

%%%%%%%%%%%%%%%%%%%%%%%%%%%%%%%%%%%%%%%%%%%%%%%%%%%%%%%%%%%%%%%%%%%%%%%%%%%%%%%%%
%%%%%%%%%%%%%%%%%%%%%%%%%%%%%%%%Section 10: Some open problems %%%%%%%%%%%%%%%%%%%%%%%%%%%%%%
%%%%%%%%%%%%%%%%%%%%%%%%%%%%%%%%%%%%%%%%%%%%%%%%%%%%%%%%%%%%%%%%%%%%%%%%%%%%%%%%%

 \typeout{-----------------------------   Section 10: Some open problems ---------------------------}

\section{Some open problems}
\label{sec:Some_open_problems}

Finally we discuss some further  open problems concerning the $L^2$-torsion function
besides the Questions~\ref{que:L2-Betti_number_and_twisting},~\ref{que:Novikov-Shubin-invariants_and_twisting},~%
\ref{que:determinant_class} and~\ref{que:Continuity_of_the_regular_determinant}.

%%%%%%%%%%%%%%%%%%%%%%%%%%%%%%%%%%%%%%%%%%%%%%%%%%%%%%%%%%%%%%%%%%%%%%%%%%%%%%%%%

\subsection{Properties of the $L^2$-torsion function}
\label{subsec:Properties_of_the_L2-torsion_function}

\begin{problem}[Continuity of the $L^2$-torsion function]
\label{prob:continuity_of_the_L2-torsion_function}
Under which conditions on $G$, $X$ and $\phi \colon G \to \IR$
is the $L^2$-torsion function $\overline{\rho}_G(\overline{X};\phi)$ continuous?
\end{problem}

A partial positive result to this question is given by the following result which is 
a direct  consequence of  Liu~\cite[Lemma~4.2 and Theorem~5.1]{Liu(2015)}.

\begin{theorem}[Continuity of the $L^2$-torsion function]
\label{the:Continuity_of_the_L2_torsion_function}
Let $X$ be a  finite free $G$-$CW$-complex. Let $[B_X]$ be  a base refinement for $X$.
Suppose that $G$ is finitely generated and residually finite and that $X$ is $L^2$-acyclic. 
Consider a group homomorphism $\phi \colon G \to \IZ^d$. 

Then  the torsion function
\[
\rho^{(2)}(X;\phi,[B_X]) \colon \IR^{>0} \to \IR
\]
is continuous.
\end{theorem}
\begin{proof}
We know already  that $X$ is $\phi$-twisted $\det$-$L^2$-acyclic, see
Theorem~\ref{the:Basic_properties_of_the_reduced_L2-torsion_function_for_finite_free_G-CW_complexes}~%
\eqref{the:Basic_properties_of_the_reduced_L2-torsion_function_for_finite_free_G-CW_complexes:L2-acyclic_and_res_fin}.

Fix a representative $B_X$ of $[B_X]$. Since each $C_n(X)$ has now an ordered $\IZ G$-basis, 
we can form the adjoint $c_n^* \colon C_{n-1}(X) \to C_n(X)$ which is a $\IZ G$-linear map. 
It corresponds on the level of matrices to replacing a matrix $(a_{i,j})$
by the matrix $(\overline{a_{j,i}})$, where 
$\overline{\sum_{g \in G} \lambda_g \cdot g} := \sum_{g \in G} \lambda_g \cdot g^{-1}$
 for $\sum_{g \in G} \lambda_g \cdot g \in \IZ G$.  Define the combinatorial Laplace operator 
to be the $\IZ G$-map
\[
\Delta_n := c_{n+1} \circ c_{n+1}^* + c_n^* \circ c_n \colon C_n(X) \to C_n(X)
\]
Then the $L^2$-torsion function is defined by 
\[
\rho^{(2)}(X;\phi,[B_X]) := \sum_{n \ge } (-1)^n \cdot n \cdot 
\ln\bigl({\det}_{\caln(G)}\bigl( \Lambda \circ \eta_{\phi^*\IC_t}(\Delta_n)\bigr)\bigr).
\]
Hence it suffices to show for each $n$ that 
${\det}_{\caln(G)}\bigl( \Lambda \circ \eta_{\phi^*\IC_t}(\Delta_n)\bigr)$ depends continuously on $t$.  This 
follows from a result of Liu~\cite[Lemma~4.2 and Theorem~5.1]{Liu(2015)} applied to the case, where
$\gamma = \id_G$ and $A$ is the matrix describing $\Delta_n \colon C_n(X) \to C_n(X)$ with
respect to the basis $B_X$.
\end{proof}

More generally, one may ask 

\begin{problem}[Continuity of the $L^2$-torsion function in $(t,\phi)$]
\label{prob:continuity_of_the_L2-torsion_function_in(t,phi)}
Under which conditions on $G$ and $X$
is the map 
\[
H^1(X,\IR) \times \IR^{>0} \to \IR, \quad (\phi,t) \mapsto \rho^{(2)}(X;\phi,[B_X]).
\]
is continuous? 
\end{problem}

Problem~\ref{prob:continuity_of_the_L2-torsion_function_in(t,phi)} is a special case of
Problem~\ref{pro:Continuity_of_rho[n]_C_ast} which in turn is a special case of
Question~\ref{que:Continuity_of_the_regular_determinant}.

A function $f \colon \IR^{>0} \to [0,\infty)$ is called \emph{multiplicatively convex} if
$f(t_0^{\lambda} \cdot t_1^{1-\lambda}) \le f(t_0)^{\lambda} \cdot f(t_1)^{1-\lambda}$ 
holds for all $t_0,t_1 \in \IR^{>0}$ and $\lambda \in (0,1)$. If it takes values in $(0,\infty)$,
this is equivalent to the condition that $\ln \circ f \circ \exp$ is a convex function $\IR \to \IR$.
Notice that a multiplicatively convex function is automatically continuous.

\begin{problem}[Continuity, Convexity and Monotonicity of the $L^2$-torsion function]
  \label{prob:Continuity,Convexity_and_Monotonicity}
  Under which conditions on $G$ is the function 
  $\det_{\caln(G)}\bigl(\Lambda \circ  \eta_{\phi^* \IC_t}(r_A)\bigr)$ for all square matrices $A$ over $\IZ G$ 
  and $\phi  \colon G \to \IZ$ multiplicatively convex, convex, continuous, or  monotone increasing?
\end{problem}

If $G$ is finitely generated and residually finite, then Liu~\cite[Lemma~4.2 and Theorem~5.1]{Liu(2015)}
shows that $\det_{\caln(G)}\bigl(\Lambda \circ \eta_{\phi^* \IC_t}(r_A)\bigr)$  is multiplicatively convex
and hence also continuous.
Problem~\ref{prob:Continuity,Convexity_and_Monotonicity} is related to the more general 
Question~\ref{que:Continuity_of_the_regular_determinant}.

The following  question is related to the Approximation Conjecture for
Fuglede-Kadison determinants, see~\cite[Conjecture~14.1]{Lueck(2016_l2approx)}.
A positive answer would make some of the proofs easier and imply a positive answer to 
Problem~\ref{prob:Continuity,Convexity_and_Monotonicity}
provided that $G$ is finitely generated residually finite.

\begin{problem}[Stronger version of Theorem~\ref{the:Twisted_Approximation_inequality}~%
\eqref{the:Twisted_Approximation_inequality:Q_i-condition}]
\label{pro:Equality_instead_of_inequality}
When is it true that in
Theorem~\ref{the:Twisted_Approximation_inequality}~\eqref{the:Twisted_Approximation_inequality:Q_i-condition}
the inequality
\begin{eqnarray*} 
{\det}_{\caln(G)}\bigl(\Lambda^G \circ \eta_{\phi^* V}(r_A)\bigr) & \ge
  & \limsup_{i \to \infty} {\det}_{\caln(Q_i)}\bigl(\Lambda^{Q_i} \circ \eta_{\phi_i^*
    V}(r_{A[i]})\bigr).
\end{eqnarray*}
is an equality? When do we get an equality if we replace $\limsup$ by $\lim$?
\end{problem}

\begin{problem}[Locally constant at $0$ and $\infty$]
\label{prob:Locally_constant_at_0_and_infty}
Under which conditions on $G$, $\phi \colon G \to \IR$ and
$A \in M_{r,r}(\IZ G)$ do there exist constants $T \ge 1$ and $C_1, C_2\in \IR$ such that
\[
\begin{array}{lclcl}
\det_{\caln(G)}\bigl(\Lambda \circ \eta_{\phi^* \IC_t}(r_A)\bigr)
& = & 
C_1 \cdot \ln(t) 
& \text{for} 
& t \le T^{-1};
\\
\det_{\caln(G)}\bigl(\Lambda \circ \eta_{\phi^* \IC_t}(r_A)\bigr)
 & = & 
C_2 \cdot \ln(t)  
& \text{for} 
& t \ge T.
\end{array}
\]
\end{problem}

The answer to Problems~\ref{prob:Continuity,Convexity_and_Monotonicity}
and~\ref{prob:Locally_constant_at_0_and_infty} are positive for mapping tori,
see Lemma~\ref{the:mapping_tori} and  in the presence of appropriate $S^1$-actions, 
see Theorem~\ref{the:S1-actions}, for appropriate $\phi \colon G \to \IZ$.

A weaker version of Problem~\ref{prob:Locally_constant_at_0_and_infty} is 

\begin{problem}[Convergence  at $0$ and $\infty$]
\label{prob:Convergence_at_0_and_infty}
Under which conditions on $G$, $\phi \colon G \to \IR$ and
$A \in M_{r,r}(\IZ G)$ do the limits
\begin{eqnarray*}
& \lim_{t \to 0+} \frac{\ln\bigl(\det_{\caln(G)}\bigl(\Lambda \circ \eta_{\phi^* \IC_t}(r_A)\bigr)\bigr)}{\ln(t)}; &
\\
& \lim_{t \to \infty} \frac{\ln\bigl(\det_{\caln(G)}\bigl(\Lambda \circ \eta_{\phi^* \IC_t}(r_A)\bigr)\bigr)}{\ln(t)}, &
\end{eqnarray*}
exist as real numbers.
\end{problem}

A consequence of a positive answer to Problem~\ref{prob:Convergence_at_0_and_infty}
would be that in the definition of the degree of the $L^2$-torsion function of
Definition~\ref{def:Degree_of_an_equivalence_class_of_functions_IR_greater_0_to_IR} one
can replace $\liminf$ and $\limsup$ by $\lim$. This has the advantage that the various sum
and product formulas for the $L^2$-torsion function imply analogous formulas for its
degree. For the $\phi$-twisted $L^2$-torsion function for the universal covering of an 
irreducible oriented compact $3$-manifold
with empty or incompressible torus boundary  a positive answer for  Problem~\ref{prob:Convergence_at_0_and_infty}
is proved by Liu~\cite[Theorem~1.2]{Liu(2015)}.

\begin{problem}[Continuity of the degree]
\label{prob:Continuity_of_the_degree}
Under which conditions on $G$, $\phi \colon G \to \IR$ and
$A \in M_{r,r}(\IZ G)$ does the degree depends continuously on $\phi$?
\end{problem}

A positive answer for  Problem~\ref{prob:Continuity_of_the_degree} follows directly from Liu~\cite[Theorem~6.1]{Liu(2015)},
 if $G$ is a finitely generated  residually finite group and $\Lambda(r_A) \colon L^2(G)^r \to L^2(G)^r$ is a weak isomorphism,
since $G$ satisfies the Determinant Conjecture,  see~\cite[Conjecture~13.2 on page~454 and Theorem~13.3 on page~454]{Lueck(2002)}. 

%%%%%%%%%%%%%%%%%%%%%%%%%%%%%%%%%%%%%%%%%%%%%%%%%%%%%%%%%%%%%%%%%%%%%%%%%%%%%%%%%

\subsection{Analytic versions}
\label{subsec:analytic_versions}

There is an analytic version of $L^2$-torsion which can be identified with the combinatorial version,
see for instance~\cite[Section~3.5]{Lueck(2002)}, where also references to the relevant literature can be found.

\begin{problem}[Analytic version]
\label{prob:Analytic_version}
Develop an analytic version of the $V$-twisted $L^2$-torsion of 
Definition~\ref{def:L2-torsion_twisted_by_a_based_finite-dimensional_representation},
and identify both.
\end{problem}

%%%%%%%%%%%%%%%%%%%%%%%%%%%%%%%%%%%%%%%%%%%%%%%%%%%%%%%%%%%%%%%%%%%%%%%%%%%%%%%%%

\subsection{$L^2$-torsion function for arbitrary representations}
\label{subsec:L2-torsion_function_for_arbitrary_representations}

It would be very interesting if one could prove
Theorem~\ref{the:Determinant_class_and_twisting} also in the case, where $\phi \colon G
\to \IZ^d$ can be replaced by $\id \colon G \to G$.  Suppose for the remainder 
of this subsection that this is possible.

 Consider an object $C_*$  in $\FBCC{\IC G}$ such that $\Lambda(C_*)$ is $L^2$-acyclic.
An example is  $C_*(X)$ for a $\det$-$L^2$acyclic finite free $G$-$CW$-complex $X$.
Then one  can twist it with any finite-dimensional $G$-representation $V$. This leads to a function
\begin{equation}
\rho[n]_{C_*} \colon \Rep_n(G) \to \IR; \quad [V] \mapsto \rho^{(2)}\bigl(\Lambda^G \circ \eta_V(C_*);\caln(G)\bigr)
\label{rho[n]_C_ast}
\end{equation}
on the space of $n$-dimensional $G$-representations, i.e., on the space $\Rep_n(G)
:= \hom(G,GL_n(\IC))/GL_n(\IC)$, where $\hom(G,GL_n(\IC))$ is the space of all group
homomorphisms $G \to GL_n(\IC)$ and $A \in GL_n(C)$ acts on $\hom(G,GL_n(\IC))$ by
composition with the conjugation automorphism 
$c_A \colon GL_n(\IC) \to GL_n(\IC), \; B \mapsto A^{-1}BA$. The function $\rho[n]_{C_*} $ is well-defined
because of Lemma~\ref{lem:dependency_on_V}~\eqref{lem:dependency_on_V:formula_for_T_u}
since $\chi_{\IC G}(C_*) = \chi^{(2)}(\Lambda(C_*)) = 0$. 
It is very likely that this is an interesting
function, in particular if $C_*$ comes from  the universal covering of a compact connected orientable irreducible
$3$-manifold whose fundamental group is infinite and whose boundary is empty or toroidal.

\begin{problem}[Continuity of $\rho{[n]}_{C_*}$] \label{pro:Continuity_of_rho[n]_C_ast}
Under which conditions on $G$ is the map $\rho[n]_{C_*}$ of~\eqref{rho[n]_C_ast}
continuous.
\end{problem}

Let $\Rep_{\IC}(G)$ be the representation ring of finite-dimensional complex
$G$-repre\-sen\-ta\-tions.  This is the abelian group whose generators are isomorphism
classes $[V]$ of finite-dimensional $G$-representations and whose relations are given by
$[V_0]- [V_1] + [V_2]$ for any exact $\IC G$-sequence $0 \to V_0 \to V_1 \to V_2 \to 0$ of
finite-dimensional $G$-representations. This becomes a ring by the tensor product over
$\IC$ with the diagonal $G$-action, namely define by $[V] \cdot [W] := [(V \otimes W)_d]$.

We conclude from  Lemma~\ref{lem:additivity_under_exact_sequences_of_reps}
that we obtain from the homomorphisms $\rho[n]_{C_*}$ of \eqref{rho[n]_C_ast} a homomorphism of abelian groups
\[
\rho_{C_*} \colon \Rep_{\IC}(G) \to \IR,
\quad [V] \mapsto \rho^{(2)}\bigl(\Lambda^G \circ \eta_{V,[B_V]}(C_*);\caln(G)\bigr),
\]
where we choose for  $V$ any equivalence class of $\IZ G$-basis $[B_V]$ and this choice turns out not to matter.

\typeout{-------------------------------------- References  ---------------------r------------------}

%\bibliographystyle{abbrv}
%\bibliography{dbpub,dbpre}

%\version{10.03.2017}

\end{document}